\documentclass[a4paper,10pt]{report}
\usepackage{mathtext}
\usepackage[T1,T2A]{fontenc}
\usepackage[cp1251]{inputenc}
\usepackage[english]{babel}
\usepackage{amsmath}
\usepackage{amsfonts}
\usepackage{amssymb}
\usepackage{mathrsfs}
\usepackage{amsthm}
\usepackage{titling}
\usepackage{arcs}
\usepackage{graphicx}
\usepackage{wrapfig}

\usepackage{makeidx}

\usepackage{color}
\usepackage{euscript}

\usepackage{bbding}

\DeclareMathOperator{\sign}{sign}
\DeclareMathOperator{\supp}{supp}
\DeclareMathOperator{\interior}{int}
\DeclareMathOperator{\var}{var}
\DeclareMathOperator{\dist}{dist}
\DeclareMathOperator{\BV}{\mathrm{BV}}
\newcommand{\df}{\buildrel\mathrm{def}\over=}
\newcommand{\Bell}{\boldsymbol{B}_{\eps}}
\newcommand{\Bellb}{\boldsymbol{B}_{\eps}^{\mathrm{b}}}
\newcommand{\Belld}{\boldsymbol{B}_{\eps}^{\mathrm{dyad}}}
\newcommand{\BMO}{\mathrm{BMO}}
\newcommand{\eps}{\varepsilon}

\newcommand{\Cii}[1]{_{{}_{\scriptstyle #1}}}
\newcommand{\cii}[1]{_{{}_{#1}}}
\newcommand{\BNorm}[1]{\|#1\|_{{}_{\BMO(I)}}}

\newcommand{\av}[2]{\langle {#1}\rangle_{{}_{#2}}}
\newcommand{\Ch}{\Omega_{\mathrm{ch}}}
\newcommand{\Bch}{B^{\mathrm{ch}}}

\newcommand{\Rt}{\Omega_{\mathrm{R}}}
\newcommand{\Lt}{\Omega_{\mathrm{L}}}

\newcommand{\Ang}{\Omega_{\mathrm{ang}}}

\newcommand{\RTroll}{\Omega_{\mathrm{tr},\mathrm{R}}}
\newcommand{\LTroll}{\Omega_{\mathrm{tr},\mathrm{L}}}

\newcommand{\Srt}{D_{\mathrm{R}}}
\newcommand{\Slt}{D_{\mathrm{L}}}

\newcommand{\mrt}{m_{\mathrm{R}}}
\newcommand{\mlt}{m_{\mathrm{L}}}
\newcommand{\bp}[1]{\mathfrak{b}_{#1}}
\newcommand{\hreff}[1]{\hyperref[#1]{\ref{#1}}}
\newcommand{\Bird}{\Omega_{\mathrm{bird}}}

\newcommand{\Fr}{F_{\mathrm{R}}}
\newcommand{\Fl}{F_{\mathrm{L}}}
\newcommand{\FFr}{\mathfrak{F}_{\mathrm{R}}}
\newcommand{\FFl}{\mathfrak{F}_{\mathrm{L}}}
\renewcommand{\le}{\leqslant}
\renewcommand{\ge}{\geqslant}
\renewcommand{\leq}{\leqslant}
\renewcommand{\geq}{\geqslant}
\newcommand{\mr}[2]{m''_{\mathrm{R}}(#1;\,#2)}
\newcommand{\ml}[2]{m''_{\mathrm{L}}(#1;\,#2)}
\newcommand{\BG}{\mathfrak{B}}
\newcommand{\BM}{\EuScript{B}}

\newcommand{\MTTR}{\Omega_{\mathrm{Mtr,R}}}
\newcommand{\MTTL}{\Omega_{\mathrm{Mtr,L}}}
\newcommand{\MTC}{\Omega_{\mathrm{Mcup}}}
\newcommand{\MTB}{\Omega_{\mathrm{Mbird}}}
\newcommand{\ClMTC}{\Omega_{\mathrm{ClMcup}}}
\newcommand{\ato}{a^{\mathrm{top}}}
\newcommand{\bto}{b^{\mathrm{top}}}
\newcommand{\abot}{a^{\mathrm{bot}}}
\newcommand{\bbot}{b^{\mathrm{bot}}}
\newcommand{\Ato}{A^{\mathrm{top}}}
\newcommand{\Bto}{B^{\mathrm{top}}}

\newcommand{\tr}{t^{\mathrm{R}}}
\newcommand{\tl}{t^{\mathrm{L}}}
\newcommand{\ur}{u_{\mathrm{R}}}
\newcommand{\ul}{u_{\mathrm{L}}}
\newcommand{\Ur}{U_{\mathrm{R}}}
\newcommand{\Ul}{U_{\mathrm{L}}}
\newcommand{\FixedBoundary}{\partial_{\mathrm{fixed}}}
\newcommand{\FreeBoundary}{\partial_{\mathrm{free}}}
\newcommand{\GammaFixed}{\Gamma^{\mathrm{fixed}}}
\newcommand{\GammaFree}{\Gamma^{\mathrm{free}}}
\newcommand{\uur}{u_{\mathrm{r}}}
\newcommand{\uul}{u_{\mathrm{l}}}
\newcommand{\evo}{\EuScript{EVO}}

\newcommand{\E}{\mathbb{E}}
\newcommand{\G}{\EuScript{G}}
\newcommand{\GFree}{\EuScript{G}^{\mathrm{free}}}

\theoremstyle{plain}
\theoremstyle{plain}\newtheorem{Le}{Lemma}[section]
\theoremstyle{definition}\newtheorem{Def}[Le]{Definition}
\theoremstyle{plain}\newtheorem{St}[Le]{Proposition}
\theoremstyle{plain}\newtheorem{Conj}[Le]{Conjecture}
\theoremstyle{plain}\newtheorem{Th}[Le]{Theorem}
\theoremstyle{plain}\newtheorem{Cor}[Le]{Corollary}
\theoremstyle{plain}\newtheorem{Rem}[Le]{Remark}
\theoremstyle{plain}\newtheorem{Fact}[Le]{Fact}
\theoremstyle{plain}\newtheorem{Cond}[Le]{Condition}
\numberwithin{equation}{section}

\renewcommand{\chaptermark}[1]{}

\makeindex
\author{Paata~Ivanisvili\and Dmitriy~M.~Stolyarov\and Vasily~I.~Vasyunin\and Pavel~B.~Zatitskiy}
\title{Bellman function for extremal problems in $\BMO$ II: evolution}
\date{}

\begin{document}
\maketitle
\begin{abstract}
In \cite{5A} the authors built the Bellman function for integral functionals on the $\BMO$ space. The present paper provides a development of the subject. We abandon the majority of unwanted restrictions on the function 
	that generates the
	functional. It is the new \emph{evolutional} approach that allows us to treat the problem in its natural setting. What is more, these new considerations lighten dynamical aspects of the Bellman function, in particular, 
	evolution of its picture. 
\end{abstract}
\newpage\pagestyle{empty}

{\bf The research is supported by RSF grant 14-41-00010}

\bigskip

We are grateful to Dmitriy~Chelkak, Pavel~Galashin, Alexander~Logunov, Nikolay~Osipov, Fedor~Petrov, Leonid~Slavin, Alexander~Volberg, and Vladimir~Zolotov for sharing their ideas with us.

\bigskip

The second named author would like to thank Institute for Mathematics of Polish Academy of Sciences: a big part of the present text was written while he was enjoying its hospitality.

\bigskip

\bigskip

\bigskip

\bigskip

\bigskip

\bigskip

\bigskip

\bigskip

\bigskip

\bigskip

\bigskip

\bigskip

\bigskip

\bigskip

\bigskip

\bigskip

\bigskip

\bigskip

\bigskip

\bigskip

\bigskip

\bigskip

\bigskip

\bigskip

\bigskip

\bigskip

\bigskip

Paata Ivanisvili

Kent State University, Kent, OH, 44243, USA

ivanishvili.paata@gmail.com

\bigskip

Dmitriy M. Stolyarov

%Institute of Mathematics, Polish Academy of Sciences, ul. Sniadeckich 8, Warsaw, 00-656, Poland

Chebyshev Laboratory, St. Petersburg State University, 14th Line, 29b, St. Petersburg, 199178, Russia

St. Petersburg Department of Steklov Mathematical Institute, Russian Academy of Sciences, 27 Fontanka, St. Petersburg, 191023, Russia

dms@pdmi.ras.ru

http://www.chebyshev.spb.ru/DmitriyStolyarov

\bigskip

Vasily I. Vasyunin

St. Petersburg Department of Steklov Mathematical Institute, Russian Academy of Sciences, 27 Fontanka, St. Petersburg, 191023, Russia 

vasyunin@pdmi.ras.ru

\bigskip

Pavel B. Zatitskiy

Chebyshev Laboratory, St. Petersburg State University, 14th Line, 29b, St. Petersburg, 199178, Russia

St. Petersburg Department of Steklov Mathematical Institute, Russian Academy of Sciences, 27 Fontanka, St. Petersburg, 191023, Russia 

paxa239@yandex.ru

http://www.chebyshev.spb.ru/pavelzatitskiy

\newpage
\tableofcontents

\newpage
\pagestyle{headings}

\chapter{Introduction}
\section{Historical remarks}\label{s11}

\subsection{$\BMO$ and John--Nirenberg inequality}\label{s111}

	We begin with clarifying our notation. The symbols $I$ and $J$ always denote intervals of the real line. The symbol~$\av{\varphi}{J}$ stands for the average of a summable function $\varphi$ over~$J$:
	\begin{equation*}
	\av{\varphi}{J} \df \frac{1}{|J|}\int\limits_{J}\varphi(t)\,dt,
  \end{equation*}
	where $|J|$ is the length of $J$. Consider a summable function $\varphi\colon I \mapsto \mathbb{R}$ for some interval $I$. This function belongs to $\BMO$\index{BMO} provided
	\begin{equation*}
	\sup_{J\subset I}\av{|\varphi-\av{\varphi}{J}|^2}{J} < \infty,
	\end{equation*}
	where the supremum is taken over all the subintervals of $I$. As usual, we equip this space with a seminorm,
	\begin{equation}\label{BMOnorm2}
	\BNorm{\varphi} \!\!\df\; \Big(\sup_{J\subset I}\av{|\varphi-\av{\varphi}{J}|^2}{J}\Big)^{\frac{1}{2}}.
	\end{equation}
	To get a Banach space, one has to factorize over the set where the seminorm vanishes. In the case of~$\BMO$, we factorize over the one-dimensional space of constant functions.
	However, this operation is inconvenient for our purposes, so we leave $\BMO$ to be a space with a seminorm. What is more, we will call it a norm. The $\BMO$ space plays a crucial role in harmonic analysis and applications, 
	so we refer the reader to the books~\cite{Koosis} and~\cite{Stein} for more information about it.
	
	 We have defined the $\BMO$ space with the help of the quadratic seminorm. One gets the same space (equivalent, but surely not isometric) if both numbers $2$ in the definition above are changed for a number~$p$,~$0 < p < \infty$. 
	 This equivalence can be expressed in terms of inequalities:
	 \begin{equation}\label{bmoequivnorms}
		c_p\BNorm{\varphi} \le\; \sup_{J\subset I}\av{|\varphi-\av{\varphi}{J}|^p}{J}^{1/p}\; \le\; C_p\BNorm{\varphi},
	 \end{equation}
	 which hold for all $p$, $0 < p < \infty$ (the same~$\BMO$-space can be defined using even weaker integral restrictions, see~\cite{LSSVZ}).
	 These inequalities follow immediately from the John--Nirenberg inequality, which forces the distribution function of a $\BMO$~function to decrease exponentially:
	 \begin{equation}\label{JN}
		\frac{1}{|I|}\Big|\big\{t\in I\mid\; |\varphi(t)-\av{\varphi}{I}| \geq \lambda\big\}\Big|\le 
		c_1 e^{-c_2\lambda/\BNorm{\varphi}}.
	\end{equation}
	The John--Nirenberg inequality can, in its turn, be reformulated in terms of integrals and becomes the integral form of the John--Nirenberg inequality: there exists $\eps_{\infty} > 0$ and a positive constant~$C = C(\eps)$, 
	${0<\eps<\eps_{\infty}}$, such that
	\begin{equation}\label{intJN}
		\av{e^{\varphi}}{I} \le C(\eps) e^{\av{\varphi}{I}}
	\end{equation}
	for all functions $\varphi \in \BMO_{\eps}(I)$. The last symbol stands for the closed ball of radius~$\eps$ in~$\BMO$. One can think of this inequality as of a reverse Jensen inequality for functions of bounded mean oscillation and the exponential function, which is convex. 
	
	There are various proofs of these three inequalities. We mention the martingale proof (see \cite{Koosis}) and the proof that implicitly uses the duality between $H^1$ and $\BMO$ (see~\cite{Stein}). The method of the Bellman function,
	which we survey in the forthcoming section, treats these three inequalities as different manifestations of a general principle. What do they have in common? The idea is that all three inequalities are estimates for averages
	of expressions $f(\varphi)$, $\varphi \in \BMO$, i.e. estimates of integral functionals
	\begin{equation}\label{functional}
	f[\varphi] \df \av{f(\varphi)}{I}.
	\end{equation}   
	The function $f\colon \mathbb{R} \mapsto \mathbb{R}$ that generates the functional should satisfy some conditions. First of all, this function must be defined everywhere, though we do not need it to be continuous. Second, it should not 
	be very big: the function $f(\varphi)$ must be summable. More detailed requirements are discussed in Subsection~\ref{s212}. Integral form of the John--Nirenberg inequality~\eqref{intJN} refers to the case of $f(t) = e^t$, 
	the original John--Nirenberg inequality~\eqref{JN} refers to the case of $f(t) = \chi\cii{(-\infty,-\lambda]\cup[\lambda,\infty)}(t)$, and the first of three inequalities,~\eqref{bmoequivnorms}, refers to the case of $f(t) = |t|^p$.
	In all three inequalities, the functional is estimated on the hyperspace of $\BMO$, defined by equation $\av{\varphi}{I} = 0$. However, it occurs to be more convenient to write an estimate in the whole space. So, we want to estimate
	the functional given by formula~\eqref{functional} not only in terms of~$\BNorm{\varphi}$, but also in terms of $\av{\varphi}{I}$. And here the Bellman function\index{Bellman! function} arrives:
	\begin{equation}\label{Bf}
		\Bell(x_1,x_2;\,f) \;\df 
		\sup\big\{f[\varphi] \mid\; \av{\varphi}{I} = x_1,\av{\varphi^2}{I} = x_2,\;\varphi \in \BMO_{\eps}(I)\big\},
	\end{equation} 
	where $x_1$ and $x_2$ are real numbers. We omit those pairs of $x_1$ and $x_2$ for which the supremum is taken over the empty set. More detailed discussion of rigorous definition and basic properties of the Bellman 
	function is located in Subsection~\ref{s211}. The purpose of this paper is to find the Bellman function for a vast class of functions $f$. If one finds the Bellman function, he achieves sharp estimates of the functional given by formula~\eqref{functional} and thus gets the sharp constants in inequalities of types~\eqref{bmoequivnorms},~\eqref{JN}, and~\eqref{intJN}. 
	
\subsection{Bellman function in analysis}

	It was Burkholder who began to apply the ideas and machinery of the Bellman function to different probability problems, see his pioneering paper~\cite{Burk}. The ideas of that paper helped to prove a great number of sharp inequalities for martingales, see the book~\cite{Os} or the survey~\cite{Osekowski3}. However, the method of the Burkholder's followers (usually called \emph{the Burkholder method}) is a bit different, though highly related to, than that of the Bellman function. We also mention the paper~\cite{Hanner} of Hanner (a reconstruction of Beurling's report at a seminar in Uppsala in 1945), where the Bellman function type technique was used to calculate the moduli of convexity of the $L^p$~spaces. Surely, one cannot find the term ``the Bellman function'' there, however, a special minimal concave function solves the problem (see the paper~\cite{ISZ} for the explanation where the Bellman function is hidden).  The papers~\cite{NaTrVol} (which appeared as a preprint as early as 1995) and~\cite{NaTr} gave the Bellman function its name and were the first to apply it systematically to various problems in harmonic analysis (see also~\cite{NTV}). The sharp constants in the integral form of the John--Nirenberg inequality~\eqref{JN} and the 
	corresponding Bellman function were found in~\cite{Slavin} and~\cite{Va2} (finally, see~\cite{SlVa}). The case of the norm equivalence inequality~\eqref{bmoequivnorms} was successfully considered in~\cite{SlVa2}. And the original John--Nirenberg 
	inequality~\eqref{JN} can be found in~\cite{Va,VV2}. We should also mention that the $\BMO$ space can be replaced with the famous Muckenhoupt classes $A_p$. The reverse Holder inequality can be treated in the same way, this was 
	done in~\cite{Va3}. For the corresponding ``tail''-estimate on the Muckenhoupt classes see~\cite{Rez}. One can also apply very similar technique to estimate the norm of the maximal operator, see~\cite{Melas,SSV}, and the Carleson embedding operator, see~\cite{VaVo,Slavin3}. We refer the reader to the lecture notes~\cite{Volberg1,Volberg} for a more detailed scenery of the subject.
	
	The main idea of the Bellman function theory is that the Bellman function satisfies some differential equations that originate from self-similarity of the extremal problem. We explain this heuristic in our case. First, it should be mentioned	that the Bellman function~\eqref{Bf} does not depend on the interval $I$. Second, from the additivity of integral with respect to an interval we can easily derive that $\Bell$ is
	locally concave on its domain. And the main idea is that it is the minimal locally concave function that satisfies special boundary conditions. This part of the work is done in Section~\ref{s22}. Therefore, we arrive at some purely geometric 
	three-dimensional problem of finding the minimal locally concave surface. Such surfaces always satisfy the homogeneous Monge--Amp\`ere equation. The connection between some class of Bellman function problems and the Monge--Amp\`ere equation was first noticed by L.~Slavin (see~\cite{SSV,VaVo}). The only thing we have to do is to find the minimal locally concave solution of this equation and then prove that the solution is the desired Bellman function. 
	
	The authors' previous work~\cite{5A} was an attempt to build a general theory of Bellman functions for arbitrary~$f$  (see the short report~\cite{CR} and the older version~\cite{5AOld, Addendum} as well). The problem was solved only partly, because we did not even tackle the cases of non-smooth functions~$f$, 	in that 
	paper the function~$f$ had to be almost three times continuously differentiable. But much worse thing was that the roots of its third derivative had to be well separated in a sense. What is more, these roots had to be thin, 
	i.e. we did not allow the function~$f'''$ to vanish on an interval (for example, our theory did not cover the case of such a  simple function~$f$:~$f(t) = 0, t < 0; f(t) = t^3, t \geqslant 0$). In this work we overcome the difficulties
	that appear in the case of an arbitrary position of those roots, though we still assume that there is only a finite number of them. We also remove the technical requirement on the thinness of the roots. We also present a dynamical approach to the problem.
	
	All ideas of the present paper were clear to the authors during the preparation of~\cite{5A}. However, it took us a lot of time to produce an adequate language for the presentation and write down the results. Since then, some heuristics has been justified rigorously, and some results have been generalized. However, we will follow a more ``old-fashioned'' way of presentation by two reasons: it is more transparent for a non-specialist and allows this text to be self-contained and relatively short. We describe the recent (after 2012) development in the last chapter (Subsection~\ref{s621}).
	
\section{Structure of the paper}
	%{\color{red}
	In this section, we survey the structure of the paper. Here we omit any definitions, reasonings, etc. Those readers who are unfamiliar with the Bellman function machinery can read the second chapter first and then return to
	this section to get the general scenery of the text.
	
	We begin with the rigorous setting of the problem. It is situated in the first half of the second chapter, Section~\ref{s21}. We list simple properties of the Bellman function, those that follow from its 
	definition immediately. We also establish easy properties of the function class~$f$ belongs to. In Section~\ref{s22}, we discuss more sophisticated properties of the Bellman function, its domain 
	and special functions, the \emph{optimizers}\index{optimizer} (informally, these functions $\varphi$ are the ones for which $f[\varphi]$ attains its maximum, i.e. the value of the Bellman function). This chapter ends with a detailed plan of the proof, now based 
	on rigorous definitions. We also list our results in the last subsection. To be honest, the majority of the material of the second chapter can be found in~\cite{5A}. We quote
	it here for the sake of completeness and clearness.
	
	The third chapter describes all the local types of foliation we will need. Since there are a plenty of them (from six to nine depending on the terminology), we begin with some general remarks on the classification of these types in Section~\ref{s31}. Some of these local types (called figures), have already been studied in~\cite{5A} and earlier papers. However, we repeat this study for the sake of completeness. Section~\ref{s32} is devoted to tangent domains, Section~\ref{s33} treats chordal domains (some examples are also included there). The linearity domains are described in Section~\ref{s34}. There are new ones, mostly they are the linearity domains with three or more points on the lower boundary. Finally, in Section~\ref{s35} we describe the general combinatorial properties of foliations. We explain how do neighbor figures interact with each other, in some cases two neighbor figures may be treated as a single one. We also introduce a special graph to describe the combinatorial properties of foliations (this notion has already been used in~\cite{CrazySine}). We end this section with examples.
	
	It is time to mention the role of examples. Not only do they illustrate the dynamics but also verify
the existence of the abstract objects we have built. What is more, they show that our abstract treatment
of the subject provides an approach to ``lively'' problems of calculating Bellman functions for particular
cases of~$f$ (for a superior example, which contains almost all constructions of the present paper, see~\cite{CrazySine}).
	
	The evolutional approach works as follows: first we build the Bellman function for very small~$\eps$ and then begin to increase this parameter. In the case of small~$\eps$, the picture of the Bellman function is relatively simple 
	(and the problem corresponds to the case of well-separated roots that we have already considered in~\cite{5A}, however, there are some difficulties concerning thick roots that are not covered by~\cite{5A}). But it begins to be more complicated for bigger~$\eps$, the figures begin to mix with each other. We have to monitor 
	this process.
	
	In the fourth chapter, we build the Bellman candidate for arbitrary~$f$ and all~$\eps$. As it was said, we first construct it for sufficiently small~$\eps$. This is done in Section~\ref{s41}, here the foliation of the Bellman candidate is simple. To do this, we need to study some monotonicity properties of forces (some of which have already been established in~\cite{5A}). We also provide some examples of~$f$ and~$\eps$ for which the foliation of the Bellman function is simple. Section~\ref{s42} contains auxiliary lemmas that will be needed for the evolution. These are mostly monotonicity properties for forces, tails, and roots of balance equations. 
	
	The Bellman candidate ``consists of pieces''. So, to study its evolution, we should study the local (in time) evolution of each of these pieces. This is done in Section~\ref{s43}. It appears that these local evolutional scenarios obey some rules that follow from the monotonicity lemmas of the previous section. 
	
	Finally, in Section~\ref{s44} we prove that there is a flow~$B(\eps)$ of Bellman candidates that obey the evolutional rules. This is done in the following way: we grow~$\eps$, then all the figures follow their evolutional scenarios. The problem can arise when some two (or more) figures crash (or some figure disappears). In such a case, we use special formulas from Section~\ref{s35} to show that after the crash the two figures may be treated as a single one. So, we can change the structure of the Bellman candidate (i.e. change its graph), and continue the evolution further. The moment~$\eps$ when the crash happens is called the critical point of the evolution. It is not clear whether there should be a finite number of such crashes during the whole evolution. So, we treat only some crashes (roughly speaking, those in which not only angles take part) as critical points and prove that there is only a finite number of them. And all the other ``crashes'' (with angles) are, in a sense, not so serious, and the evolution can get through them (this also needs some justifications). All in all, we are able to build the Bellman candidate for all~$\eps$. We finish the chapter with examples in Section~\ref{s45}.
			
			By the word ``build the Bellman function'' we mean that we provide an algorithm that allows to calculate the Bellman function. In our paper,
the word ``algorithm'' is used in an informal way. By an algorithm we mean the way to build the Bellman
function in a finite number of steps, where a step can be either an integration or a differentiation of a
function. It can also consist of solving some equations, but we prove that there are only a finite number of
solutions. 
The algorithm of the present paper differs dramatically from the one we had in~\cite{5A}.
	
	Chapter five contains the theory of optimizers. Section~\ref{s51} suggests general principles of constructing the optimizers, as well as the core convexity lemma. As with the Bellman candidates, we first study the local behavior of optimizers, and then glue them together. Section~\ref{s52} provides a detailed study of the optimizers for each figure. In Section~\ref{s53}, we glue them together. More or less, this is done by a simple induction over the graph of the Bellman candidate. Since we have built the Bellman candidate and constructed the optimizers for it, we have built the Bellman function for all~$\eps$.
	
	The final chapter is for supplementary material. Secion~\ref{s61} contains answers to two questions on the regularity of the Bellman function. In a sense, we show that the difficulty of some our constructions comes not from our method, but from the problem itself. Section~\ref{s62} includes some remarks on the recent development of the field, conjectures, and suggestions for further study. %} 
	
	\chapter{Setting and sketch of proof}

\section{Setting}\label{s21}

\subsection{Extremal problem and the Bellman function}\label{s211}

In this subsection, we summarize some easy properties of the Bellman function defined by formula~\eqref{Bf} and try to explain the choice of the Bellman function in view of the extremal problem we study. Expressions $\Bell(x_1,x_2)$, $\Bell(x ;\,f)$, or even $\Bell(x)$, where $x = (x_1,x_2)$, stand for the Bellman function. 
Now, we formulate the easiest properties of the Bellman function  that do not need any conditions on $f$.%\index{Bellman function}
\begin{Rem}\label{Rescaling}
The Bellman function $\Bell$ does not depend on the interval~$I$ \textup{(}the interval the $\BMO$ space is defined on\textup{)}.
\end{Rem}

Indeed, using a linear change of variables one can transform a function $\varphi \in \BMO(I)$ into 
another function $\tilde{\varphi} \in \BMO(\tilde{I})$ so that all the values of the integral functionals defined by formula~\eqref{functional} do not change. Since
$$\BNorm{\varphi}^2 = \sup_{J \subset I}\big(\av{\varphi^2}{J} - \av{\varphi}{J}^2\big),$$ 
it does not change as well. 
Thus, the supremum defined by formula~\eqref{Bf} is taken over the same subset of the real numbers.

The next remark allows us to estimate integral functionals $f[\varphi]$ from below. One can consider another Bellman function,
$$
\Bell^{\min}(x_1,x_2;\,f) = 
\inf\big\{f[\varphi] \mid\; \av{\varphi}{I} = x_1,\av{\varphi^2}{I} = x_2,\;\varphi \in \BMO_{\eps}(I)\big\}.
$$
Of course, the minimal Bellman function can be easily expressed in terms of the maximal one.
\begin{Rem}\label{minBell}
$\Bell^{\min}(x_1,x_2;\,f) = -\Bell(x_1,x_2;\,-f)$.
\end{Rem}

We also easily calculate the change of the Bellman function if we add a quadratic polynomial to $f$.
\begin{Rem}\label{quadratische}
The following equation holds for all reals $a,b,c,d$\textup:
 \begin{equation*}
\Bell\big(x_{1},x_{2};t \mapsto af(t)+bt^{2}+ct+d\big) = |a|\Bell\big(x_{1},x_{2};(\sign{a})f\big)+bx_{2}+cx_{1}+d.
\end{equation*}
\end{Rem}
We can also modify the boundary condition with the help of a linear change of variables.

\begin{Rem}\label{remm2}
	For any real numbers $\alpha$ and $\beta$\textup, we have
	$$
	\Bell\big(x_{1},x_{2};\,t \mapsto f(\alpha t +\beta)\big)=\boldsymbol{B}_{|\alpha|\eps}\big(\alpha x_{1}+\beta, \alpha^{2}x_{2}+2\alpha \beta   x_{1}+\beta^{2}; \, f\big).
	$$
	\end{Rem}

Now we begin to study the domain of $\Bell$. The following definition seems to be useful for all further reasoning.

\begin{Def}\label{BellmanPoint}\index{Bellman! point}
Consider a function $\varphi \in L^2(I)$. We call the point $\bp{\varphi} \in \mathbb{R}^2$, 
$$\bp{\varphi} \df (\av{\varphi}{I}, \av{\varphi^2}{I}),$$
the Bellman point of $\varphi$.
\end{Def} 

The Bellman function~\eqref{Bf} is defined for all pairs $(x_1,x_2) \in \mathbb{R}^2$. However, this function is uninteresting for some pairs of reals because the supremum is taken over the empty set. We drop such points
from the domain of the Bellman function. 

\begin{St}\label{parstr}
The parabolic strip
\begin{equation}\label{ParStr}
		\Omega_\eps \df \big\{(x_1, x_2) \in \mathbb{R}^2\mid\; x_1^2 \le x_2 \le x_1^2+\eps^2\big\}
\end{equation}
is the domain of $\Bell$\textup, i.e. $\Omega_{\eps}$ is the set of points $x$ such that there exists a function $\varphi \in \BMO_{\eps}$ with $\bp{\varphi} = x$.
\end{St} 

\begin{figure}[h]
\begin{center}
\includegraphics{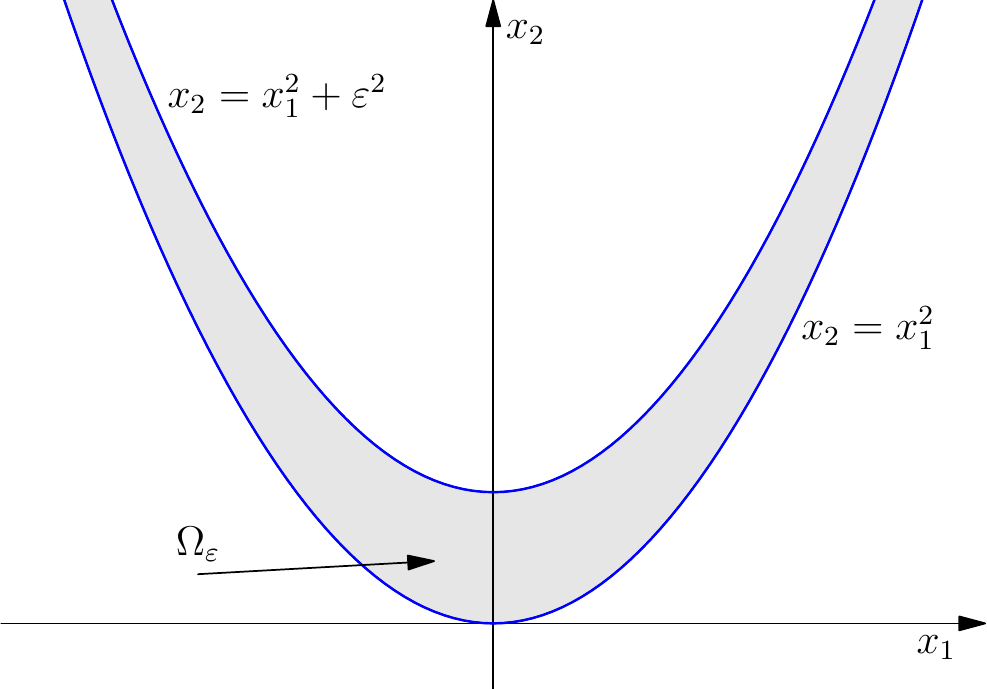}
\caption{Parabolic strip $\Omega_{\varepsilon}$.}
\label{fig:pp}
\end{center}
\end{figure}

\begin{proof}
First, we verify that $\bp{\varphi} \in \Omega_{\eps}$ for all the functions $\varphi \in \BMO_{\eps}$. It suffices to prove two inequalities: $\av{\varphi^2}{I} - \av{\varphi}{I}^2 \geqslant 0$ and $\av{\varphi^2}{I} - \av{\varphi}{I}^2 \leq \eps^2$. The first inequality follows from the Cauchy--Schwarz inequality and the second one is a straightforward consequence of the condition $\varphi \in \BMO_{\eps}$. So, the parabolic strip contains the domain of $\Bell$. Second,  we find functions~$\varphi$ such that~$\bp{\varphi}$ is an arbitrary point in~$\Omega_{\eps}$. For each point~$x = (x_1, x_2) \in \Omega_{\eps}$ the function 
\begin{equation*}			\varphi(t)=
				\begin{cases}
					 x_1+\sqrt{x_2 - x_1^2},& t\in [0,\frac{1}{2});\\
					 x_1-\sqrt{x_2 - x_1^2},& t\in [\frac{1}{2},1]
				\end{cases} 
\end{equation*}
belongs to~$\BMO_{\eps}$ (this is an easy computation). Consequently,~$\Omega_{\eps}$ is contained 
in the domain of~$\Bell$.
\end{proof}
This proposition says nothing whether the Bellman function is finite on~$\Omega_\eps$, it only cuts off those points of the plane where it is \emph{a priori} infinite.

Though this is the aim of this work and seemingly a difficult problem to calculate the Bellman function, its values at some points are already known.

\begin{St}
$\Bell(x_1, x_1^2; f) = f(x_1)$ for all $x_1 \in \mathbb{R}$.
\end{St}

\begin{proof}
The function whose Bellman point lies on the parabola $\{x\in\mathbb{R}^2\mid x_2 = x_1^2\}$ is constant, because the Cauchy--Schwarz inequality turns into equality only for these functions. And we know this constant, because we know the average of the function (see the proof of Proposition~\ref{parstr}). Consequently, the set we are taking the supremum over consists of a single number $f(x_1)$, thus, $\Bell(x_1, x_1^2; f) = f(x_1)$.
\end{proof}

The lower parabola $\{x \in\mathbb{R}^2\mid x_2 = x_1^2\}$ is called the \emph{fixed boundary}\index{fixed boundary} and denoted by~$\FixedBoundary\Omega_{\eps}$ and the upper one, $\{x \in \mathbb{R}^2\mid x_2 = x_1^2 + \eps^2\}$, is the \emph{free boundary}\index{free boundary} denoted by~$\FreeBoundary\Omega_{\eps}$: it moves when we vary~$\eps$. We know the value of the Bellman function on the fixed boundary,
the value on the free boundary is unknown. 

The reader may have asked a question why did we fix $h[\varphi]$ for $h(t) = t$ and $h(t) = t^2$ in the defintion of the Bellman function, formula~\eqref{Bf}. The meaning of this definition is to design some function that has two properties. First,
it should estimate the functional $f[\varphi]$, so it is appropriate to take the supremum of $f[\varphi]$ over some set of functions $\varphi$. Second, it should have good analytic properties (for example, it should be concave). We 
have already noticed in Section~\ref{s11} that we should fix $h[\varphi]$ for $h(t) = t$, because all the classical estimates~\eqref{bmoequivnorms},~\eqref{JN},~\eqref{intJN} contain this expression. The idea of what is the 
second functional to fix comes from
the definition of the $\BMO$-norm.

\subsection{Conditions on $f$}\label{s212}

The function~$f$ will be under some conditions. These conditions are of two types. The first type comes from quantitative aspects. We are interested in finite Bellman functions, at least, we want the functional~\eqref{functional} to be well defined for all the points of $\BMO_{\eps}$. These conditions are expressed in terms of summability properties of $f$. The second type of conditions corresponds to continuity properties of $f$. These conditions make 
the structure of the Bellman function less complicated, thus able to be described. 

We begin with the condition of the second type. We require $f \in C^2(\mathbb{R})$. We need $f''$ to be continuous, piecewise monotone, and have only finite number of monotonicity intervals. We reformulate these conditions in terms of $f'''$ and also introduce some useful notions. The third derivative of $f$ is a measure. By virtue of the condition $f \in C^2$, $f'''$ does not have atoms. Also, this measure is regular in a sense. We introduce an object that was completely ignored in~\cite{5A}.

\begin{Def}\label{solidroot}
Let $\mu$ be a measure on the line. The complement of its support is an open subset of the line, so it is a union of several intervals (finite or countable number of them). We call the closure of each such interval \emph{a solid root} of $\mu$. If $\mu$ is neither positive nor negative in every neighborhood of its solid root, then such a solid root is called \emph{essential}. 
\end{Def}\index{solid root}\index{essential root}

The measure $f'''$ can be thought of as a measure on $\supp f'''$. It has the Hahn decomposition on this set, $f''' = f'''_+ + f'''_-$. The set $\supp f'''_+ \cap \supp f'''_-$ is the set of points where $f'''$ ``changes its sign''. This set is closed. The points of~$\supp f'''_+ \cap \supp f'''_-$ are also called essential roots. Therefore, an essential root is a maximal by inclusion connected subset of the line such that $f'''$ vanishes on it but is neither negative nor positive in every its neighborhood. The regularity condition we impose on $f'''$ is that it has only finite number of essential roots. If $f$ were $C^3$-smooth, this condition would be the same as if the function $f'''$ had only finite number of changes of sign.
The reformulation in terms of $f'''$ is
more useful, because the Bellman function depends on this expression in a more direct way. %(this is a consequence of the fact that the planes' traces on the lower parabola are quadratic polynomials in $x_1$; this space is exactly the kernel of the operator of third derivation).

The function $f$ also requires some summability conditions at infinity. The considerations below will help us to guess these conditions. We know that the Bellman
function does not depend on the interval, so we may assume that $I = [0,1]$. One can see that the function $\varphi(t) = \eps \ln t$ lies in $\BMO_{\eps}$, it follows from direct calculations. If the Bellman function is finite, then $f[\eps \ln t]$ is finite:
\begin{equation*}
f[\eps\ln t] = \int\limits_0^1 f(\eps \ln t) dt = \frac{1}{\eps}\int\limits_{-\infty}^0 f(x) e^{\frac{x}{\eps}} dx.
\end{equation*}
Therefore, the integral on the right must converge. We can also substitute the function $\varphi(t) = - \eps \ln t$ into the functional $f[\varphi]$. Thus, $f$ must be summable with the weight $e^{-\frac{|t|}{\eps}}$. A more detailed discussion of the summability properties of $f$ whose Bellman function is finite is postponed until Subsection~\ref{s612} (now we are searching only for some sufficient conditions, not for the necessary ones). For our purposes, we need the following summability condition: $w_{\eps_{\infty}}(t) \df e^{-{|t|}/{\eps_{\infty}}} \in L^1(f''')$ for some $\eps_{\infty}$. This condition differs a bit from $w_{\eps} \in L^1(f)$, which is necessary for the finiteness of the Bellman function $\Bell$. Here are the requirements for $f$\index{condition! conditions on~$f$}.
\begin{Cond}\label{reg}
The function~$f$ is two times continuously differentiable,~$f''$ is piecewise monotone
and has only finite number of monotonicity intervals.
\end{Cond}
\begin{Cond}\label{sum}
The integral~$\int_{-\infty}^{\infty}w_{\eps_{\infty}}(t)\,df''(t)$ is absolutely convergent.
\end{Cond}
%\begin{align}
%\label{reg}&1.\hbox{The function~$f$ is two times continuously differentiable,~$f''$ is piecewise monotone}\\ 
%&\hbox{and has only finite number of monotonicity intervals.}\\
%\label{sum}&2.\hbox{The integral~$\int_{-\infty}^{\infty}w_{\eps_{\infty}}(t)\,df'''(t)$ is absolutely convergent.}
%\end{align}
The essential roots of $f'''$ will play a significant role in what follows, therefore we fix the notation for them.

\begin{Def}\label{roots}
The essential roots of~$f'''$ are closed intervals (which can be single points or rays) $c_0, c_1, \ldots, c_n$ and $v_1, v_2, \ldots, v_n$ such that $c_0 < v_1 < c_1 < v_2 < \cdots < v_n < c_n$ and $(\cup_i v_i) \bigcup (\cup_i c_i)$ is complement to the set of the growth points of $f''$ (i.e. the points, in the neighborhood of which $f''$ either strictly increases or decreases). The function $f'''$ ``changes sign'' from '$-$' to '+' at $v_i$, from '+' to '$-$' at~$c_i$. 
\end{Def} 
We make an agreement that if in a neighborhood of $-\infty$ we have $f''' < 0$, then $c_0 = -\infty$. Similarly, if in a neighborhood of $+\infty$ we have $f''' > 0$, then $c_n = \infty$. What is more, $v_i$ or $c_i$ is an interval (not a point) if and only if it is an essential solid root in the sense of Definition~\ref{solidroot}. 

In the light of our definition, sometimes we will have to treat the intervals as if they were points. We write $\dist(x,y)$ for the usual distance between subsets $x$ and $y$ of the real line. We will need it only to denote the distance between either two intervals or an interval and a point. Moreover, sometimes we will write, for example, $a_n \rightarrow w$ where $w$ is a root, e.g. can be an interval. In such situations we mean that for every neighborhood of $w$ all but finite number of members of $\{a_n\}_n$ lie in it. What is more, the set of intervals has an essential ordering: $[a,b]$ is less than $[c,d]$ if and only if $b < c$. We have already used this ordering in Defition~\ref{roots}. We will also often use the notation~$\mathfrak{a}^{\mathrm{r}}$ and~$\mathfrak{a}^{\mathrm{l}}$ for the right and the left endpoints of the interval~$\mathfrak{a}$.

We also establish a corollary of our summability assumptions. We begin with an easy proposition. In it, the symbol~$\BV_{\mathrm{loc}}$ denotes the set of functions with locally bounded variation.
\begin{St}\label{summability}
Let $g \in \BV_{\mathrm{loc}}$ and let $w_{\eps_{\infty}} \in L^1(dg)$. Then~$w_{\eps_{\infty}} \in L^1(g)$. 
\end{St}
\begin{proof}
First, using the definition of $dg$ we can write
$$
g(x)-g(0)=\int\limits_0^x dg(t)
$$
for every $x$. Thus, 

\begin{equation*}
|g(x)| \leq |g(0)| + \int_{0}^{x} \sign x |dg(t)|.
\end{equation*}
Multiplying this equation by $w_{\eps_{\infty}}(x)$, we get
\begin{align*}
|g(x)|w_{\eps_{\infty}}(x) \leq |g(0)|w_{\eps_{\infty}}(x) + \int_{0}^{x} \sign x \, w_{\eps_{\infty}}(x) |dg(t)|.
\end{align*}
Integrating over the real line with respect to $x$, we get
\begin{align*}
&\|w_{\eps_{\infty}}\|_{L^{1}(g)} \leq 2 \eps_{\infty} |g(0)|+\int_{-\infty}^{\infty}\int_{0}^{x}\sign x \, w_{\eps_{\infty}}(x) \,|dg(t)|\, dx = \\
&2 \eps_{\infty} |g(0)| +\int_{-\infty}^{\infty}\int_{|t|}^{\infty} w_{\eps_{\infty}}(x) \,dx \,|dg(t)| = \eps_{\infty}\big(2 |g(0)|+\|w_{\eps_{\infty}}\|_{L^{1}(dg)}\big).
\end{align*}
The proposition is proved. 
\end{proof}

\begin{Le}\label{emb}
Suppose $f$ satisfies the summability Condition~\textup{\ref{sum}}. Then the functions $f''$, $f'$, and~$f$ are also in $L^1(w_{\eps_{\infty}})$ and
\begin{equation*}
			f^{(r)}(u)e^{-|u|/\eps_{\infty}} \to 0\quad\mbox{as}\quad u\to\pm\infty\quad\mbox{for}\quad r = 0,1,2.
\end{equation*}
\end{Le}
\begin{proof}
First, we apply Proposition~\ref{summability} with $g = f''$. We obtain that $\int_{-\infty}^{\infty} |f''(t)|w_{\eps_{\infty}}(t) \,dt$ is finite. Then we can apply the same proposition with $g = f'$, this yields $f' \in L^1(w_{\eps_{\infty}})$. Finally, applying it once more with $g = f$, we get $f \in L^1(w_{\eps_{\infty}})$.
We only have to cope with the limits. It is easy to see that those limits do exist, because
\begin{equation*}
\lim\limits_{u\to \pm\infty} f^{(r)}(u)e^{-u/\eps_{\infty}}=
%f^{(r)}(u_1)e^{-u_1/\eps_{\infty}}+\lim\limits_{u\to \pm\infty}\Big(
%\int\limits_{u_1}^u e^{-t/\eps_{\infty}}df^{(r)}(t) -\frac{1}{\eps_{\infty}} \int\limits_{u_1}^u e^{-t/\eps_{\infty}}f^{(r)}(t)dt\Big)=\\
f^{(r)}(u_1)e^{-u_1/\eps_{\infty}}+ \int\limits_{u_1}^{\pm\infty} e^{-t/\eps_{\infty}}df^{(r)}(t) -
\frac{1}{\eps_{\infty}} \int\limits_{u_1}^{\pm\infty} e^{-t/\eps_{\infty}}f^{(r)}(t)dt,
\end{equation*}
and the integrals on the right-hand side are finite.
%$$|f^{(r)}(u_1)e^{-u_1/\eps_{\infty}} - f^{(r)}(u_2)e^{-u_2/\eps_{\infty}}| \leq \eps_{\infty}^{-1} \int\limits_{u_2}^{u_1} e^{-t/\eps_{\infty}} |d f^{r}(t)|,$$
%if $0 \leq u_2 \leq u_1$. But the integral on the right tends to zero as $u_1,u_2 \rightarrow \infty$ for $r = 0,1,2$. Therefore, $f^{(r)}(u)e^{-u/\eps_{\infty}}$ is a Cauchy sequence in $u$.
The function $f^{(r)}(u)e^{-u/\eps_{\infty}}$ is summable itself, so the limit equals zero. 
\end{proof}
\section{On concavity of surfaces and functions}\label{s22}

\subsection{Main inequality}

We have already noted that our extremal problem is self-similar. Now we try to exploit this fact. Assume~$I$ to be divided into two smaller intervals, $I_-$ and $I_+$. Let
$\varphi_{-}$ be a function in $\BMO_{\eps}(I_-)$ and let $\varphi_{+}$ be a function in $\BMO_{\eps}(I_+)$. Assume that the point 
$$x \df \frac{|I_-|}{|I|}\bp{\varphi_-} + \frac{|I_+|}{|I|}\bp{\varphi_+}$$
belongs to $\Omega_{\eps}$ (we use the notation for a Bellman point, see Definition~\ref{BellmanPoint}). 
We can choose functions~$\varphi_-$ and~$\varphi_+$ almost realizing the supremum in the definition of the Bellman function, formula~\eqref{Bf}. That is
$$\langle{f(\varphi\Cii\pm)}\rangle\Cii{I_{\pm}} \ge \Bell(\bp{\varphi_{\pm}}) - \eta$$
for a small positive $\eta$. Therefore, their concatenation, the function $\varphi \in L^2(I)$ defined by the formula
    $$
		\varphi(t)=\left\{
			\begin{aligned}
				&\varphi\Cii{-}(t),\quad t\in I_-;\\
				&\varphi\Cii{+}(t),\quad t\in I_+,
			\end{aligned} \right.
	$$
corresponds to the point $x$, $\bp{\varphi} = x$. Suppose that $\varphi \in \BMO_{\eps}$ (we will comment this assumption a bit later). Then, 
\begin{multline*}
		\Bell\bigg(\frac{|I_-|}{|I|}\bp{\varphi_-} + \frac{|I_+|}{|I|}\bp{\varphi_+}\bigg) 
		\ge \av{f(\varphi)}{I} = 
		\frac{|I_-|}{|I|}\av{f(\varphi\Cii{-})}{I_{-}}+\frac{|I_+|}{|I|}\av{f(\varphi\Cii{+})}{I_{+}} \\\ge
		\frac{|I_-|}{|I|}\Bell(\bp{\varphi_-}) + \frac{|I_+|}{|I|}\Bell(\bp{\varphi_+})-\eta,
	\end{multline*}
so we get the \emph{concavity} of the function $\Bell$, because $\eta$ can be arbitrarily small. Indeed, we can argue vise versa: first choose three points $x, x_+, x_-$ such that $x = \alpha_+x_+ +  \alpha_-x_-$, then divide
$I$ into $I_+$ and $I_-$ in such a way that $\alpha_+ = \frac{|I_+|}{|I|}$, $\alpha_- = \frac{|I_-|}{|I|}$. In such a case, the achieved inequality turns into 
\begin{equation}\label{MIneq}
\Bell(\alpha_-x_- + \alpha_+x_+) 
		\ge
		\alpha_-\Bell(x_-)+\alpha_+\Bell(x_+).
\end{equation}
We did not mark this reasoning as a proof because the assumption we did ($\varphi \in \BMO_{\eps}$) is significantly irreducible. What is more, the Bellman function is not concave, but only \emph{locally concave}\index{locally concave function}.
\begin{Def}\label{LocalConcavity}
The function $G\colon \Omega \mapsto \mathbb{R}$ is said to be locally concave on $\Omega$ if it is concave on every convex subdomain of $\Omega$. 
\end{Def}

It is not difficult to see that a function is locally concave if and only if it is concave on every segment that belongs to $\Omega$ entirely. %{\color{green}This is also the place where the dyadic case (see~\cite{SlVa}) differs dramatically from the continuous one. The dyadic Bellman function is globally concave, in the dyadic setting
%the reasoning from the beginning of this subsection works without any additional assumptions.}

We suppose that the Bellman function is locally concave and try to guess it. After we have an appropriate locally concave candidate, we try to prove that it is the true Bellman function. If we succeed, we get the local concavity 
of the Bellman function.

The inequality that follows from self-similarity of the problem (\eqref{MIneq} in our case) is usually called \emph{the main inequality}\index{main inequality}.

\subsection{Locally concave majorants}

In the previous subsection, we saw that concavity properties play a significant role in the subject. We introduce a useful class of functions:
\begin{equation}\label{Lambda}
		\Lambda_{\eps,f} \df \big\{G\in C(\Omega_\eps)\mid\;  \mbox{$G$ is locally concave};\; \forall u\in\mathbb{R} \quad
		G(u,u^2) = f(u)\big\}.
\end{equation}
The functions from this class have one very important property: they majorize the Bellman function.
\begin{Le}\label{LMaj}
		Let a continuous function~$f$ satisfy Condition~\textup{\ref{sum}}. If $G \in \Lambda_{\eps,f}$\textup, then $\Bell(x;\,f) \le G(x)$ for all $x \in \Omega_\eps$. 
\end{Le}
Here we may not suppose that $f$ satisfies Condition~\ref{reg} from Subsection~\ref{s212}. We recall easy statements about functions in the $\BMO$ space and simple geometry.

\begin{St}\label{L5}
		Suppose $\tilde{\eps} > \eps$\textup, $I\subset\mathbb{R}$ and $\varphi \in \BMO_\eps(I)$. Then there exists a partition $I={I_{-}\cup I_{+}}$ such that the straight line segment $[\bp{\varphi \mid_{I_{-}}}, \bp{\varphi \mid_{I_{+}}}]$ lies
		inside $\Omega_{\tilde{\eps}}$. What is more\textup, the parameters $\alpha_{\pm} = |I_\pm|/|I|$ can be chosen uniformly \textup{(}in~$I$ and $\varphi$\textup{)} separated from $0$ and $1$.
	\end{St} 
The proof can be found either in~\cite{Va2} or in~\cite{SlVa}.

\begin{St}\label{Stcutoff}
		Suppose $\varphi \in \BMO(I)$\textup, $c,d \in \mathbb{R}$\textup{,} and $c<d$.
		Consider the truncation of $\varphi$\textup{:}
		\begin{equation}\label{truncation}
			\varphi_{c,d}(t) \df 
			\begin{cases} 
				\ d,&\ \varphi(t) > d;\\
				\varphi(t),&  c \le \varphi(t) \le d;\\
				\ c,&\ \varphi(t) < c.
			\end{cases}  
		\end{equation}
		Then $\av{\varphi_{c,d}^2}{J} - \av{\varphi_{c,d}}{J}^2 \le \av{\varphi^2}{J} - \av{\varphi}{J}^2$ for every segment $J \subset I$.
	\end{St}	
This statement is accurately proved in~\cite{SlVa2} (Lemma~$6.3$ of that paper). Such statements are usual for the Bellman function method, for example,  see~\cite{RVV} for a similar statement for Muckenhoupt weights. Here we give a shorter reasoning, communicated to us by F. V. Petrov.
\begin{proof}
Combining the formula
\begin{equation*}
%\label{FedorFormula}
\av{h^2}{J} - \av{h}{J}^2 = \frac{1}{2|J|^2} \int\limits_{J^2}|h(x) - h(y)|^2\, dx\, dy
\end{equation*}
with the evident inequality
$$
|\varphi_{c,d}(x)-\varphi_{c,d}(y)|^2 \le|\varphi(x) - \varphi(y)|^2,
$$
we obtain the desired estimate. 
\end{proof}

\begin{Cor}
	If ${\varphi \in \BMO_\eps(I)}$\textup{,} then $\varphi_{c,d} \in \BMO_\eps(I)$.
	\end{Cor} 
We note that the proof above gives more: a composition of any $1$-Lipshitz function with a $\BMO$-function does not increase its $\BMO$-norm. 

\paragraph{Proof of lemma \ref{LMaj}.}

By the Bellman function definition, formula~\eqref{Bf}, we need to prove that~$f[\varphi] \leq G(\bp{\varphi})$ (here we use the notion of a Bellman point, see Definition~\ref{BellmanPoint}) for all the functions~$\varphi \in\BMO_{\eps}$. We begin with proving this assertion for essentially bounded $\varphi$ and then use
a limit argument to verify this inequality for the remaining functions. But before that we have to make a small trick that allows us to enlarge $\Omega_{\eps}$. Let $0 < \tau < 1$. Consider two new functions, namely,
$$
G_\tau(x_1,x_2) = G(\tau x_1,\tau^2x_2), \qquad f_\tau(x_1) = f(\tau x_1).
$$
The function $G_{\tau}$ is still continuous and locally concave on $\Omega_{\eps/\tau}$. What is more, these two functions,~$f_{\tau}, G_{\tau}$, satisfy the same equation on the boundary:
$$
G_\tau(x_1,x_1^2) = f_\tau(x_1).
$$
We fix some function $\varphi \in \BMO_{\eps}(I) \cap L^{\infty}(I)$. By Proposition~\ref{L5}, there exist two intervals $I_-$ and $I_+$ such that the whole segment $[\bp{\varphi \mid_{I_-}},\bp{\varphi \mid_{I_+}}]$ belongs to $\Omega_{\eps/\tau}$. Therefore, by the local concavity of $G_{\tau}$, we have
$$
G_{\tau}(\bp{\varphi}) \geqslant \frac{|I_-|}{|I|}G_{\tau}(\bp{\varphi \mid_{I_-}}) + \frac{|I_+|}{|I|}G_{\tau}(\bp{\varphi \mid_{I_+}}).
$$
We can repeat this procedure for each of the intervals $I_-$ and $I_+$ to divide the Bellman point of $\varphi$ into a sum of four Bellman points. Then we subdivide each of them, and so on. On the $n$-th step we get
$$
G_{\tau}(\bp{\varphi}) \geqslant \sum\limits_{\sigma} \frac{|\sigma|}{|I|}G_{\tau}(\bp{\varphi \mid_{\sigma}}),
$$
where $\sigma$ runs over the $2^n$ subintervals of $I$. The sum on the right is nothing but an integral. Indeed, we can introduce a step function $x_n(t)\colon I \mapsto \mathbb{R}^2$ that equals $\bp{\varphi \mid_{\sigma}}$ on
$\sigma$ for every $\sigma$ of the $n$-th partition of $I$. Therefore,
$$
G_{\tau}(\bp{\varphi}) \geqslant \frac{1}{|I|}\int\limits_{I} G_{\tau}(x_n(t))\, dt.
$$
Now we let $n \rightarrow \infty$. By Proposition~\ref{L5}, the size of the partition tends to zero as $n$ tends to infinity, therefore, by the Lebesgue  dominated convergence theorem, $x_n(t) \rightarrow (\varphi(t),\varphi^2(t))$ pointwise a.e. Thus, $G_{\tau}(x_n(t)) \rightarrow f_{\tau}(\varphi(t))$. All these functions are essentially bounded, because we have supposed that $\varphi$ is bounded and $G_{\tau}$ is continuous. Consequently, 
$$
\frac{1}{|I|}\int\limits_{I} G_{\tau}(x_n(t))\, dt \rightarrow f_{\tau}[\varphi],
$$  
so, sending $\tau$ to $1$, we get
$$
G(\bp{\varphi}) \geqslant f[\varphi].
$$
Now we have to get rid of the boundedness assumption. Let $\varphi$ be an arbitrary function belonging to~$\BMO_{\eps}$. We can make this function bounded by cutting off its bigger part:
$$
				\varphi_m(t) =
				\begin{cases}
					\phantom{-}m,& \varphi(t) > m;\\
					\varphi(t),&  |\varphi(t)| \le m;\\
					 -m,& \varphi(t) < -m.
				\end{cases}  
	  $$
	  By virtue of Proposition~\ref{Stcutoff}, $\varphi_m \in \BMO_{\eps}$. Therefore,
$$
G(\bp{\varphi_m}) \geqslant f[\varphi_m].
$$
Obviously, $\bp{\varphi_m} \rightarrow \bp{\varphi}$ as $m \rightarrow \infty$. The convergence of the right parts of the inequalities is a bit more puzzling. We need to verify that
$$
\int\limits_I f(\varphi_m(t))\, dt \rightarrow \int\limits_I f(\varphi(t))\, dt.
$$
We have the pointwise convergence of the integrands. So we seek a summable majorant. It is the function~$e^{\frac{|t|}{\eps_{\infty}}}$ that plays the role of the majorant for $f$. Indeed, by Lemma~\ref{emb},
$$
|f(\varphi_m(t))| \leq C e^{\frac{|\varphi_m(t)|}{\eps_{\infty}}} \leq C e^{\frac{|\varphi(t)|}{\eps_{\infty}}}.
$$
The latter expression is summable for~$\eps < \eps_{\infty}$ by the integral form of John--Nirenberg inequality~\eqref{intJN} in its sharp form  proved in~\cite{SlVa}. 
So we are finished. \qed

The procedure just described is usually called the \index{Bellman! induction}\emph{Bellman induction}. The last implication of the proof may seem confusing: we have used the John--Nirenberg inequality to prove something very similar. However, one can first prove Lemma~\ref{LMaj} for the case~$f(t) = e^{|t|/\eps}$ using the monotone convergence theorem instead of the Lebesgue theorem, then find the Bellman function for this very particular case (this can be done with the theory of Subsection~\ref{s32}) and verify that it is finite, thus, providing the wanted majorant (more or less, this is what is done in~\cite{SlVa}).
	
\subsection{Monge--Amp\`ere equation}\label{s223}
We see that functions from $\Lambda_{\eps,f}$ provide good estimates for the Bellman function. It may be useful to find the (pointwise) minimal function from this class. Surely, such a function exists, because an infimum of an arbitrary 
set of locally concave functions is also locally concave. Denote such a minimal function by~$\BG$. Using some easy convex geometry arguments one can see that~$\BG$ has to be linear in some directions. The precise statement looks like this.
\begin{Th}\label{ConvexGeometry}
Let~$f$ satisfy Conditions~\textup{\ref{reg}, \ref{sum}}. The minimal function $\BG$ from the set $\Lambda_{\eps,f}$ satisfies the following conditions.
	\textup{\begin{enumerate}
	\item \emph{For every point $x \in \interior\Omega_{\eps}$ there is a dichotomy\textup: either there exists a vector $\Theta(x)$ such that $\BG$ is linear along the line $\ell(x)= x + \mathbb{R}\Theta(x)$ in a neighborhood of $x$ or $\BG$ is linear in a neighborhood of~$x$. We call the lines $\ell(x)$ the {\bf extremals}\index{extremal}.}
	\item \emph{The function~$\BG$ is differentiable and its differential is constant for all the points on a single extremal.}
	\item \emph{The extremals cannot intersect the free boundary\textup, but only touch upon it.}
	\end{enumerate}}
\end{Th}
%By a superdifferential at a point~$x \in \Omega_{\eps}$ for a function~$G \in \Lambda_{\eps,f}$, we mean the set of all linear functions~$L:\mathbb{R}^2 \to \mathbb{R}$ such that~$G(y) \leq G(x) + L(y-x)$ in a neighborhood of~$x$.

This theorem provides a partition of $\Omega_{\eps}$ into sets of two types. The sets of the first type are extremals, line segments, along which $\BG$ is linear. We note (and this can be easily proved) that the extremals
cannot ``stop'': either both ends of an extremal lie on the fixed boundary or one of them is the point of tangency with the free boundary. Sets of the second type are the two-dimensional domains where $\BG$ is linear.

We should make a remark on the term ``domain''. We call a domain every open connected set united with some (or none) part of its boundary.

We will not prove Theorem~\ref{ConvexGeometry} right now, because, from a formal point of view, we do not need it (however, it will follow from our general considerations, e.g. Theorem~\ref{Final} far below). It only helps us to guess the Bellman function. It leads us to the notion of a \emph{Bellman candidate}\index{Bellman! candidate}.
\begin{Def}\label{candidate}
Suppose $\Omega$ to be some subdomain of $\Omega_{\eps}$, which in its turn is a union of several domains,~$\Omega = \cup_{i}\Omega^i$. Let the continuous function $B$ be locally concave and satisfy the boundary condition on the intersection of $\Omega$ with the fixed boundary, 
let also $B \in C^1({\Omega^i})$ for all $i$. Then we call $B$ a Bellman candidate provided for each index $i$ the function $B$ is either linear in $\Omega^i$ or $\Omega^i$ is foliated by the extremals along which the differential
of $B$ is constant.
\end{Def}
%{\color{blue}This definition is a bit different from the one we had in~\cite{5A}. Here the number of subdomains may be infinite. In fact, we will not need it because the number of these subdomains does not exceed the number of essential roots of $f'''$ plus one. However, this definition can be used in further development, in the case of infinite number of essential roots.}
 
If $B$ is twice differentiable at some inner point $x \in \Omega_{\eps}$, then $\frac{d^2\!B}{dx^2} \le 0$, i.e. the second differential of $B$ is negative-definite, because $B$ is locally concave. On the other hand, this matrix has a non-trivial kernel, because the function $B$ is linear in the direction of $\Theta(x)$. So, the determinant of the second differential is zero. This remark clarifies the name of the subsection: the achieved equation is called the \emph{homogeneous Monge--Amp\`ere equation}\index{Monge--Amp\'ere equation}:
\begin{equation}\label{MA}
B_{x_1x_1}B_{x_2x_2} - B_{x_1x_2}^2 = 0.
\end{equation}
The homogeneous Monge-Amp\'ere equation must hold a.e. for the function $\BG$, because a locally concave function is a.e. twice differentiable. However, it does not have to hold everywhere even for very smooth boundary values $f$, because the Bellman function is rarely $C^2$-smooth.
 
\subsection{Optimizers}\label{s224}

Let $B$ be a Bellman candidate in the whole domain $\Omega_{\eps}$, i.e. let it satisfy Definition \ref{candidate} with $\Omega = \Omega_{\eps}$. This subsection provides a method of verification that the candidate $B$ coincides with the Bellman function. By Lemma~\ref{LMaj}, there is the inequality $\Bell \leq B$. To prove the reverse inequality, $B(x) \leq \Bell(x)$, for a point $x$, $x \in \Omega_{\eps}$, it is sufficient to find a function $\varphi \in \BMO_{\eps}$ with $\bp{\varphi} = x$ such that $B(x) \leq f[\varphi]$. Indeed, by the definition of the Bellman function, formula \eqref{Bf}, $f[\varphi] \leq \Bell(x)$, consequently, $B(x) \leq \Bell(x)$. We introduce some notions.
\begin{Def}\label{TestFunction}\index{test function}
Let $x \in \Omega_{\eps}$. We call a function $\varphi \in \BMO_{\eps}$ a \emph{test function} for $x$ if $\bp{\varphi} = x$. 
\end{Def} 
\begin{Def}\label{Opt}\index{optimizer}
Let $x \in \Omega_{\eps}$ and let $B$ be a Bellman candidate. We call a measurable function $\varphi$ an \emph{optimizer} for $B$ at $x$ if it satisfies two conditions:
  \begin{itemize}
  \item $\varphi$ is a test function for $x$; 
  
  \item $B(x) = f[\varphi]$.
 
 \end{itemize} 
\end{Def}
So, in order to prove that $B$ coincides with the true Bellman function, one has to provide at least one optimizer for each point $x$ in $\Omega_{\eps}$ and this candidate $B$. What is the way to do this? We can consider only non-decreasing optimizers. This follows from the fact that the monotonic rearrangement\index{monotonic rearrangement} does not increase the $\BMO$ norm. We remind the reader this useful statement.
\begin{Th}\label{MonotonicRearrangement}
Let $\varphi \in \BMO (I)$\textup, let $\varphi^*$ be its monotonic rearrangement \textup{(}i.e. a non-decreasing function such that the sets $\{t \mid \varphi(t) \geqslant \lambda\}$ and \mbox{$\{t \mid \varphi^*(t) \geqslant \lambda\}$} have equal measure for every $\lambda$\textup{)}. Then 
$$\BNorm{\varphi^*} \leq \BNorm{\varphi}.$$
\end{Th}
This theorem was proved in~\cite{Ivo}. The importance of this theorem and its relationship with sharp constants in inequalities of the form~\eqref{bmoequivnorms},~\eqref{JN},~\eqref{intJN} are emphasized by the paper~\cite{Kor}, where this theorem was used to find the exact value of maximal $\eps_{\infty}$ for \eqref{intJN} in the case of the~$\BMO$~$1$-semi-norm. There is a very close relationship between monotonic rearrangements and Bellman functions, but we postpone this discussion to Section~\ref{s62}.

Monotonic rearrangement has one useful property:
$$f[\varphi] = f[\varphi^*],$$
which shows, together with Theorem~\ref{MonotonicRearrangement} that $\varphi^*$ is an optimizer provided $\varphi$ is an optimizer. We do not need this consideration formally, but it helps us to guess the optimizers. In the light of it, we consider only monotone optimizers. All  other considerations concerning optimizers are postponed until the fifth chapter.

\subsection{General principles and description of results}

We have seen that to find the Bellman function~$\Bell$, one has to construct the minimal function~$\BG$ from~$\Lambda_{\eps,f}$ and then prove that~$\BG = \Bell$ using optimizers. We recall that $\BG$ is a Bellman candidate in the whole domain~$\Omega_{\eps}$, i.e. the domain~$\Omega_{\eps}$ is foliated by its extremals and linearity domains~(see Subsection~\ref{s223}). We note that we used Theorem~\ref{ConvexGeometry} to verify that $\BG$ is a Bellman candidate, but in what follows we will first construct some Bellman candidate in the whole domain $\Omega_{\eps}$ that seems to be minimal, and then prove that it coincides with $\BG$ and $\Bell$ using optimizers. 

We call such a foliation\index{foliation} by extremals and linearity domains a \emph{picture} of a candidate~$B$. We use this notion only for descriptive needs, not for the proofs. We will not formalize it, because there are some difficulties in such a formalization, which we do not need to overcome. For example, some functions $B$ have several pictures (for the case of the constant function~$B$ we can draw any picture we want). But if one knows the picture for~$B$ and the boundary condition~$f$, he can calculate the Bellman candidate with such a picture in many situations. There is also some non-uniqueness of the Bellman candidate $\BG$, which disappears if one recalls that $\BG$ is minimal. Anyway, the structure of the picture plays the crucial role in the proof, all the reasonings deal not with the candidate itself, but with its picture. The expression ``to build the Bellman function'' informally means ``to build the foliation for the minimal candidate''. The global foliations are built from the local ones, which are of several different types. First, we study the local structure of each such~\emph{figure} and then ``glue'' the picture from them. 

We formulate the main theorem of the paper. We recall that~$\BG$ is the minimal locally concave function in~$\Omega_{\eps}$ that coincides with~$f$ on the boundary.
\begin{Th}\label{MT}
Let $f$ satisfy Conditions~\textup{\ref{reg},~\ref{sum}}. Then 
\begin{equation*}
\Bell(x;\,f) = \BG_{\eps,f}.
\end{equation*}
\end{Th}
This theorem holds in a much more general setting, see Section~\ref{s62}. 

\bigskip

\bigskip

\bigskip

\bigskip

\bigskip

\centerline{\bf Main results}
\begin{enumerate}
\item We build the Bellman function for $f$ satisfying Conditions~\ref{reg},~\ref{sum} and describe its evolution in~$\eps$;
\item We provide an algorithm that allows one to calculate the Bellman function; %For piecewise analytic~$f$ there are only  a finite number of steps;
\item Theorem \ref{MT}.
\end{enumerate} 
Some explanation is needed. By building the Bellman function we mean only that there is a foliation of $\Omega_{\eps}$ of rather simple type (for example, it has only finite number of linearity domains) over which the Bellman function $\Bell(f)$ is built. This foliation evolves continuously and obey some monotonicity rules that are also described. In the second point we intend to provide some expression for $\Bell$ that contains integrals, differentiation, and solution of some implicit equations. We always prove that those equations are well solvable, i.e. do not have infinite number of solutions. The strength of the evolutional approach is that it provides the Bellman function for all $\eps$ simultaneously.

Some miscellaneous results are also situated in the sixth chapter.

\chapter{Patterns for Bellman candidates}

\section{Preliminaries}\label{s31}
The purpose of this section is to construct Bellman candidates (see Definition~\ref{candidate}) on various domains. The global foliation for the Bellman function may occur to be rather complicated, but its local structure is easy to describe. We give some heuristics to classify local Bellman candidates.

Consider a minimal locally concave function and its foliation provided by Theorem~\ref{ConvexGeometry}. We recall that this foliation consists of segments, which are called extremals, and linearity domains\index{domain! linearity domain}. 

The extremals are of two types: those that connect two points on the lower boundary and those that connect a point on the lower (fixed) boundary with a point on the upper (free) one. First type extremals are called \emph{chords}\index{chords}, second type extremals are called \emph{tangents}\index{tangents}. We note that a chord can touch the upper parabola. Such a chord is called a \emph{long} one. Long chords are the ones whose projection on the~$x_1$-axis has length~$2\eps$. Other chords have smaller projections.

It is convenient to classify linearity domains by the number of their points on the lower boundary. Indeed, if a linearity domain has at least three points on the lower boundary, then the value of the Bellman candidate can be calculated immediately. Moreover, a linearity domain that has more than three points on the lower boundary can be present in the foliation of a Bellman candidate not for all functions~$f$. If this linearity domain crosses the lower boundary at the points~$A_i = \big(a_i,a_i^2\big)$, then the points~$\big(a_i,a_i^2,f(a_i)\big)$ lie in one plane in~$\mathbb{R}^3$. This provides a restriction on the function~$f$. Summing up all we have said about linearity domains, we distinguish the ones that have one point on the lower boundary, the ones that have two points on the lower boundary, and all the others. A more detailed classification will be provided later.

A global foliation is glued from local ones. We explain the informal meaning of the word ``glue'' we use. Consider two subdomains~$\Omega^1$ and~$\Omega^2$ of~$\Omega_{\eps}$. Let~$B_1$ be a Bellman candidate on~$\Omega^1$, let~$B_2$ be a Bellman candidate on~$\Omega^2$. Suppose that~$B_1 = B_2$ on~$\Omega^1 \cap \Omega^2$. Consider the function~$B$ defined on the union domain~$\Omega = \Omega^1 \cup \Omega^2$ as a concatenation of~$B_1$ and~$B_2$ (i.e.~$B = B_1$ on~$\Omega^1$ and~$B = B_2$ on~$\Omega^2$). Suppose that this function~$B$ is~$C^1$-smooth. In such a case, it is locally concave, provided the functions~$B_1$ and~$B_2$ are locally concave. Thus it is a Bellman candidate on~$\Omega$. Its foliation coincides with the foliation for~$B_1$ on~$\Omega^1$ and with the foliation for~$B_2$ on~$\Omega^2$. We see that the foliation for~$B$ is glued from the foliations for~$B_1$ and~$B_2$. 

We have used the following fact in the explanation : a~$C^1$-concatenation of two locally concave functions is locally concave. To formulate this claim rigorously, we need a new notion\footnote{This definition was suggested to us by Pavel Galashin and Vladimir Zolotov, we are grateful to them.}. 
\begin{Def}\label{InducedConvexSet}\index{induced convex set}\index{induced convex set}
Suppose that~$\Omega$ is a subdomain of~$\Omega_{\eps}$. We call~$\Omega$ an induced convex set if for every segment~$l \subset \Omega_{\eps}$ the set~$\Omega \cap l$ is convex.
\end{Def}
%{\bf [we have either to find this definition in some book or thank the guys]}

All the domains we use for building Bellman candidates are induced convex.

\begin{St}\label{ConcatenationOfConcaveFunctions}
Suppose that the domains~$\Omega^1$ and~$\Omega^2$ are induced convex in~$\Omega_{\eps}$. Suppose that a~$C^1$-smooth function~$B$ is locally concave on each of the domains~$\Omega^i$\textup,~$i=1,2$. Then it is locally concave on~$\Omega^1 \cup\Omega^2$.
\end{St}

\begin{proof}
To prove the claim we establish that the restriction of~$B$ to every segment~$l \subset \Omega^1 \cup\Omega^2$ is concave. Obviously,~$l = \big(l\cap\Omega^1\big) \cup \big(l\cap\Omega^2\big)$. Each of the sets~$l\cap\Omega^1$ and~$l\cap\Omega^2$ is convex, i.e. they are either segments or empty sets (the latter case is obvious). By the hypothesis,~$B$ is concave on each of these segments. Using~$C^1$-smoothness of~$B$ in a common point of these segments, we get that~$B|_{l}$ is concave.  
\end{proof}
Surely, this proposition can be generalized in many ways, for us it is only a useful tool. Now we can state that a~$C^1$-smooth concatenation of two Bellman candidates is a Bellman candidate provided their domains are induced convex. 

We turn to building Bellman candidates. Usually, we will give only sufficient conditions for a foliation and a function~$f$ that generate a Bellman candidate. However, to be ready to construct the Bellman function, we have to examine all possible local Bellman candidates. So, the conditions we provide are usually also necessary. To make the story shorter, sometimes we will not prove this necessity, because we do not need it.

To describe combinatorial properties of foliations, we associate a special oriented graph with each foliation. Generally, its vertices correspond to the linearity domains, whereas its edges correspond to the domains of extremals. A vertex is incident to an edge if the corresponding two domains are adjacent. We postpone a more detailed description of the graph (concerning the orientation of edges and the numbers that correspond to different vertices and edges) to Subsection~\ref{s345}. 

\section{Tangent domains}\label{s32}\index{domain! tangent domain}
In this section we build Bellman candidates on \emph{tangent domains}, i.e. domains that consist of tangents to the free boundary. These domains can be formed either by the right tangents, or by the left ones (see Fig.~\ref{fig:tang}). The right tangents are those that lie on the right of their tangency points.
\begin{figure}[h!]
\includegraphics[height=3.8cm]{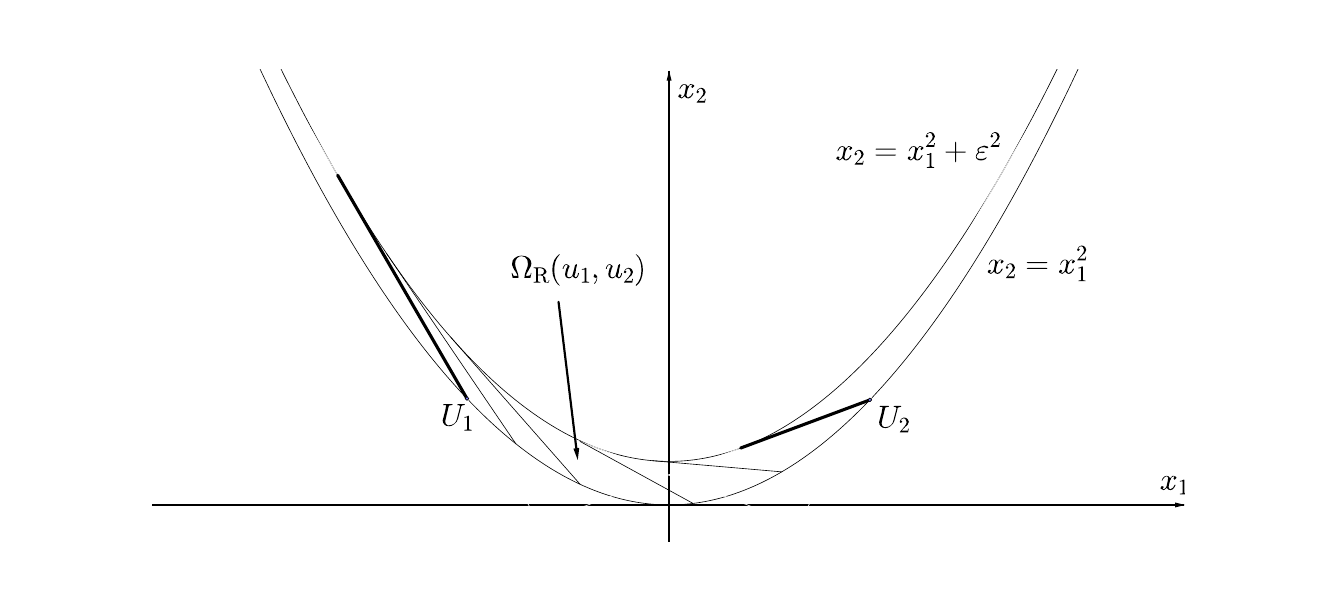}
\raisebox{-10pt}{\includegraphics[height=4.1cm]{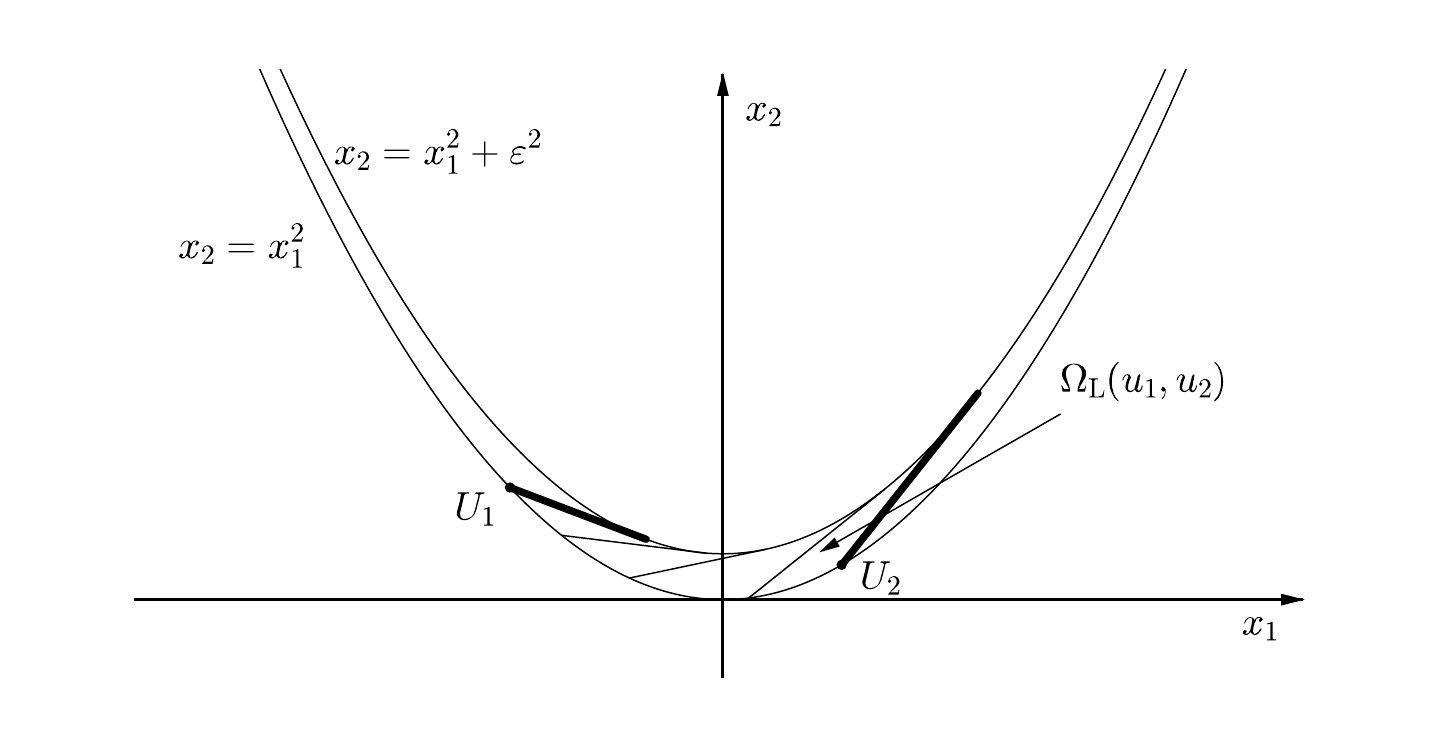}}
\caption{Domains~$\Rt$ and~$\Lt$ with right and left tangents.}
\label{fig:tang}
\end{figure}
%\begin{wrapfigure}{o}{250pt}
%\begin{center}
%\includegraphics[width = 1 \linewidth]{TangentsR.jpg}
%\caption{A domain $\Rt$ with the right tangents.}
%\label{fig:opk}
%\end{center}
%\end{wrapfigure}	

We introduce notation. Points on the lower parabola are denoted by capital letters, whereas their first coordinates are denoted by the corresponding small letters (e.g.~$A = (a,a^2)$). For each point~$x = (x_1,x_2)$ in~$\Omega_{\eps}$ there is exactly one right tangent passing through it. We denote the point where it intersects the lower boundary by~$\Ur = \Ur(x)$. Then 
\begin{equation}\label{ur}
		\ur = x_1 + \eps - \sqrt{x_1^2 + \eps^2 - x_2}.
\end{equation}
In the symmetric case, when we draw the left tangent through~$x$, the equation for~$\Ul = \Ul(x)$ and~$x$ becomes
\begin{equation}\label{ul}
		\ul = x_1 - \eps + \sqrt{x_1^2 + \eps^2 - x_2}.
\end{equation}
Now we can define the right tangent domain
\begin{equation*}%\label{Rt}
\Rt(u_1,u_2;\eps) \df \{x \in \Omega_{\eps} \mid \ur(x) \in [u_1,u_2]\}.
\end{equation*}
Similarly,
\begin{equation*}%\label{Lt}
\Lt(u_1,u_2;\eps) \df \{x \in \Omega_{\eps} \mid \ul(x) \in [u_1,u_2]\}.
\end{equation*}
We will often omit~$\eps$ in the notation for tangent domains. We turn to the description of the Bellman candidates on the domains introduced. The Bellman candidate is linear on each extremal (Definition~\ref{candidate}),
\begin{equation}\label{linearity}	
		B(x_1,x_2) = m(u)(x_1 - u) + f(u),
\end{equation}
where~$u$ is either~$\ur$ or~$\ul$ depending on the orientation of the tangent family. The coefficient~$m$ depends  on the tangent only (i.e. on~$u$).
We cite two propositions from~\cite{5A} that give sufficient conditions for the function~\eqref{linearity} to be a Bellman candidate.
%\begin{wrapfigure}[12]{i}{250pt}
%\includegraphics[width = 1 \linewidth]{TangentsL.jpg}
%\caption{A domain $\Lt$ with the left tangents.}
%\label{fig:olk}
%\end{wrapfigure}

\begin{St}\label{RightTangentsCandidate}
Suppose that the function~$m$ satisfies the differential equation
\begin{equation}\label{difeq}
\eps m'(u) + m(u) - f'(u) = 0
\end{equation}
and the inequality~$m''(u) \leq 0$ for~$u \in (u_1,u_2)$. Then the function~$B$ given by the formula~\eqref{linearity} is a Bellman candidate on~$\Rt(u_1,u_2)$.
\end{St}

\begin{proof}
Differentiating formula~\eqref{ur}, we get
\begin{equation*}
\frac{\partial u}{\partial x_2} = \frac{1}{2(x_1-u+\eps)},
\end{equation*}
thus, formula~\eqref{linearity} leads to
\begin{equation}\label{BSubX2Tangents}
\frac{\partial B}{\partial x_2} = \frac{m'(u)}{2} - \frac{\eps m'(u) + m(u) - f'(u)}{2(x_1 - u +\eps)} \stackrel{\scriptscriptstyle{\eqref{difeq}}}{=} \frac{m'(u)}{2}.
\end{equation}
So, we see that~$\frac{\partial B}{\partial x_2}$ is constant along the extremals. Since~$B$ is linear on each extremal and the extremals are not parallel to the~$x_2$-axis,~$\nabla B$ is constant along each extremal as well. It remains to notice that
\begin{equation*}
\frac{\partial^2 B}{\partial x_2^2} = \frac{m''(u)}{4(x_1 - u + \eps)} \leq 0.
\end{equation*} 
For every~$x\in\Rt(u_1,u_2)$, the matrix~$\frac{\partial^2 B}{\partial x^2} (x)$ is symmetric and has zero eigenvalue with the eigenvector collinear with the right tangent passing through~$x$. Since~$\frac{\partial^2 B}{\partial x_2^2}(x) \leq 0$ and the~$x_2$-direction is never collinear with the tangent, the second eigenvalue of~$\frac{\partial^2 B}{\partial x^2} (x)$ is non-positive, and thus~$\frac{\partial^2 B}{\partial x^2} (x) \leq 0$ as a matrix. We have proved that~$B$ is locally concave.
\end{proof}

\begin{Def}\label{StandardCandidateTangentRight}\index{standard candidate! in right tangent domain}
The function~$B$ constructed in Proposition~\textup{\ref{RightTangentsCandidate}} is called a {\it standard candidate} on the domain~$\Rt(u_1,u_2)$ if $m''<0$ on $(u_1,u_2)$.
\end{Def}

\begin{St}\label{LeftTangentsCandidate}
Suppose that the function~$m$ satisfies the differential equation
\begin{equation}\label{difeq2}
-\eps m'(u) + m(u) - f'(u) = 0
\end{equation}
and the inequality~$m''(u) \geq 0$ for~$u \in (u_1,u_2)$. Then the function~$B$ given by the formula~\eqref{linearity} is a Bellman candidate on~$\Lt(u_1,u_2)$.
\end{St}

The proof of this proposition is completely similar to the previous one.

\begin{Def}\label{StandardCandidateTangentLeft}\index{standard candidate! in left tangent domain}
The function~$B$ constructed in Proposition~\textup{\ref{LeftTangentsCandidate}} is called a {\it standard candidate} on the domain~$\Lt(u_1,u_2)$ if $m''>0$ on $(u_1,u_2)$.
\end{Def}

Differential equations~\eqref{difeq} and~\eqref{difeq2} have one-parameter families of solutions~$m$ for given~$f$. Namely, a solution of~\eqref{difeq} is given by
\begin{equation}\label{ExplicitFormulaForm}
m(u) = e^{-u/\eps}\bigg(e^{u_1/\eps}m(u_1) + \eps^{-1}\int\limits_{u_1}^u f'(t)e^{t/\eps}\,dt\bigg).
\end{equation}
The inequality~$m''(u) \leq 0$ can be rewritten as
\begin{equation}\label{mr''_firstformula}
m''(u) = e^{(u_1-u)/\eps}m''(u_1) + 
\eps^{-1}e^{-u/\eps}\int\limits_{u_1}^u e^{t/\eps}\,df''(t) \leq 0.
\end{equation}
We note that the constant terms in the previous two formulas are related by the equation
\begin{equation}\label{difeqSecondDer}
m''(u_1) = \frac{m(u_1) - f'(u_1) + \eps f''(u_1)}{\eps^2},
\end{equation}
which follows from~\eqref{difeq}.

Similarly, a solution of~\eqref{difeq2} is given by
\begin{equation}\label{ExplicitFormulaForm2}
  	m(u) = e^{u/\eps}\bigg(e^{-u_2/\eps}m(u_2) + \eps^{-1}\int\limits_{u}^{u_2} f'(t)e^{-t/\eps}\,dt\bigg).  
  \end{equation}
The inequality~$m''(u) \geq 0$ in this situation is
\begin{equation*}
  	m''(u) = e^{(u-u_2)/\eps}m''(u_2) + \eps^{-1}e^{u/\eps}\int\limits_{u}^{u_2} e^{-t/\eps}\,df''(t) \geq 0.
  \end{equation*}
Similarly, here
\begin{equation}\label{difeq2SecondDer}
m''(u_2) = \frac{m(u_2) - f'(u_2) - \eps f''(u_2)}{\eps^2}.
\end{equation}
Using limit relations provided by Lemma~\ref{emb} for the function~$f$ in the case of right tangents, one can get Bellman candidates on the domains~$\Rt(-\infty,u_2)$. 

\begin{St}\label{RightTangentsCandidateInfty}
Suppose that the function~$m$ is given by the formula
\begin{equation}\label{minfty}
m(u) = \eps^{-1}e^{-u/\eps}\int\limits_{-\infty}^u f'(t)e^{t/\eps}\,dt,
\end{equation}
and the inequality~$m''(u) \leq 0$ is fulfilled for~$u \in (-\infty,u_2)$. Then the function~$B$ given by formula~\eqref{linearity} is a Bellman candidate on~$\Rt(-\infty,u_2)$.
\end{St}

\begin{Def}\label{StandardCandidateTangentRightInfty}\index{standard candidate! in infinite right tangent domain}
The function~$B$ constructed in Proposition~\textup{\ref{RightTangentsCandidateInfty}} is called a {\it standard candidate} on the domain~$\Rt(-\infty,u_2)$ if $m''<0$ on $(-\infty,u_2)$.
\end{Def}

\begin{St}\label{LeftTangentsCandidateInfty}
Suppose that the function~$m$ is given by the formula
\begin{equation}\label{minfty2}
m(u) = \eps^{-1}e^{u/\eps}\int\limits_u^{+\infty} f'(t)e^{-t/\eps}\,dt,
\end{equation}
and the inequality~$m''(u) \geq 0$ is fulfilled for~$u \in (u_1,+\infty)$. Then the function~$B$ given by formula~\eqref{linearity} is a Bellman candidate on~$\Lt(u_1,+\infty)$.
\end{St}

\begin{Def}\label{StandardCandidateTangentLeftInfty}\index{standard candidate! in infinite left tangent domain}
The function~$B$ constructed in Proposition~\textup{\ref{LeftTangentsCandidateInfty}} is called a {\it standard candidate} on the domain~$\Lt(u_1,+\infty)$ if $m''>0$ on $(u_1,+\infty)$.
\end{Def}

In the case of Proposition~\ref{RightTangentsCandidateInfty} the inequality~$m''(u) \leq 0$ turns into
\begin{equation*}
m''(u) =  \eps^{-1}e^{-u/\eps}\int\limits_{-\infty}^u e^{t/\eps}\,df''(t) \leq 0.
\end{equation*}
In the case of Proposition~\ref{LeftTangentsCandidateInfty} the inequality~$m''(u) \geq 0$ turns into
\begin{equation*}
m''(u) = \eps^{-1}e^{u/\eps}\int\limits_{u}^{+\infty} e^{-t/\eps}\,df''(t) \geq 0.
\end{equation*}

\section{Around the cup}\label{s33}
As it was mentioned in Section~\ref{s31}, the extremals are of two types: chords and tangents. In Section~\ref{s32}, we dealt with the tangent domains. This section provides a study of \emph{chordal domains}, i.e. domains that consist of chords (see Fig.~\ref{fig:chd}).\index{domain! chordal domain}
\begin{figure}[!h]
\begin{center}
\includegraphics[width = 1 \linewidth]{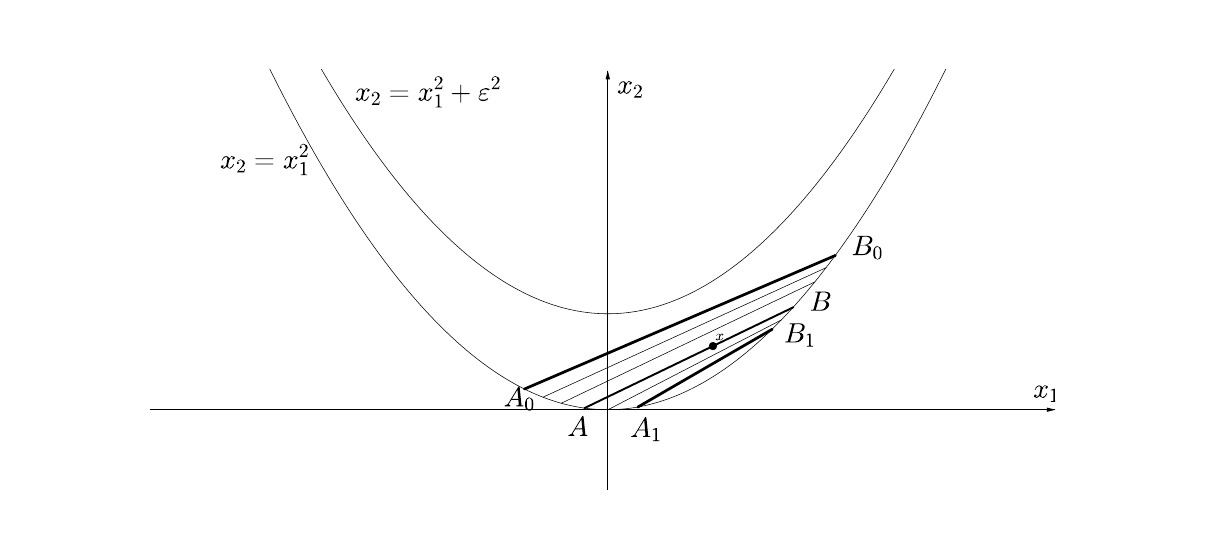}
\caption{Chordal domain $\Ch([a_0,b_0],[a_1,b_1])$.}
\label{fig:chd}
\end{center}
\end{figure}
  
\subsection{Chordal domain}\label{s331}
We introduce notation. Consider two chords,~$[A_0,B_0]$ and~$[A_1,B_1]$, that belong to~$\Omega_{\eps}$ entirely,~$a_0 < a_1 < b_1 < b_0$. Suppose that the whole parabolic quadrilateral~$A_0A_1B_1B_0$ is foliated by the chords that do not have common interior points and connect the points of its left side with the points of its right side (see Fig.~\ref{fig:chd}). In such a case, we call this quadrilateral a chordal domain and denote by~$\Ch([a_0,b_0],[a_1,b_1])$. A chordal domain with~$b_1=a_1$ is called a \emph{cup}\index{cup}, the point~$a_1$ is its \emph{origin}\index{cup! origin of a cup}.
Surely, for each interior point~$x$ of~$\Ch([a_0,b_0],[a_1,b_1])$ there is a unique chord of the family passing through~$x$. Denote the length of its projection onto the~$x$-axis by~$\ell(x)$, thus~$\ell\colon \Big(\Ch([a_0,b_0],[a_1,b_1]) \setminus \FixedBoundary\Omega_{\eps}\Big) \to [b_1 - a_1, b_0 - a_0]$. If the chords do not intersect even on the boundary, then this function is well defined on~$\Ch([a_0,b_0],[a_1,b_1])$.

On the other hand, for each~$l$ between~$b_1 - a_1$ and~$b_0 - a_0$ there is a unique chord of the family whose projection length equals~$l$. Denote this chord by~$[A(l),B(l)]$, where~$a\colon [b_1 - a_1, b_0 - a_0] \to [a_0,a_1]$ and~$b\colon [b_1 - a_1, b_0 - a_0] \to [b_1,b_0]$ (we remind the reader our notation~$P = (p,p^2)$). Surely,~$a$ is a non-increasing function,~$b$ is a non-decreasing function,~$b(l) - a(l) = l$, and~$\big[A\big(\ell(x)\big),B\big(\ell(x)\big)\big]$ is the chord that passes through~$x$. Though all this notation may seem bulky, it is rather useful for formalization. 

By Definition~\ref{candidate}, the Bellman candidate is linear along the extremals and satisfies the boundary conditions. Therefore, 
\begin{equation}\label{vallun}
		B(x)=\frac{f(b)-f(a)}{b-a}x_1+\frac{bf(a)-af(b)}{b-a},
	\end{equation}
where~$a = a\big(\ell(x)\big)$ and~$b = b\big(\ell(x)\big)$.

Assume that the foliation is sufficiently smooth. Namely, we suppose the functions~$a$ and~$b$ 
to be differentiable. We need two more definitions.

\begin{Def}\label{CupEquation}\index{equation! cup equation}
We say that a pair of points~$(a,b)$\textup,~$a < b$\textup, satisfies the cup equation 
for the function~$f$ if
\begin{equation}
		\label{urlun}
		\av{f'}{[a,b]}=\frac{f'(a)+f'(b)}{2}.
	\end{equation}
\end{Def}

\begin{Def}\label{differentials}\index{differentials}
Suppose that the pair~$(a,b)$ satisfies the cup equation~\eqref{urlun}. We call the following two expressions the differentials, the left and the right one correspondingly\textup:
	\begin{equation}\label{e334}
		\Slt(a,b) \df f''(a)-\av{f''}{[a,b]}\quad\textrm{and}\quad
		\Srt(a,b) \df f''(b)-\av{f''}{[a,b]}.
	\end{equation}
\end{Def}
We formulate a proposition from~\cite{5A} that gives sufficient conditions for the function~$B$ defined by formula~\eqref{vallun} to be a Bellman candidate on the domain~$\Ch([a_0,b_0],[a_1,b_1])$.

\begin{St}\label{LightChordalDomainCandidate}
Consider~$\Ch([a_0,b_0],[a_1,b_1])$ with the functions~$\ell$\textup,~$a$\textup, and~$b$ associated with it. Assume the following\textup:
\begin{itemize}
\item the functions~$a$ and~$b$ are differentiable on the interior of their domain\textup,~$a' < 0$\textup,~$b' > 0$\textup;
\item for each~$l \in [b_1 - a_1, b_0 - a_0]$ the pair~$(a(l),b(l))$ satisfies the cup equation~\eqref{urlun}\textup;
\item $\Slt\big(a(l),b(l)\big) \leq 0$,~$\quad\Srt\big(a(l),b(l)\big) \leq 0$.
\end{itemize}
Then the function~$B$ defined by formula~\eqref{vallun} is a Bellman candidate on~$\Ch([a_0,b_0],[a_1,b_1])$.
\end{St}
\begin{proof}
We begin with the equation of the chord~$[A,B]$: 
\begin{equation*}
x_2 = (a + b)x_1 - ab;\quad a = a\big(\ell(x)\big),\  b = b\big(\ell(x)\big).
\end{equation*}
Differentiating this equation with respect to~$x_2$ (keeping~$x_1$ fixed), we see that
\begin{equation*}
\frac{\partial \ell}{\partial x_2} = \frac{1}{(a' + b')x_1 - (a'b + ab')} = \frac{1}{a'(x_1 - b) + b'(x_1 - a)}.
\end{equation*}
Note that this value is finite since~$a' < 0$,~$b' > 0$. After some calculations involving the cup equation~\eqref{urlun} and formula~\eqref{vallun}, we see that
\begin{equation}\label{BsubX2Chords}
\frac{\partial B}{\partial x_2}(x) = \frac{f'(b) - f'(a)}{2(b-a)}; \quad a = a\big(\ell(x)\big),\  b = b\big(\ell(x)\big).
\end{equation}
As in the proof of Proposition~\ref{RightTangentsCandidate}, this shows that~$\nabla B$ is constant along each chord. Again, as it was explained in the proof of Proposition~\ref{RightTangentsCandidate}, it suffices to prove that~$\frac{\partial^2 B}{\partial x_2^2} \leq 0$. Since~$\ell$ is an increasing function of~$x_2$ (when~$x_1$ is fixed), this inequality is equivalent to~$\frac{\partial^2 B}{\partial x_2\partial \ell} \leq 0$, which can be easily computed:
\begin{equation*}
\frac{\partial^2 B}{\partial x_2\partial \ell} = \frac{b'\Srt(a,b) - a'\Slt(a,b)}{2(b-a)} \leq 0.
\end{equation*}
\end{proof}

\begin{Def}\label{StandardCandidateChordalDomain}\index{standard candidate! in chordal domain}
The function~$B$ given by formula~\eqref{vallun} on~$\Ch([a_0,b_0],[a_1,b_1])$ described in Proposition~\textup{\ref{LightChordalDomainCandidate}} is called a {\it standard candidate} on this domain if $\Slt(a(l),b(l))<0$ and $\Srt(a(l),b(l))<0$ for~$l \in (b_1-a_1,b_0-a_0)$.
\end{Def}

We note that the condition~$a' > 0$,~$b' > 0$ is not needed (in the proof above we used these inequalities to have the function~$\ell$ well defined and differentiable up to the boundary). Namely, one can prove the following.
\begin{St}\label{ChordalDomainCandidate}
Consider~$\Ch([a_0,b_0],[a_1,b_1])$ with the functions~$\ell$\textup,~$a$\textup, and~$b$ associated with it. Assume that\textup:
\begin{itemize}
\item the functions~$a$ and~$b$ are differentiable\textup;
\item for each~$l \in [b_1 - a_1, b_0 - a_0]$ the pair~$(a(l),b(l))$ satisfies the cup equation~\eqref{urlun}\textup;
\item $\Slt\big(a(l),b(l)\big) \leq 0$,~$\quad\Srt\big(a(l),b(l)\big) \leq 0$.
\end{itemize}
Then the function~$B$ defined by formula~\eqref{vallun} is a Bellman candidate on~$\Ch([a_0,b_0],[a_1,b_1])$.
\end{St}
However, we will not need this proposition, moreover it will follow from considerations in Section~\ref{s411} below.
The essential difference between the two propositions is that we forbid the extremals to intersect even on the boundary. 

We want to lighten the relationship between the cup equation~\eqref{urlun} and the differentials given by Definition~\ref{differentials}. For that purpose, we introduce the function~$\Phi\colon \mathbb{R}^2 \to \mathbb{R}$,
\begin{equation}\label{Phi}
\Phi(a,b) \df f'(a)+f'(b) - 2\av{f'}{[a,b]}.
\end{equation}
Surely, the cup equation is equivalent to~$\Phi(a,b) = 0$. The partial derivatives of~$\Phi$ on the set~$\{\Phi(a,b) = 0\}$ are expressed in terms of the differentials:
\begin{equation*}
d\Phi(a,b) = \Slt(a,b)\,da + \Srt(a,b)\,db.
\end{equation*}
Therefore, if at least one of the differentials is non-zero for the pair~$(a,b) = (a_0,b_0)$, the set~$\{(a,b) \in \mathbb{R}^2\mid \Phi(a,b) = 0\}$ is a one-dimensional submanifold of~$\mathbb{R}^2$ in a neighborhood of~$(a_0,b_0)$.
\begin{Le}
The relations
\begin{align}
\label{dDlt}
d\Slt(a,b)=df''(a)+\frac{2\Slt(a,b)}{b-a}da;
\\
\label{dDrt}
d\Srt(a,b)=df''(b)-\frac{2\Srt(a,b)}{b-a}db,
\end{align}
hold on the manifold~$\{\Phi(a,b) = 0\}$. 
\end{Le}
This lemma is Lemma~$6.1$ of~\cite{5A}, it is a straightforward calculation.

%\begin{wrapfigure}[14]{o}{200pt}
%\begin{center}
%\includegraphics[width = 1 \linewidth]{DifferentialsPicture.jpg}
%\caption{Differentials correspond to the marked angles.}
%\label{fig:diff}
%\end{center}
%\end{wrapfigure}

The cup equation and the differentials can be visualized. Indeed, consider the graph of the function~$f'$. The expression~$\av{f'}{[a,b]}$ is the subgraph area, whereas~$\frac{f'(a)+f'(b)}{2}$ is the trapezoid area (see Fig.~\ref{fig:area_differentials}). The cup equation states that these two areas are equal. Surely, we are talking about oriented areas (integrals), i.e. the part of the area gained below the~$x_1$-axis is calculated with the minus sign.

\begin{figure}[!h]
\begin{center}
\includegraphics[width = 0.54 \linewidth]{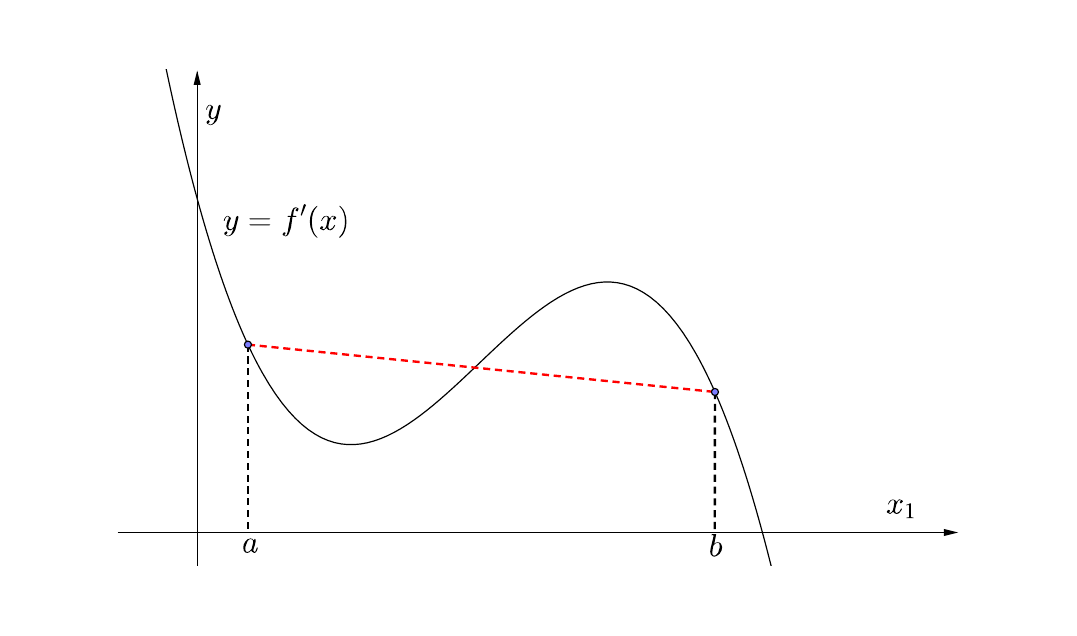}
\includegraphics[width = 0.45 \linewidth]{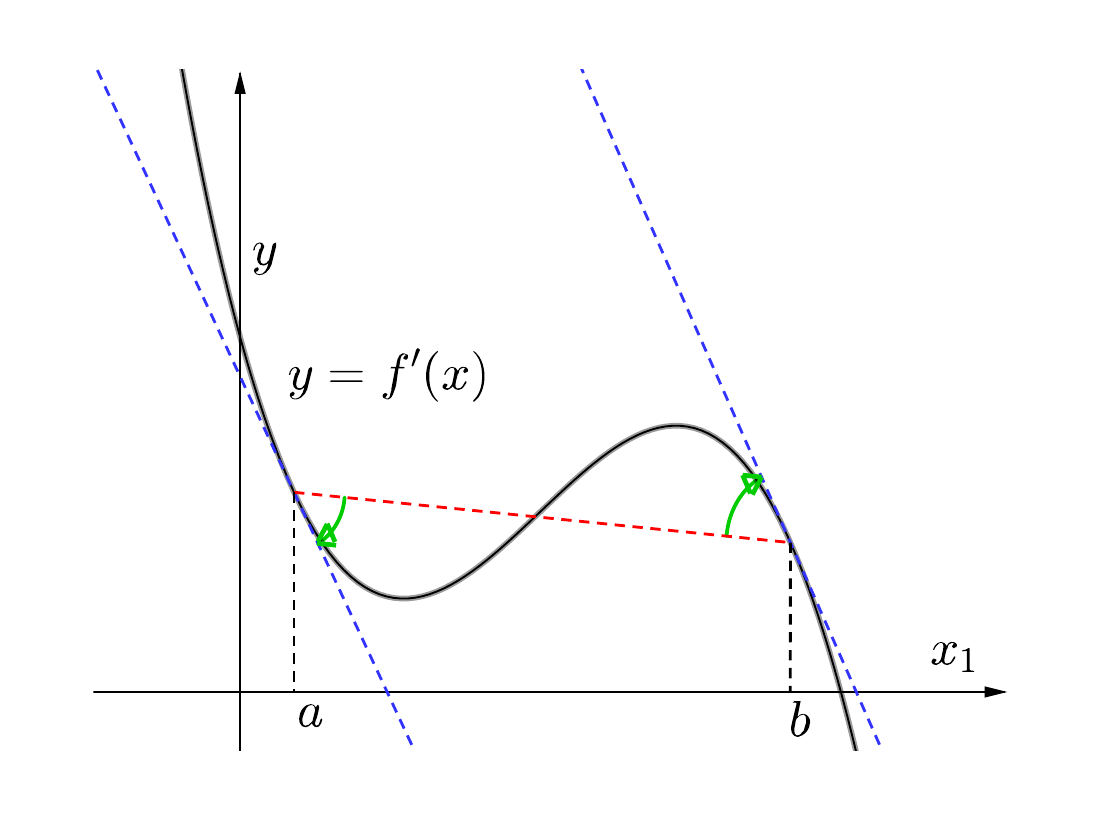}
\caption{Areas of the trapezoid and the subgraph must be equal. 
Differentials correspond to the marked angles.}
\label{fig:area_differentials}
\end{center}
\end{figure}

A line that passes through the points~$\big(s,f'(s)\big)$ and~$\big(t,f'(t)\big)$ is denoted by~$L_{s,t}$. By~$L_{s,s}$ we mean the tangent to~$\{(u,f'(u))\}$ at the point~$s$. The expression~$f''(a)$ is the slope of~$L_{a,a}$, whereas~$\av{f''}{[a,b]}$ is the slope of~$L_{a,b}$. Therefore, the left differential is the difference between these slopes. We are more concerned with the signs of the differentials than with themselves. It follows that the signs of the differentials coincide with the signs of the angles marked on Figure~\ref{fig:area_differentials}.

\subsection{Gluing chords with tangents}\label{s332}
We have studied both types of extremals. Now we investigate some foliations that contain the extremals of both types. Namely, consider a full chordal domain~$\Ch([a_0,b_0],*)$ (i.e.~$b_0 - a_0 = 2\eps$; the star denotes a chord whose name is unimportant) and the standard candidate~$\Bch$ on it. We want to glue the ``wings'' to the chordal domain (see Fig.~\ref{fig:cup_graph}). By this we mean the tangent domains~$\Rt(b_0,u_2)$ and~$\Lt(u_1,a_0)$ with standard candidates there. 

We first deal with the right tangent domain. We want to obtain a Bellman candidate on the union domain~$\Ch([a_0,b_0],*) \cup \Rt(b_0,u_2)$ from standard candidates on~$\Ch([a_0,b_0],*)$ and~$\Rt(b_0,u_2)$. Such a construction in a general setting was described in Section~\ref{s31}. We are going to use Proposition~\ref{ConcatenationOfConcaveFunctions}. We note that the domains we are dealing with are induced convex. 

We remind the reader that there is a one-parameter family of standard candidates on~$\Rt(b_0,u_2)$ given by Definition~\ref{StandardCandidateTangentRight}. The value~$m(b_0)$ (this is the slope of the Bellman candidate restriction to~$[A_0,B_0]$) plays the role of the parameter (see formula~\eqref{ExplicitFormulaForm}). We choose this value by the following rule:
\begin{equation}\label{m(b_0)}
m(b_0) = \frac{f'(a_0) + f'(b_0)}{2}.
\end{equation} 
Such a choice of~$m(b_0)$ guarantees the continuity of the concatenation. Luckily, this concatenation is also $C^1$-smooth. In fact, the proposition below is equivalent to Propostion~$5.3$ in~\cite{5A} (however, the latter is formulated for cups).

\begin{St}\label{CupMeetsRightTangents}
Consider~$\Ch([a_0,b_0],*) \cup \Rt(b_0,u_2)$\textup,~$b_0 - a_0 = 2\eps$. Let the function~$B$ coincide with the standard candidates on~$\Ch([a_0,b_0],*)$ and~$\Rt(b_0,u_2)$\textup{,} where the parameter~$m(b_0)$ is determined by~\eqref{m(b_0)}. Then $B$ is a $C^1$-smooth Bellman candidate on~$\Ch([a_0,b_0],*) \cup \Rt(b_0,u_2)$.
\end{St}

\begin{proof}
By Proposition~\ref{ConcatenationOfConcaveFunctions}, it suffices to prove that~$B$ is~$C^1$-smooth. Since the~$x_2$-axis is never parallel with a chord or a tangent, we may verify that~$\frac{\partial B}{\partial x_2}$ is a continuous function in a neighborhood of the common boundary of~$\Ch([a_0,b_0],*)$ and~$\Rt(b_0,u_2)$. Indeed,
\begin{equation*}
\frac{\partial B}{\partial x_2}\Big|_{\Rt(b_0,u_2)}(b_0,b_0^2) \stackrel{\scriptscriptstyle{\eqref{BSubX2Tangents}}}{=} \frac{m'(b_0)}{2} \stackrel{\scriptscriptstyle{\eqref{difeq}}}{=} \frac{f'(b_0) - m(b_0)}{2\eps} \stackrel{\scriptscriptstyle{\eqref{m(b_0)}}}{=} \frac{f'(b_0) - f'(a_0)}{2(b_0-a_0)} \stackrel{\scriptscriptstyle{\eqref{BsubX2Chords}}}{=}\frac{\partial B}{\partial x_2}\Big|_{\Ch([a_0,b_0],*)}(b_0,b_0^2).
\end{equation*}
\end{proof}

%\begin{Def}\label{StandardCandidateCupMeetsRightTangents}
%The function~$B$ described in Proposition~\textup{\ref{CupMeetsRightTangents}} is called a {\it standard candidate} on $\Ch([a_0,b_0],*) \cup \Rt(b_0,u_2)$.
%\end{Def}

\begin{St}\label{CupMeetsLeftTangents}
Consider~$\Lt(u_1,a_0)\cup \Ch([a_0,b_0],*)$\textup,~$b_0 - a_0 = 2\eps$. Let the function~$B$ coincide with the standard candidates on~$\Ch([a_0,b_0],*)$ and~$\Lt(u_1,a_0)$\textup{,} where the parameter~$m(a_0)$ is determined by
\begin{equation}\label{m(a_0)}
m(a_0) = \frac{f'(a_0) + f'(b_0)}{2}.
\end{equation}
Then $B$ is a $C^1$-smooth Bellman candidate on~$\Lt(u_1,a_0)\cup \Ch([a_0,b_0],*)$.
\end{St}

As usual, the proof of this proposition is completely symmetric.

%\begin{Def}\label{StandardCandidateCupMeetsLeftTangents}
%The function~$B$ described in Proposition~\textup{\ref{CupMeetsLeftTangents}} is called a {\it standard candidate} on $\Lt(u_1,a_0)\cup  \Ch([a_0,b_0],*)$.
%\end{Def}

The inequality~$m''(u) < 0$ required in Proposition~\ref{CupMeetsRightTangents} for the candidate~$B$ to be standard on~$\Rt(b_0,u_2)$ can be rewritten as 
\begin{equation}\label{InequalitiesForConcatenationRight}
m''(u) = \eps^{-1}\Srt(a_0,b_0)e^{(b_0-u)/\eps} + 
		\eps^{-1}e^{-u/\eps}\int\limits_{b_0}^u e^{t/\eps}\,df''(t) < 0, \quad u \in (b_0,u_2),
\end{equation}
with the help of equations~\eqref{m(b_0)},~\eqref{difeqSecondDer}, and~\eqref{mr''_firstformula}.
Similarly, in the case of Proposition~\ref{CupMeetsLeftTangents}, the inequality~$m''(u) > 0$ turns into
\begin{equation}\label{InequalitiesForConcatenationLeft}
	m''(u)	= -\eps^{-1}\Slt(a_0,b_0)e^{(u-a_0)/\eps} + 
		\eps^{-1}e^{u/\eps}\int\limits_{u}^{a_0} e^{-t/\eps}\,df''(t) > 0, \quad u \in (u_1,a_0).
	\end{equation}
We introduce one of the main ``heroes'' of the story: the \emph{force function}\index{force! function}.

\begin{Def}\label{Forces}
Let~$a_0$ and~$b_0$ be a pair of points satisfying the cup equation~\eqref{urlun}. Then the function~$\Fr$ given by the formula
\begin{equation}\label{RightForce}
\Fr(u; a_0, b_0; \eps) \df \Srt(a_0,b_0)e^{(b_0-u)/\eps} + 
		e^{-u/\eps}\int\limits_{b_0}^u e^{t/\eps}\,df''(t), \quad u \geq b_0,
\end{equation}
is the right force of~$(a_0,b_0)$. The left force is the function~$\Fl$ defined as
\begin{equation}\label{LeftForce}
\Fl(u; a_0, b_0; \eps) \df -\Slt(a_0,b_0)e^{(u-a_0)/\eps} + 
		e^{u/\eps}\int\limits_{u}^{a_0} e^{-t/\eps}\,df''(t), \quad u \leq a_0.
\end{equation}
There are also forces coming from the infinities\textup:
\begin{equation}\label{RightForceInfinity}
\Fr(u; -\infty; \eps) \df  
		e^{-u/\eps}\int\limits_{-\infty}^u e^{t/\eps}\,df''(t),\quad u \in \mathbb{R},
\end{equation}
\begin{equation}\label{LeftForceInfinity}
\Fl(u; \infty; \eps) \df  
		e^{u/\eps}\int\limits_{u}^{\infty} e^{-t/\eps}\,df''(t),\quad u \in \mathbb{R}.
\end{equation}
\end{Def}

Note that inequalities~\eqref{InequalitiesForConcatenationRight} and~\eqref{InequalitiesForConcatenationLeft} turn into
\begin{equation*}
\Fr(u;a_0,b_0;\eps) < 0,\ u \in (b_0,u_2), \quad \hbox{and} \quad \Fl(u;a_0,b_0;\eps) > 0,\ u \in (u_1,a_0),
\end{equation*}
correspondingly. In a similar way, the inequalities in Definitions~\ref{StandardCandidateTangentRightInfty} and~\ref{StandardCandidateTangentLeftInfty} turn into
\begin{equation*}
\Fr(u;-\infty;\eps) < 0,\ u \in (-\infty,u_2), \quad \hbox{and} \quad \Fl(u;+\infty;\eps) > 0,\ u \in (u_1,+\infty).
\end{equation*}

Though the expressions for the forces are well defined for arbitrary pairs of points~$(A_0, B_0)$, when we write a force concerning such a pair, we always assume that the pair~$(a_0,b_0)$ satisfies the cup equation. We study differential properties of forces. First,
\begin{equation}\label{DifForcePoU}
\frac{\partial \Fr}{\partial u} = f''' - \frac{\Fr}{\eps}; \quad
\frac{\partial \Fl}{\partial u} = -f''' + \frac{\Fl}{\eps}.
\end{equation}
These equations should be understood in the distributional sense. The second claim concerns differentiation with respect to the second and the third arguments.  

\begin{Le}\label{DifForcePoL}
Let~$\Ch([a_0,b_0],[a_1,b_1])$ be a chordal domain satisfying the assumptions of Proposition~\textup{\ref{LightChordalDomainCandidate}}. Then for all~$l \in (b_1 - a_1, b_0 - a_0)$ we have the following\textup:
\begin{equation}\label{Fpoell}
\begin{aligned}
\frac{\partial}{\partial l}\bigg[\Fl\big(u; a(l), b(l); \eps\big)\bigg]=
\phantom{-}\Big(\frac{2}{l}-\frac{1}{\eps}\Big)\frac{\Slt\Srt}{\Slt+\Srt}
e^{-(a(l)-u)/\eps};\\
\frac{\partial}{\partial l}\bigg[\Fr\big(u; a(l), b(l); \eps\big)\bigg] = -\Big(\frac{2}{l}-\frac{1}{\eps}\Big)\frac{\Slt\Srt}{\Slt+\Srt}e^{-(u-b(l))/\eps}.
\end{aligned}
\end{equation} 
Here the differentials take their values at the points~$a(l)$ and~$b(l)$.
\end{Le}

\begin{proof}
Let us prove the formula for the left force, the one for the right is similar. By the definition of forces and formula~\eqref{dDlt},
\begin{equation*}
\frac{\partial}{\partial l}\bigg[\Fl\big(u; a(l), b(l); \eps\big)\bigg] = \Slt(a(l),b(l))\Big(\frac{1}{\eps} - \frac{2}{l}\Big)e^{-\frac{a(l)-u}{\eps}}a'(l).
\end{equation*}
It remains to prove that
\begin{equation*}
a' = -\frac{\Srt}{\Slt + \Srt},
\end{equation*}
which is a consequence of the trivial identities~$b'-a' = 1$ and~$a'\Slt + b'\Srt = 0$ (the second identity follows by differentiation of the cup equation~\eqref{urlun}).
\end{proof}

%\begin{wrapfigure}[14]{o}{200pt}
%\begin{center}
%\includegraphics[width = 1\linewidth]{CupMeetsTangentsGraph.jpg}
%\caption{A graph for a full chordal domain.}
%\label{fig:ctg}
%\end{center}
%\end{wrapfigure}
This lemma is a particular case of Lemma~$6.12$ in~\cite{5A}. The next statement deals with differentiation with respect to~$\eps$. For a moment, we let chordal domains fall out the parabolic strip. Surely, when we were working with chordal domains, we did not need the upper boundary, therefore, such an assumption does not break all the results concerning chordal domains. The last lemma of this subsection is the first moment when a chordal domain interacts with the upper boundary.

\begin{Le}\label{difequal}
Let~$\Ch([a_0,b_0],[a_1,b_1])$ be a chordal domain satisfying the assumptions of Proposition~\textup{\ref{LightChordalDomainCandidate}}. Then\textup, for all~$\eps \in (\frac{b_1 - a_1}{2}, \frac{b_0 - a_0}{2})$
\begin{equation*}
\frac{\partial}{\partial \eps}\bigg[F\big(u; a(l), b(l); \eps\big)\bigg]\Bigg|_{l = 2 \eps}= \frac{\partial}{\partial \eps}\bigg[F\big(u; a(l), b(l); \eps\big)\Big|_{l = 2 \eps}\bigg],
\end{equation*}
where $F$ stands either for $\Fr$ or $\Fl$. 
\end{Le}

\begin{proof}
Surely,
\begin{equation*} 
\frac{\partial}{\partial \eps}\bigg[F\big(u; a(2\eps), b(2\eps); \eps\big)\bigg] =
2\frac{\partial}{\partial l}\bigg[F\big(u; a(l), b(l); \eps\big)\bigg]\Bigg|_{l = 2 \eps} + 
\frac{\partial}{\partial \eps}\bigg[F\big(u; a(l), b(l); \eps\big)\bigg]\Bigg|_{l = 2 \eps}.
\end{equation*}
The first summand is zero by formula~\eqref{Fpoell}, so, the lemma is proved.
\end{proof}

Finally, we give a graphical representation of Propositions~\ref{CupMeetsRightTangents} and~\ref{CupMeetsLeftTangents}. We draw a fictious vertex that corresponds to the ``linearity domain'' that consists of the single chord~$[A_0,B_0]$ that has three outcoming edges representing the chordal domain and the tangent domains. See Figure~\ref{fig:cup_graph}. Numbers written on vertices and edges correspond to the horizontal length of domains, and the direction of edges correspond to the direction of tangent domains. In particular, on Figure~\ref{fig:cup_graph}, the vertex with the number~$2\eps$ visualizes the long chord, and the edges with the numbers~$a-u_1$ and~$u_2-b$ indicate the tangent domains~$\Lt(u_1,a)$ and~$\Rt(b,u_2)$.  
\begin{figure}[!h]
\hskip10pt
\includegraphics[width = 0.55 \linewidth]{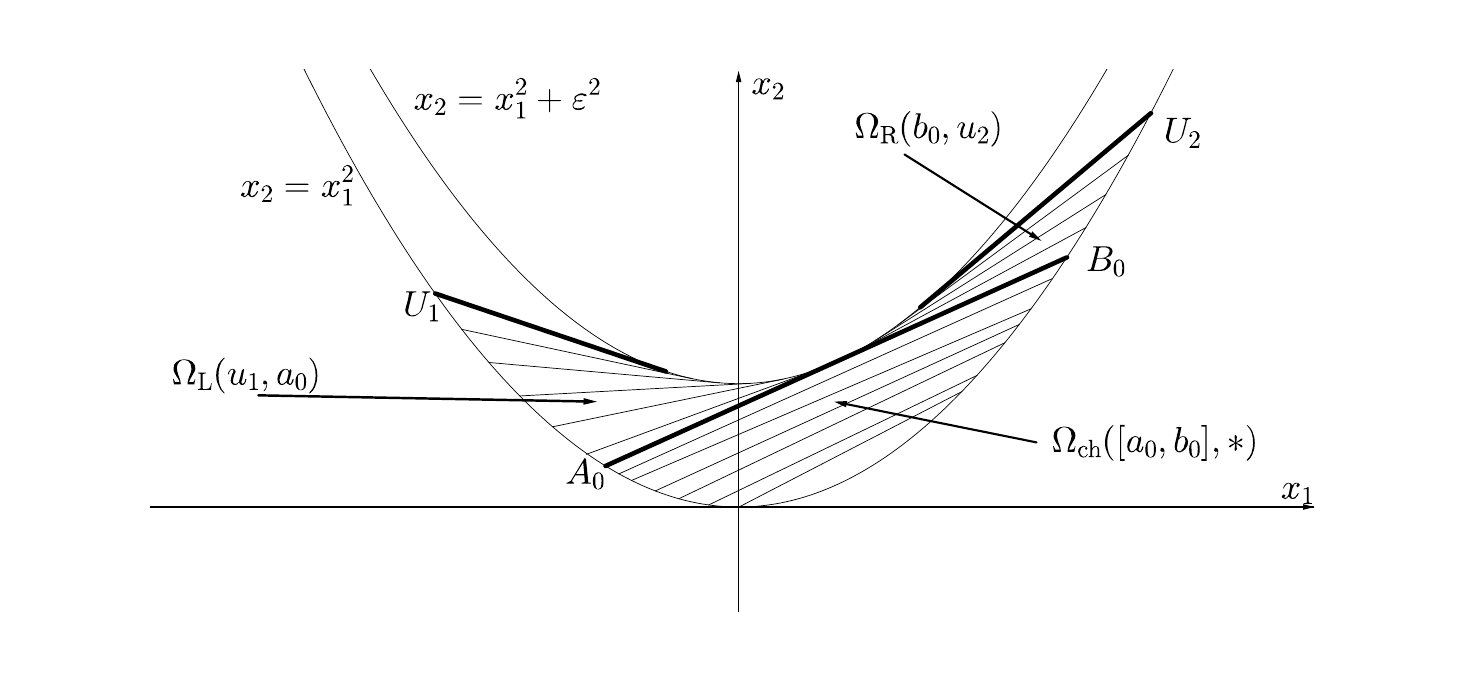}
\hskip-40pt
\raisebox{-10pt}{\includegraphics[width = 0.6\linewidth]{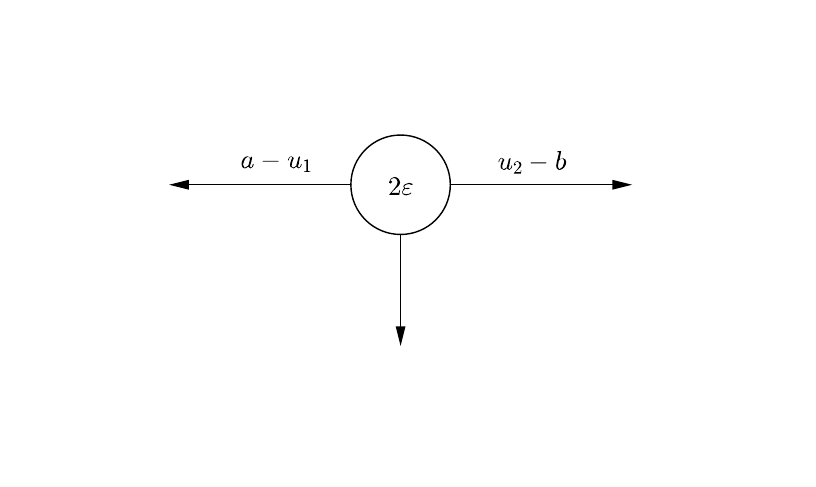}}
\vskip-10pt
\caption{A chordal domain $\Ch([a_0, b_0],*)$ with tangent domains attached to it and the corresponding graph.}
\label{fig:cup_graph}
\end{figure}

\subsection{Examples}\label{s333}
\paragraph{Exponential function.}
Consider the case~$f(t)= e^t$, which corresponds to the integral John--Nirenberg inequality~\eqref{intJN}, and was considered in~\cite{SlVa}. First, we have to verify Conditions~\ref{reg} and~\ref{sum}. The first one is satisfied, because~$f'''$ does not have any root. The second one is equivalent to the inequality~$\eps_{\infty} < 1$. What is more, it is easy to see that the Bellman function is everywhere infinite for~$f(t) = e^t$ and~$\eps \geq 1$. So, we take~$\eps_{\infty}$ to be some number less than~$1$.

The function~$f'''$ is everywhere positive, so we can use Proposition~\ref{LeftTangentsCandidateInfty} and construct a Bellman candidate on the domain~$\Lt(-\infty,\infty)$ given by formulas~\eqref{linearity} and~\eqref{minfty2}. We do not prove that it coincides with the Bellman function, this will be done in Subsection~\ref{s521}. For the explicit formulas for this Bellman function, see~\cite{SlVa}.

For the case~$f(t) = -e^t$, we can use Proposition~\ref{RightTangentsCandidateInfty} to construct a Bellman candidate on the domain~$\Rt(-\infty,\infty)$. This candidate also coincides with the Bellman function.

\paragraph{Polynomial of third degree.}
Consider the case where~$f$ is a polynomial of third degree. Surely, such a function satisfies Conditions~\ref{reg} and~\ref{sum}. We see that~$f'''$ is a constant. If this constant is positive, then the situation falls under the scope of Proposition~\ref{LeftTangentsCandidateInfty}. In the opposite case~$f'''$ is negative, we can use Proposition~\ref{RightTangentsCandidateInfty}. These Bellman candidates coincide with the true Bellman functions. For more details, see~\cite{5A}.

\paragraph{A chordal domain sitting on a solid root.}
This example illustrates the notion of a solid root given in Definition~\ref{solidroot}. Consider the function 
$$
f(t)=\begin{cases}
-(|t|-1)^3\quad&\text{for }|t|\geq1,
\\
0&\text{for }|t|<1.
\end{cases}
$$
It is not difficult to see that for each~$\eps > 0$ the foliation~$\Lt(-\infty,-\eps) \cup \Ch([-\eps,\eps],[0,0]) \cup \Rt(\eps,\infty)$ provides a Bellman candidate. Let us prove this claim. We start from the chordal domain. The foliation is symmetric, thus~$a(l) = -\frac{l}{2}$,~$b(l) = \frac{l}{2}$. It is not difficult to verify that such a foliation satisfies the assumptions of Proposition~\ref{LightChordalDomainCandidate}. We note that~$f'''(t) \leq 0$ when~$t \geq 0$ and~$f'''(t) \geq 0$ when~$t \leq 0$. Therefore, the functions built by Proposition~\ref{CupMeetsRightTangents} on~$\Rt(\eps,\infty)$ and by Proposition~\ref{CupMeetsLeftTangents} on~$\Lt(-\infty, -\eps)$, are Bellman candidates. Indeed, the corresponding right force function is non-positive, whereas the corresponding left force function is non-negative (see formulas~\eqref{RightForce} and~\eqref{LeftForce}).

We will see in Subsection~\ref{s343} that it is convenient to treat the part~$\Ch([-1,1],[0,0])$ of the cup as a special linearity domain.

\section{Linearity domains}\label{s34}\index{domain! linearity domain}
\subsection{Angle}\label{s341}
As it was stated in Section~\ref{s31}, we classify the linearity domains by the number of points on the lower parabola. The first linearity domain we study, an \emph{angle}\index{angle}, has only one point~$W$ on the lower parabola. See Figure~\ref{fig:ang} to get the idea. 
We define the angle at~$w$ by the formula
\begin{equation*}
\Ang(w;\eps) = \big\{x \in \mathbb{R}^2 \mid w-\eps \leq x_1 \leq w+\eps,\; 2wx_1-w^2+2\eps|w-x_1|\leq x_2 \leq x_1^2+\eps^2\big\}.
\end{equation*}
The point~$W$ is called the vertex of~$\Ang(w)$ (we will often omit~$\eps$ in the notation for angles). We suppose that a function~$B$ coincides with the Bellman candidate on~$\Ang(w)$ and is~$C^1$-smooth at~$W$. The function~$B$ is linear on~$\Ang(w)$, therefore, 
\begin{equation}\label{ValAng}
 B(x) = \beta_0+  \beta_1 x_1 + \beta_2 x_2, \quad x \in \Ang(w),
\end{equation}
for some reals~$\beta_i$,~$i=0,1,2$. We have two conditions on the coefficients~$\beta$. First, the boundary condition
$B(W)=f(w)$. Second, by the~$C^1$-smoothness assumption, the tangent vector $(1,2w,f'(w))$ to the curve $\{(t,t^2,f(t))\,|\; t \in \mathbb{R}\}$ at $w$ lies in the plane~$\{(x_1, x_2, \beta_0+  \beta_1 x_1 + \beta_2 x_2)\,|\; x_1,x_2 \in \mathbb{R}\}$. Therefore, 
\begin{equation}\label{betaequations}
\begin{aligned}
\beta_1 =& f'(w) - 2\beta_2 w; \\
\beta_0 =& f(w) - wf'(w)+\beta_2 w^2.
\end{aligned}
\end{equation} 
So, we have the following family of linear functions:  
\begin{equation}\label{StandardCandidateAngleFormula}
B(x)= f(w) - wf'(w)+\beta_2 w^2+ (f'(w) - 2\beta_2 w)x_1+ \beta_2 x_2
\end{equation}
parametrized by~$\beta_2 \in \mathbb{R}$.

\begin{Def}\label{StandardCandidateAngle}\index{standard candidate! in angle}
The function~$B$ defined by formula~\eqref{StandardCandidateAngleFormula} in~$\Ang(w)$ is called a {\it standard candidate} there. 
\end{Def}

We suppose that an angle $\Ang(w)$ is adjacent to a right tangent domain~$\Rt(u_1,w)$ from the right. Our aim is to glue a Bellman candidate on~$\Rt(u_1,w)\cup\Ang(w)$ from standard candidates on~$\Rt(u_1,w)$ and~$\Ang(w)$. The continuity of the glued function implies that the parameters~$\mrt(w)$ of the standard candidate on~$\Rt(u_1,w)$ and~$\beta_2$ of the standard candidate on $\Ang(w)$ satisfy the equation
\begin{equation}\label{AngleMeetRightTangentsFormula}
\mrt(w) = f'(w)-2\eps \beta_2. 
\end{equation}

\begin{St}\label{AngleMeetRightTangents}
%Consider~$\Rt(u_1,w) \cup \Ang(w)$. 
Let the function~$B$ coincide with standard candidates on~$\Rt(u_1,w)$ and~$\Ang(w)$. If the corresponding parameters satisfy  relation~\eqref{AngleMeetRightTangentsFormula}\textup{,} then~$B$ is a~$C^1$-smooth Bellman candidate on the domain~$\Rt(u_1,w) \cup \Ang(w)$.
\end{St} 

\begin{proof}
By Proposition~\ref{ConcatenationOfConcaveFunctions}, it suffices to prove that~$B$ is~$C^1$-smooth. Since the~$x_2$-axis is never parallel to a tangent, it suffices to verify that~$\frac{\partial B}{\partial x_2}$ is a continuous function in a neighborhood of the common boundary of~$\Ang(w)$ and~$\Rt(u_1,w)$. This is a simple calculation:
\begin{equation*}
\frac{\partial B}{\partial x_2}\Big|_{\Rt(u_1,w)}(w,w^2) \stackrel{\scriptscriptstyle{\eqref{BSubX2Tangents}}}{=} \frac{\mrt'(w)}{2} \stackrel{\scriptscriptstyle{\eqref{difeq}}}{=} \frac{f'(w) - \mrt(w)}{2\eps} \stackrel{\scriptscriptstyle{\eqref{AngleMeetRightTangentsFormula}}}{=} \beta_2 \stackrel{\scriptscriptstyle{\eqref{StandardCandidateAngleFormula}}}{=}\frac{\partial B}{\partial x_2}\Big|_{\Ang(w)}.
\end{equation*}
\end{proof}

A similar proposition for a left tangent domain is symmetric.
\begin{St}\label{AngleMeetLeftTangents}
%Consider~$\Ang(w) \cup \Lt(w,u_2)$. 
Let the function~$B$ coincide with standard candidates on~$\Ang(w)$ and~$\Lt(w,u_2)$. If the corresponding parameters satisfy the relation $\mlt(w) = f'(w)+2\eps \beta_2$\textup{,} then~$B$ is a~$C^1$-smooth Bellman candidate on~$\Ang(w) \cup \Lt(w,u_2)$.
\end{St} 

% surrounded by tangent domains. Our aim is to glue a Bellman candidate on~$\Rt(u_1,w)\cup\Ang(w)\cup\Lt(w,u_2)$ from the candidates on each of the domains~$\Rt(u_1,w)$,~$\Ang(w)$, and~$\Lt(w,u_2)$. The Bellman candidate on~$\Ang(w)$ is linear, therefore 
%\begin{equation}\label{ValAng}
% B(x) = \alpha_0+  \alpha_1 x_1 + \alpha_2 x_2, \quad x \in \Ang(w),
%\end{equation}
%for some reals~$\alpha_i$,~$i=0,1,2$.
%If~$\mrt$ and~$\mlt$ are the corresponding coefficients in formula~\eqref{linearity} for the Bellman candidates on~$\Rt(u_1,w)$ and~$\Lt(w,u_2)$, then
%\begin{equation}\label{bang}	
%	\begin{split}
%		\alpha_1 &= \frac{\mrt(w)+\mlt(w)}{2} - \frac{\mlt(w)-\mrt(w)}{2\eps}\,w;\\
%		\alpha_2 &= \frac{\mlt(w)-\mrt(w)}{4\eps};\\
%		\alpha_0 &= \frac{\mlt(w)-\mrt(w)}{4\eps}\,w^2 - \frac{\mrt(w)+\mlt(w)}{2}\,w + f(w).
%	\end{split}
%	\end{equation}
%This follows from the continuity of the candidate (we compare the values given by formulas~\eqref{ValAng} and~\eqref{linearity} at the points~$W$,~$((w-\eps),(w-\eps)^2 + \eps^2)$, and~$((w+\eps),(w+\eps)^2 + \eps^2)$).
The following proposition is a straightforward corollary of Propositions~\ref{AngleMeetRightTangents} and~\ref{AngleMeetLeftTangents} (one can find its analog in~\cite{5A}, which is Proposition~$4.1$ in that paper).

\begin{St}\label{AngleProp}
%Consider~$\Rt(u_1,w)\cup \Ang(w) \cup \Lt(w,u_2)$. 
Let the function~$B$ coincide with standard candidates on~$\Rt(u_1,w)$\textup,~$\Ang(w)$\textup, and~$\Lt(w,u_2)$. If the corresponding parameters satisfy the relation 
\begin{equation*}
\mrt(w)+\mlt(w) = 2f'(w),
\end{equation*}
and the parameter~$\beta_2$ is determined by 
\begin{equation}\label{beta2}
\beta_2 = \frac{\mlt(w)-\mrt(w)}{4\eps},
\end{equation}
then~$B$ is a~$C^1$-smooth Bellman candidate on~$\Rt(u_1,w)\cup \Ang(w) \cup \Lt(w,u_2)$.\
\end{St}

%\begin{wrapfigure}[16]{i}{200pt}
%\begin{center}
%\includegraphics[width = 1 \linewidth]{AngleMeetsTangents.jpg}
%\caption{A graph for an angle.}
%\label{fig:amt}
%\end{center}
%\end{wrapfigure} 
For purposes of further development, it is more convenient to rewrite the relation~$\mrt(w)+\mlt(w) = 2f'(w)$ in terms of forces. %For that we introduce a new notion. 
\begin{Def}\index{equation! balance equation}\label{BalanceEquationDefinition}
Suppose~$\Fr$ and~$\Fl$ to be the right and left forces of the two chords\textup,~$[A_1,B_1]$ and~$[A_2,B_2]$ \textup(or forces coming from the infinities\textup) correspondingly\textup,~$b_1 \leq a_2$. The equation
\begin{equation*}
\Fr(u) + \Fl(u) = 0
\end{equation*}
for the point~$u$ in~$[b_1,a_2]$ is called the balance equation for the forces~$\Fr$ and~$\Fl$. 
\end{Def}
\begin{figure}[!h]%{o}{250pt}
\includegraphics[width = 0.55 \linewidth]{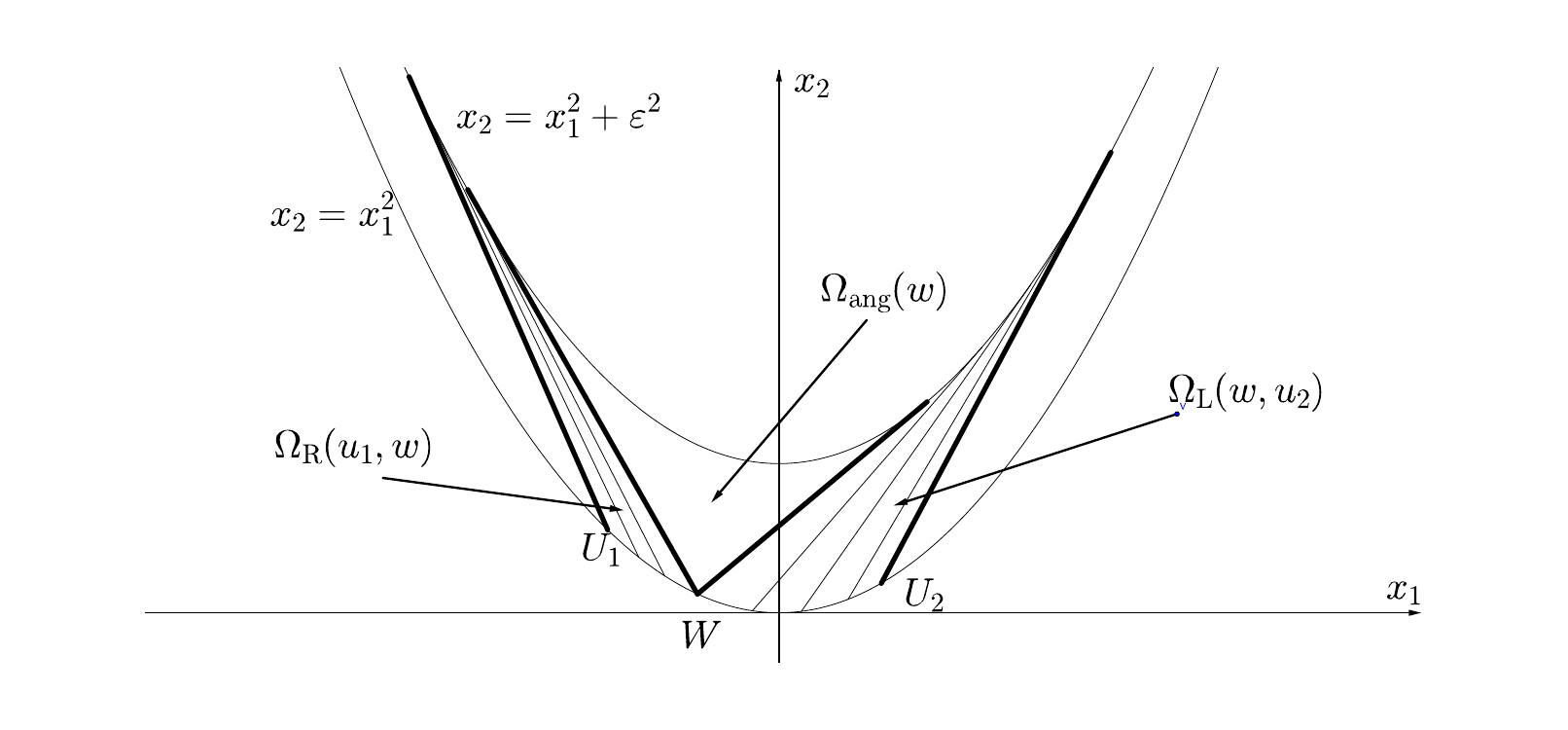}
\hskip-50pt
\raisebox{-10pt}{\includegraphics[width = 0.6 \linewidth]{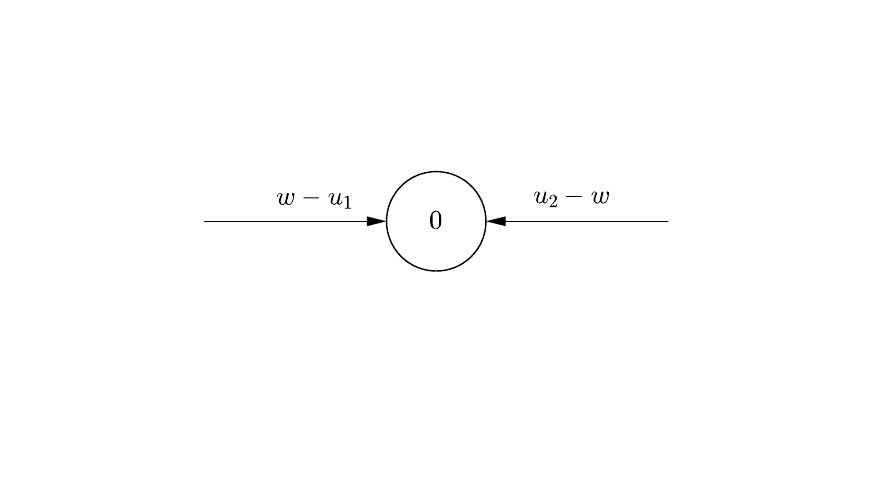}}
\vskip-10pt
\caption{An angle $\Ang(w)$ with adjacent domains and their graph.}
\label{fig:ang}
\end{figure}
The situation described above is represented graphically on Figure~\ref{fig:ang}. The vertex corresponds to the angle~$\Ang(w)$, it has two incoming edges representing the tangent domains.

The balance equation is equivalent to the relation~$\mrt(w)+\mlt(w) = 2f'(w)$. Indeed, by formulas~\eqref{difeq} and~\eqref{difeq2}, this condition is equivalent to~$\mrt''(w)+\mlt''(w) = 0$. Suppose that there are some full chordal domains~$\Ch([u_1 - 2\eps,u_1],*)$ and~$\Ch([u_2,u_2 + 2\eps],*)$ that are attached to~$\Rt(u_1,w)$ and~$\Lt(w,u_2)$ correspondingly in the sense described in Subsection~\ref{s332}. 
Then~$\Fr(u) = \eps\mrt''(u)$, where~$\Fr(u) = \Fr(u;u_1-2\eps,u_1;\eps)$, and~$\Fl(u) = \eps\mlt''(u)$, where~$\Fl(u) = \Fl(u;u_2,u_2 + 2\eps;\eps)$ (compare~\eqref{InequalitiesForConcatenationRight}, \eqref{InequalitiesForConcatenationLeft} with the formulas from Definition~\ref{Forces}). All these leads to the following proposition.
\begin{St}\label{AngleBetweenChords}
Suppose that $\Ch([a_1,b_1],*)$ and~$\Ch([a_2,b_2],*)$\textup, $b_1 \leq a_2$\textup, are full chordal domains. %satisfying Proposition~\textup{\ref{LightChordalDomainCandidate}}. 
Suppose there exists a point~$w$\textup,~$w \in [b_1,a_2]$\textup, that is a root of the balance equation for~$\Fr(\cdot\,;a_1,b_1;\eps)$ and~$\Fl(\cdot\,;a_2,b_2;\eps)$\textup, i.e.
\begin{equation*}
\Fr(w;a_1,b_1;\eps) + \Fl(w;a_2,b_2;\eps) = 0.
\end{equation*}
%We also suppose the inequalities~$\Fr(u) < 0$\textup,~$u \in (b_1,w)$\textup, and~$\Fl(u) > 0$\textup,~$u \in (w,a_2)$\textup, to be valid. 
Let the continuous function~$B$ coincide with the standard candidates on~$\Rt(b_1,w)$\textup, $\Lt(w,a_2)$,\textup, $\Ch([a_2,b_2],*)$\textup, and~$\Ang(w)$ with the parameter $\beta_2$ given by~\eqref{beta2}. Then~$B$ is a~$C^1$-smooth Bellman candidate on
%on defined by formula~\eqref{vallun} inside the chordal domains\textup, by formulas~\eqref{linearity} and~\eqref{difeq} in~$\Rt(b_1,w)$\textup, by formulas~\eqref{linearity} and~\eqref{difeq2} in~$\Lt(w,a_2)$\textup, and by formulas~\eqref{ValAng} and~\eqref{bang} in~$\Ang(w)$, is the Bellman candidate on the union domain
\begin{equation*}
\Ch([a_1,b_1],*) \cup \Rt(b_1,w) \cup \Ang(w) \cup \Lt(w,a_2) \cup \Ch([a_2,b_2],*).
\end{equation*}
\end{St}

\begin{Rem}
The necessary and sufficient conditions for the existence of the standard candidates on the domains~$\Ch([a_1,b_1],*)\cup \Rt(b_1,w)$ and $\Lt(w,a_2)\cup \Ch([a_2,b_2],*)$ are~$\Fr(u;a_1,b_1;\eps) < 0$\textup,~$u \in (b_1,w)$\textup, and~$\Fl(u;a_2,b_2;\eps) > 0$\textup,~$u \in (w,a_2)$\textup, respectively.
\end{Rem}

This situation gives an example of a more interesting graph, see Figure~\ref{fig:abtcd}.
\begin{figure}[!h]%[13]{o}{300pt}
\begin{center}
\vskip-20pt
\includegraphics{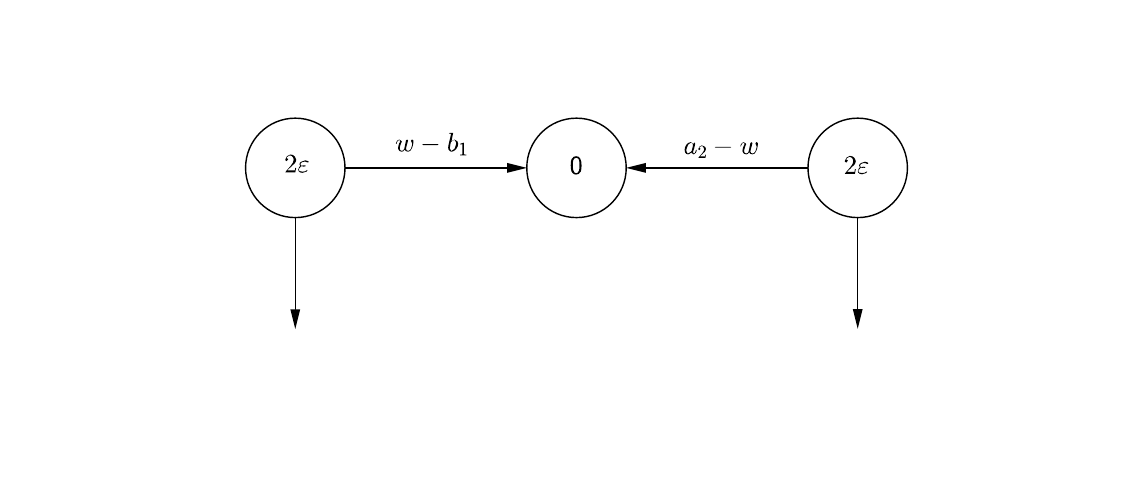}
\vskip-20pt
\caption{A graph for an angle between two full chordal domains.}
\label{fig:abtcd}
\vskip-20pt
\end{center}
\end{figure}
\begin{Rem}
Any of the two chordal domains in Proposition~\textup{\ref{AngleBetweenChords}} may be replaced by the corresponding infinity\textup, i.e. one may consider~$\Rt(-\infty,w)$ instead of~$\Ch([a_1,b_1],*) \cup \Rt(b_1,w)$ and~$\Lt(w,\infty)$ instead of~$\Lt(w,a_2) \cup \Ch([a_2,b_2],*)$. 
\end{Rem}

The corresponding graphs are presented on Figure~\ref{fig:abtcd_infty}.
\begin{figure}[!h]
\begin{center}
\vskip-20pt
\hskip-40pt
\includegraphics[width = 0.6 \linewidth]{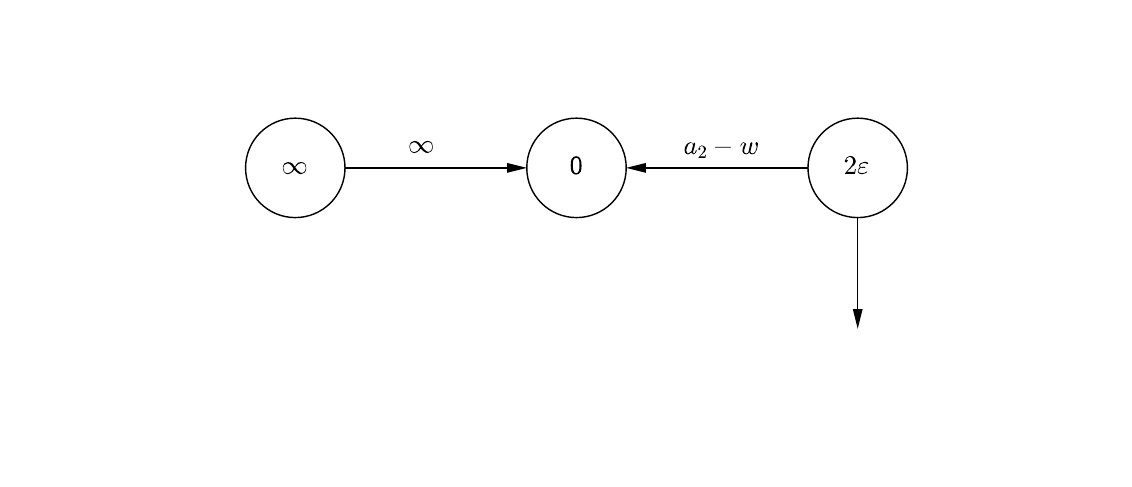}
\hskip-80pt
\includegraphics[width = 0.6 \linewidth]{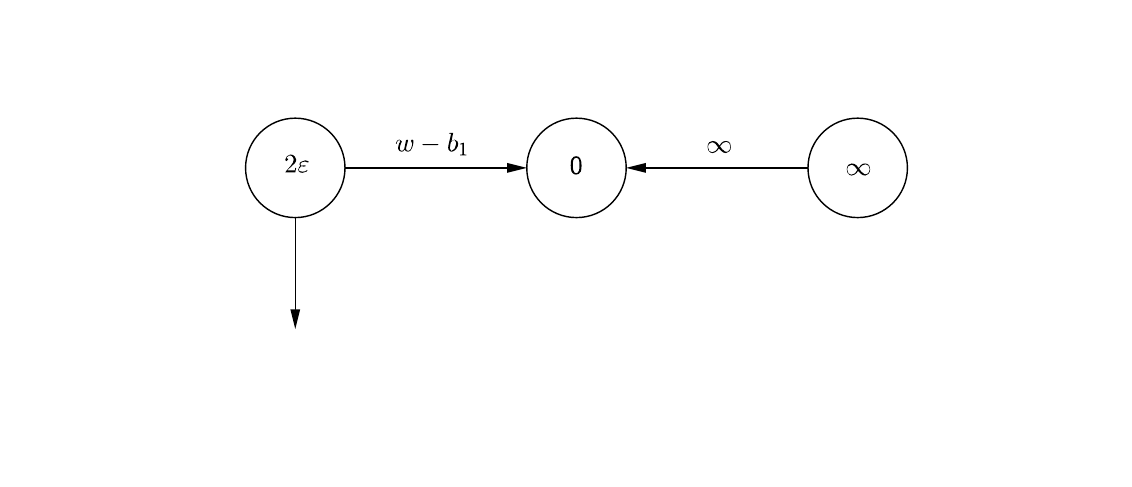}
\vskip-20pt
\caption{Graphs for an angle between a full chordal domain and infinity.}
\label{fig:abtcd_infty}
\vskip-20pt
\end{center}
\end{figure}

\subsection{Linearity domains with two points on the lower boundary}\label{s342}
Consider a linearity domain~$\mathfrak{L}$ that has two points~$A_0$ and~$B_0$ on the lower boundary. Surely, the segment~$[A_0,B_0]$ is a part of the boundary for the linearity domain. It is natural to assume that there are two extremals touching the free boundary, ending at~$A_0$ and~$B_0$, and bounding our linearity domain from the left and right. If they have the same orientation, then the linearity domain is called a~\emph{trolleybus}\index{trolleybus} (see Figure~\ref{fig:ptr}). The right trolleybus is denoted by~$\RTroll(a_0,b_0;\eps)$, and~$\LTroll(a_0,b_0;\eps)$ stands for the left one.
\begin{figure}[!h]%{o}{250pt}
\begin{center}
\includegraphics[width = 0.49 \linewidth]{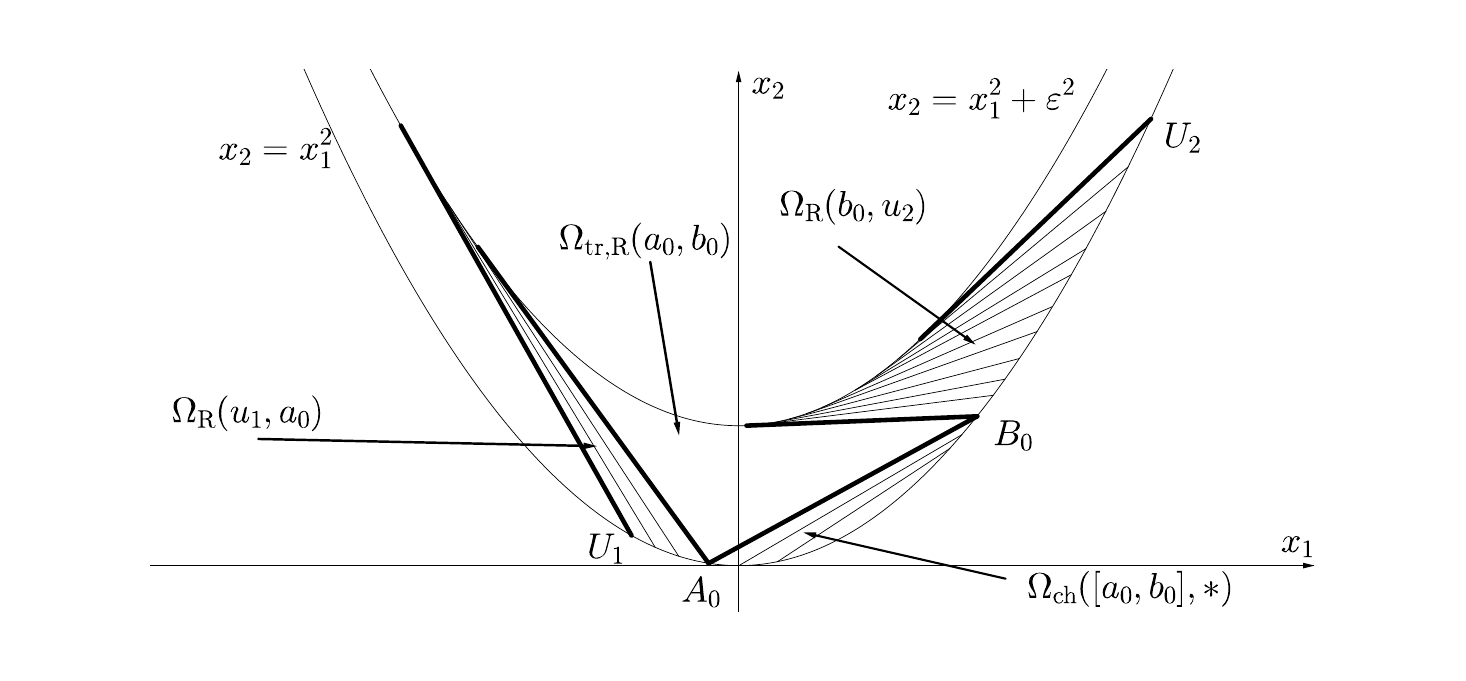}
\includegraphics[width = 0.49 \linewidth]{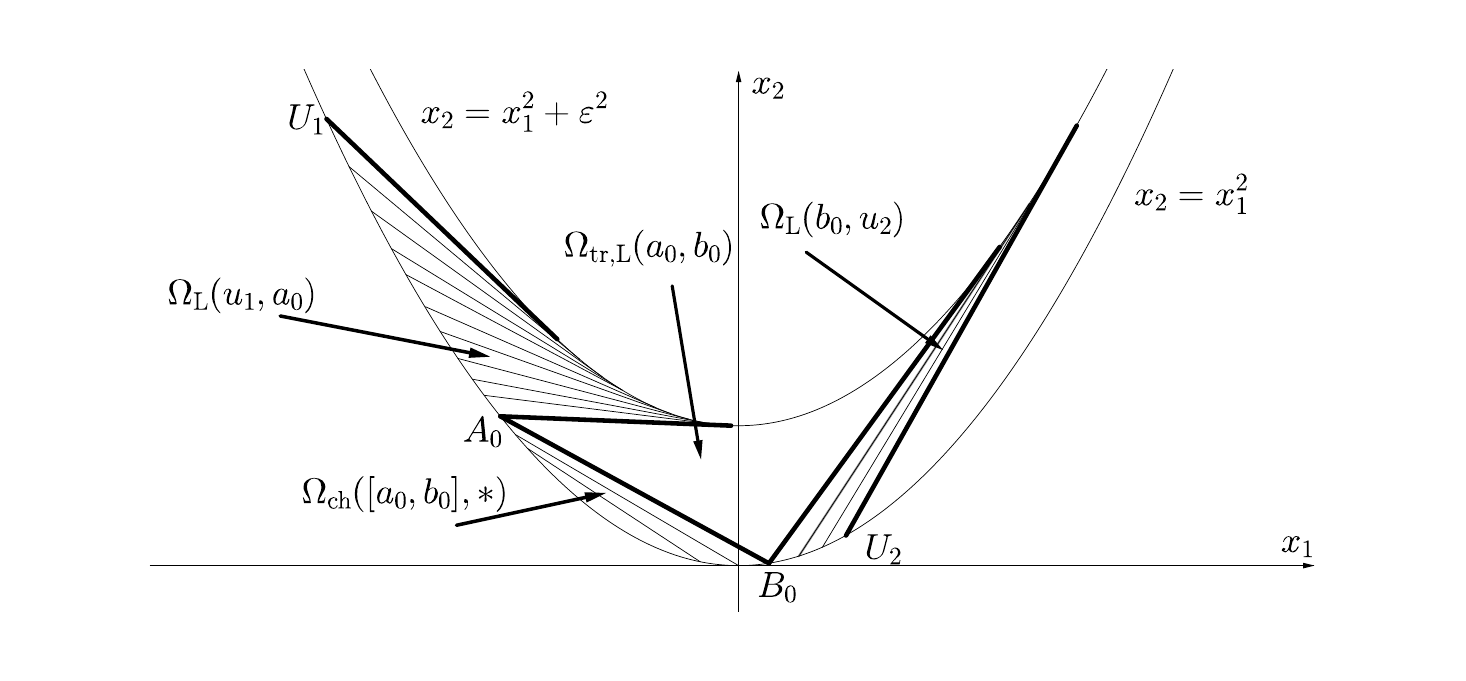}
\caption{Trolleybuses~$\RTroll$ and~$\LTroll$ and their adjacent domains.}
\label{fig:ptr}
\label{fig:ltr}
%\label{fig:gtr}
\end{center}
\end{figure}
The linearity domain whose border tangents have different orientation is called a~\emph{birdie}\index{birdie}, see Figure~\ref{fig:birdie}. We denote it by~$\Bird(a_0,b_0;\eps)$. We will often omit~$\eps$ in the notation above. 

As in the case of an angle, we look for a Bellman candidate~$B$ linear in~$\mathfrak{L}$:
\begin{equation}\label{LinearInTrolleybus}
B(x)=\beta_0+\beta_1 x_1 + \beta_2 x_2.
\end{equation}
We have a pair of conditions on the parameters~$\beta$ at each of the points~$A_0$ and~$B_0$ (see~\eqref{betaequations}):
\begin{align}
%\begin{aligned}
\beta_1 =& f'(a_0) - 2\beta_2 a_0 =  f'(b_0) - 2\beta_2 b_0; \label{betaequations1}\\
\beta_0 =& f(a_0) - a_0f'(a_0)+\beta_2 a_0^2=  f(b_0) - b_0f'(b_0)+\beta_2 b_0^2. \label{betaequations2}
%\end{aligned}
\end{align} 
Solving the pair of equations~\eqref{betaequations1} we find
\begin{equation}\label{coef1}
\begin{split}
&\beta_2 = \frac{f'(b_0)-f'(a_0)}{2(b_0-a_0)}=\frac{1}{2}\av{f''}{[a_0,b_0]};\\
&\beta_1 = \frac{f'(b_0)+f'(a_0)}{2}-\frac{1}{2}(b_0+a_0)\av{f''}{[a_0,b_0]}.
\end{split}
\end{equation} 
The compatibility condition for the pair of equations~\eqref{betaequations2} is equivalent to the cup equation~\eqref{urlun} for the pair~$(a_0,b_0)$. Thus we can write down the following:
\begin{equation}\label{coef}
\begin{split}
&\beta_0=\frac{b_0 f(a_0)-a_0 f(b_0)}{b_0-a_0}+\frac{1}{2}a_0 b_0\av{f''}{[a_0,b_0]};\\
&\beta_1 = \av{f'}{[a_0,b_0]} - \frac{1}{2}(b_0+a_0)\av{f''}{[a_0,b_0]};\\
&\beta_2 = \frac{1}{2}\av{f''}{[a_0,b_0]}.
\end{split}
\end{equation} 

\begin{Def}\index{standard candidate! in trolleybus}\index{standard candidate! in birdie}
The function~$B$ defined by formulas~\eqref{LinearInTrolleybus} and~\eqref{coef} in the linearity domain~$\mathfrak{L}$ with two points on the fixed boundary is called a \emph{standard candidate} there. 
\end{Def}

So, we get the standard candidates in trolleybuses and birdies.

\begin{St}\label{DomainOfLinearityWithTwoPointsOnTheLowerBoundaryMeetsChordalDomain}
Let~$\mathfrak{L}$ be a linearity domain with the points $A_0$ and $B_0$\textup, $a_0<b_0\leq a_0+2\eps$\textup, on the fixed boundary. Let the function~$B$ coincide with the standard candidates on~$\mathfrak{L}$ and on~$\Ch([a_0,b_0],*)$. Then~$B$ is a~$C^1$-smooth Bellman candidate on~$\mathfrak{L}\cup \Ch([a_0,b_0],*)$.
\end{St}

\begin{proof}
In order to prove the~$C^1$-smoothness of the function~$B$, we verify the continuity of~$B_{x_2}$ comparing~\eqref{BsubX2Chords} with the formula for $\beta_2$ in~\eqref{coef}.  
\end{proof}

\begin{St}\label{DomainOfLinearityWithTwoPointsOnTheLowerBoundaryMeetsTangentDomain}
Let~$\mathfrak{L}$ be a linearity domain with the points $A_0$ and $B_0$\textup, $a_0<b_0\leq a_0+2\eps$\textup, on the fixed boundary. Let~$W$ be one of the points~$A_0$ and~$B_0$. Let the function~$B$ coincide with the standard candidates on~$\mathfrak{L}$ and on a tangent domain adjacent to~$\mathfrak{L}$ along a tangent ending at~$W$. If~$B$ is continuous\textup, then it is a~$C^1$-smooth Bellman candidate on the union of the domains.
\end{St}

\begin{proof}
For the case of the right tangent domain, the continuity of the function~$B$ implies $\mrt(w) = f'(w)-2\eps \beta_2$, which in its turn yields the continuity of $B_{x_2}$ (see the proof of Propositions~\ref{AngleMeetRightTangents}).
For the case of the left tangent domain we come to the same result using the relation $\mlt(w) = f'(w)+2\eps \beta_2$.
\end{proof}

% satisfies equality
%we are interested in sufficient conditions that allow to glue the linearity domain with the extremals surrounding it. We begin with the right trolleybus. There is a chordal domain below~$[A_0,B_0]$ and two tangent domains,~$\Rt(u_1,a_0)$ and~$\Rt(b_0,u_2)$, around~$\RTroll$. We are looking for the Bellman candidate on the union domain
%\begin{equation*}
%\Rt(u_1,a_0) \cup \RTroll(a_0,b_0) \cup \Ch([a_0,b_0],*) \cup \Rt(b_0,u_2).
%\end{equation*}
%The Bellman candidate is linear in~$\RTroll(a_0,b_0)$:
%\begin{equation}\label{LinearInTrolleybus}
%B(x)=\beta_1 x_1 + \beta_2 x_2 + \beta_0.
%\end{equation}
%Using continuity of the Bellman candidate (and formula~\eqref{linearity} inside~$\Rt(u_1,a_0)$), we obtain:
%\begin{equation}\label{e8}
%	\left\{ 
%	\begin{aligned}
%		&\beta_1 a_0 + \beta_2 a_0^2 + \beta_0 = f(a_0);\\
%		&\beta_1 b_0 + \beta_2 b_0^2 + \beta_0 = f(b_0);\\
%		&m(a_0) = \beta_1 + 2(a_0-\eps) \beta_2,
%	\end{aligned} \right.
%\end{equation}
%Here~$m$ is the slope coefficient, defined by formula~\eqref{linearity}, for the tangent domain~$\Rt(u_1,a_0)$. As for the continuity on the common boundary between~$\RTroll(a_0,b_0)$ and~$\Rt(b_0,u_2)$, it is equivalent to the equation~$m(b_0) = \beta_1+2(b_0 - \eps)\beta_2$, where~$m$ is the slope coefficient in~$\Rt(b_0,u_2)$.

%\begin{wrapfigure}[15]{i}{270pt}
%\begin{center}
%\includegraphics[width = 1 \linewidth]{TrolleybusL.jpg}
%\caption{A left trolleybus $\LTroll$.}
%\label{fig:ltr}
%\end{center}
%\end{wrapfigure}
We obtain Proposition~$6.1$ from~\cite{5A} that gives sufficient conditions for a function in a neighborhood of a trolleybus to be a Bellman candidate.

\begin{St}\label{S15}
Let $u_1<a_0<b_0<u_2$ and let $b_0-a_0\le 2\eps$. Let the function~$B$ coincide with the standard candidates on~$\Ch([a_0,b_0],*)$\textup, $\Rt(u_1,a_0)$\textup, $\RTroll(a_0,b_0)$\textup, and $\Rt(b_0,u_2)$. Suppose that 
\begin{align}
\mrt(a_0) = f'(a_0) - \eps \av{f''}{[a_0,b_0]}\label{mrta0};\\ 
\mrt(b_0) = f'(b_0) - \eps \av{f''}{[a_0,b_0]}\label{mrtb0}.
\end{align}
%\begin{equation}\label{RTrEq}
%\eps m''(a_0) + \Fl(a_0; a_0, b_0; \eps) = 0.
%\end{equation}
Then\textup, the function~$B$ is a~$C^1$-smooth Bellman candidate on
$$
\Rt(u_1,a_0) \cup \RTroll(a_0,b_0) \cup \Ch([a_0,b_0],*) \cup \Rt(b_0,u_2).
$$
%\textup, and the linearity domain
%. The function given by formula~\textup{\eqref{vallun}} in~$\Ch([a_0,b_0])$\textup,
%by formulas~\textup{\eqref{linearity}} and~\textup{\eqref{difeq}} in~$\Rt(u_1,a_0)$\textup, by formulas~\eqref{LinearInTrolleybus} and~\eqref{e8} in~$\RTroll(a_0,b_0)$\textup, and by formulas~\textup{\eqref{linearity},}~\textup{\eqref{difeq},} and~$m(b_0) =\beta_1 + 2(b_0 - \eps)\beta_2$ in~$\Rt(b_0,u_2)$\textup, is a Bellman candidate if the chordal domain satisfies the hypothesis of Proposition~\textup{\ref{LightChordalDomainCandidate}}\textup, the inequality~$m'' \leq 0$ is valid on both tangent domains\textup, and
\end{St}

\begin{proof}
Relations~\eqref{mrtb0} and~\eqref{mrta0} guarantee the continuity of~$B$. Proposition~\ref{DomainOfLinearityWithTwoPointsOnTheLowerBoundaryMeetsChordalDomain} 
and~Proposition~\ref{DomainOfLinearityWithTwoPointsOnTheLowerBoundaryMeetsTangentDomain} imply the~$C^1$-smoothness of~$B$.
\end{proof}

We state a symmetric proposition for the left trolleybus.
\begin{St}\label{S16}
Let $u_1<a_0<b_0<u_2$ and let $b_0-a_0\le 2\eps$. Let the function~$B$ coincide with the standard candidates on~$\Ch([a_0,b_0],*)$\textup, $\Lt(u_1,a_0)$\textup, $\LTroll(a_0,b_0)$\textup, and $\Lt(b_0,u_2)$. Suppose that 
\begin{align}
\mlt(a_0) = f'(a_0) + \eps \av{f''}{[a_0,b_0]}\label{mlta0};\\ 
\mlt(b_0) = f'(b_0) + \eps \av{f''}{[a_0,b_0]}\label{mltb0}.
\end{align}
%\begin{equation}\label{RTrEq}
%\eps m''(a_0) + \Fl(a_0; a_0, b_0; \eps) = 0.
%\end{equation}
Then\textup, the function~$B$ is a~$C^1$-smooth Bellman candidate on
$$
\Lt(u_1,a_0) \cup \LTroll(a_0,b_0) \cup \Ch([a_0,b_0],*) \cup \Lt(b_0,u_2).
$$
%\textup, and the linearity domain
%. The function given by formula~\textup{\eqref{vallun}} in~$\Ch([a_0,b_0])$\textup,
%by formulas~\textup{\eqref{linearity}} and~\textup{\eqref{difeq}} in~$\Rt(u_1,a_0)$\textup, by formulas~\eqref{LinearInTrolleybus} and~\eqref{e8} in~$\RTroll(a_0,b_0)$\textup, and by formulas~\textup{\eqref{linearity},}~\textup{\eqref{difeq},} and~$m(b_0) =\beta_1 + 2(b_0 - \eps)\beta_2$ in~$\Rt(b_0,u_2)$\textup, is a Bellman candidate if the chordal domain satisfies the hypothesis of Proposition~\textup{\ref{LightChordalDomainCandidate}}\textup, the inequality~$m'' \leq 0$ is valid on both tangent domains\textup, and
\end{St}
\begin{figure}[!h]
\vskip-30pt
\begin{center}
\includegraphics{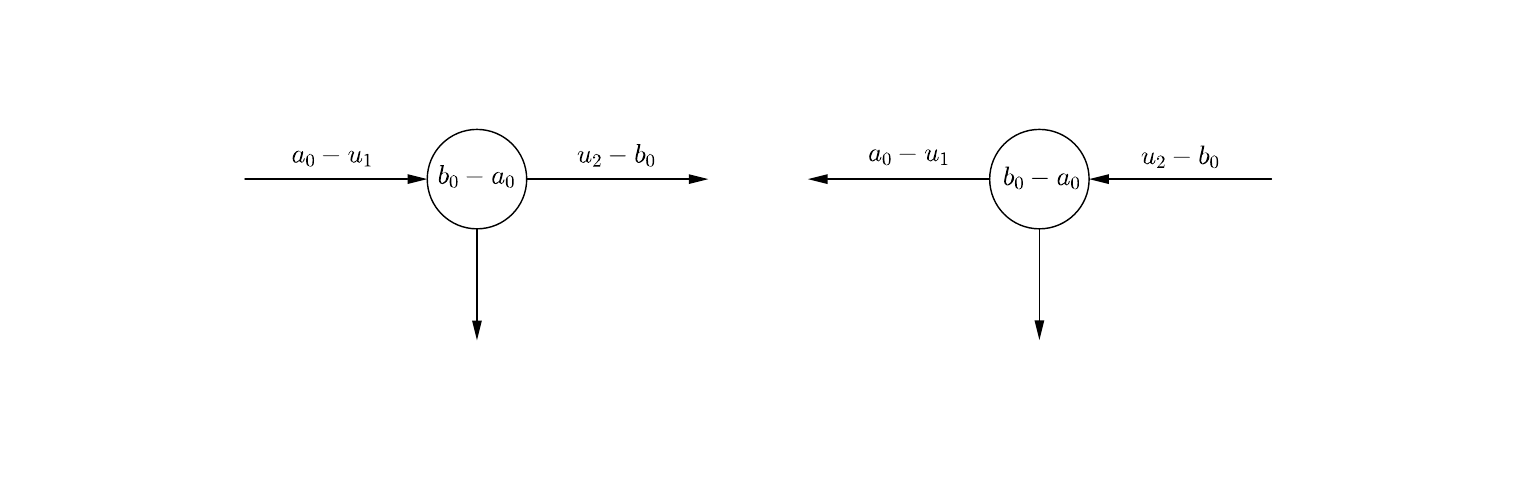}
\vskip-40pt
\caption{Graphs for the trolleybuses with the adjacent domains.}
\label{fig:gtr}
\end{center}
\end{figure}
Both types of trolleybuses are represented graphically on Figure~\ref{fig:gtr}.

%\begin{St}\label{S16}
%Let $u_1<a_0<b_0<u_2$ and let $b_0-a_0\le2\eps$. Consider the chordal domain 
%$\Ch([a_0,b_0],*)$\textup, two domains $\Lt(u_1,a_0)$ and $\Lt(b_0,u_2)$\textup, and the linearity domain
%$\LTroll(a_0,b_0)$ located between them all. A function given by formula~\eqref{vallun} in~$\Ch([a_0,b_0])$\textup, by formulas~\eqref{LinearInTrolleybus} and~\eqref{coef} in~$\LTroll(a_0,b_0)$\textup, and
%by formulas~\eqref{linearity} and~\eqref{difeq2} in~$\Lt(b_0,u_2)$ and~$\Lt(u_1,a_0)$ \textup(with the boundary data given by the formula~$m(a_0) = \beta_1 + 2(a_0+\eps)\beta_2$ in~$\Lt(u_1,a_0)$\textup)\textup, is a Bellman candidate if the chordal domain satisfies the hypothesis of Proposition~\textup{\ref{LightChordalDomainCandidate}}\textup, the inequality~$m'' \geq 0$ is valid on both tangent domains\textup, and
%\begin{equation*}
%\eps m''(b_0) + \Fr(b_0; a_0, b_0; \eps) = 0.
%\end{equation*}
%\end{St}
We turn to the birdie (see Figure~\ref{fig:bgr} for the graphical representation). 

\begin{figure}[!h]%{o}{250pt}
\vskip-10pt
\hskip10pt
\includegraphics[width = 0.6\linewidth]{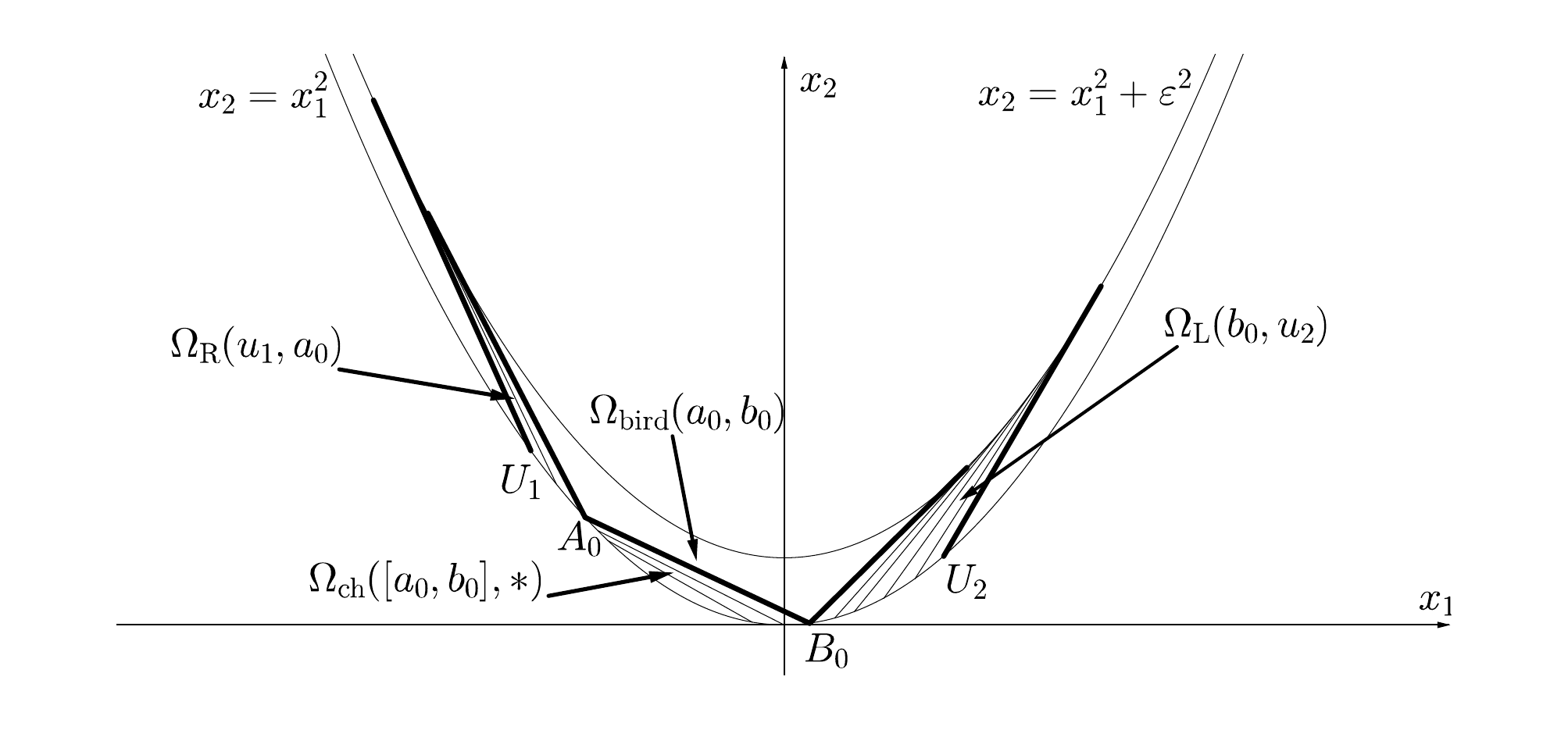}
\hskip-50pt
\raisebox{-10pt}{\includegraphics[width = 0.6\linewidth]{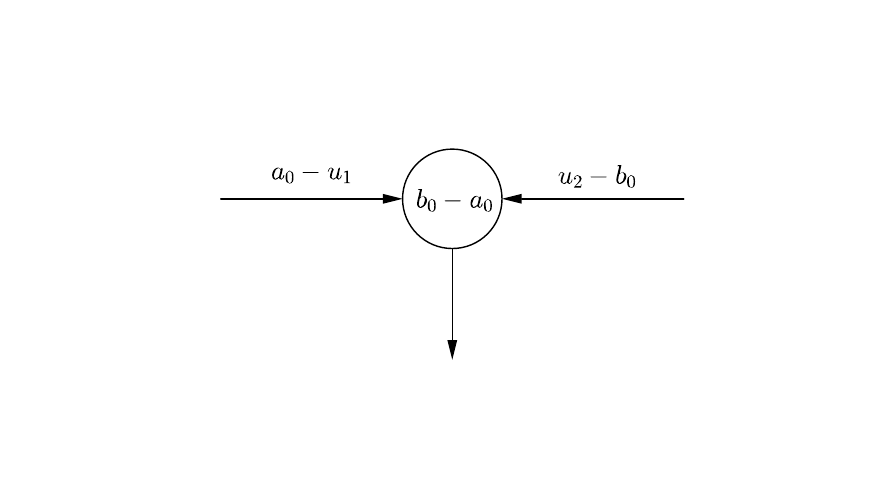}}
\caption{A birdie $\Bird(a_0,b_0)$ with the adjacent domains and their graph.}
\label{fig:birdie}
\label{fig:bgr}
\end{figure}

The corresponding proposition looks like this.

\begin{St}\label{S17}
Let $u_1<a_0<b_0<u_2$ and let $b_0-a_0\le 2\eps$. Let the function~$B$ coincide with the standard candidates on~$\Ch([a_0,b_0],*)$\textup, $\Rt(u_1,a_0)$\textup, $\Bird(a_0,b_0)$\textup, and $\Lt(b_0,u_2)$. Suppose that  
\begin{align*}
\mrt(a_0) = f'(a_0) - \eps \av{f''}{[a_0,b_0]};\\ 
\mlt(b_0) = f'(b_0) + \eps \av{f''}{[a_0,b_0]}.
\end{align*}
Then\textup, the function~$B$ is a~$C^1$-smooth Bellman candidate on
$$
\Rt(u_1,a_0) \cup \Bird(a_0,b_0) \cup \Ch([a_0,b_0],*) \cup \Lt(b_0,u_2).
$$
\end{St}
%\begin{St}\label{S17}
%Let $u_1<a_0<b_0< u_2$ and let $b_0-a_0\le2\eps$. Consider the chordal domain~$\Ch([a_0,b_0],*)$\textup, two domains $\Rt(u_1,a_0)$ and $\Lt(b_0,u_2)$\textup, and the linearity domain
%$\Bird(a_0,b_0)$. The function~$B$ defined by formula~\eqref{vallun} in~$\Ch([a_0,b_0])$\textup,
%by formulas~\eqref{linearity} and~\eqref{difeq} in~$\Rt(u_1,a_0)$\textup, by formulas~\eqref{linearity} and~\eqref{difeq2} in ~$\Lt(b_0,u_2)$\textup, and by formulas~\eqref{LinearInTrolleybus} and~\eqref{coef} in~$\Bird(a_0,b_0)$ is a Bellman candidate if the chordal domain satisfies the hypothesis of Proposition~\textup{\ref{LightChordalDomainCandidate}}\textup, the inequality~$m'' \leq 0$ is valid in~$\Rt(u_1,a_0)$, the inequality~$m'' \geq 0$ is valid in~$\Lt(b_0,u_2)$\textup, and
%\begin{align*}
%\eps m''(a_0) + \Fl(a_0; a_0, b_0; \eps) = 0;\\
%\eps m''(b_0) + \Fr(b_0; a_0, b_0; \eps) = 0.
%\end{align*}
%\end{St}
%\begin{wrapfigure}[12]{i}{250pt}
%\begin{center}
%\includegraphics[width = 1 \linewidth]{BirdieGraph.jpg}
%\caption{A graph for the birdie.}
%\label{fig:bgr}
%\end{center}
%\end{wrapfigure}

We may rewrite equality~\eqref{mrta0} in terms of forces:
\begin{equation}\label{RightTrollBalanceLeft}
\eps \mrt''(a_0) + \Fl(a_0; a_0, b_0; \eps) \!\!
\stackrel{\scriptscriptstyle{\eqref{LeftForce}}}{=} \!\!\eps \mrt''(a_0) - \Slt(a_0,b_0)\!\! 
\stackrel{{\genfrac{}{}{0pt}{-2}{\scriptscriptstyle\eqref{difeqSecondDer}}{\scriptscriptstyle\eqref{e334}}}}{=}\!\! \frac{\mrt(a_0)-f'(a_0)}{\eps}+\av{f''}{[a_0,b_0]}\!\!\!
\stackrel{\scriptscriptstyle{\eqref{mrta0}}\phantom{a}}{=}\!\!0.
\end{equation}
Similarly, equality~\eqref{mrtb0} is equivalent to
\begin{equation}\label{RightTrollBalanceRight}
\eps \mrt''(b_0) - \Fr(b_0; a_0, b_0; \eps)=0,
\end{equation}
which by formulas~\eqref{mr''_firstformula} and~\eqref{RightForce} implies that
\begin{equation}\label{ForceFromRightTrolleybus}
\eps \mrt''(u) = \Fr(u;a_0,b_0;\eps)
\end{equation} 
for~$u \in \Rt(b_0,u_2)$.

In the case of a left trolleybus we can rewrite relation~\eqref{mltb0} in the form
\begin{equation}\label{LeftTrollBalanceRight}
\eps \mlt''(b_0) + \Fr(b_0; a_0, b_0; \eps) =0,
\end{equation}
and~\eqref{mlta0} in the form
\begin{equation*}%\label{RightTrollBalanceLeft}
\eps \mlt''(a_0) - \Fl(a_0; a_0, b_0; \eps) =0,
\end{equation*}
which implies
\begin{equation*}
\eps \mlt''(u) - \Fl(u; a_0, b_0; \eps) =0
\end{equation*}
for~$u \in \Lt(u_1,a_0)$. In the case of a birdie, we obtain relations~\eqref{LeftTrollBalanceRight} and~\eqref{RightTrollBalanceLeft} in a similar way.

A birdie is a union of a trolleybus and an angle:
\begin{equation}\label{FirstFormula}
\Bird(a_0,b_0) = \RTroll(a_0,b_0) \biguplus \Rt(b_0,b_0) \biguplus \Ang(b_0) = \Ang(a_0)\biguplus \Lt(a_0,a_0) \biguplus \LTroll(a_0,b_0). 
\end{equation}
This equality can be presented in the terms of graphs as it is shown on Figure~\ref{fig:BirdieAsAnglePlusTrol}
\begin{figure}[!h]%{o}{250pt}
%\begin{center}
\vskip-30pt
\hskip-80pt\includegraphics[scale=1.3]{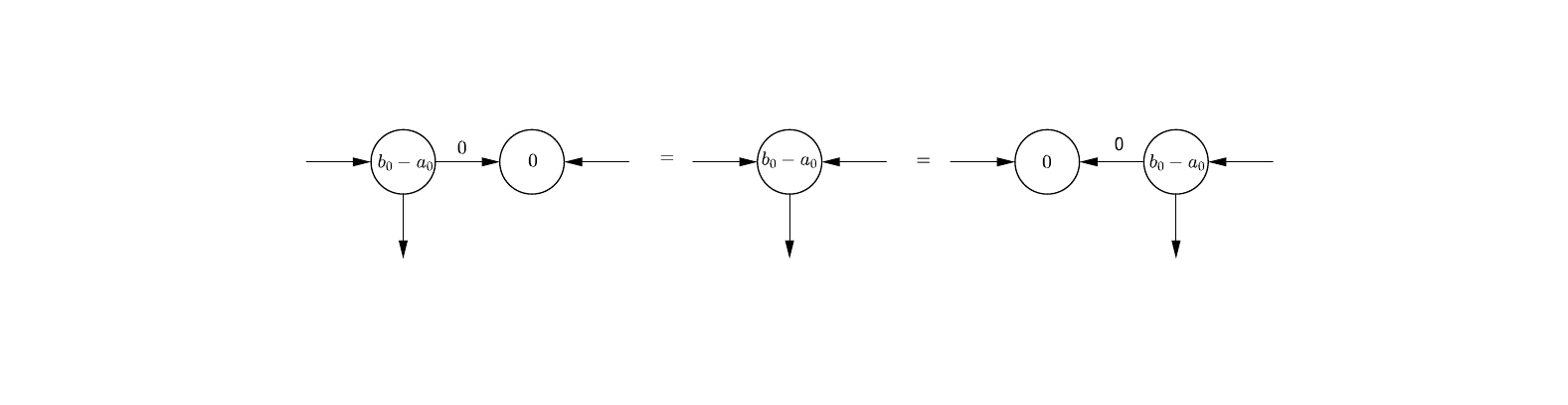}
\vskip-40pt
\caption{The equality ``birdie = angle + troleybus'' in terms of graphs.}
\label{fig:BirdieAsAnglePlusTrol}
%$\end{center}
\end{figure}
The symbol~$\biguplus$ in equality~\eqref{FirstFormula} means the following: if a function~$B$ on this domain (note that all three parts of the equation are equal as planar domains provided we substitute~$\bigcup$ for~$\biguplus$) is continuous and its restriction to each single subdomain of one part of the formula is a standard candidate, then this function~$B$ is a standard candidate for each subdomain of another part of the formula. Indeed, consider the foliation~$\RTroll(a_0,b_0) \cup \Rt(b_0,b_0)\cup \Ang(b_0)$. Let us apply Propositions~\ref{S15} and~\ref{AngleMeetRightTangents}. The parameter~$\mrt(b_0)$ of the standard candidate in~$\Rt(b_0,b_0)$ is then defined by formula~\eqref{mrtb0}. The parameter~$\beta_2$ in an angle~$\Ang(b_0)$ is determined by~\eqref{AngleMeetRightTangentsFormula}, therefore it coincides with the corresponding parameter in~$\RTroll(a_0,b_0)$ given by~\eqref{coef}. Thus,~$B$ is linear on~$\Bird(a_0,b_0)$ and is a standard candidate there. Note that the conditions required for the existence of such a continuous candidate for different sides of formula~\eqref{FirstFormula}, are the same. What is more, the equation arising from gluing a neighbor tangent or chordal domain to~$\Bird(a_0,b_0)$ is the same as when we glue with the corresponding summand in formula~\eqref{FirstFormula} instead of the whole birdie.

%Suppose that the birdie~$\Bird(a_0,b_0)$ is surrounded by~$\Rt(u_1,a_0)$,~$\Ch([a_0,b_0],*)$, and~$\Lt(b_0,u_2)$. Let us comment the following formula.
%\begin{equation*}%\label{FirstFormula}
%\begin{split}
%\Rt(u_1&,a_0) \biguplus \Ch([a_0,b_0],*) \biguplus\Bird(a_0,b_0) \biguplus \Lt(b_0,u_2) \\ 
%&=\Rt(u_1,a_0) \biguplus \Ch([a_0,b_0],*)  \biguplus \RTroll(a_0,b_0) \biguplus \Rt(b_0,b_0) \biguplus \Ang(b_0) \biguplus \Lt(b_0,u_2).\\
%\end{split}
%\end{equation*}
%~\eqref{ForceFromRightTrolleybus} (here~$\mrt''(b_0)$ is nothing but a symbol). 
%We see that the balance equation needed to paste the angle turns into~$\Fr(b_0;a_0,b_0;\eps) + \eps\mlt''(b_0)=0$, where~$\mlt$ is the slope coefficient for~$\Lt(b_0,u_2)$. This equation is exactly~\eqref{LeftTrollBalanceRight} which is equivalent to what we asked for in Proposition~\ref{S17}. So, we see that the conditions we require to build~$\Bird(a_0,b_0)$ and~$\RTroll(a_0,b_0)\cup\Ang(b_0)$ are the same. %It is easy to see that the Bellman candidates constructed either by Proposition~\ref{S17} or by Propositions~\ref{S15} and~\ref{AngleProp} are the same. Indeed, it suffices to verify that the function obtained by the second method is~$C^1$-smooth. This follows from the fact that~$\frac{\partial B}{\partial x_2} = \frac{\mrt'(b_0)}{2}$ both in the trolleybus and the angle (here~$\mrt'(b_0)$ is defined formally by equation~\eqref{difeq}).  This deduces Proposition~\ref{S17} from the propositions for trolleybuses and Proposition~\ref{AngleProp}. 

We did have birdies in the previous paper~\cite{5A}, however we treated them as a union of a trolleybus and an angle, not giving them this name. The birdie is a very capricious figure from the evolutional point of view, that is why it needs a separate study.

\index{standard candidate! in a single chord}
It is convenient to introduce two more ``linearity domains'' for the purposes of formalization. First, sometimes we will treat a single chord~$[A_0,B_0]$,~$b_0 - a_0 \leq 2\eps$,~$(a_0,b_0)$ satisfies the cup equation~\eqref{urlun}, as a linearity domain. The standard candidate~$B$ inside~$[A_0,B_0]$ is then given by the formula~\eqref{LinearInTrolleybus}. Since this linearity domain has two points on the fixed boundary, the~$\beta$ are given by~\eqref{coef}\footnote{Note that the coefficients~$\beta$ are not uniquely defined by the trace of~$B$ on~$[A_0,B_0]$. To restore them, we use the traditional smoothness assumption.}. With this definition at hand, we see that Propositions~\ref{CupMeetsRightTangents} and~\ref{CupMeetsLeftTangents} say the usual truth: if~$B$ is continuous on the union of a long chord~$[A_0,B_0]$, a chordal domain~$\Ch([a_0,b_0],*)$, and two tangent domains adjacent to~$[A_0,B_0]$, and coincides with the standard candidates there, then it is a~$C^1$-smooth Bellman candidate on this union. One can glue the standard candidate on a chord~$[A_0,B_0]$,~$b_0 - a_0 < 2\eps$, with the standard candidates on~$\Ch(*,[a_0,b_0])$ and~$\Ch([a_0,b_0],*)$ in a similar manner.

Second, sometimes it is useful to treat a single tangent~$\Rt(w,w)$ or~$\Lt(w,w)$ as a linearity domain. Moreover, no matter how strange it seems, it is natural to think of it as of a domain with two points on the fixed boundary\footnote{In a sense, we treat this single tangent as a trolleybus of zero width, i.e. its base is the chord~$[W,W]$.}. We will consider such a construction only when~$w = c_i$ for some~$i$, where~$c_i$ is a single point root from Definition~\ref{roots}. Therefore, the standard candidate~$B$ in this domain is given by formulas~\eqref{LinearInTrolleybus} and~\eqref{betaequations} with
\begin{equation}\label{Beta2InTheFifthType}\index{standard candidate! in a single tangent}
\beta_2 = \frac12 f''(w)
\end{equation} (compare with~\eqref{coef}). The concatenation of this ``linearity domain'' with adjacent tangent domains is performed in the same way as for any linearity domain with the given~$\beta$ (see Proposition~\ref{DomainOfLinearityWithTwoPointsOnTheLowerBoundaryMeetsTangentDomain}).

\subsection{Multifigures}\index{multifigures}\label{s343}
We begin with a structural agreement. For each linearity domain~$\mathfrak{L}$, consider its intersection with the fixed boundary. We assume that it is a union of finite number of arcs (however, one or two of these arcs may be infinite),
\begin{equation*}
\mathfrak{L} \cap \FixedBoundary\Omega_{\eps} = \cup_{i=1}^k \{(t,t^2) \mid t \in \mathfrak{a}_i\},
\end{equation*} 
where~$\{\mathfrak{a}_i\}_{i=1}^k$ is a finite set of disjoint closed intervals, which can be single points. The parabolic arc that corresponds to~$\mathfrak{a}_i$ is called~$\mathfrak{A}_i$. We remind the reader the notation introduced in Subsection~\ref{s212}: the left endpoint of~$\mathfrak{a}_i$ is~$\mathfrak{a}_{i}^{\mathrm l}$ and the right endpoint is~$\mathfrak{a}_i^{\mathrm r}$. As we will see, all the linearity domains needed to construct the Bellman function for~$f$ satisfy this finiteness assumption due to Condition~\ref{reg}.

Consider some linearity domain~$\mathfrak{L}$. We know that all the points~$\big(a, a^2, f(a)\big)$, $A \in \mathfrak{L} \cap \FixedBoundary\Omega_{\eps}$, lie in one and the same two-dimensional plane in~$\mathbb{R}^3$. Therefore, there exists a quadratic polynomial~$P_{\mathfrak{L}}$ such that 
\begin{equation}\label{PolynomialForLinearityDomain}
f(a) = P_{\mathfrak{L}}(a), \quad A \in \mathfrak{L} \cap \FixedBoundary\Omega_{\eps}.
\end{equation}
Surely, the converse is also true: if there exists a quadratic polynomial~$P_{\mathfrak{L}}$ such that equality~\eqref{PolynomialForLinearityDomain} holds true, then there exists a linear function~$B_{\mathfrak{L}}$ such that~$B_{\mathfrak{L}}(a,a^2) = f(a)$ for all~$A \in \mathfrak{L} \cap \FixedBoundary\Omega_{\eps}$ (this is the assertion of Remark~\ref{quadratische}). Specifically, if~$P_{\mathfrak{L}}(t) = \beta_0+\beta_1 t+ \beta_2 t^2$, then
\begin{equation}\label{CandidateInMultifigure}
B_{\mathfrak{L}}\big(x_{1},x_{2}\big) = \beta_0+\beta_1 x_1+\beta_2 x_2.
\end{equation}
This function~$B_{\mathfrak{L}}$ is a Bellman candidate in~$\mathfrak{L}$. Similar to the case where the linearity domain has only two points on the fixed boundary, the system~\eqref{coef} and~\eqref{urlun} holds true for any~$A_0,B_0 \in \mathfrak{L} \cap \FixedBoundary\Omega_{\eps}$. Indeed, under the same smoothness assumptions, we can write equations~\eqref{betaequations1},~\eqref{betaequations2} for~$A_0$ and~$B_0$ and then derive the system~\eqref{coef} from them, which, in its turn, leads to the cup equation~\eqref{urlun} for the pair~$(a_0,b_0)$. In particular, all the 
%We assume that the Bellman candidate~$B$ for some domain containing~$\mathfrak{L}$ is~$C^1$-smooth in a neighborhood of~$\mathfrak{L}$. Therefore, the lines
%\begin{equation*}
%\Big\{\big(a,a^2,f(a)\big) + t\big(1,2a,f'(a)\big)\,\Big|\; t \in \mathbb{R}\Big\}, \quad A \in \mathfrak{L} \cap \FixedBoundary\Omega_{\eps},
%\end{equation*}
%lie in the same plane with the graph of~$B$ over~$\mathfrak{L}$. As a result, all the
 points~$\big(a,f'(a)\big) \in \mathbb{R}^2$, $A \in\mathfrak{L} \cap \FixedBoundary\Omega_{\eps}$, lie on one line, whose slope is~$2\beta_2$. 
 
\begin{Def}\index{standard candidate! in multifigure}\index{standard candidate! in multifigure}
The function~$B$ defined by formulas~\eqref{CandidateInMultifigure} and~\eqref{coef} in the linearity domain~$\mathfrak{L}$, where~$A_0$ and~$B_0$ are arbitrary points from $\mathfrak{L} \cap \FixedBoundary\Omega_{\eps}$\textup,
is called a \emph{standard candidate} there. 
\end{Def}
 
As we have verified, the standard candidate in~$\mathfrak{L}$ does not depend on the choice of~$A_0$ and~$B_0$ in the definition.
 
 In the following lemma, we use  Definition~\ref{differentials}.
\begin{Le}\label{ThreePointsOnOneLine}
Let~$A_1$\textup,~$A_2$\textup, and~$A_3$ be three points such that the points~$\big(a_i,f'(a_i)\big)$\textup, $i = 1,2,3$\textup, lie on one line. If~$a_1 \leq a_2 \leq a_3$\textup, then
\begin{align*}
\Slt(a_1,a_2)&=\Slt(a_1,a_3);
\\
\Srt(a_1,a_3)&=\Srt(a_2,a_3);
\\
\Srt(a_1,a_2)&=\Slt(a_2,a_3).
\end{align*}
\end{Le}
\begin{proof}
One can easily ``observe'' this lemma from the geometric interpretation of the differentials, Figure~\ref{fig:area_differentials}. In all these cases the slopes of the same lines determine the values of the differentials on the left-hand side and the right-hand side of each equality.
\end{proof}

%Consider now the points~$\mathfrak{A}_i^r$ and~$\mathfrak{A}_{i+1}^l$ for some~$i = 1,2, \ldots, k-1$. It is natural to assume that there is a chordal domain~$\Ch([\mathfrak{a}_i^r, \mathfrak{a}_{i+1}^l],*)$ below the chord~$\mathfrak{A}_i^r\mathfrak{A}_{i+1}^l$. Thus,~$(\mathfrak{a}_i^r,\mathfrak{a}_{i+1}^l)$ satisfies the cup equation. Surely, if some points~$a$ and~$\tilde{a}$,~$a \leq \tilde{a}$, belong to some~$\mathfrak{A}_i$, then~$(a,\tilde{a})$ satisfies the cup equation, because the function~$f'$ is linear on~$[a,\tilde{a}]$. It occurs that every two points of~$\mathfrak{L} \cap \FixedBoundary\Omega_{\eps}$ satisfy the cup equation. 
\begin{Le}\label{ThreePointsAndTheArea}
Let~$A_1$\textup,~$A_2$\textup, and~$A_3$ be three points such that the points~$\big(a_i,f'(a_i)\big)$\textup, $i = 1,2,3$\textup, lie on one line. Suppose that the pairs~$(a_1,a_2)$ and~$(a_2,a_3)$ satisfy the cup equation. Then~$(a_1,a_3)$ satisfies the cup equation as well.
\end{Le}
\begin{proof}
We have to prove that the subgraph area for the pair~$(a_1,a_3)$ equals the corresponding trapezoid area (see the geometric interpretation of the cup equation provided at the end of Subsection~\ref{s331} and Figure~\ref{fig:area_differentials} as well). Surely, each of these quantites is a sum of the corresponding quantities for~$(a_1,a_2)$ and~$(a_2,a_3)$. This is always true for the subgraph area. For the trapezoid area it is a consequence of the lemma hypothesis. %The lemma is proved.
\end{proof}

\begin{Rem}
We recall that for any two points~$A_1,A_2 \in \mathfrak{L}\cap \FixedBoundary \Omega_{\eps}$\textup, where~$\mathfrak{L}$ is a multifigure\textup, we have
\begin{equation*}
\frac{\partial B}{\partial x_2}\Big|_{\mathfrak{L}} = \beta_2 = \frac{1}{2}\av{f''}{[a_1,a_2]} =\frac{f'(a_2) - f'(a_1)}{2(a_2-a_1)}.
\end{equation*}
\end{Rem}

Now we are equipped to describe all the remaining linearity domains. We start with the domains 
that are not separated from the upper parabola. 
\begin{figure}[h]
\vskip-20pt
\hskip10pt
\includegraphics[width = 0.5\linewidth]{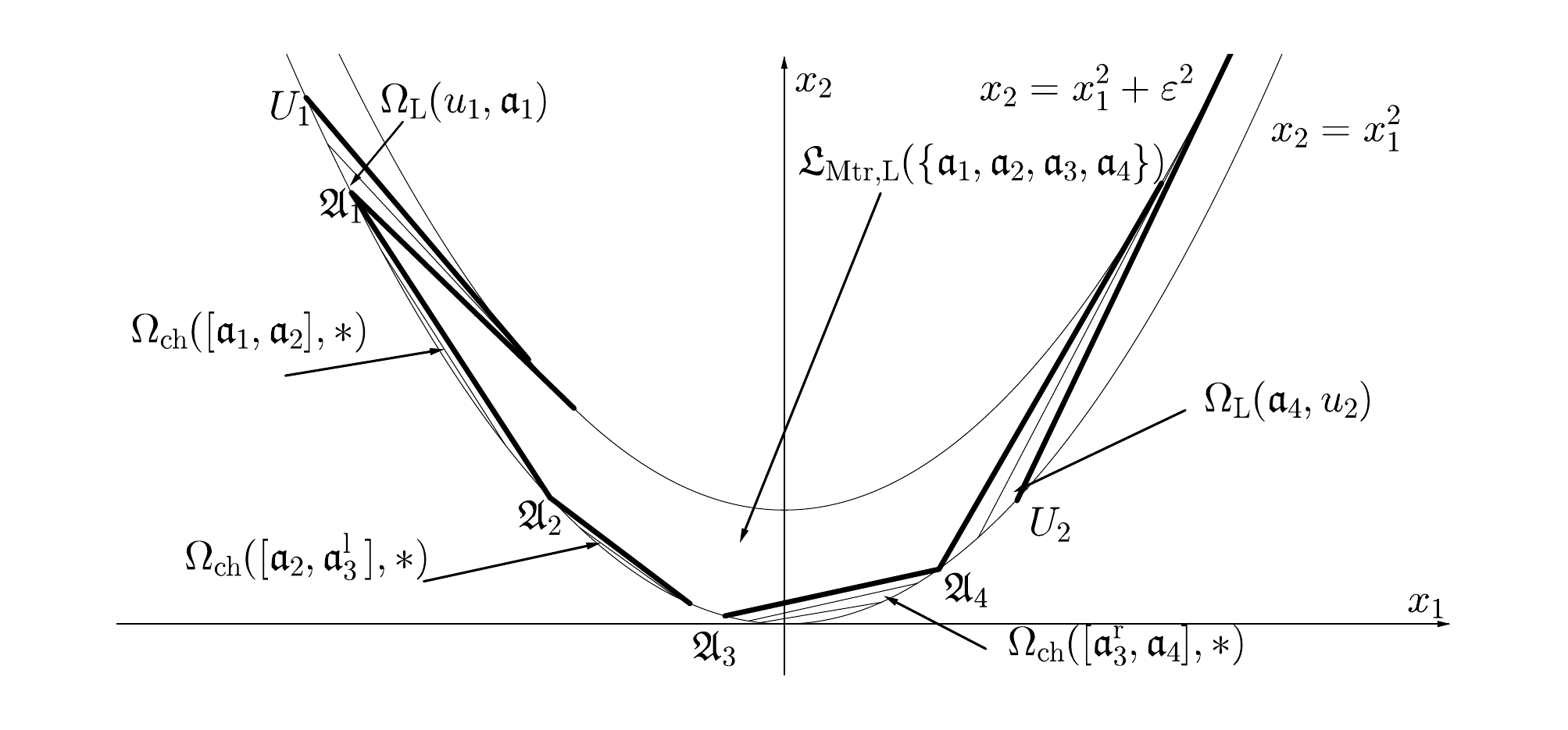}
\hskip-30pt
\raisebox{-20pt}{{\includegraphics[scale=1.4]{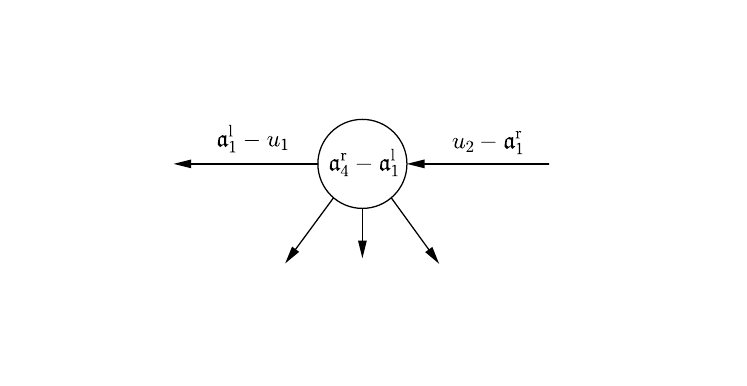}}}
%\begin{center}
%\includegraphics[width = 0.59\linewidth]{MultitrolleybusM4}
%\includegraphics[width = 0.39\linewidth]{3_10.pdf}
%\includegraphics[width = 0.49 \linewidth]{MtrGraph.jpg}
\caption{A multitrolleybus for $k = 4$ with the adjacent domains and their graphical representation.}
\label{fig:multitroll}
\label{fig:MtrGr}
%\end{center}
\end{figure}
The boundary of such a bounded domain consists of the arcs~$\mathfrak{A}_i$, $i = 1,2,\ldots, k$, 
the chords~$\mathfrak{A}_i^{\mathrm r}\mathfrak{A}_{i+1}^{\mathrm l}$, $i = 1,2, \ldots, k-1$, two tangents 
from the points~$\mathfrak{A}_1^{\mathrm l}$ and~$\mathfrak{A}_{k}^{\mathrm r}$, and the arc of the free boundary. 
We classify the multifigures with respect to the orientation of these tangents. Namely, if 
the edge tangents are both right, then we get a right \emph{multitrolleybus}\index{multitrolleybus} 
denoted by~$\MTTR(\{\mathfrak{a}_i\}_{i=1}^k;\eps)$; if they are both left, then we have a left 
multitrolleybus denoted by~$\MTTL(\{\mathfrak{a}_i\}_{i=1}^k;\eps)$ (see Figure~\ref{fig:multitroll}).
If both tangents ``look inside'' the domain, then it is called a~\emph{multicup}\index{multicup}, 
see Figure~\ref{fig:multicup4},~$\MTC(\{\mathfrak{a}_i\}_{i=1}^k;\eps)$. 
\begin{figure}[h]%{i}{300pt}
\hskip10pt
\includegraphics[width = 0.55\linewidth]{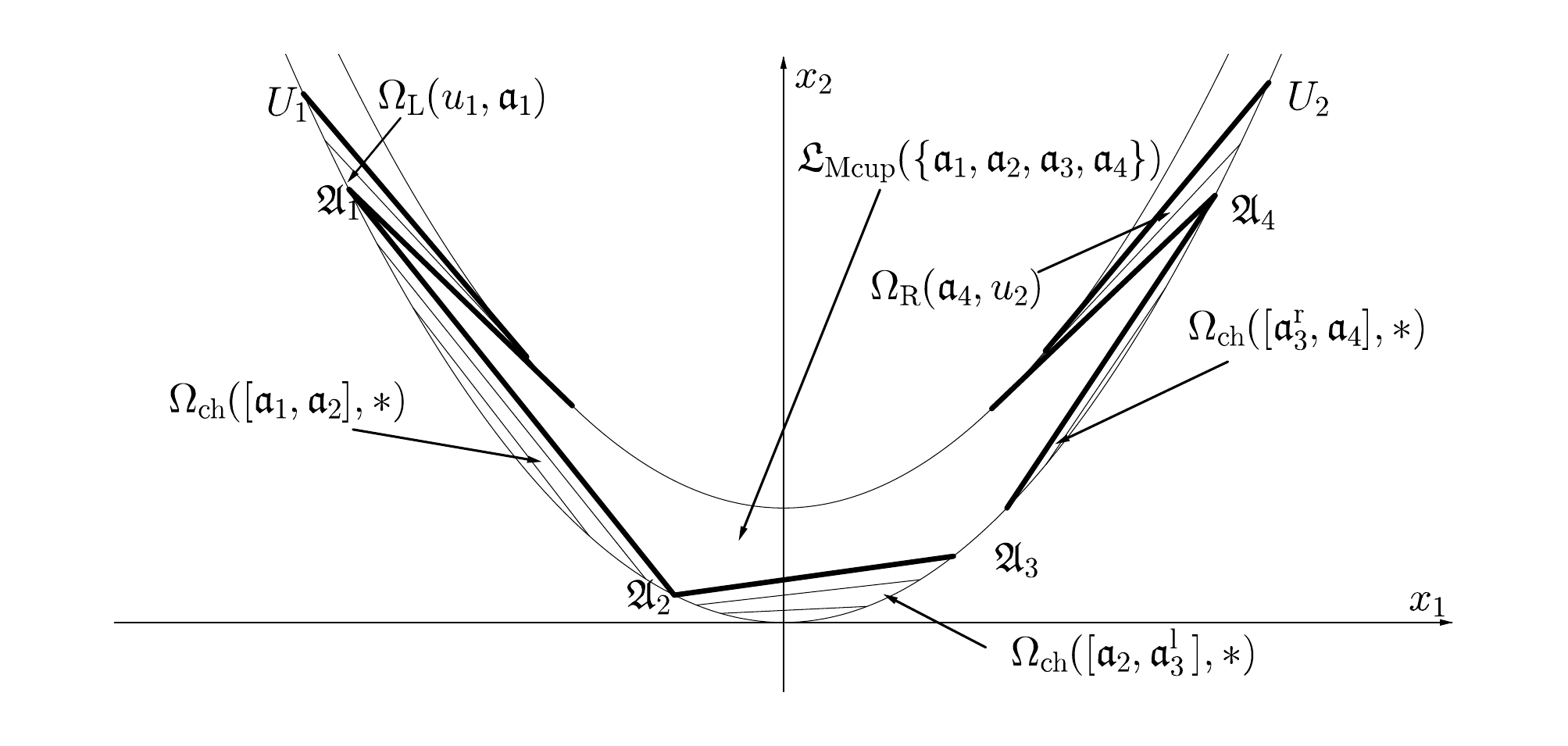}
\hskip-40pt
\raisebox{-10pt}{{\includegraphics[scale=1.4]{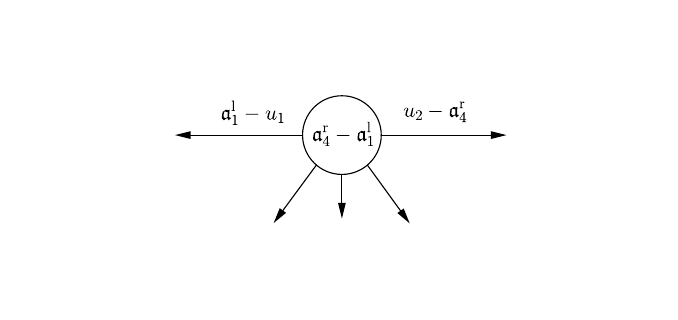}}}
\caption{A mutlicup for $k=4$ with the adjacent domains and their graphical representation.}
\label{fig:multicup4}
\label{fig:McupGr}
\end{figure}
We distinguish the case where the two border tangents lie on one line 
(i.e.~$\mathfrak{a}_k^{\mathrm r} - \mathfrak{a}_1^{\mathrm l} = 2\eps$) 
and say that in this case the multicup is full. Finally, if both tangents ``look outside'' 
the domain, then it is called a~\emph{multibirdie}\index{multibirdie}, see 
Figure~\ref{fig:multibirdie},~$\MTB(\{\mathfrak{a}_i\}_{i=1}^k;\eps)$.
\begin{figure}[h]%[16pt]{o}{300pt}
\hskip10pt
\includegraphics[width = 0.55\linewidth]{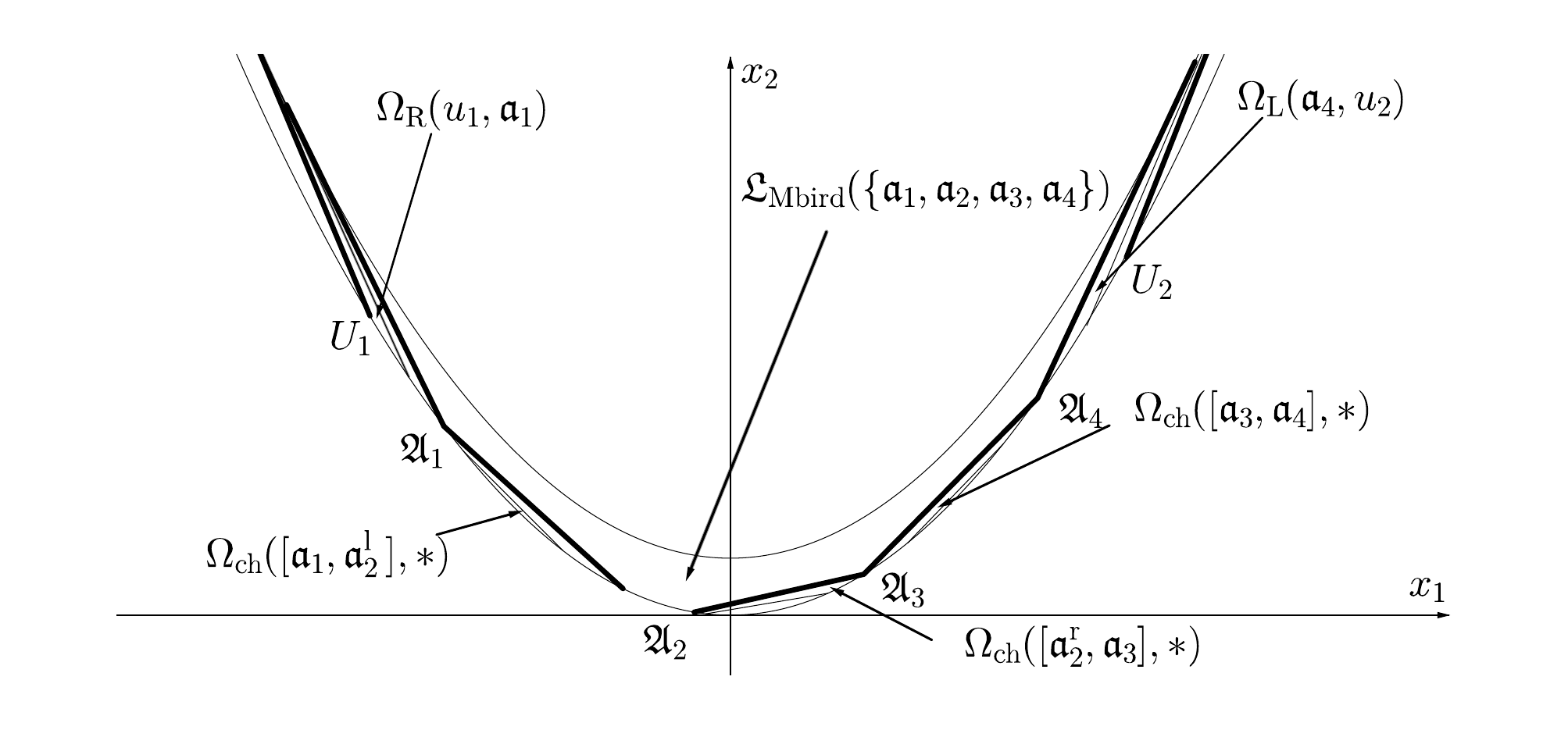}
\hskip-40pt
\raisebox{-10pt}{{\includegraphics[scale=1.4]{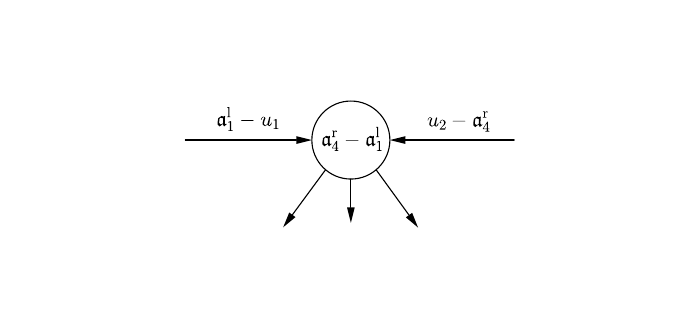}}}
\caption{A multibirdie for $k=4$ with the adjacent domains and their graphical representation.}
\label{fig:multibirdie}
\label{fig:MtbGr}
\end{figure}
Graphical representation for a multifigure~$\mathfrak{L}$ built over~$\{\mathfrak{a}_i\}_{i=1}^k$ 
is drawn by the following rule. The domain~$\mathfrak{L}$ corresponds to a single vertex. 
It has~$k-1$ outcoming edges representing the chordal
domains~$\Ch([\mathfrak{a}_i^{\mathrm r},\mathfrak{a}_{i+1}^{\mathrm l}],*)$,~$i = 1,2,\ldots,k-1$. 
There are two more edges corresponding to two tangent domains surrounding~$\mathfrak{L}$. 
They are both outcoming if~$\mathfrak{L}$ is a multicup and both incoming in the case
where~$\mathfrak{L}$ is a multibirdie. If~$\mathfrak{L}$ is a multitrolleybus, then it 
has one incoming and one outcoming edge. There is one exception: in the case where 
a multicup or a multitrolleybus lasts to infinity (i.e. one or both of its border arcs are rays), 
then it does not have a border tangent. In such a case, its vertex does not have the corresponding
outcoming edge. We provide examples of graphs for the multifigures drawn on
Figures~\ref{fig:multitroll},~\ref{fig:multicup4}, and~\ref{fig:multibirdie}.

%\begin{wrapfigure}{o}{250pt}
%\begin{center}
%\includegraphics[width = 1 \linewidth]{MtrGraph.jpg}
%\caption{A graphical representation for the Picture~\ref{fig:multitroll}.}
%\label{fig:MtrGr}
%\end{center}
%\end{wrapfigure}
Now we treat the multifigures separately. Our aim is to give sufficient conditions for concatenation 
with the tangent domains surrounding the linearity domain. We begin with a multicup. 

\begin{St}\label{MulticupCandidate}
Suppose~$\mathfrak{a}_i = [\mathfrak{a}_i^{\mathrm l}, \mathfrak{a}_i^{\mathrm r}]$\textup,~$i = 1,2,\ldots,k$\textup, to be disjoint intervals on~$\mathbb{R}$ \textup(these intervals can be single points or rays\textup) such that~$0 < \mathfrak{a}_{i+1}^{\mathrm l} - \mathfrak{a}_{i}^{\mathrm r} \leq 2\eps$\textup, and~$\mathfrak{a}_k^{\mathrm r} - \mathfrak{a}_1^{\mathrm l} \geq 2\eps$. Assume that all the points~$\big(a,f'(a)\big)$\textup,~$a \in \cup_{i=1}^k\mathfrak{a}_i$\textup, lie on one line\textup, and any pair of points from~$\cup_{i=1}^k\mathfrak{a}_i$ satisfies~\eqref{urlun}. Let the function~$B$ coincide with the standard candidates on~$\MTC(\{\mathfrak{a}_i\}_{i=1}^k)$\textup, on each~$\Ch([\mathfrak{a}_i^{\mathrm r}, \mathfrak{a}_{i+1}^{\mathrm l}],*)$\textup, $i=1,\dots, k-1$\textup, on~$\Lt(u_1,\mathfrak{a}_1^{\mathrm l})$ with the parameter~$\mlt(\mathfrak{a}_1^{\mathrm l})$ given by formula~\eqref{mlta0}\textup(with $a_0=\mathfrak{a}_1^{\mathrm l}$ and $b_0=\mathfrak{a}_k^{\mathrm r}$\textup)\textup, and on~$\Rt(\mathfrak{a}_k^{\mathrm r},u_2)$ with the parameter~$\mrt(\mathfrak{a}_k^{\mathrm r})$ given by formula~\eqref{mrtb0} \textup(with $a_0=\mathfrak{a}_1^{\mathrm l}$ and $b_0=\mathfrak{a}_k^{\mathrm r}$\textup). 
%
%
% built over this set of arcs and a linear function~$B_{\mathrm{Mcup}(\{\mathfrak{a}_i\}_{i=1}^k)}$ that corresponds to it \textup(defined by formulas~\eqref{PolynomialForLinearityDomain} and~\eqref{PolynomialFormula}\textup). The function~$B$ is defined by formula~\eqref{vallun} in each~$\Ch([\mathfrak{a}_i^{\mathrm r}, \mathfrak{a}_{i+1}^{\mathrm l}],*)$\textup, by formulas~\eqref{linearity}\textup,~\eqref{difeq2}\textup, and
%\begin{equation}\label{FormulaForLeftmMulticup}
%\eps m''(u) = \Fl(u;\mathfrak{a}_1^{\mathrm{l}},\mathfrak{a}_k^{\mathrm{r}};\eps)
%\end{equation}
%in~$\Lt(u_1,\mathfrak{a}_1^{\mathrm l})$\textup, by formulas~\eqref{linearity}\textup,~\eqref{difeq}\textup, and
%\begin{equation}\label{FormulaForRightmMulticup}
%\eps m''(u) = \Fr(u;\mathfrak{a}_1^{\mathrm{l}},\mathfrak{a}_k^{\mathrm{r}};\eps)
%\end{equation}
%in~$\Rt(\mathfrak{a}_k^{\mathrm r},u_2)$\textup, and by formula~$B = B_{\mathrm{Mcup}(\{\mathfrak{a}_i\}_{i=1}^k)}$ in the multicup. 
%Let all the chordal domains satisfy the hypothesis of Lemma~\textup{\ref{LightChordalDomainCandidate}} and let the inequalities~$m'' \geq 0$ and~$m'' \leq 0$ be fulfilled on~$\Lt(u_1,\mathfrak{a}_1^{\mathrm l})$ and~$\Rt(\mathfrak{a}_k^{\mathrm r},u_2)$ correspondingly. 
Then the function~$B$ %(\{\mathfrak{a}_i\}_{i=1}^k)$ 
is a~$C^1$-smooth Bellman candidate on the domain
\begin{equation*}
\Lt(u_1,\mathfrak{a}_1^{\mathrm l}) \cup \MTC(\{\mathfrak{a}_i\}_{i=1}^k) \cup \Big(\cup_{i=1}^{k-1} \Ch([\mathfrak{a}_i^{\mathrm r}, \mathfrak{a}_{i+1}^{\mathrm l}],*)\Big) \cup \Rt(\mathfrak{a}_k^{\mathrm r},u_2).
\end{equation*}
\end{St}
%\begin{wrapfigure}[14pt]{i}{250pt}
%\begin{center}
%\includegraphics[width = 1 \linewidth]{MtcGraph.jpg}
%\caption{A graphical representation for the Picture~\ref{fig:multicup4}.}
%\label{fig:McupGr}
%\end{center}
%\end{wrapfigure}
If one of the intervals~$\mathfrak{a}_1$ and~$\mathfrak{a}_k$ is infinite, then we do not consider the corresponding domain~$\Lt$ or~$\Rt$. 
\begin{proof}
This is a direct consequence of Propositions~\ref{DomainOfLinearityWithTwoPointsOnTheLowerBoundaryMeetsChordalDomain} and~\ref{DomainOfLinearityWithTwoPointsOnTheLowerBoundaryMeetsTangentDomain}.
%As usual, it suffices to prove only the~$C^1$-smoothness of the Bellman candidate. It has already been done for the case of common boundaries of cups and~$\MTC(\{\mathfrak{a}_i\}_{i=1}^k)$. As we have seen in the proof of Proposition~\ref{S15}, equation~\eqref{FormulaForRightmMulticup} leads to~$m'(\mathfrak{a}_k^{\mathrm{r}}) = \frac{f'(\mathfrak{a}_k^{\mathrm{r}}) - f'(a)}{\mathfrak{a}_k^{\mathrm{r}} - a}$ (we implicitly use Lemma~\ref{ThreePointsOnOneLine} here). Here~$a$ is any other than~$\mathfrak{a}_k^{\mathrm{r}}$ point on the lower parabola that belongs to~$\MTC(\{\mathfrak{a}_i\}_{i=1}^k)$, and~$m$ is the slope coefficient for~$\Rt(\mathfrak{a}_k^{\mathrm{r}}, u_2)$. As usually, the~$C^1$-smoothness on the common boundary of the multicup and~$\Rt(\mathfrak{a}_k^{\mathrm{r}}, u_2)$ follows from formulas~\eqref{BSubX2Tangents} and~\eqref{BsubX2Chords}.
\end{proof}

A similar proposition holds for the multibirdie. 

\begin{St}\label{MultibirdieCandidate}
Suppose~$\mathfrak{a}_i = [\mathfrak{a}_i^{\textup l}, \mathfrak{a}_i^{\textup r}]$\textup,~$i = 1,2,\ldots,k$\textup, to be disjoint intervals on~$\mathbb{R}$ \textup(these intervals can be single points\textup) such that~$0 < \mathfrak{a}_{i+1}^{\textup l} - \mathfrak{a}_{i}^{\textup r} \leq 2\eps$. Assume that all the points~$\big(a,f'(a)\big)$\textup,~$a \in \cup_{i=1}^k\mathfrak{a}_i$\textup, lie on one line\textup, and any pair of points from~$\cup_{i=1}^k\mathfrak{a}_i$ satisfies~\eqref{urlun}. Let the function~$B$ coincide with the standard candidates on~$\MTB(\{\mathfrak{a}_i\}_{i=1}^k)$\textup, on each~$\Ch([\mathfrak{a}_i^{\mathrm r}, \mathfrak{a}_{i+1}^{\mathrm l}],*)$\textup, $i=1,\dots, k-1$\textup, on~$\Rt(u_1,\mathfrak{a}_1^{\mathrm l})$ with the parameter~$\mrt(\mathfrak{a}_1^{\mathrm l})$ given by formula~\eqref{mrta0} \textup(with $a_0=\mathfrak{a}_1^{\mathrm l}$ and $b_0=\mathfrak{a}_k^{\mathrm r}$\textup)\textup, and on~$\Lt(\mathfrak{a}_k^{\mathrm r},u_2)$ with the parameter~$\mlt(\mathfrak{a}_k^{\mathrm r})$ given by formula~\eqref{mltb0} \textup(with $a_0=\mathfrak{a}_1^{\mathrm l}$ and $b_0=\mathfrak{a}_k^{\mathrm r}$\textup). 
Then the function~$B$ is a~$C^1$-smooth Bellman candidate on the domain
%Suppose~$\mathfrak{a}_i = [\mathfrak{a}_i^{\mathrm l}, \mathfrak{a}_i^{\mathrm r}]$\textup,~$i = 1,2,\ldots,k$\textup,~$0 < \mathfrak{a}_{i+1}^{\mathrm l} - \mathfrak{a}_{i}^{\mathrm r} \leq 2\eps$\textup, to be disjoint intervals on~$\mathbb{R}$ \textup(these intervals can be single points\textup). Assume that all the points~$\big(a,f'(a)\big)$\textup,~$a \in \cup_{i=1}^k\mathfrak{a}_i$\textup, lie on one line. 
%Consider the multibirdie~$\MTB(\{\mathfrak{a}_i\}_{i=1}^k)$ built over this set of arcs and a linear function~$B_{\mathrm{Mbird}(\{\mathfrak{a}_i\}_{i=1}^k)}$ \textup(defined by formulas~\eqref{PolynomialForLinearityDomain} and~\eqref{PolynomialFormula}\textup) that corresponds to it. The function~$B$ is defined by formula~\eqref{vallun} in each~$\Ch([\mathfrak{a}_i^{\mathrm r}, \mathfrak{a}_{i+1}^{\mathrm l}],*)$\textup, by formulas~\eqref{linearity}\textup,~\eqref{difeq2} in~$\Lt(\mathfrak{a}_k^{\mathrm r},u_2)$ \textup(in this domain we denote the coefficient~$m$ by~$m_{\mathrm{L}}$\textup)\textup, by formulas~\eqref{linearity}\textup,~\eqref{difeq} in~$\Rt(u_1, \mathfrak{a}_1^{\mathrm l})$ \textup(in this domain we denote the coefficient~$m$ by~$m_{\mathrm{R}}$\textup)\textup, and by formula~$B = B_{\mathrm{Mbird}(\{\mathfrak{a}_i\}_{i=1}^k)}$ in the multibirdie. The function~$B$ is a Bellman candidate on the domain
\begin{equation*}
\Rt(u_1, \mathfrak{a}_1^{\mathrm l}) \cup \MTB(\{\mathfrak{a}_i\}_{i=1}^k) \cup \Big(\cup_{i=1}^{k-1} \Ch([\mathfrak{a}_i^{\mathrm r}, \mathfrak{a}_{i+1}^{\mathrm l}],*)\Big) \cup \Lt(\mathfrak{a}_k^{\mathrm r},u_2).
\end{equation*}
%provided all the chordal domains satisfy the hypothesis of Lemma~\textup{\ref{LightChordalDomainCandidate}}\textup, the inequalities~$\mlt'' \geq 0$ and~$\mrt'' \leq 0$ are fulfilled on~$\Lt(\mathfrak{a}_k^{\mathrm r},u_2)$ and~$\Rt(u_1, \mathfrak{a}_1^{\mathrm l})$ correspondingly\textup, and two balance equations \textup(see Definition~\textup{\ref{BalanceEquationDefinition}}\textup)
%\begin{equation}\label{MultibirdieRightBalance}
%\eps m_{\mathrm{R}}''(\mathfrak{a}_1^{\mathrm l}) + \Fl(\mathfrak{a}_1^{\mathrm l}; \mathfrak{a}_1^{\mathrm l}, \mathfrak{a}_k^{\mathrm r}; \eps) = 0;
%\end{equation}
%\begin{equation}\label{MultibirdieLeftBalance}
%\eps m_{\mathrm{L}}''(\mathfrak{a}_k^{\mathrm r}) + \Fr(\mathfrak{a}_k^{\mathrm r}; \mathfrak{a}_{1}^{\mathrm l}, \mathfrak{a}_k^{\mathrm r}; \eps) = 0
%\end{equation}
%hold true.
\end{St}
%\begin{proof}
%We have to verify the~$C^1$-smoothness on the common boundaries of the multibirdie and tangent domains. Let us consider the case of~$\Rt(u_1, \mathfrak{a}_1^{\mathrm l})$ (the second is symmetric). Using equations~\eqref{difeq},~\eqref{RightForce},~\eqref{e334} (with Lemmas~\ref{ThreePointsOnOneLine} and~\ref{ThreePointsAndTheArea}), we see that
%\begin{equation*}
%\mrt'(\mathfrak{a}_1^{\mathrm{l}}) = \frac{f'(\mathfrak{a}_1^{\mathrm{l}}) - f'(\mathfrak{a}_k^{\mathrm{r}})}{\mathfrak{a}_1^{\mathrm{l}} - \mathfrak{a}_k^{\mathrm{r}}},
%\end{equation*}
%which leads to the desired smoothness by formulas~\eqref{BSubX2Tangents} and~\eqref{BsubX2Chords}.
%\end{proof}
%%\begin{wrapfigure}{i}{250pt}
%%\begin{center}
%%\includegraphics[width = 1 \linewidth]{MtbGraph.jpg}
%%\caption{A graphical representation for the Picture~\ref{fig:multibirdie}.}
%%\label{fig:MtbGr}
%%\end{center}
%%\end{wrapfigure}
Finally, the propositions for the right multitrolleybus looks like this.% We leave their proofs to the reader.
\begin{St}\label{MultitrolleybusCandidate}
Suppose~$\mathfrak{a}_i = [\mathfrak{a}_i^{\textup l}, \mathfrak{a}_i^{\textup r}]$\textup,~$i = 1,2,\ldots,k$\textup, to be disjoint intervals on~$\mathbb{R}$ \textup(these intervals can be single points\textup, and the last one can be a ray\textup) such that~$0 < \mathfrak{a}_{i+1}^{\textup l} - \mathfrak{a}_{i}^{\textup r} \leq 2\eps$. Assume that all the points~$\big(a,f'(a)\big)$\textup,~$a \in \cup_{i=1}^k\mathfrak{a}_i$\textup, lie on one line\textup, and any pair of points from~$\cup_{i=1}^k\mathfrak{a}_i$ satisfies~\eqref{urlun}. Let the function~$B$ coincide with the standard candidates on~$\MTTR(\{\mathfrak{a}_i\}_{i=1}^k)$\textup, on each~$\Ch([\mathfrak{a}_i^{\mathrm r}, \mathfrak{a}_{i+1}^{\mathrm l}],*)$\textup, $i=1,\dots, k-1$\textup, on~$\Rt(u_1,\mathfrak{a}_1^{\mathrm l})$ with the parameter~$\mrt(\mathfrak{a}_1^{\mathrm l})$ given by formula~\eqref{mrta0} \textup(with $a_0=\mathfrak{a}_1^{\mathrm l}$ and $b_0=\mathfrak{a}_k^{\mathrm r}$\textup)\textup, and on~$\Rt(\mathfrak{a}_k^{\mathrm r},u_2)$ with the parameter~$\mrt(\mathfrak{a}_k^{\mathrm r})$ given by formula~\eqref{mrtb0} \textup(with $a_0=\mathfrak{a}_1^{\mathrm l}$ and $b_0=\mathfrak{a}_k^{\mathrm r}$\textup). 
Then the function~$B$ is a~$C^1$-smooth Bellman candidate on the domain
%Suppose~$\mathfrak{a}_i = [\mathfrak{a}_i^{\mathrm l}, \mathfrak{a}_i^{\mathrm r}]$\textup,~$i = 1,2,\ldots,k$\textup,~$0 < \mathfrak{a}_{i+1}^{\mathrm l} - \mathfrak{a}_{i}^{\mathrm r} \leq 2\eps$\textup, to be disjoint intervals on~$\mathbb{R}$ \textup(these intervals can be single points\textup, and the last of them can be a ray\textup). Assume that all the points~$\big(a,f'(a)\big)$\textup,~$a \in \cup_{i=1}^k\mathfrak{a}_i$\textup, lie on one line. Consider the right multitrolleybus~$\MTTR(\{\mathfrak{a}_i\}_{i=1}^k)$ built over this set of arcs and a linear function~$B_{\mathrm{Mtr,R}(\{\mathfrak{a}_i\}_{i=1}^k)}$ \textup(defined by formulas~\eqref{PolynomialForLinearityDomain} and~\eqref{PolynomialFormula}\textup) that corresponds to it. The function~$B$ is defined by formula~\eqref{vallun} in each~$\Ch([\mathfrak{a}_i^r, \mathfrak{a}_{i+1}^l],*)$\textup, by formulas~\eqref{linearity}\textup,~\eqref{difeq} in~$\Rt(u_1, \mathfrak{a}_1^{\mathrm r})$\textup, by formulas~\eqref{linearity}\textup,~\eqref{difeq}\textup, and~\eqref{FormulaForRightmMulticup} in~$\Rt(\mathfrak{a}_k^{\mathrm r},u_2)$\textup, and by formula~$B = B_{\mathrm{Mtr,R}(\{\mathfrak{a}_i\}_{i=1}^k)}$ in the multitrolleybus. The function~$B$ is a Bellman candidate on the domain
\begin{equation*}
\Rt(u_1, \mathfrak{a}_1^{\mathrm l}) \cup \MTTR(\{\mathfrak{a}_i\}_{i=1}^k) \cup \Big(\cup_{i=1}^{k-1} \Ch([\mathfrak{a}_i^{\mathrm r}, \mathfrak{a}_{i+1}^{\mathrm l}],*)\Big) \cup \Rt(\mathfrak{a}_k^{\mathrm r},u_2).
\end{equation*}
%provided all the chordal domains satisfy the hypothesis of Proposition~\textup{\ref{LightChordalDomainCandidate}}\textup, the inequality~$m'' \leq 0$ is fulfilled on~$\Rt(u_1, \mathfrak{a}_1^{\mathrm l})$ and~$\Rt(\mathfrak{a}_k^{\mathrm r},u_2)$\textup, and the balance equation~\textup{\eqref{MultibirdieRightBalance}} is valid.
\end{St}
The statement for the left multitrolleybus is symmetric.
\begin{St}\label{LeftMultitrolleybusCandidate}
Suppose~$\mathfrak{a}_i = [\mathfrak{a}_i^{\textup l}, \mathfrak{a}_i^{\textup r}]$\textup,~$i = 1,2,\ldots,k$\textup, to be disjoint intervals on~$\mathbb{R}$ \textup(these intervals can be single points\textup, and the first one can be a ray\textup) such that~$0 < \mathfrak{a}_{i+1}^{\textup l} - \mathfrak{a}_{i}^{\textup r} \leq 2\eps$. Assume that all the points~$\big(a,f'(a)\big)$\textup,~$a \in \cup_{i=1}^k\mathfrak{a}_i$\textup, lie on one line\textup, and any pair of points from~$\cup_{i=1}^k\mathfrak{a}_i$ satisfies~\eqref{urlun}. Let the function~$B$ coincide with the standard candidates on~$\MTTL(\{\mathfrak{a}_i\}_{i=1}^k)$\textup, on each~$\Ch([\mathfrak{a}_i^{\mathrm r}, \mathfrak{a}_{i+1}^{\mathrm l}],*)$\textup, $i=1,\dots, k-1$\textup, on~$\Lt(u_1,\mathfrak{a}_1^{\mathrm l})$ with the parameter~$\mlt(\mathfrak{a}_1^{\mathrm l})$ given by formula~\eqref{mlta0} \textup(with $a_0=\mathfrak{a}_1^{\mathrm l}$ and $b_0=\mathfrak{a}_k^{\mathrm r}$\textup)\textup, and on~$\Lt(\mathfrak{a}_k^{\mathrm r},u_2)$ with the parameter~$\mlt(\mathfrak{a}_k^{\mathrm r})$ given by formula~\eqref{mltb0} \textup(with $a_0=\mathfrak{a}_1^{\mathrm l}$ and $b_0=\mathfrak{a}_k^{\mathrm r}$\textup). 
Then the function~$B$ is a~$C^1$-smooth Bellman candidate on the domain
%Suppose~$\mathfrak{a}_i = [\mathfrak{a}_i^{\mathrm l}, \mathfrak{a}_i^{\mathrm r}]$\textup,~$i = 1,2,\ldots,k$\textup,~$0 < \mathfrak{a}_{i+1}^{\mathrm l} - \mathfrak{a}_{i}^{\mathrm r} \leq 2\eps$\textup, to be disjoint intervals on~$\mathbb{R}$ \textup(these intervals can be single points\textup, and the first of them can be a ray\textup). Assume that all the points~$\big(a,f'(a)\big)$\textup,~$a \in \cup_{i=1}^k\mathfrak{a}_i$\textup, lie on one line. Consider the left multitrolleybus~$\MTTL(\{\mathfrak{a}_i\}_{i=1}^k)$ built over this set of arcs and a linear function~$B_{\mathrm{Mtr,L}(\{\mathfrak{a}_i\}_{i=1}^k)}$ \textup(defined by formulas~\eqref{PolynomialForLinearityDomain} and~\eqref{PolynomialFormula}\textup) that corresponds to it. The function~$B$ is defined by formula~\eqref{vallun} in each~$\Ch([\mathfrak{a}_i^r, \mathfrak{a}_{i+1}^l],*)$\textup, by formulas~\eqref{linearity}\textup,~\eqref{difeq2} in~$\Lt(u_1, \mathfrak{a}_1^{\mathrm r})$\textup, by formulas~\eqref{linearity}\textup,~\eqref{difeq2}\textup, and~\eqref{FormulaForLeftmMulticup} in~$\Rt(\mathfrak{a}_k^{\mathrm r},u_2)$\textup, and by formula~$B = B_{\mathrm{Mtr,L}(\{\mathfrak{a}_i\}_{i=1}^k)}$ in the multitrolleybus. The function~$B$ is a Bellman candidate on the domain
\begin{equation*}
\Lt(u_1, \mathfrak{a}_1^{\mathrm l}) \cup \MTTL(\{\mathfrak{a}_i\}_{i=1}^k) \cup \Big(\cup_{i=1}^{k-1} \Ch([\mathfrak{a}_i^{\mathrm r}, \mathfrak{a}_{i+1}^{\mathrm l}],*)\Big) \cup \Lt(\mathfrak{a}_k^{\mathrm r},u_2).
\end{equation*}
%provided all the chordal domains satisfy the hypothesis of Proposition~\textup{\ref{LightChordalDomainCandidate}}\textup, the inequality~$m'' \geq 0$ is fulfilled on~$\Lt(u_1, \mathfrak{a}_1^{\mathrm l})$ and~$\Lt(\mathfrak{a}_k^{\mathrm r},u_2)$\textup, and the balance equation~\textup{\eqref{MultibirdieLeftBalance}} is valid.
\end{St}

The case of a linearity domain that is separated from the free boundary is easier than the ones considered. Indeed, in such a case the boundary of the linearity domain consists of the arcs~$\mathfrak{A}_i$,~$i = 1,2,\ldots,k$, the chords~$\mathfrak{A}_{i}^{\textup r}\mathfrak{A}_{i+1}^{\textup l}$,~$i = 1,2,\ldots,k-1$, and the chord~$\mathfrak{A}_{1}^l\mathfrak{A}_{k}^{r}$. Surely, such a construction, called a \emph{closed multicup}\index{multicup! closed multicup}, requires the inequality~$\mathfrak{a}_{k}^r - \mathfrak{a}_1^l < 2\eps$. The following proposition is a consequence of Proposition~\ref{DomainOfLinearityWithTwoPointsOnTheLowerBoundaryMeetsChordalDomain},
%Lemmas~\ref{ThreePointsOnOneLine} and~\ref{ThreePointsAndTheArea}
and the reasoning that proved the~$C^1$-smoothness of the concatenation of a multifigure with 
the chordal domains lying below it (we claim that applies to the case where the chordal domain 
lies above the linearity domain without any modifications).

\begin{St}\label{ClosedMulticupCandidate}
Suppose~$\mathfrak{a}_i = [\mathfrak{a}_i^{\textup l}, \mathfrak{a}_i^{\textup r}]$\textup,~$i = 1,2,\ldots,k$\textup, to be disjoint intervals on~$\mathbb{R}$ \textup(these intervals can be single points) such that~$\mathfrak{a}_k^{\mathrm r} - \mathfrak{a}_{1}^{\mathrm l} < 2\eps$. Assume that all the points~$\big(a,f'(a)\big)$\textup,~$a \in \cup_{i=1}^k\mathfrak{a}_i$\textup, lie on one line\textup, and any pair of points from~$\cup_{i=1}^k\mathfrak{a}_i$ satisfies~\eqref{urlun}. Let the function~$B$ coincide with the standard candidates on~$\ClMTC(\{\mathfrak{a}_i\}_{i=1}^k)$\textup, on each~$\Ch([\mathfrak{a}_i^{\mathrm r}, \mathfrak{a}_{i+1}^{\mathrm l}],*)$\textup, $i=1,\dots, k-1$\textup, and on~$\Ch(*,[\mathfrak{a}_1^{\mathrm l},\mathfrak{a}_k^{\mathrm r}])$. Then the function~$B$ is a~$C^1$-smooth Bellman candidate on the domain
%
%Suppose~$\mathfrak{a}_i = [\mathfrak{a}_i^{\mathrm l}, \mathfrak{a}_i^{\mathrm r}]$\textup,~$i = 1,\ldots,k$\textup,~$\mathfrak{a}_k^{\mathrm r} - \mathfrak{a}_{1}^{\mathrm l} \leq 2\eps$\textup, to be disjoint intervals on~$\mathbb{R}$ \textup(these intervals can be single points\textup). Assume that all the points~$\big(a,f'(a)\big)$\textup,~$a \in \cup_{i=1}^k\mathfrak{a}_i$\textup, lie on one line. Consider the multicup~$\ClMTC(\{\mathfrak{a}_i\}_{i=1}^k)$ built over this set of arcs and a linear function~$B_{\mathrm{ClMcup}(\{\mathfrak{a}_i\}_{i=1}^k)}$ \textup(defined by formulas~\eqref{PolynomialForLinearityDomain} and~\eqref{PolynomialFormula}\textup) that corresponds to it. The function~$B$ is defined by formula~\eqref{vallun} in each~$\Ch([\mathfrak{a}_i^{\mathrm r}, \mathfrak{a}_{i+1}^{\mathrm l}],*)$ and in~$\Ch(*,[\mathfrak{a}_1^{\mathrm l},\mathfrak{a}_k^{\mathrm r}])$ and by formula~$B = B_{\mathrm{ClMcup}(\{\mathfrak{a}_i\}_{i=1}^k)}$ in the closed multicup. The function~$B$ is a Bellman candidate on the domain
\begin{equation*}
\Ch(*,[\mathfrak{a}_1^{\mathrm l},\mathfrak{a}_k^{\mathrm r}]) \cup \ClMTC(\{\mathfrak{a}_i\}_{i=1}^k) \cup \Big(\cup_{i=1}^{k-1} \Ch([\mathfrak{a}_i^{\mathrm r}, \mathfrak{a}_{i+1}^{\mathrm l}],*)\Big).
\end{equation*}
%provided all the chordal domains satisfy the hypothesis of Lemma~\textup{\ref{LightChordalDomainCandidate}}.
\end{St}

A closed multicup is represented graphically in the following way. It has one incoming edge representing~$\Ch(*,[\mathfrak{a}_1^{\mathrm l},\mathfrak{a}_k^{\mathrm r}])$ and several outcoming edges corresponding to the chordal domains~$\Ch([\mathfrak{a}_i^{\mathrm r},\mathfrak{a}_{i+1}^{\mathrm l}],*)$,~$i=1,2,\ldots,k-1$. For example, it may look like the one on Figure~\ref{fig:ClMcupGr}.
\begin{figure}[h]%{o}{200pt}
\vskip-32pt
\begin{center}
\includegraphics[scale=1.1]{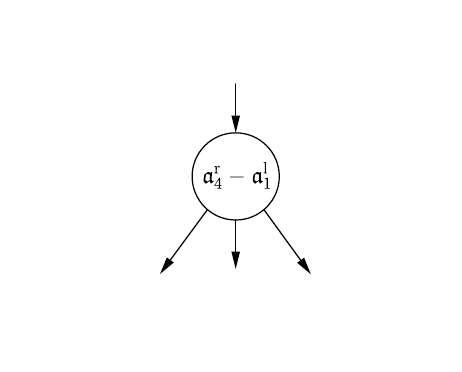}
\vskip-30pt
\caption{
An example of the graph for a closed multicup with the adjacent chordal domains.}
\label{fig:ClMcupGr}
\end{center}
\end{figure}

\section{Combinatorial properties of foliations}\label{s35}
\subsection{Gluing composite figures}\label{s344}
We begin with a detailed example. Consider a long chord~$[A_0,B_0]$ and an angle~$\Ang(a_0)$. Both these figures have their own standard candidates. In our foliations, an angle is usually surrounded by two tangent domains. Though its neighbor on the right is a long chord~$[A_0,B_0]$, it is natural to consider a degenerate tangent domain~$\Lt(a_0,a_0)$ as such a neighbor. Consider a function~$B$ that is continuous and whose restriction to each of the domains~$[A_0,B_0]$,~$\Lt(a_0,a_0)$, and~$\Ang(a_0)$, is a standard candidate there. By Proposition~\ref{DomainOfLinearityWithTwoPointsOnTheLowerBoundaryMeetsTangentDomain}, the function~$m$ for~$\Lt(a_0,a_0)$ is given by the formula
\begin{equation*}
\mlt(a_0) = f'(a_0) + \eps\av{f''}{[a_0,b_0]}.
\end{equation*}
By Proposition~\ref{AngleMeetLeftTangents}, we see that the standard candidate in~$\Ang(a_0)$ is chosen by the formula~$2\beta_2 = \av{f''}{[a_0,b_0]}$. So,~$B$ is~$C^1$-smooth and coincides with the standard candidate in~$\RTroll(a_0,b_0)$. All this leads to the following formula visualized on Figure~\ref{CupPlusAngleRGraph}:
\begin{equation}\label{CupPlusAngleR}
\Ang(a_0) \biguplus \Lt(a_0,a_0) \biguplus [A_0,B_0] = 
\RTroll(a_0,b_0),\quad b_0-a_0 = 2\eps.
\end{equation}
\begin{figure}[h]%{i}{250pt}
%\begin{center}
\vskip-50pt
\hskip10pt
\includegraphics[width = 0.5\linewidth]{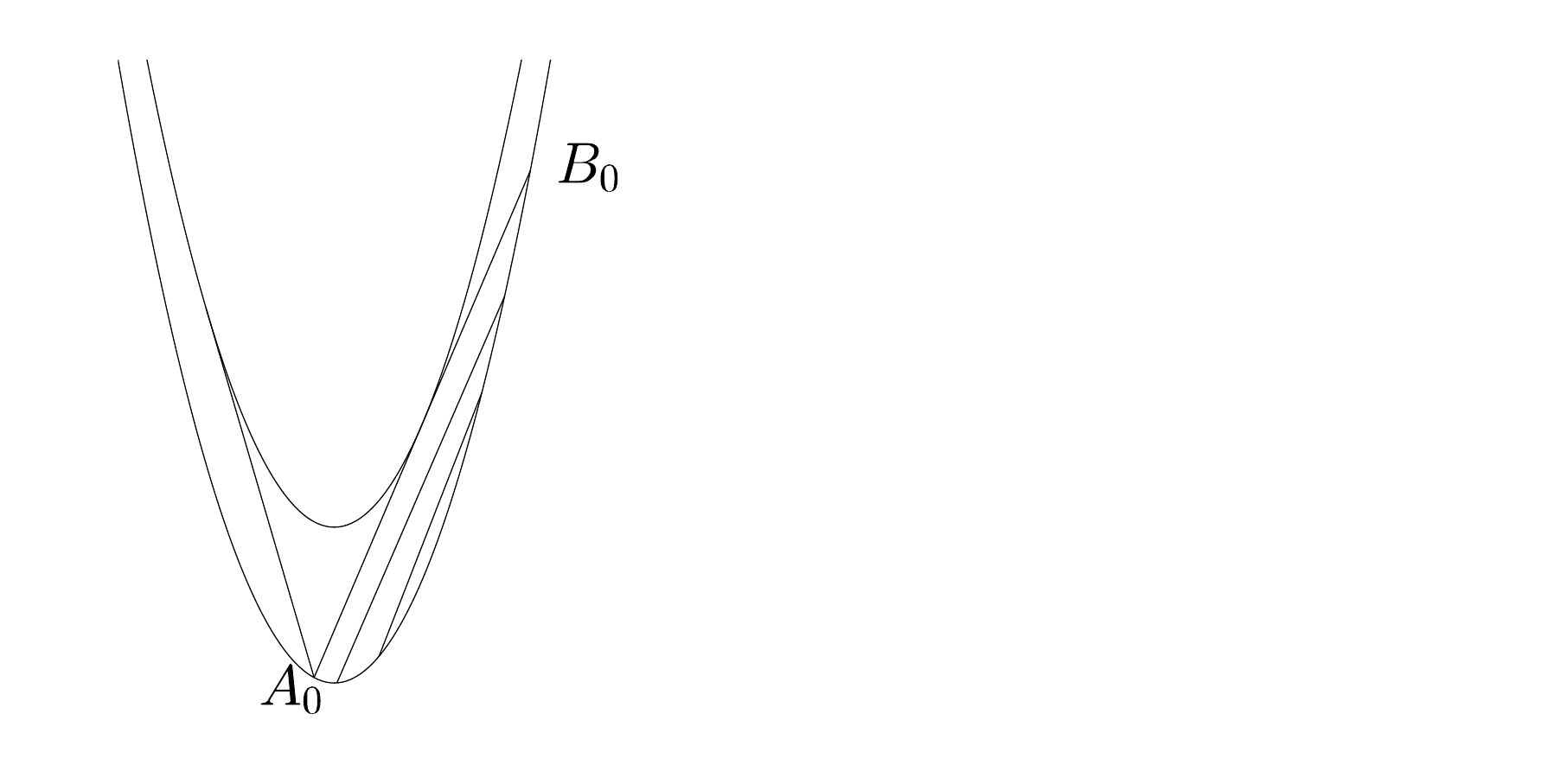}
\hskip-180pt
\raisebox{-35pt}{\includegraphics[width = 1.0\linewidth]{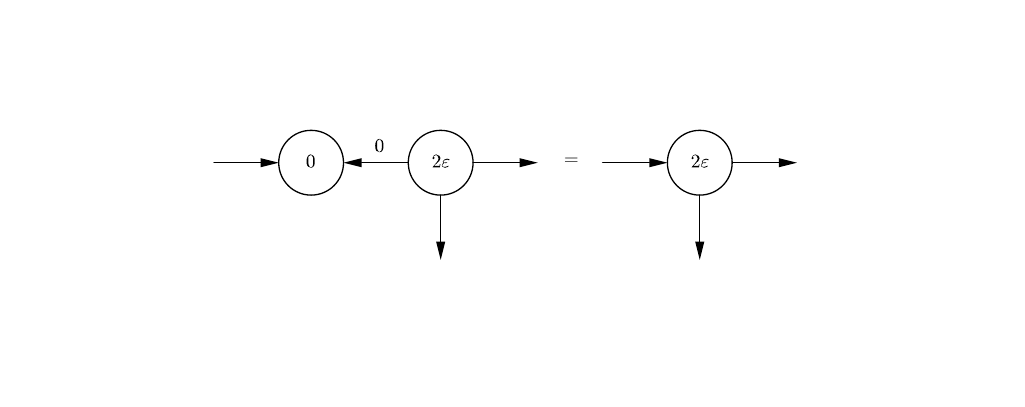}}
\vskip-50pt
\caption{A graphical representation of formula~\eqref{CupPlusAngleR}.}
\label{CupPlusAngleRGraph}
%\end{center}
\end{figure}

%\begin{figure}%[17]{i}{250pt}
%\begin{center}
%\includegraphics[width = 1 \linewidth]{SmallMultitrolleybus.jpg}
%\caption{A graphical representation of formula~\eqref{SmallMtrR}.}
%\label{SmallMtrRGraph}
%\end{center}
%\end{figure} 

Similarly,
\begin{equation}\label{CupPlusAngleL}
\Ang(b_0) \biguplus \Rt(b_0,b_0) \biguplus [A_0,B_0] = \LTroll(a_0,b_0), \quad b_0-a_0 = 2\eps.
\end{equation}
Both these formulas can be informally named as~``{\bf angle + long chord = trolleybus}''.

We have already considered an example of a more complicated formula~\eqref{FirstFormula}:
\begin{equation}\label{RTrolleybusPlusAngle}
\RTroll(a_0,b_0)  \biguplus \Rt(b_0,b_0) \biguplus \Ang(b_0) = \Bird(a_0,b_0);
\end{equation}
\begin{equation}\label{LTrolleybusPlusAngle}
\Ang(a_0) \biguplus \Lt(a_0,a_0)  \biguplus \LTroll(a_0,b_0)= \Bird(a_0,b_0),
\end{equation}
which can be informally named as ``{\bf birdie = angle + trolleybus}''.

%\begin{figure}
%\begin{center}
%\includegraphics[width = 1 \linewidth]{Trolleybus_plus_angle.jpg}
%\caption{A graphical representation of formula~\eqref{RTrolleybusPlusAngle}.}
%\label{RTrolleybusPlusAngleGraph}
%\end{center}
%\end{figure}

%Another example is a destruction of a multitrolleybus sitting on a single arc. Suppose we have an arc~$\mathfrak{A}_1$ such that~$f''' = 0$ on~$\mathfrak{a}$. %%Assume that~$f'''$ is positive on the right and on the left of~$\mathfrak{a}_1$. 
%Then we can build a family of right tangents both on the left and on the right of~$\mathfrak{a}$, and a multitrolleybus on~$\mathfrak{a}$. However, it is the same as the family of tangents on the whole domain considered (See Figure~\ref{SmallMtrRGraph}). These leads us to the two formulas ({\bf multitrolleybus on single arc = tangent domain}):
%\begin{equation}\label{SmallMtrR}
%\Rt(u_1,\mathfrak{a}^{\mathrm l}) \biguplus \MTTR(\{\mathfrak{a}\}) \biguplus \Rt(\mathfrak{a}^{\mathrm r},u_2) = \Rt(u_1,u_2).
%\end{equation}
%\begin{equation}\label{SmallMtrL}
%\Lt(u_1,\mathfrak{a}^{\mathrm l}) \biguplus \MTTL(\{\mathfrak{a}\}) \biguplus \Lt(\mathfrak{a}^{\mathrm r},u_2) = \Lt(u_1,u_2).
%\end{equation}
%We note that these formulas differ from the previous ones: they should be read only from left to right (i.e. not every standard candidate in a tangent domain can be represented as a continuous concatenation of three standard candidates as prescribed by formulas~\eqref{SmallMtrR},~\eqref{SmallMtrL}).

We provide the same-fashioned formulas for other domains. We leave their verification to the reader.

%\begin{figure}
%\begin{center}
%\includegraphics[width = 1 \linewidth]{Multicup_plus_angle.jpg}
%\caption{A graphical representation of formula~\eqref{AnglePlusMulticupLeft}.}
%\label{AnglePlusMulticupLeftGraph}
%\end{center}
%\end{figure}
\paragraph{Angle + multicup = multitrolleybus}
\begin{equation}\label{AnglePlusMulticupLeft}
\Ang(\mathfrak{a}_1^{\mathrm l}) \biguplus \Lt(\mathfrak{a}_1^{\mathrm l},\mathfrak{a}_1^{\mathrm l}) \biguplus \MTC(\{\mathfrak{a}_i\}_{i=1}^k) = \MTTR(\{\mathfrak{a}_i\}_{i=1}^k);
\end{equation}

\begin{figure}[h]
\vskip-42pt
\hskip-70pt
\includegraphics[width=0.8\linewidth]{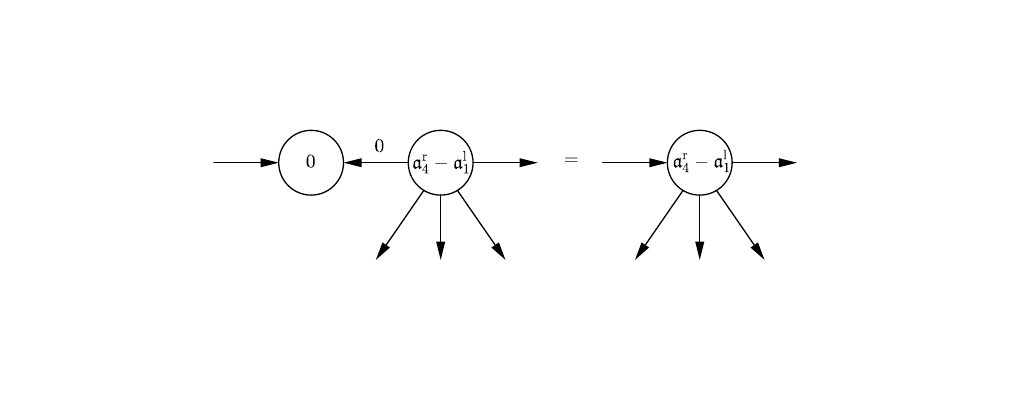}
\hskip-130pt
\includegraphics[width=0.8\linewidth]{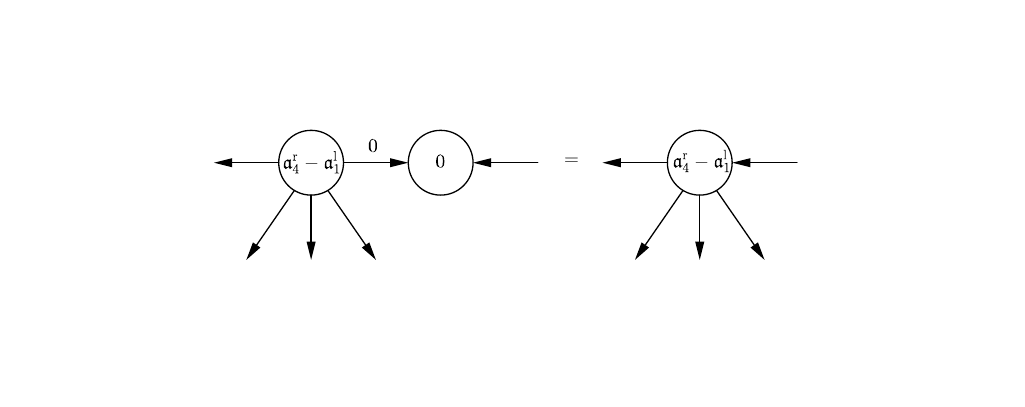}
\hskip-60pt
\vskip-60pt
\caption{The graphs for formulas~\eqref{AnglePlusMulticupLeft} and~\eqref{AnglePlusMulticupRight}.}
\label{fig_AnglePlusMcup}
\end{figure}

\begin{equation}\label{AnglePlusMulticupRight}
\MTC(\{\mathfrak{a}_i\}_{i=1}^k) \biguplus  \Rt(\mathfrak{a}_k^{\mathrm r},\mathfrak{a}_k^{\mathrm r}) \biguplus \Ang(\mathfrak{a}_k^{\mathrm r})  = \MTTL(\{\mathfrak{a}_i\}_{i=1}^k).
\end{equation}
%\begin{figure}
%\begin{center}
%\includegraphics[width = 1 \linewidth]{FullChordalDomain_plus_trolleybus.jpg}
%\caption{A graphical representation of formula~\eqref{ChordalDomainPlusLTrolleybus}.}
%\label{ChordalDomainPlusLTrolleybusGraph}
%\end{center}
%\end{figure}

%\paragraph{Full chordal domain + trolleybus = multicup}
%\begin{equation}\label{ChordalDomainPlusRTrolleybus}
%\Ch([a_1,a_2],*) \biguplus \Rt(a_2,a_2) \biguplus \RTroll(a_2,a_3) = \Ch([a_1,a_2],*) \biguplus \MTC\big(\{a_1,a_2,a_3\}\big);
%\end{equation}
%\begin{equation}\label{ChordalDomainPlusLTrolleybus}
%\LTroll(a_1,a_2) \biguplus \Lt(a_2,a_2) \biguplus \Ch([a_2,a_3],*) = \MTC\big(\{a_1,a_2,a_3\}\big) \biguplus \Ch([a_2,a_3],*).
%\end{equation}

\paragraph{Long chord + multibirdie = multitrolleybus}
\begin{equation}\label{ChordalDomainPlusMultibirdieL}
[\mathfrak{A}_0, \mathfrak{A}_1^{\mathrm l}] \biguplus \Rt(\mathfrak{a}_1^{\mathrm l},\mathfrak{a}_1^{\mathrm l})\biguplus \MTB(\{\mathfrak{a}_i\}_{i=1}^k) = \MTTL(\{\mathfrak{a}_0\} \cup\{\mathfrak{a}_i\}_{i=1}^k), \quad \mathfrak{a}_1^{\mathrm{l}} - \mathfrak{a}_0 = 2\eps;
\end{equation}
%\begin{figure}[h]
%\vskip-42pt
%\hskip-70pt
%\includegraphics[width=0.8\linewidth]{fig_LChordPlusMbirdie}
%\hskip-130pt
%\includegraphics[width=0.8\linewidth]{fig_RChordPlusMbirdie}
%\hskip-60pt
%\vskip-60pt
%\caption{The graphs for formulas~\eqref{AnglePlusMulticupLeft} and~\eqref{AnglePlusMulticupRight}.}
%\label{fig_ChordPlusMbirdie}
%\end{figure}
\begin{equation}\label{ChordalDomainPlusMultibirdieR}
\MTB(\{\mathfrak{a}_i\}_{i=1}^k) \biguplus \Lt(\mathfrak{a}_k^{\mathrm r},\mathfrak{a}_k^{\mathrm r}) \biguplus [\mathfrak{A}_k^{\mathrm r}, \mathfrak{A}_{k+1}] =\MTTR(\{\mathfrak{a}_i\}_{i=1}^k \cup \{\mathfrak{a}_{k+1}\}),\quad \mathfrak{a}_{k+1} - \mathfrak{a}_{k}^{\mathrm r} = 2\eps.
\end{equation}

\begin{figure}[h]
\vskip-40pt
\hskip10pt
\includegraphics[width = 0.6\linewidth]{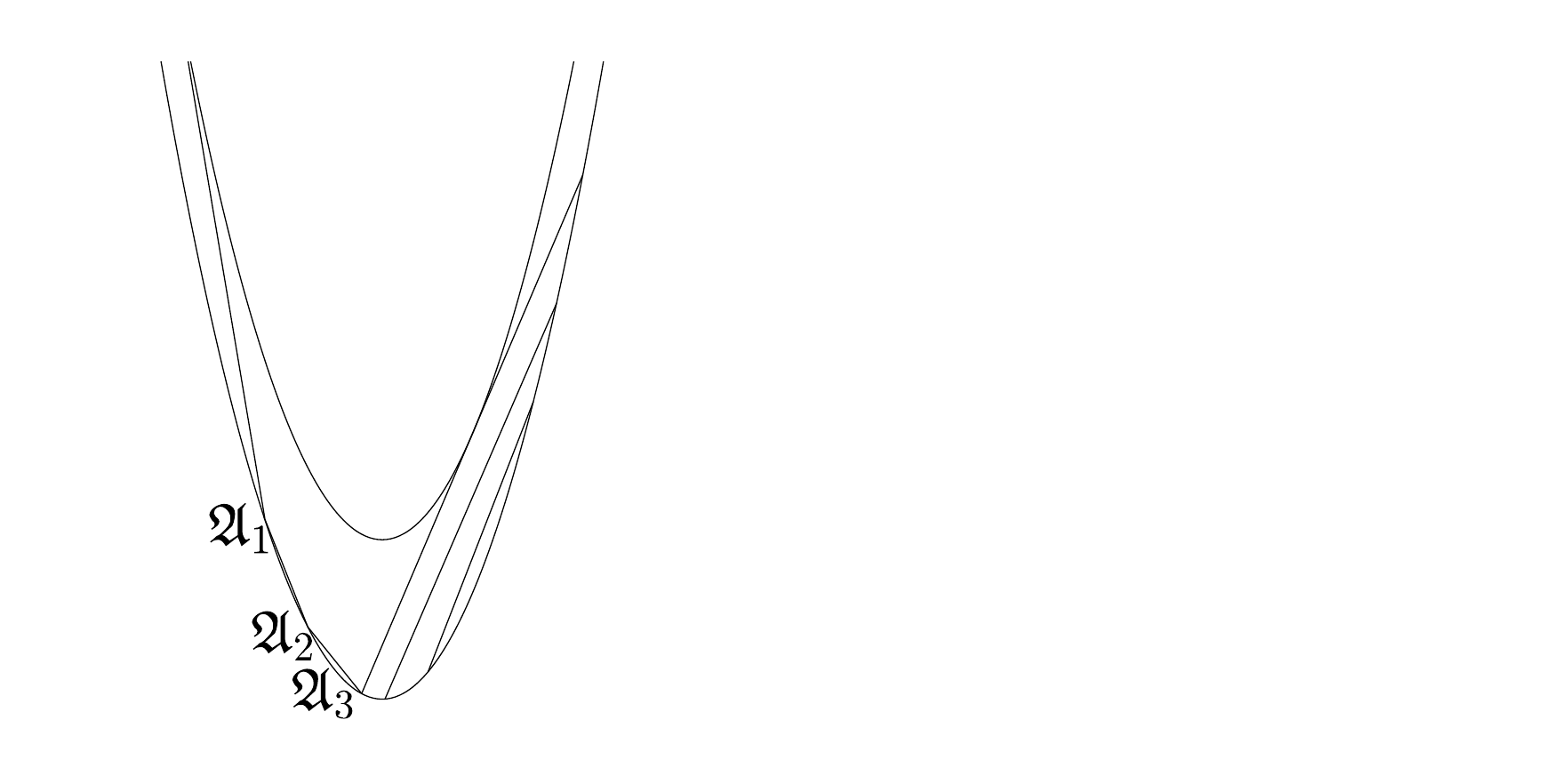}
\hskip-250pt
\raisebox{-35pt}{\includegraphics[width = 1.1\linewidth]{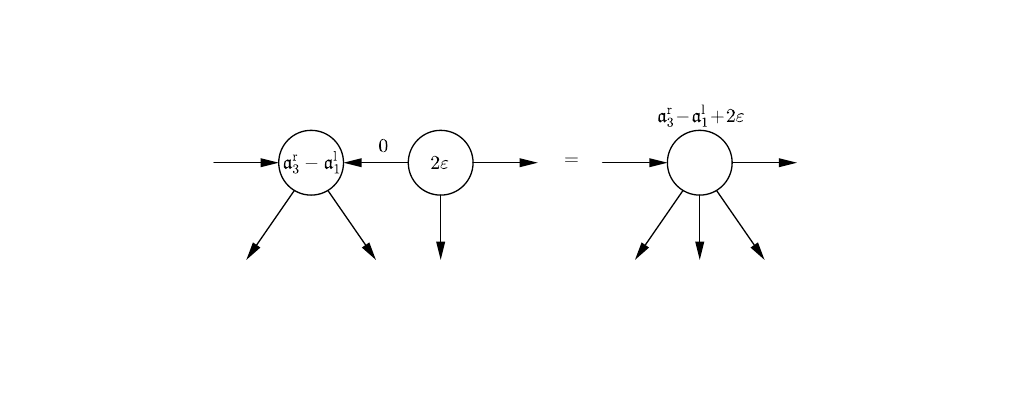}}
\vskip-40pt
\caption{A graphical representation of formula~\eqref{ChordalDomainPlusMultibirdieR}.}
\label{ChordalDomainPlusMultibirdieRGraph}
\end{figure}

\paragraph{Angle + multitrolleybus = multibirdie}
\begin{equation}\label{AnglePlusMultiTrollebusR}
\MTTR(\{\mathfrak{a}_i\}_{i=1}^k) \biguplus \Rt(\mathfrak{a}_k^{\mathrm r},\mathfrak{a}_k^{\mathrm r}) \biguplus \Ang(\mathfrak{a}_k^{\mathrm r}) = \MTB(\{\mathfrak{a}_i\}_{i=1}^k);
\end{equation}
\begin{figure}[h]
\vskip-20pt
\hskip-70pt
\includegraphics[width=0.8\linewidth]{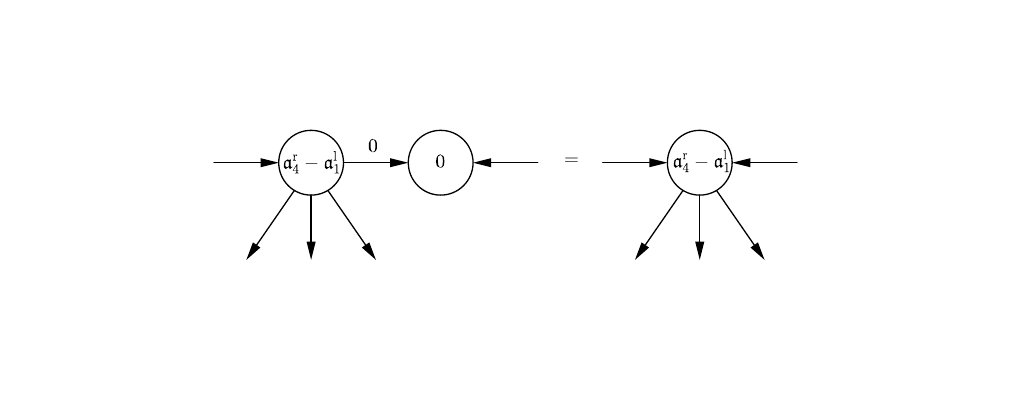}
\hskip-130pt
\includegraphics[width=0.8\linewidth]{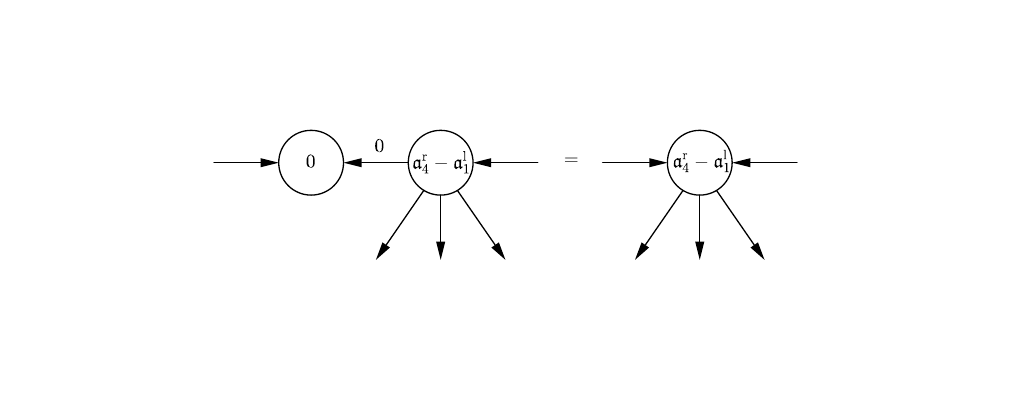}
\hskip-60pt
\vskip-60pt
\caption{The graphs for formulas~\eqref{AnglePlusMultiTrollebusR} and~\eqref{AnglePlusMultiTrollebusL}.}
%\label{fig_AnglePlusMTroll}
\end{figure}
\begin{equation}\label{AnglePlusMultiTrollebusL}
\Ang(\mathfrak{a}_1^{\mathrm l}) \biguplus \Lt(\mathfrak{a}_1^{\mathrm l},\mathfrak{a}_1^{\mathrm l}) \biguplus \MTTL(\{\mathfrak{a}_i\}_{i=1}^k)= \MTB(\{\mathfrak{a}_i\}_{i=1}^k).
\end{equation}
%\begin{figure}
%\begin{center}
%\includegraphics[width = 1 \linewidth]{Angle_plus_Multitrolleybus.jpg}
%\caption{A graphical representation of formula~\eqref{AnglePlusMultiTrollebusR}.}
%\label{AnglePlusMultiTrollebusRGraph}
%\end{center}
%\end{figure}
%\begin{figure}
%\begin{center}
%\includegraphics[width = 1 \linewidth]{FullChordalDomain_plus_Multibirdie.jpg}
%\caption{A graphical representation of formula~\eqref{ChordalDomainPlusMultibirdieR}.}
%\label{ChordalDomainPlusMultibirdieRGraph}
%\end{center}
%\end{figure}

%\begin{figure}
%\begin{center}
%\includegraphics[width = 1 \linewidth]{FullChordalDomain_plus_Multitrolleybus.jpg}
%\caption{A graphical representation of formula~\eqref{ChordalDomainPlusMultitrolleybusL}.}
%\label{ChordalDomainPlusMultitrolleybusLGraph}
%\end{center}
%\end{figure}
\paragraph{Long chord + multitrolleybus = multicup}
\begin{equation}\label{ChordalDomainPlusMultitrolleybusR}
[\mathfrak{A}_0, \mathfrak{A}_1^{\mathrm l}] \biguplus \Rt(\mathfrak{a}_1^{\mathrm l},\mathfrak{a}_1^{\mathrm l})\biguplus \MTTR(\{\mathfrak{a}_i\}_{i=1}^k) = \MTC(\{\mathfrak{a}_0\} \cup\{\mathfrak{a}_i\}_{i=1}^k),\quad \mathfrak{a}_1^{\mathrm l} - \mathfrak{a}_0 = 2\eps;
\end{equation}
\begin{figure}[h]
\vskip-20pt
\hskip-70pt
\includegraphics[width=0.8\linewidth]{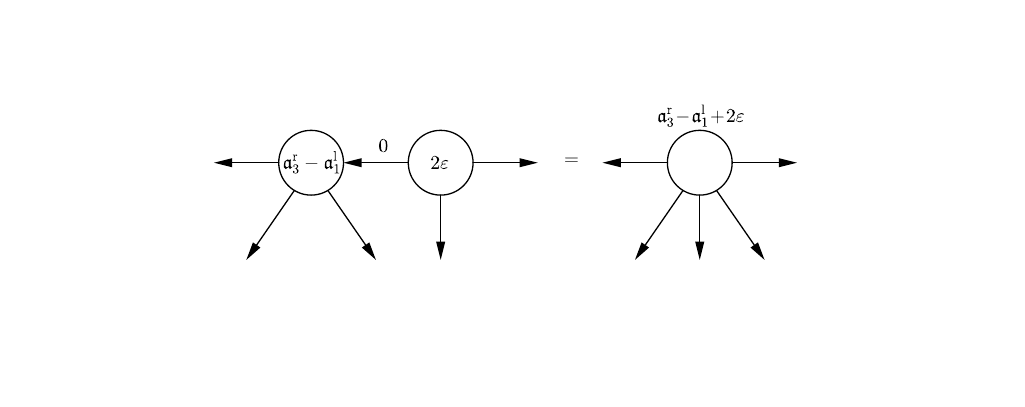}
\hskip-130pt
\includegraphics[width=0.8\linewidth]{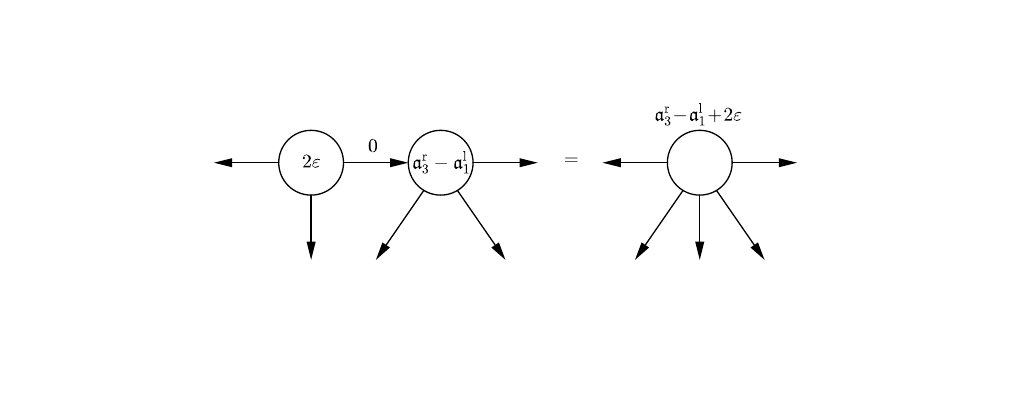}
\hskip-60pt
\vskip-60pt
\caption{The graphs for formulas~\eqref{ChordalDomainPlusMultitrolleybusR} and~\eqref{ChordalDomainPlusMultitrolleybusL}.}
%\label{fig_AnglePlusMTroll}
\end{figure}
\begin{equation}\label{ChordalDomainPlusMultitrolleybusL}
\MTTL(\{\mathfrak{a}_i\}_{i=1}^k) \biguplus \Lt(\mathfrak{a}_k^{\mathrm r},\mathfrak{a}_k^{\mathrm r}) \biguplus [\mathfrak{A}_k^{\mathrm r}, \mathfrak{A}_{k+1}] = \MTC(\{\mathfrak{a}_i\}_{i=1}^k \cup \{\mathfrak{a}_{k+1}\})\quad \mathfrak{a}_{k+1} - \mathfrak{a}_k^{\mathrm r} = 2\eps.
\end{equation}

\paragraph{Multicup + multitrolleybus = multicup}
\begin{equation}\label{MTCplusMTTR}
\MTC(\{\mathfrak{a}_i\}_{i=1}^k) \biguplus \Rt(\mathfrak{a}_k^{\mathrm r},\mathfrak{a}_k^{\mathrm r}) \biguplus \MTTR(\{\mathfrak{a}_i\}_{i=k+1}^{m}) = \MTC(\{\mathfrak{a}_i\}_{i=1}^m),\quad \mathfrak{a}_k^{\mathrm r} = \mathfrak{a}_{k+1}^{\mathrm l}, m \geq k+1;
\end{equation}
\begin{equation}\label{MTCplusMTTL}
\MTTL(\{\mathfrak{a}_i\}_{i=1}^k) \biguplus \Lt(\mathfrak{a}_k^{\mathrm r},\mathfrak{a}_k^{\mathrm r}) \biguplus \MTC(\{\mathfrak{a}_i\}_{i=k+1}^{m}) = \MTC(\{\mathfrak{a}_i\}_{i=1}^m),\quad \mathfrak{a}_k^{\mathrm r} = \mathfrak{a}_{k+1}^{\mathrm l}, m \geq k+1;
\end{equation}

\paragraph{Multicup + birdie = multitrolleybus}
\begin{equation}\label{MTCplusBirdieR}
\MTC(\{\mathfrak{a}_i\}_{i=1}^k) \biguplus \Rt(\mathfrak{a}_k^{\mathrm r},\mathfrak{a}_k^{\mathrm r}) \biguplus \MTB(\{\mathfrak{a}_i\}_{i=k+1}^{m}) = \MTTL(\{\mathfrak{a}_i\}_{i=1}^m),\quad \mathfrak{a}_k^{\mathrm r} = \mathfrak{a}_{k+1}^{\mathrm l}, m \geq k+1;
\end{equation}
\begin{equation}\label{MTCplusBirdieL}
\MTB(\{\mathfrak{a}_i\}_{i=1}^k) \biguplus \Lt(\mathfrak{a}_k^{\mathrm r},\mathfrak{a}_k^{\mathrm r}) \biguplus \MTC(\{\mathfrak{a}_i\}_{i=k+1}^{m}) = \MTTR(\{\mathfrak{a}_i\}_{i=1}^m),\quad \mathfrak{a}_k^{\mathrm r} = \mathfrak{a}_{k+1}^{\mathrm l}, m \geq k+1;
\end{equation}

\paragraph{Multitrolleybus = trolleybus parade}
\begin{equation}\label{RMultitrolleybusDesintegration}
\MTTR(\{\mathfrak{a}_i\}_{i=1}^k) = \Big(\biguplus_{i=1}^{k}\MTTR(\{\mathfrak{a}_i\})\Big) \biguplus \Big(\biguplus_{i=1}^{k-1}\RTroll(\mathfrak{a}_i^{\mathrm r},\mathfrak{a}_{i+1}^{\mathrm l})\Big);
\end{equation}
\begin{equation}\label{LMultitrolleybusDesintegration}
\MTTL(\{\mathfrak{a}_i\}_{i=1}^k) = \Big(\biguplus_{i=1}^{k}\MTTL(\{\mathfrak{a}_i\})\Big) \biguplus \Big(\biguplus_{i=1}^{k-1}\LTroll(\mathfrak{a}_i^{\mathrm r},\mathfrak{a}_{i+1}^{\mathrm l})\Big).
\end{equation}
%\begin{figure}
%\begin{center}
%\includegraphics[width = 1 \linewidth]{TrolleybusParade.jpg}
%\caption{A graphical representation of formula~\eqref{LMultitrolleybusDesintegration}.}
%\label{LMultitrolleybusDesintegrationGraph}
%\end{center}
%\end{figure}
In these formulas, the multitrolleybuses on the right should be changed for the degenerate tangent domains~$\Rt(\mathfrak{a}_i,\mathfrak{a}_i)$ or~$\Lt(\mathfrak{a}_i,\mathfrak{a}_i)$, provided the corresponding~$\mathfrak{a}_i$ is a single point. %However, multitrolleybuses on solid arcs can also be changed for tangent domains  using formulas~\eqref{SmallMtrR} and~\eqref{SmallMtrL}.

\paragraph{Multibirdie = trolleybus parade + angle + trolleybus parade}
\begin{multline}\label{MultibirdieDesintegration}
\MTB(\{\mathfrak{a}_i\}_{i=1}^k) = \Bigg(\Big(\biguplus_{i=1}^{j-1}\RTroll(\mathfrak{a}_i^{\mathrm r},\mathfrak{a}_{i+1}^{\mathrm l})\Big) \biguplus \Big(\biguplus_{i=1}^{j-1}\MTTR(\{\mathfrak{a}_i\})\Big)\Bigg) \biguplus \\
\MTB(\{\mathfrak{a}_j\}) \biguplus \Bigg(\Big(\biguplus_{i=j+1}^{k}\MTTL(\{\mathfrak{a}_i\})\Big) \biguplus \Big(\biguplus_{i=j}^{k-1}\LTroll(\mathfrak{a}_i^{\mathrm r},\mathfrak{a}_{i+1}^{\mathrm l})\Big)\Bigg)
\end{multline}
Here~$j$ is an arbitrary number,~$j = 1,2, \ldots,k$. Similarly to the previous case, the multitrolleybuses on the right-hand side of the formula should be changed for the degenerate tangent domains, provided the arcs they are sitting on are single points. If~$\mathfrak{a}_j$ is a single point, then one should change~$\MTB(\{\mathfrak{a}_j\})$ for~$\Ang(\mathfrak{a}_j)$. %We note that the birdie in its turn can be decomposed as a sum of a trolleybus and an angle with the help of formula~\eqref{FirstFormula}.
\begin{figure}[h]
\vskip-37pt
\hskip-5pt
\includegraphics[width = 0.4\linewidth]{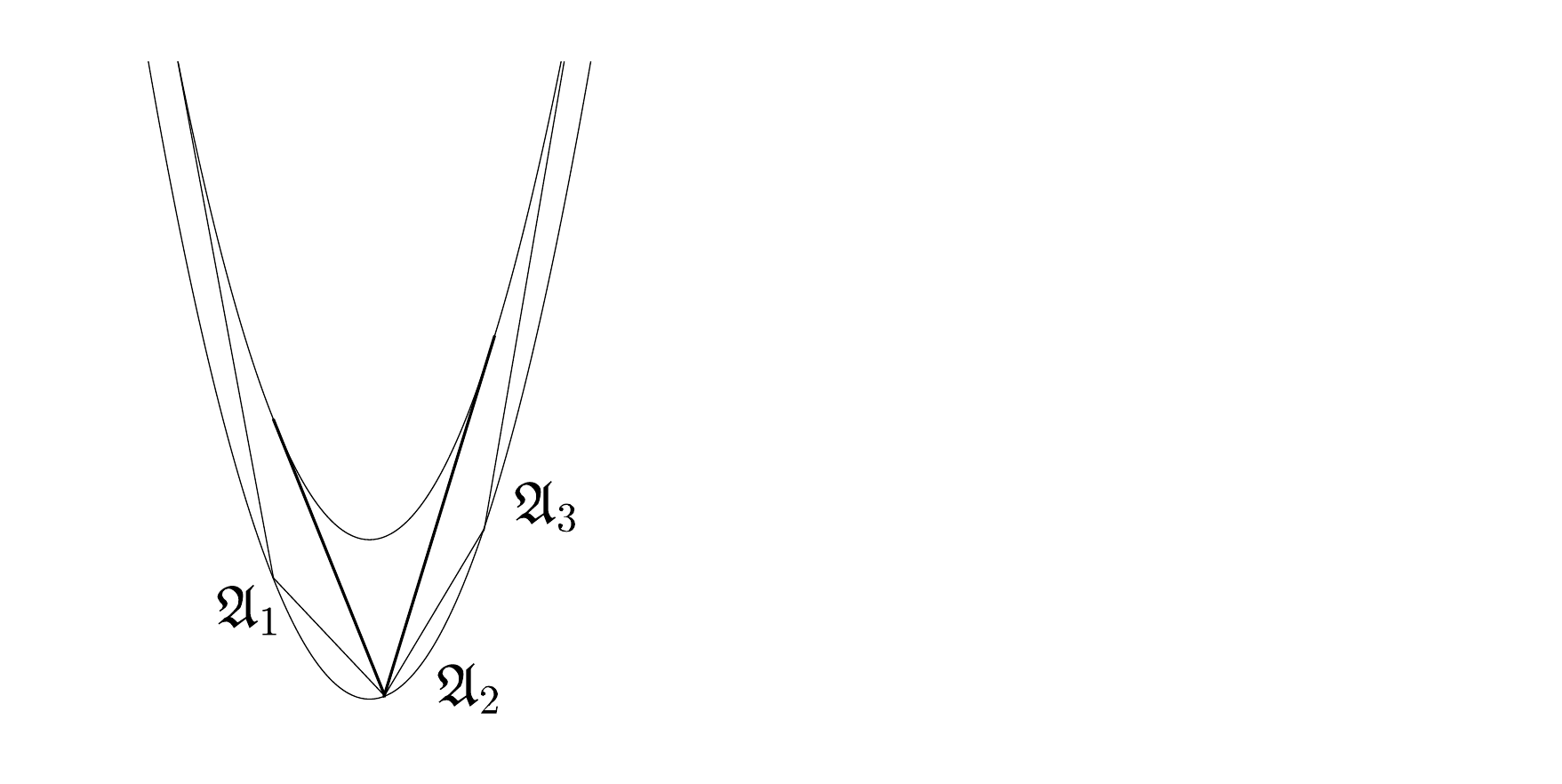}
\hskip-205pt
\raisebox{-25pt}{\includegraphics[width = 1.3\linewidth]{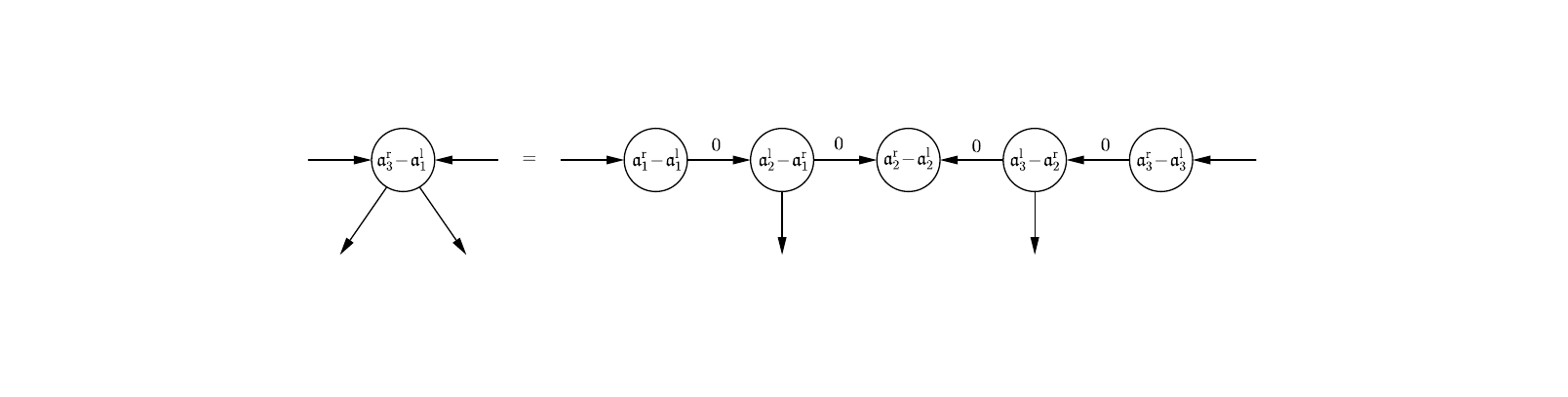}}
\vskip-40pt
\caption{A graphical representation of formula~\eqref{MultibirdieDesintegration}.}
\label{MultibirdieDesintegrationGraph}
\end{figure}

\paragraph{Closed multicup + trolleybus = multitrolleybus}
\begin{equation}\label{ClosedMulticupTrolleybusR}
\ClMTC(\{\mathfrak{a_i}\}_{i=1}^k)\biguplus\RTroll(\mathfrak{a}_1^{\mathrm l},\mathfrak{a}_k^{\mathrm r}) = \MTTR(\{\mathfrak{a_i}\}_{i=1}^k).
\end{equation}
\begin{equation}\label{ClosedMulticupTrolleybusL}
\ClMTC(\{\mathfrak{a_i}\}_{i=1}^k)\biguplus\LTroll(\mathfrak{a}_1^{\mathrm l},\mathfrak{a}_k^{\mathrm r}) = \MTTL(\{\mathfrak{a_i}\}_{i=1}^k).
\end{equation}
\paragraph{Closed multicup + birdie = multibirdie}
\begin{equation}\label{ClosedMulticupBirdie}
\ClMTC(\{\mathfrak{a_i}\}_{i=1}^k)\biguplus\Bird(\mathfrak{a}_1^{\mathrm l},\mathfrak{a}_k^{\mathrm r}) = \MTB(\{\mathfrak{a_i}\}_{i=1}^k).
\end{equation}

\subsection{General foliations}\label{s345}
It is natural to draw a special graph~$\Gamma$\index{graph!} corresponding to a foliation to describe its combinatorial properties. The vertices correspond to the linearity domains. Two vertices are joined with an edge if there is a domain of extremals that is their common neighbor. Such a graph is drawn in the plane by the mapping~$\nabla B\colon \Omega_{\eps} \to \mathbb{R}^2$. However, we need to clarify some details.

We will use a small amount of graph terminology. Since we study very special graphs, the use of the terminology will also be special. Our graphs are oriented trees (i.e.  trees whose edges possess orientation). We call a vertex that does not have incoming edges a root, a vertex that does not have outcoming edges a leaf (a leaf may have several incoming edges). By a path we call an oriented path, i.e. we move from the beginning of the edge to its end while exploring the path. Other terminology is clear.

The vertices of the graph will be denoted by~$\{\mathfrak{L}_i\}_i$, the edges will be denoted by~$\{\mathfrak{E}_i\}_i$. Edges and vertices are of different types, moreover, they are also equipped with numerical parameters to be specified later. We begin the description with the edges.

Each edge~$\mathfrak{E}$ represents either a chordal domain~$\Ch([\ato,\bto],[\abot,\bbot])$ or a tangent domain~$\Rt(\uul,\uur)$ or~$\Lt(\uul,\uur)$. The edge of~$\Rt(\uul,\uur)$ is oriented from the vertex of its left neighbor to the vertex of its right neighbor. The edge representing~$\Lt(\uul,\uur)$ is oriented symmetrically. 
The edge representing a chordal domain~$\Ch([\ato,\bto],[\abot,\bbot])$ is oriented from its upper neighbor to its lower neighbor. We consider the functions~$a$ and~$b$ associated to a chordal domain as its numeric parameters.

The vertices correspond to the linearity domains. For angles, trolleybuses, birdies, and multifigures the graphical representation was given in the subsections where they were introduced. These vertices are of their individual types (i.e. there are several vertices of the type ``angle'' in the graph, several vertices of the type ``birdie'', etc.). Each such vertex is equipped with its points on the lower parabola as its numerical characteristics. For example, a vertex of the type ``angle'' has one numerical parameter~---  the first coordinate~$w$ of the point~$W$ the angle is sitting on, whereas the collection of the intervals~$\{\mathfrak{a}_i\}_{i=1}^k$ plays the role of the numerical parameter for a vertex that has the type ``multicup'', or ``right multitrolleybus'', or ``left multitrolleybus'', or ``multibiridie'', or ``closed multicup''.

However, we also need some fictious vertices, which do not correspond to any linearity domain of non-zero area. For example, on Figure~\ref{fig:abtcd} two vertices representing long chords are fictious. There will be five types of such vertices.

First, there will be some~$\mathfrak{L}_i$ that correspond to the long chords (the chords that touch the upper parabola). Namely, suppose that we have a full chordal domain~$\Ch([\ato,\bto],*)$ such that~$\Srt(\ato,\bto) \ne 0$ and~$\Slt(\ato,\bto) \ne 0$, and two tangent domains,~$\Rt(\bto,u_2)$ and~$\Lt(u_1,\ato)$. In other words, this part of the foliation falls under the scope of Propostions~\ref{CupMeetsRightTangents} and~\ref{CupMeetsLeftTangents}. Then, the vertex~$\mathfrak{L}$ corresponding to the chord~$[\Ato,\Bto]$ has three outcoming edges representing~$\Ch([\ato,\bto],*)$, $\Rt(\bto,u_2)$, and~$\Lt(u_1,\ato)$. The pair~$(\ato,\bto)$ is the numerical parameter for~$\mathfrak{L}$. The example is given on Figure~\ref{fig:cup_graph}. 

Second, there will be some vertices~$\mathfrak{L}_i$ that correspond to points of the fixed boundary. Suppose we have a chordal domain~$\Ch(*,[\abot,\bbot])$ with~$\abot = \bbot$ (we recall that such chordal domains are called cups). In our foliations, all such points will coincide with some~$c_j$ from Definition~\ref{roots}. Then, the vertex~$\mathfrak{L}$ corresponding to~$\abot = \bbot$ has one incoming edge matching~$\Ch(*,[\abot,\bbot])$ and one numerical parameter that equals~$\abot$.

Third, sometimes we will need to paste a chord between two chordal domains (this will be done when one of the differentials, see Definition~\ref{differentials}, vanishes). Suppose we have two chordal domains,~$\Ch([a_1,b_1],[a_2,b_2])$ and~$\Ch([a_2,b_2],[a_3,b_3])$. In such a case, there is some vertex~$\mathfrak{L}_i$ that corresponds to the chord~$[A_2,B_2]$. It has one incoming edge and one outcoming edge and two numerical parameters~$a_2$ and~$b_2$. Long chords, one or both differentials of which vanish, are also considered as fictious vertices of the third type.  

Fourth, there might be one or two vertices at infinity. If we have the domain~$\Rt(-\infty,\uur)$, then there is a vertex~$\mathfrak{L}$ that corresponds to~$-\infty$. It has the numerical parameter~$-\infty$ and one outcoming edge represnting~$\Rt(-\infty,\uur)$. Similarly, if~$\uur = \infty$, then there is a vertex~$\mathfrak{L}$ that corresponds to~$+\infty$ and has one incoming edge and~$+\infty$ as a numerical parameter. We also draw vertices at infinity by the symmetric rule in the case of left tangents running to infinity.

Fifth, there might be vertices corresponding to single tangents. Suppose we have a tangent domain~$\Rt(u_1,u_2)$ such that~$m'' >0$ on~$[u_1,u_2]$ except for some point~$u$, where~$m''$ equals zero\footnote{In such a situation,~$u = c_j$ for some root~$c_j$ (see Definition~\ref{roots}).}. Then, it is useful to decompose~$\Rt(u_1,u_2)$ as
\begin{equation*}
\Rt(u_1,u) \biguplus \Rt(u,u) \biguplus \Rt(u,u_2)
\end{equation*}
and paste a vertex representing~$\Rt(u,u)$ (alternatively, one can consider it as a multitrolleybus on a single point). It has one outcoming edge~$\Rt(u,u_2)$ and one incoming edge~$\Rt(u_1,u)$. Its numerical parameter is~$u$. The same things can be done for the case of left tangents. We note that the fictious vertices of the fifth type may be right and left (as well as the trolleybuses).

The rules listed above define the graph of the foliation. However, we provide further description to make its structure more transparent. It is useful to introduce a partial ordering on the set of linearity domains. 
\begin{Def}\label{Ordering}
Let~$A$ and~$B$ be two linearity domains. We say that~$B$ is subordinate to~$A$ and write~$B \prec A$ if~$A$ separates~$B$ from the upper parabola. 
\end{Def} 
Note that in such a case~$B$ is a closed multicup.  We can also let~$A$ and~$B$ to be chords, and let~$B$ to be a point on the fixed boundary. Therefore, the partial ordering introduced can be extended naturally to the set~$\mathfrak{L}$ of all vertices of~$\Gamma$. One more thing to notice is that the numerical parameters of the vertices~$\mathfrak{L}_1$ and~$\mathfrak{L}_2$ are sufficient to define whether the statement~$\mathfrak{L}_1 \prec \mathfrak{L}_2$ is true.  

We explain how to construct the graph from a foliation. Our graph is a tree if we disregard the orientation. First, we describe its subgraph~$\GammaFree$ spanned by the edges representing tangent domains. This subgraph describes the trace of the foliation on the free boundary. Formally, we can define~$\GammaFree$ to be the set of vertices that are not subordinated by any other vertex, and the edges between them. If we forget the orientation of edges,~$\GammaFree$ is a path, i.e. a tree whose vertices have degree two, except, possibly for two leaves at infinity. The leaves are usually fictious vertices of the fourth type, however, if there is a multicup or a multitrolleybus that lasts to infinity, its vertex is a leaf (in such a case there is no fictious vertex representing the corresponding infinity). The orientation of edges has already been described. We only say that the roots of~$\GammaFree$ are the fictious vertices of the first and third type (the latter, of course, should belong to~$\GammaFree$, i.e. represent a long chord), the vertices that correspond to the multicups, and possibly, the vertices at infinity. The leaves  in~$\GammaFree$ correspond to angles, birdies, multibirdies, and possibly, vertices at infinity. Surely, the necessary and sufficient condition for~$\GammaFree$ to be a subgraph spanned by the edges corresponding to tangent domains of some foliation is that the foliation reconstructed from it covers the free boundary without intersections. We note that the graph~$\GammaFree$ was called a signature in our previous paper~\cite{5A}.
\begin{Def}\label{Height}\index{height}
Define the height of a vertex in~$\GammaFree$ as the number of edges in the shortest path from a root of~$\GammaFree$ to this vertex.
\end{Def}

Second, we describe the graph~$\GammaFixed$ spanned by the edges corresponding to chordal domains. The graph~$\GammaFixed$ is a forest (i.e. a finite collection of trees). Each tree of the forest is oriented from its root, being any vertex of~$\GammaFree$ except for fictious vertices of the fourth or fifth types, and multifigures sitting on single arcs, to its leaves. The leaves of~$\GammaFixed$ are the fictious vertices of the second type (corresponding to the origins of cups) and closed multicups sitting on single arcs. All other vertices are closed multicups and fictious vertices of the third type. We note that this graph is generated by the ordering introduced in Definition~\ref{Ordering}: each edge~$\mathfrak{E}$ goes from~$\mathfrak{L}_1$ to $\mathfrak{L}_2$ if and only if~$\mathfrak{L}_2 \prec \mathfrak{L}_1$ and there are no vertices~$\mathfrak{L}_3$ such that~$\mathfrak{L}_2 \prec \mathfrak{L}_3 \prec \mathfrak{L}_1$. The necessary and sufficient condition for~$\GammaFixed$ to be a subgraph spanned by the edges corresponding to chordal domains of some foliation is that the linearity domains built from its vertices do not intersect, the edges are generated by the ordering from Definition~\ref{Ordering}, and the roots and the leaves correspond to the vertices of the type they should correspond to. 

So, the graph of the foliation is a finite oriented tree whose vertices and edges have type (they correspond either to some figures or to fictious constructions described above) and several numerical characteristics regarding their type.

We warn the reader that we do not write down all the numerical parameters when we draw graphs, this makes our illustrations more clear. We only indicate some of the parameters, for example, we usually write the number~$b_0-a_0$ on the vertex represnting a long chord~$[A_0,B_0]$ or a trolleybus~$\RTroll(a_0,b_0)$, etc.  

\subsection{Examples}\label{s353}
The examples we provide here do not only demonstrate how to use the methods we have developed, but also give some hints to the evolutional behavior of foliations to be discussed in the forthcoming section. So, after the Bellman candidate is constructed, we give a short summary concerning its evolution in~$\eps$.
\paragraph{Stable angle.}
This example originates from inequality~\eqref{bmoequivnorms}. It was already considered in~\cite{SlVa2}. Consider the function $f(t) = |t|^p$.  If $p>2$, this function satisfies Conditions~\ref{reg}~and~\ref{sum}. We do not write down the formula for the Bellman function itself (this was done in~\cite{SlVa2}). But we verify that the situation falls under the scope of Proposition~\ref{AngleProp}. Indeed, the expression
\begin{equation*}
\mrt''(u;\,-\infty) + \mlt''(u;\,+\infty)
=\eps^{-1}\int\limits_{-\infty}^{\infty}\sign t \cdot p(p-1)(p-2)|t|^{p-3}e^{-|u-t|/\eps}\,dt
\end{equation*}
has a single root on the line, $u = 0$ (here~$\mrt$ denotes the slope coefficient for the domain of right tangents,~$\mlt$ denotes the slope coefficient for the domain of left tangents). Moreover, the function~$f'''$ is positive on the left of~$u=0$ and is negative on the right (and thus~$\mrt''$ and~$\mlt''$ have the signs required by Proposition~\ref{AngleProp}). So, the vertex of the angle sits at $(0,0)$ and does not move at all when one changes $\eps$.

\paragraph{Escaping angle.}
In the previous paper~\cite{5A} (see Example~$6$ there), we gave an example of an angle, which staggers from side to side. Now we illustrate another phenomenon: an escaping angle. Take any~$f\in C^2$ such that
\begin{equation*}
f'''(t)=\begin{cases}
-1\quad&\text{for }t<0,
\\
e^{-t}&\text{for }t>0.
\end{cases}
\end{equation*}

The vertex of the angle (see Proposition~\ref{AngleProp}) is the point on the lower boundary whose first
coordinate is the root of the equation~$g_\eps(u)=0$, where
$$
g_\eps(u)\df \eps\mrt''(u;\,-\infty) + \eps\mlt''(u;\,+\infty) = \int_{-\infty}^\infty e^{-|u-t|/\eps}f'''(t)\,dt.
$$
Let us find this root.
It is clear that~$g_\eps(u)<0$ for~$u<0$, therefore we consider the case~$u\ge0$ only:
\begin{align*}
g_\eps(u)&=-\int_{-\infty}^0 e^{(t-u)/\eps}\,dt+\int_0^u e^{(t-u)/\eps-t}\,dt
+\int_u^{\infty} e^{(u-t)/\eps-t}\,dt
\\
&=\frac{\eps(2-\eps)}{\eps-1} e^{-u/\eps}-\frac{2\eps}{\eps^2-1} e^{-u}\,.
\end{align*}
This expression is negative everywhere if~$\eps\ge2$. For smaller~$\eps$ we
have one root
$$
u=\frac\eps{\eps-1}\log\frac{2}{(2-\eps)(1+\eps)}\,,
$$
which goes to infinity as~$\eps\to2$.

The evolution scenario is rather simple. For small~$\eps$ we have a single angle near zero surrounded by two tangent domains. As~$\eps$ grows, it moves to the right. At the moment~$\eps = 2$ the angle reaches infinity and disappears. For~$\eps \geq 2$ the whole parabolic strip is foliated by the right tangents.

\paragraph{Polynomial of fourth degree.}
This example has already been considered in the previous paper~\cite{5A}, so we only sketch the results.
In the previous series of examples in Subsection~\ref{s333}, we considered the case of a third degree polynomial. Now we are sufficiently equipped to raise the degree by one. The Bellman function depends heavily on the sign of the leading coefficient, so we deal with the cases of different signs separately. We begin with the case of a positive one. We recall Remark~\ref{quadratische} and observe that it suffices to consider only the polynomials of the form $f = \frac{1}{24}(t - v)^{4}$, $v \in \mathbb{R}$. Therefore, $f'''(t) = t-v$. As usually, we do not write down the exact formula for the Bellman function, we are more interested in the picture itself. The function $f(t) = \frac{1}{24}(t - v)^{4}$ satisfies the conditions of Proposition~\ref{AngleProp}. Indeed, the expression
\begin{equation*}
\mrt''(u;\,-\infty) + \mlt''(u;\,+\infty) = \eps^{-1}\int\limits_{-\infty}^{\infty}(t-v)e^{-|u-t|/\eps}dt=2(u-v)
\end{equation*}
has a unique root $u=v$. So, the Bellman foliation consists of the stable angle at the point $(v,v^{2})$ and two families of tangents. 

The structure of the picture does not change at all as $\eps$ grows.

In the case of a negative leading coefficient it is more convenient to consider polynomials of the type $f(t)=-\frac{1}{24}(t-c)^{4}$, it is the same Remark~\ref{quadratische} that permits us to deal with the described case only. The third derivative of this function is $-(t-c)$, so it is positive on the left of~$c$, zero at $c$, and negative on the right. Thus one can build a symmetric full cup with the origin at $t = c$ (i.e.~$a(\ell) = c - \frac{\ell}{2}$ and~$b(\ell) = c + \frac{\ell}{2}$). The hypothesis of Proposition~\ref{LightChordalDomainCandidate} for this chordal domain is easily verified. Using Propositions~\ref{RightTangentsCandidate} and~\ref{LeftTangentsCandidate}, one can foliate the remaining two domains with the tangents. More detailed information on the formula for the Bellman function can be found in the seventh example in~\cite{5A}. 

The structure of the picture does not change at all as $\eps$ grows, only the cup grows as the free boundary rises.

\paragraph{Polynomial of fifth degree.}
In this example, we treat the cases of a fifth degree polynomial, i.e.~$f(t) = t^{5} + pt^{4}+qt^{3}+P_{2}(t)$ and~$f(t) = -t^{5} + pt^{4}-qt^{3}+P_{2}(t)$, where~$P_{2}$ is an arbitrary quadratic polynomial. Bellman functions for these examples were calculated in~\cite{5A}. The picture is completely determined by the discriminant~$d$ of $f'''$, where~$d=\frac{p^2}{25}-\frac{q}{10}$. We note that if~$d \leq 0$, then $f'''$ has no essential roots and the situation falls under the scope of Subsection~\ref{s32}: the parabolic strip is fully foliated by the left tangents if the leading coefficient of~$f$ is positive and by the right ones in the other case. So, we assume~$d$ to be positive. Then~$f'''$ has two essential roots,~$u_{+}$ and~$u_-$,~$u_+ > u_-$. We see that~$u_+ - u_- = 2 \sqrt{d}$. 
\begin{figure}[h]
\vskip-30pt
\hskip-50pt
\includegraphics[width = 1.2 \linewidth]{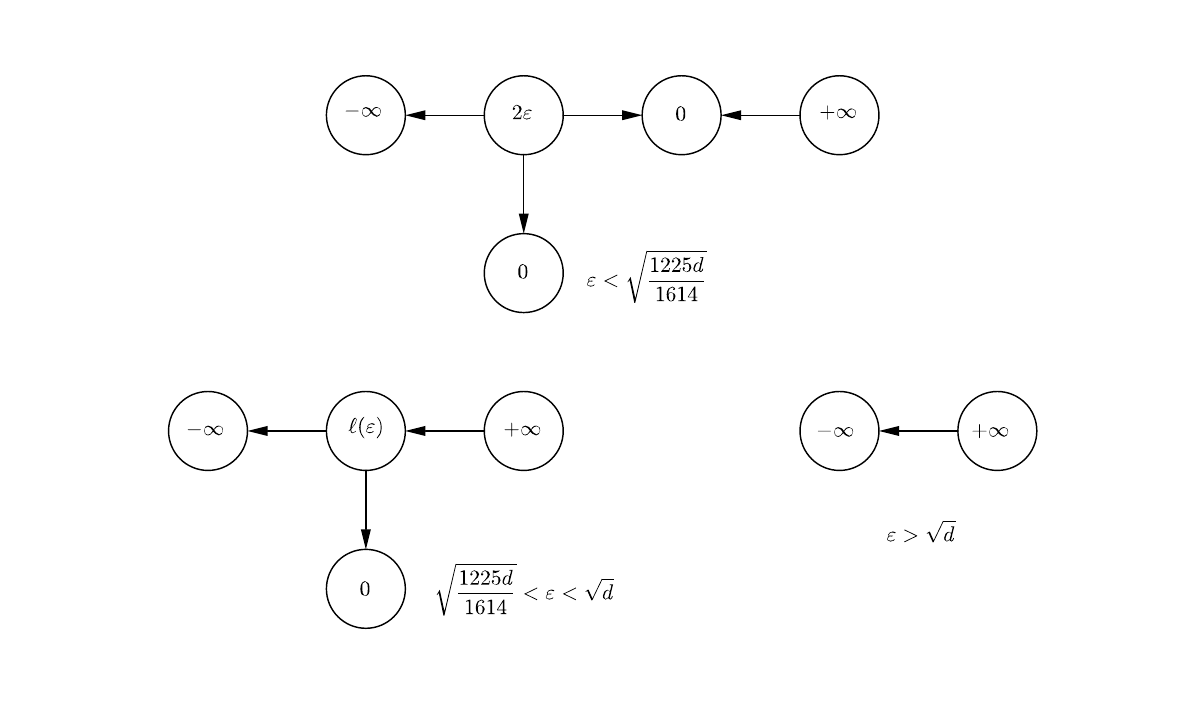}
\vskip-20pt
%\begin{center}
%\includegraphics[width = 1 \linewidth]{EvolutionFifthDegree.jpg}
\caption{Evolution of~$\Gamma$ for the polynomial of fifth degree with positive leading coefficient and positive discriminant~$d$ of~$f'''$.}
\label{fig:efd}
%\end{center}
\end{figure}

If the leading coefficient is positive and~$\eps \leq \sqrt{\frac{1225d}{1614}}$, then the picture is simple. There is a cup with the origin~$u_-$ and an angle on the right of it (we do not give a precise expression for its vertex, it is rather complicated. However, an interested reader may consult~\cite{5A}). The angle meets the cup at the moment~$\eps = \sqrt{\frac{1225d}{1614}}$ and they form a trolleybus by formula~\eqref{CupPlusAngleL}.

If the leading coefficient is negative and~$\eps \leq \sqrt{\frac{1225d}{1614}}$, then the picture is symmetric, there is a cup with the origin in~$u_+$ and an angle on the left of it. 

If~$\sqrt{\frac{1225d}{1614}} \leq \eps \leq \sqrt{d}$, then the foliation consists of a trolleybus and two families of tangents. The trolleybus is left for the case of positive leading coefficient and is right in the opposite case. 

If~$\eps > \sqrt{d}$, then the situation is described by Section~\ref{s32}. The parabolic strip is foliated by the left tangents in the case of positive leading coeffictient and by the right ones in the remaining case. 

Therefore, for the fifth-degree polynomial the evolution consists of three periods: for~$\eps \leq \sqrt{\frac{1225d}{1614}}$ the picture is simple, it consists of a cup and an angle, at the moment~$\eps = \sqrt{\frac{1225d}{1614}}$ they form a trolleybus, which begins to decrease and dies at the moment~$\eps = \sqrt{d}$, after which there are no figures at all. See Figure~\ref{fig:efd} for the evolution of the corresponding graph.

\chapter{Evolution of Bellman candidates}
To prove Theorem~\ref{MT}, we have to build a Bellman candidate for each~$\eps$,~$0 < \eps < \eps_{\infty}$, in the whole parabolic strip~$\Omega_{\eps}$. We begin with sufficiently small~$\eps$. In such a case, the foliation for the Bellman candidate can be composed of cups (multicups), angles, and tangent domains. Then we grow~$\eps$, building the Bellman candidates for greater~$\eps$. Formally, there will be statements of two kinds (they can be called ``induction steps of the first and second kinds'')\index{induction steps}. The first ones state that the set of~$\eps$ for which there is a Bellman candidate of a given structure, is open. They are of the form: ``if for some~$\eta$ there is a Bellman candidate with the graph~$\Gamma$, then there is some~$\delta$ such that for all~$\eps$ in $[\eta, \eta + \delta]$ the foliation with the graph~$\Gamma$ and perturbed numerical parameters provides a Bellman candidate for~$f$ in~$\Omega_{\eps}$''. The second ones state that the set of those~$\eps$ for which there is a graph~$\Gamma$ and a collection of numerical parameters that provide a Bellman candidate for~$f$ and~$\eps$, is closed. They are of the form: ``if for each~$\eps_n$ there is a Bellman candidate with the graph~$\Gamma$,~$\eps_n \nearrow \eps$, and the numerical parameters converge to some limits as~$\eps_n \to \eps$, then~$\Gamma$ with the limiting numerical parameters provide a foliation for~$f$ in~$\Omega_{\eps}$''. We note that the limits of numerical parameters may be degenerate in a sense (for example, a trolleybus may become a fictious vertex of the fifth type), so~$\Gamma$ changes after passing to the limit. Each such induction step, in its turn, can be reduced to similar local statements, i.e. statements about the evolutional behavior of lonely figures, e.g. cups, angles, etc.

The main law that rules the evolution of the foliation is ``the forces grow in absolute value as~$\eps$ grows''. As a result, long chords and multicups grow (Propositions~\ref{InductionStepForChordalDomain} and~\ref{InductionStepForMulticup}), trolleybuses decrease (Propositions~\ref{InductionStepForRightTrolleybus} and~\ref{InductionStepForLeftTrolleybus}), multitrolleybuses, birdies, and multibirdies disintegrate (Propositions~\ref{InductionStepForMultitrolleybus},~\ref{InductionStepForLeftMultitrolleybus}, and~\ref{MultibirdieDesintegrationSt}). What is more, single figures can crash, formally this happens in the induction steps of the second kind when one of the edges has zero length at the limit. In the case of a crash, we use formulas from Subsection~\ref{s344} to continue the evolution.

\section{Simple picture}\label{s41}
\begin{Def}\label{Simple picture}
Let~$\Gamma$ be a foliation graph. We call it \textup(and the foliation itself\textup) simple if it has no oriented paths longer than one and no closed multicups. 
\end{Def}\index{graph! simple graph}
For example, the first and the third graphs on Figure~\ref{fig:efd} are simple, whereas the second one is not (there is a path of length two).

Simple foliations consist of alternating cups (or multicups on single arcs; by a multicup on a single arc we mean~$\MTC(\{\mathfrak{a}\})$, where~$\mathfrak{a}$ is an interval) and angles connected by tangent domains. If~$\Gamma$ is a graph of a simple foliation consisting of~$N$ edges, then there are either~$\frac{N}{2}$, or~$\frac{N-1}{2}$, or~$\frac{N-2}{2}$ angles in the foliation. In~$\GammaFree$, the vertices corresponding to angles alternate the vertices representing multicups and long chords. Each multicup is sitting on an arc, whose projection on the~$x$-axis exceeds~$2\eps$. Each long chord has a cup below it. See Figure~\ref{fig:exsimpict} for the visualization. 

\begin{figure}%[16pt]{o}{300pt}
\begin{center}
\includegraphics[width = 0.99 \linewidth]{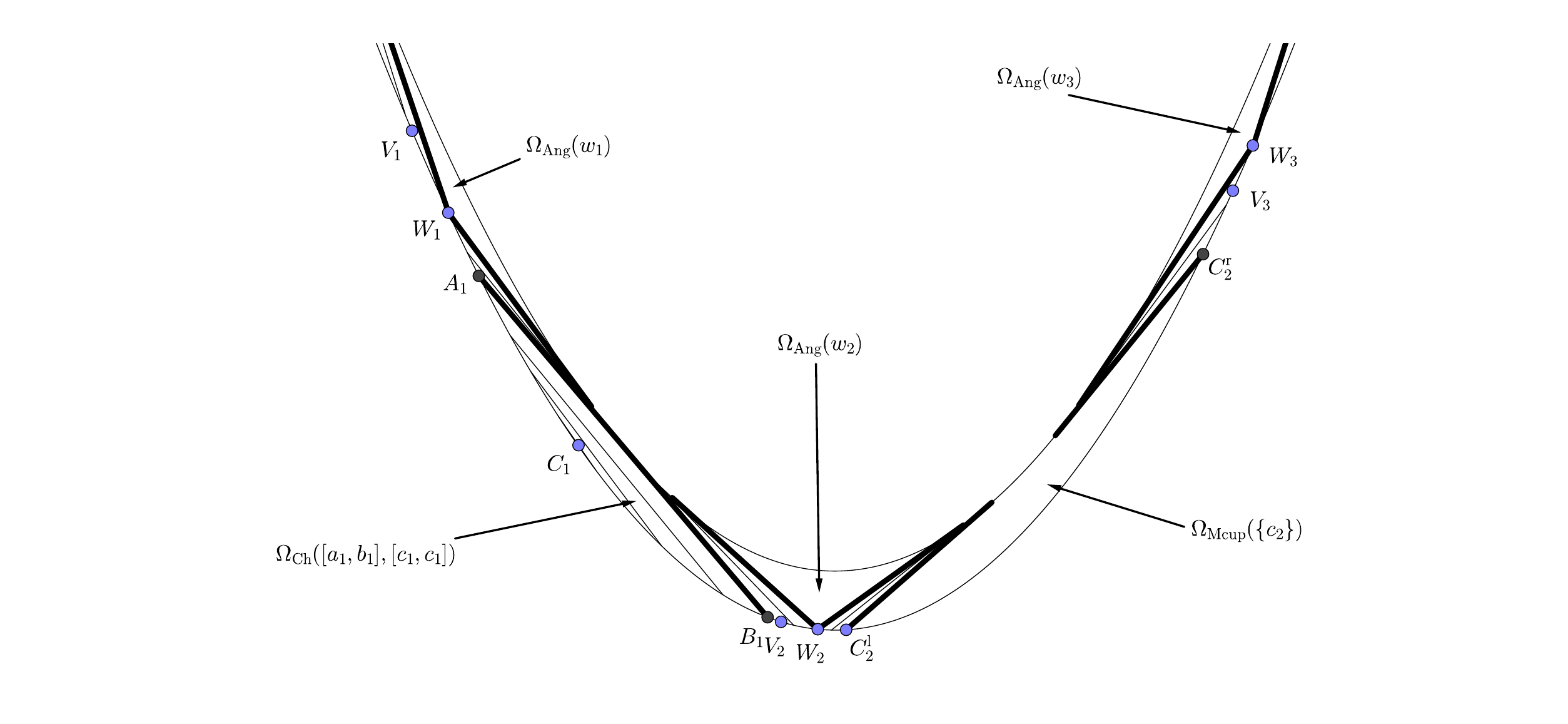}
\caption{An example of simple picture.}
\label{fig:exsimpict}
%\label{fig:MtbGr}
\end{center}
\end{figure}

For didactic reasons, we explain how does a simple graph generate a Bellman candidate~$B$ (similar essence for general graphs will be explained in Section~\ref{s44}). Suppose~$\Gamma$ to be a simple graph. First, we consider its roots, which are long chords, multicups on single arcs, and vertices at infinity. For long chords, we build the standard candidate with the help of formula~\eqref{vallun}, for multicups~--- with the help of formulas~\eqref{LinearInTrolleybus} and~\eqref{coef}. Second, consider the edges of~$\Gamma$. For the edges corresponding to chordal domains, we construct standard candidates with the help of formula~\eqref{vallun}. For each edge corresponding to a tangent domain, we continuously glue a standard candidate in it to the already built standard candidate corresponding to its source (for tangent domains whose source is infinity, we do not have to glue anything, we simply consider the standard candidate in such a domain given by formulas~\eqref{linearity} and~\eqref{minfty} or~\eqref{minfty2}). In the angles, we choose the standard candidate by Proposition~\ref{AngleProp}. Proposition~\ref{DomainOfLinearityWithTwoPointsOnTheLowerBoundaryMeetsTangentDomain} guarantees that~$B$ is~$C^1$-smooth, and thus, by Proposition~\ref{ConcatenationOfConcaveFunctions}, it is locally concave.

In the theorem below, we use the notation for the essential roots of~$f'''$, see Definition~\ref{roots}.
\begin{Th}\label{SimplePicture}
For any function~$f$ satisfying Conditions~\textup{\ref{reg}} and~\textup{\ref{sum}} there exists~$\eps_1 > 0$ such that for any~$\eps < \eps_1$ there exists a simple graph and a collection of numerical parameters such that the function~$B$ constructed from this graph\textup,~$f$\textup, and~$\eps$ as described above is a~$C^1$-smooth Bellman candidate. Moreover\textup, its foliation satisfies the following properties\textup: the origins of the cups coincide with those~$c_i$ that are single points\textup; the multicups are sitting on those~$c_i$ that are intervals\textup; for any~$k = 1,2,\ldots,n$ the vertex~$w_k$ of the~$k$-th angle in~$\GammaFree$ tends to~$v_k$ as~$\eps \to 0$. 
\end{Th}
The proof of this theorem needs preparation and will occupy us for almost a section.

\subsection{How to grow a chordal domain}\label{s411}\index{domain! chordal domain}
We begin with investigation on chordal domains. In Subsection~\ref{s331}, we gave sufficient conditions for a function given by formula~\eqref{vallun} to be a Bellman candidate inside the domain. However, we said nothing where such figures can be built. Now we are going to fill this gap. We do a bit more than needed to prove Theorem~\ref{SimplePicture}. This is useful for further study. We will be growing a chordal domain either from a single chord (whose endpoints satisfy the cup equation), or from some~$c_i$ that is a point (not an interval). These two cases are very similar, and the second one is even easier (it was considered in~\cite{5A}). So, we treat the case of a non-zero chord in full detail and after that comment on the case of a root~$c_i$.

\begin{Le}\label{SmoothChordalDomain}
Suppose that the pair~$(a_0,b_0)$ satisfies the cup equation~\eqref{urlun} and that the inequalities~$\Slt(a_0,b_0) < 0$ and~$\Srt(a_0,b_0) < 0$ are fulfilled. Then\textup, for any sufficiently small~$\delta$ there exists a unique pair~$(a,b)$ of real-valued functions defined on~$[l_0 - \delta,l_0+\delta]$ such that~$b(l)-a(l) = l$\textup, the pair~$(a(l),b(l))$ satisfies the cup equation\textup,~$a' < 0$\textup,~$b' > 0$\textup, and~$a(l_0) = a_0$\textup,~$b(l_0) = b_0$. Here~$l_0 = b_0 - a_0$.
\end{Le}
\begin{proof}
Consider the function~$\Phi \colon \mathbb{R}^2 \mapsto \mathbb{R}$ defined by the equality~$\Phi(a,b) \df f'(a)+f'(b)-2\av{f'}{[a,b]}$ (we have already used this function in Subsection~\ref{s331}, formula~\eqref{Phi}). We see that~$\frac{\partial}{\partial a}\Phi = \Slt(a,b)$,~$\frac{\partial}{\partial b}\Phi = \Srt(a,b)$ on the set~$\Phi(a,b) = 0$. So, by the implicit function theorem, the cup equation has a unique solution~$(a,b)$ with~$b-a = l$, and this solution is a~$C^1$-smooth function of~$l$, when~$l$ is in a neighborhood of~$l_0$. What is more, we can calculate~$a'$ and~$b'$:
\begin{equation}\label{Derivativesabb}
\begin{aligned}
a' = \frac{-\Srt(a,b)}{\Slt(a,b)+\Srt(a,b)};\\ 
b' = \frac{\Slt(a,b)}{\Slt(a,b)+\Srt(a,b)},
\end{aligned}
\end{equation}
so~$b' > 0$ and~$a' < 0$ in a neighborhood of~$(a_0,b_0)$, because the differentials are continuous functions. 
\end{proof}
By~Proposition~\ref{LightChordalDomainCandidate}, the function built by formula~\eqref{vallun} provides a Bellman candidate for the domain
\begin{equation*}
\Ch\Big([a(l_0+\delta),b(l_0+\delta)],[a(l_0-\delta),b(l_0-\delta)]\Big).
\end{equation*}
For the next two lemmas, we use the notation introduced at the end of Subsection~\ref{s331}.
\begin{Le}\label{nonsmoothgrowth}
Suppose that the pair~$(a_0,b_0)$\textup,~$a_0 < b_0$\textup, satisfies the cup equation. Suppose that  the graph of~$f'$ lies strictly below~$L_{a_0,b_0}$ in a right neighborhood of~$b_0$\textup, and in a left neighborhood of~$a_0$\textup, it lies strictly above~$L_{a_0,b_0}$. Then\textup, for every sufficiently small~$\delta$ there exists a pair of points~$(a_{\delta},b_{\delta})$ satisfying the cup equation and such that~$a_\delta<a_0$\textup,~$b_\delta>b_0$\textup, and~$b_\delta-a_\delta=l_\delta\df l_0+\delta$. 
\end{Le}
As in the previous lemma,~$l_{0} = b_0 -a_0$.
\begin{proof}
To make the notation simpler, we assume that~$L_{a_0,b_0}$ coincides with the $x$-axis. If it does not, one can subtract a quadratic polynomial from~$f$ to move this line onto the~$x$-axis. Such a transform changes neither the cup equation and the differentials nor the inequalities prescribed for~$f'$.

Let~$\delta$ be fixed. Consider the function~$\Phi(s,s+l)\df f'(s)+f'(s+l) - 2\av{f'}{[s,s+l]}$. Geometrically, this function times~$\frac12l$ is the difference between the area of the trapezoid with vertices~$(s,0)$,~$(s,f'(s))$,~$(s+l,f'(s+l))$, and~$(s+l,0)$ and the area of the subgraph of~$f'$ between the points~$s$ and~$s+l$. As usually, we consider oriented areas. 

We are going to prove that~$\Phi(a_0 - \delta, b_0) > 0$ and~$\Phi(a_0, b_0 + \delta) < 0$ provided~$\delta$ is sufficiently small. Assuming these two inequalities and using the Bolzano--Weierstrass principle for the continuous function~$t \mapsto \Phi(t, t+ l_{\delta})$, we see that there exists a point~$a_{\delta}$, $a_0 - \delta < a_{\delta} < a_0$, such that~$\Phi(a_{\delta}, a_{\delta} + l_{\delta})=0$. This means that the pair~$(a_{\delta}, b_{\delta})$, where~$b_{\delta} \df a_{\delta} + l_{\delta}$, satisfies the cup equation. 

The inequalities~$\Phi(a_0 - \delta, b_0) > 0$ and~$\Phi(a_0, b_0 + \delta) < 0$ are symmetric, we prove only the first one. We analyze the behavior of each of the two areas (of the trapezoid and of the subgraph) separately.
 
The area of the subgraph of~$f'$ between the points~$a_0 - \delta$ and~$b_0$ equals the area of the subgraph between the points~$a_0 - \delta$ and~$a_0$ (by the cup equation and the assumption that~$L_{a_0,b_0}$ coincides with the~$x$-axis). This figure is contained in the rectangle of height~$f'(a_0 - \delta)$ and width~$\delta$ (note that for~$\delta$ sufficiently small the function~$f'$ is decreasing on~$[a_0 - \delta,a_0]$, because there exists a left neighborhood of~$a_0$ in which~$f'$ is either convex or concave, due to Condition~\ref{reg}), so its area is less than~$f'(a_0 - \delta)\delta$.

By virtue of the assumption that~$L_{a_0,b_0}$ coincides with the~$x$-axis, the trapezoid becomes a triangle with vertices~$\big(a_0 - \delta,0\big)$, $\big(a_0 - \delta, f'(a_0 - \delta)\big)$, and~$\big(b_0,0\big)$. The area of this triangle equals~$\frac12(l_0 + \delta)f'(a_0 - \delta)$. 

We see that the desired value (the difference of the two areas) exceeds~$\frac12 f'(a_0 - \delta)(l_0 - \delta)$. This value is strictly positive, because~$b_0 \ne a_0$ and~$f'$ is strictly positive in a left neighborhood of~$a_0$. The lemma is proved.
\end{proof}
\begin{Rem}\label{nonsmoothgrowthClassical}
Let~$c_i$ be a root of~$f'''$ that is a point \textup(see Definition~\textup{\ref{roots}}\textup). Then\textup, for every~$\delta \leq \min(\dist(c_i,v_i),\dist(c_i,v_{i+1}))$ there exists a pair of points~$(a_{\delta},b_{\delta})$ satisfying the cup equation~\eqref{urlun} and such that~$a_\delta<c_i$\textup,~$b_\delta>c_i$\textup, and~$b_\delta-a_\delta=\delta$.
\end{Rem}
\begin{proof}
One can prove the remark in the same manner as Lemma~\ref{nonsmoothgrowth}. The only difference is that the inequalities~$\Phi(c_i - \delta, c_i) > 0$ and~$\Phi(c_i, c_i + \delta) < 0$ are much easier to show in this case (the first follows from the convexity of~$f'$ on the left of~$c_i$, the second from the concavity of~$f'$ on the right of~$c_i$).
\end{proof}

\begin{Le}\label{nonvanishdif2}
Suppose that the points~$a_0$ and~$b_0$ are under the same conditions as in the previous lemma and let the points~$a_{\delta}$ and~$b_{\delta}$ be such that~$a_{\delta} < a_0 < b_0 < b_{\delta}$  and~$(a_{\delta}, b_{\delta})$ satisfies the cup equation. Then\textup,~$\Slt(a_{\delta},b_{\delta}) < 0$\textup,~$\Srt(a_{\delta},b_{\delta}) < 0$ provided~$\delta$ is sufficiently small. 
\end{Le}
This lemma says that every chord we may build with the help of the previous lemma has negative differentials. One can prove that there is only one such chord of the fixed length, however, we do not need this fact.

\begin{proof}
The inequalities are symmetric, we treat the case of the right differential only. As in the proof of Lemma~\ref{SmoothChordalDomain}, we are going to argue geometrically. We assume that~$L_{a_0,b_0}$ coincides with the~$x$-axis. Therefore,~$f'(a_0) = f'(b_0) = 0$ and~$f''(b_0) = \Srt(a_0,b_0)$. Consider the two cases:~$f''(b_0) < 0$ and~$f''(b_0) = 0$ (this value cannot exceed zero, because the function~$f'$ is negative on the right of~$b_0$). 

In the first case the inequality~$\Srt(a_{\delta},b_{\delta}) < 0$ is a consequence of the continuity of~$f''$. 

The proof for the second case uses the structural properties we imposed on the function~$f$ in Subsection~\ref{s212} (Condition~\ref{reg}). By assumption that~$f'$ is strictly negative on the right of~$b_0$, we see that~$b_0$ neither belongs to the interior of any solid root of~$f'''$, nor can be the left endpoint of such a solid root (indeed, in both cases described,~$f''' = 0$ in a right neighborhood of~$b_0$ and thus~$f' = 0$ there). Therefore,~$f'$ is either convex or concave in a right neighborhood of~$b_0$. If it is convex, then~$f'$ is positive in a right neighborhood of~$b_0$ (we recall that~$f''(b_0) = 0$). So,~$f'$ is concave on the right of~$b_0$. 

Consider now the tangent line~$L_{b_{\delta},b_{\delta}}$ to the graph of~$f'$ at the point~$b_{\delta}$. The inequality~$\Srt(a_{\delta},b_{\delta}) < 0$ is equivalent to the fact that the point~$(a_{\delta},f'(a_{\delta}))$ lies below this tangent line (see the geometric interpretation of the differentials at the end of Subsection~\ref{s331}). Assume the contrary, i.e. that the point~$(a_{\delta},f'(a_{\delta}))$ does not lie below~$L_{b_{\delta},b_{\delta}}$. We claim that in such a case the whole part of the graph on the intervals~$[a_{\delta},a_0]$ and~$[b_0,b_{\delta}]$ lies below the line~$L_{a_{\delta},b_{\delta}}$. Once the claim is proved, we get a contradiction with the cup equation for the pair~$(a_{\delta},b_{\delta})$ (the area of the subgraph is strictly smaller than the area of the trapezoid, here we use the cup equation for the pair~$(a_0, b_0)$).

To finish the proof of the lemma, we must verify the claim. The part of the graph that corresponds to the interval~$[b_0,b_{\delta}]$ lies under~$L_{a_{\delta},b_{\delta}}$ because it lies below~$L_{b_{\delta},b_{\delta}}$ due to the concavity of~$f'$ on this interval. To deal with the remaining interval, we mark the point where~$L_{a_{\delta}, b_{\delta}}$ crosses the~$x$-axis by~$X = (x,0)$. By concavity and the assumption, we have~$x \geq b_0$. The function~$f'$ is either convex or concave on the left of~$a_0$, due to the structural properties of~$f$. In the first case, the claim follows form the fact that the graph of~$f'$ on~$[a_{\delta},a_0]$ lies below the line~$L_{a_{\delta},a_0}$, which lies below~$L_{a_{\delta}, b_{\delta}}$ because~$x > a_0$. In the second case, the graph of~$f'$ on the interval~$[a_{\delta},a_0]$ lies below the tangent at~$a_{\delta}$. By continuity, the slope of this line is uniformly negative (otherwise~$f'$ is negative on the left of~$a_0$), whereas the slope of~$L_{a_{\delta},b_{\delta}}$ tends to zero. Thus, the tangent at~$a_{\delta}$ lies below~$L_{a_{\delta},b_{\delta}}$ on the right of~$\big(a_{\delta},f'(a_{\delta})\big)$. The claim is proved together with the lemma.
\end{proof}
\begin{Rem}\label{nonvanishdif2Classsical}
Let~$c_i$ be a root of~$f'''$ that is a point\textup,~$\delta \leq \min(\dist(c_i,v_i),\dist(c_i,v_{i+1}))$. Suppose that the points~$a_{\delta}$ and~$b_{\delta}$ are such that~$a_{\delta} < c_i < b_{\delta}$\textup, $b_{\delta} - a_{\delta} = \delta$\textup, and the pair~$(a_{\delta},b_{\delta})$ satisfies the cup equation. Then\textup,~$\Srt(a_{\delta},b_{\delta}) < 0$ and~$\Slt(a_{\delta},b_{\delta}) < 0$.
\end{Rem}
\begin{proof}
The proof is verbatim the proof of Lemma~\ref{nonvanishdif2} (and even a bit easier, because we do not have to investigate the structure of~$f'''$ in a neighborhood of~$c_i$).
\end{proof}
Lemmas~\ref{nonsmoothgrowth} and~\ref{nonvanishdif2} have an obvious corollary.
\begin{Cor}\label{ChordDomFromAChord}
Suppose that the pair~$(a_0,b_0)$\textup,~$a_0 < b_0$\textup, satisfies the cup equation. Suppose that in a right neighborhood of~$b_0$ the graph of~$f'$ lies strictly below~$L_{a_0,b_0}$\textup, and in a left neighborhood of~$a_0$ it lies strictly above. Then\textup, there exists a chordal domain satisfying the assumptions of Proposition~\textup{\ref{LightChordalDomainCandidate}} whose bottom chord is~$[A_0,B_0]$.
\end{Cor}
\begin{proof}
We use Lemma~\ref{nonsmoothgrowth} to construct a pair~$(a_{\delta},b_{\delta})$ with~$b_\delta-a_\delta=\delta+b_0-a_0$ that satisfies the cup equation and~$a_{\delta} < a_0< b_0 < b_{\delta}$ (here~$\delta$ is sufficiently small). By Lemma~\ref{nonvanishdif2}, $\Srt(a_{\delta},b_{\delta}) < 0$ and~$\Slt(a_{\delta},b_{\delta}) < 0$. Consider the infimum~$l^*$ of the set of~$l$, $l \leq l_0$, such that there exist a unique pair of differentiable functions~$a,b\colon(l,l_{\delta}]\to\mathbb{R}$ such that~$(a,b)$ satisfies the cup equation,~$b(s) - a(s) = s$, $\Slt(a,b) < 0$,~$\Srt(a,b) < 0$,~$a < a_0$, and~$b > b_0$. By Lemma~\ref{SmoothChordalDomain}, we have~$l^* < l_{\delta}$. To prove the corollary, it suffices to prove that~$l^* = l_0 = b_0-a_0$ (the strict monotonicity of~$a$ and~$b$ follows from formulas~\eqref{Derivativesabb}). Note that for~$l_1, l_2 \in (l^*,l_0)$, $l_1<l_2$, the corresponding unique functions~$a$ and~$b$ constructed for~$l_2$ are restrictions of the functions constructed for~$l_1$. Therefore we have functions~$a$ and~$b$ uniquely defined on~$(l^*,l_0]$. 

By passing to the limit, we see that the pair~$(a(l^*),b(l^*))$, where
\begin{equation*}
a(l^*) = \lim_{l \to l^*\!{\scriptscriptstyle+}} a(l), \qquad b(l^*) = \lim_{l \to l^* \!{\scriptscriptstyle+}} b(l),
\end{equation*}
satisfies the cup equation, and moreover,~$a(l^*) \leq a_0 < b_0 \leq b(l^*)$. % (these points are constructed by the rule~$a(l^*) = \lim_{l \searrow l^*} a(l)$, $b(l^*) = \lim_{l \searrow l^*} b(l)$, these limits are defined by the uniqueness of the functions~$a$ and~$b$ provided by Lemma~\ref{SmoothChordalDomain}). 
It is easy to see that if~$a(l^*) = a_0$ or~$b(l^*) = b_0$, then both these equalities are valid and~$l^* = l$. If not, then we can use Lemma~\ref{nonvanishdif2} to see that~$\Srt(a(l^*),b(l^*)) < 0$ and~$\Slt(a(l^*),b(l^*)) < 0$, and then apply Lemma~\ref{SmoothChordalDomain} to obtain that~$l^*$ is not minimal.
\end{proof}

Remarks~\ref{nonsmoothgrowthClassical} and~\ref{nonvanishdif2Classsical} lead to a similar conclusion. Namely, now we can show that for every~$c_i$ that is a single point one can build a small cup whose origin is~$c_i$. This was Lemma~$5.5$ in~\cite{5A}.

\begin{Le}\label{CupNearRoot}
Consider the interval $\Delta = [c-l_0,c+l_0]$. Suppose ${f''}$ strictly increases on the left half 
		${[c-l_0,c]}$ 
		of $\Delta$ and ${f''}$ strictly decreases on the right half ${[c,c+l_0]}$. 		
		Then there exist two functions $a=a(l)$ and $b=b(l)= a(l) + l$\textup, 
		$l \in (0,l_0]$\textup, with the
		following properties\textup:
		\begin{enumerate}
			\item[\textup{1)}] $a(l)<c<b(l)$\textup;	
			\item[\textup{2)}] the pair $(a(l),b(l))$ solves the cup equation\textup;
			\item[\textup{3)}] 
				$\Slt\big(a(l),b(l)\big)<0$ and $\Srt\big(a(l),b(l)\big)<0$\textup;
			\item[\textup{4)}] $a$ and~$b$ are differentiable functions such that $a'<0$ and $b'>0$.
		\end{enumerate}
\end{Le}
\begin{proof}
Similar to the proof of Corollary~\ref{ChordDomFromAChord}.
\end{proof}
Lemma~\ref{CupNearRoot} states, in particular, that if~$c_i$ is a point and~$l_0 = \min(\dist(c_i,v_i),\dist(c_i,v_{i+1}))$, then one can build a cup with the origin~$c_i$ and width~$l_0$.

\begin{Rem}\label{MulticupInARoot}
When we are building cups\textup, the situation for those~$c_i$ that are intervals\textup, is much simpler. 
Consider some~$c_i$ and assume that~$\eps$ is less than a half of the length of~$c_i$. Then\textup, the multicup based on the interval~$c_i$ surely satisfies the assumptions of Proposition~\textup{\ref{LightChordalDomainCandidate}}.
\end{Rem}

%We end this subsection with a lemma that describes the places where the differentials can vanish.

\subsection{Monotonicity properties of forces}\label{s412}
We begin this subsection with the definition of a force and a tail of a chordal domain.
\begin{Def}\label{ForceOfChordalDomain}\index{force! of a chordal domain}
Let~$\Ch([a_0,b_0],[a_1,b_1])$ be a chordal domain. The function
\begin{equation}
\Fr(u;\Ch([a_0,b_0],[a_1,b_1]);\eps) =
\begin{cases}
\Srt\big(a(l),b(l)\big), \quad u=b(l);\\
\Fr(u;a_0,b_0;\eps),\quad u > b_0  
\end{cases}
\end{equation}
acting from~$(b_1,\infty)$ to~$\mathbb{R}$ is called the right force of~$\Ch([a_0,b_0],[a_1,b_1])$. The function
\begin{equation}
\Fl(u;\Ch([a_0,b_0],[a_1,b_1]);\eps) =
\begin{cases}
\Fl(u;a_0,b_0;\eps),\quad u < a_0;\\
-\Slt\big(a(l),b(l)\big), \quad u=a(l)
\end{cases}
\end{equation}
defined on~$(-\infty,a_1)$ is called the left force.
\end{Def}
We note that the forces are continuous functions.

\begin{St}\label{difforcepou}
Let~$\Ch([a,b],*)$ be a chordal domain. Then its forces satisfy the following differential equations\textup{:}
\begin{equation}\label{leftder}
\Fr'(u) = \begin{cases}
f'''(u) - \frac{2\Fr(u)}{\ell(u)}, &u < b;\\
f'''(u) - \frac{\Fr(u)}{\eps}, &u > b,
\end{cases}
\end{equation}
the function~$\ell$ is defined at the beginning of Subsection~\textup{\ref{s331}}\footnote{There is an ambiguity here:~$\ell$ is a function on the chordal domain (a domain in the plane)\textup, we write~$\ell(u)$ instead of~$\ell(u,u^2)$ for brevity.}.
Here we write derivatives with respect to~$u$\textup,~$u$ is supposed to belong to the domain of~$\Fr$. Similar formula holds for the left force\textup{:}
\begin{equation}\label{rightder}
\Fl'(u) = \begin{cases}
-f'''(u) + \frac{\Fl(u)}{\eps}, &u < a;\\
-f'''(u) + \frac{2\Fl(u)}{\ell(u)}, &u > a
\end{cases}
\end{equation}
with the same words about the domain.
\end{St}
\begin{proof}
The formulas have already been calculated, they follow from formulas~\eqref{dDlt} and~\eqref{dDrt} inside the chordal domain and formula~\eqref{DifForcePoU} outside the chordal domain.
\end{proof}
\begin{Def}\label{Tail}\index{tail}
Consider~$\Ch([a_0,b_0],[a_1,b_1])$. Let
\begin{align*}
\tr(\Ch([a_0,b_0],[a_1,b_1]);\eps) = \sup \{t \mid t > b_1, \quad \forall s \in (b_0,t) \quad \Fr(s) < 0\}; \\
\tl\,(\Ch([a_0,b_0],[a_1,b_1]);\eps) = \inf\, \{t \mid t < a_1, \quad \forall s \in (t,a_0) \quad \Fl(s) > 0\}.
\end{align*}
The interval~$(b_1,\tr)$ is called the right tail of~$\Ch([a_0,b_0],[a_1,b_1])$\textup, the interval~$(\tl,a_1)$ is called the left tail of~$\Ch([a_0,b_0],[a_1,b_1])$. We also define the tails of a chord~$[A_0,B_0]$ as the intervals~$(b_0,\tr)$ and~$(\tl,a_0)$\textup, where the points~$\tl$ and~$\tr$ are given by the same formulas with the forces of this chord.
\end{Def}
We note that the tails of a chordal domain do not coincide with the tails of its upper chord. The tails of the upper chord (as well as forces) do not remember any information about the situation below the chord. Sometimes it is convenient to use this information. The forces and tails of chordal domains are designed for this purpose.
\begin{Def}\label{TailInfinity}
The ray~$(-\infty,\tr)$ is called the right tail of~$-\infty$\textup, the interval~$(\tl,\infty)$ is called the left tail of~$\infty$, where
\begin{align*}
\tr(-\infty;\eps) = &\sup \{t \in\mathbb{R}\mid  \forall s \in (-\infty,t) \quad \Fr(s;-\infty;\eps) < 0\}; \\
\tl\,(\infty;\eps) = &\inf\, \{t \in \mathbb{R} \mid  \forall s \in (t,\infty) \quad \Fl(s;\infty;\eps) > 0\}.
\end{align*}
\end{Def}
The tails indicate up to what extent one can proceed the tangent family from a full chordal domain, a trolleybus, or a multifigure. We note that the definition above differs from the one we had in~\cite{5A} (Definition~$6.6$ there), this definition is more convenient for the treatment of solid roots and evolution. We also need the same notions for multicups.
\begin{Def}\label{ForceForMulticup}\index{force! of a multicup}
Let~$\MTC(\{\mathfrak{a}_i\}_{i=1}^k)$ be a multicup. Then its forces are defined by the formulas
\begin{align*}
\Fr(u;\MTC(\{\mathfrak{a}_i\}_{i=1}^k);\eps) = \Fr(u;\mathfrak{a}_1^l,\mathfrak{a}_k^r,\eps), \quad &u \in (\mathfrak{a}_k^r,\infty);\\
\Fl(u;\MTC(\{\mathfrak{a}_i\}_{i=1}^k);\eps) = \Fl(u;\mathfrak{a}_1^l,\mathfrak{a}_k^r,\eps), \quad &u \in (-\infty,\mathfrak{a}_1^l).
\end{align*}
\end{Def} 
\begin{Def}
Let~$\MTC(\{\mathfrak{a}_i\}_{i=1}^k)$ be a multicup. Let
\begin{align*}
&\tr(\MTC(\{\mathfrak{a}_i\}_{i=1}^k);\eps) = \sup \{t \mid t > \mathfrak{a}_k^{r}, \quad \forall s \in (\mathfrak{a}_k^{r},t)\quad \Fr(s) < 0\}; \\
&\tl\,(\MTC(\{\mathfrak{a}_i\}_{i=1}^k);\eps) = \inf \,\{t \mid t < \mathfrak{a}_1^{l}, \quad \forall s \in (t,\mathfrak{a}_1^{l}) \quad \Fl(s) > 0\},
\end{align*}
here the forces are the multicup forces. The segment~$(\mathfrak{a}_k^r,\tr)$ is the right tail of~$\MTC(\{\mathfrak{a}_i\}_{i=1}^k)$\textup, the segment~$(\tl,\mathfrak{a}_1^l)$ is the left tail.
\end{Def}

By definition, the tails of a chordal domain always contain some part of~$[a_0,b_0]$. We say that a tail is non-zero, if it has some part outside~$[a_0,b_0]$ (for example, this holds true if the differential at the corresponding endpoint is non-zero).
 
Consider two neighbor roots of~$f'''$,~$c_k$ and~$v_{k+1}$. By Lemma~\ref{CupNearRoot}, we can build a cup or a multicup around~$c_k$. Let its upper chord be~$[A_k,B_k] = [A_k(\eps),B_k(\eps)]$ (for the cup, we take its upper chord, whereas for the multicup, we consider the chord connecting its endpoints). Then, by~\eqref{RightForce}, we have~$\Fr(u;a_k,b_k;\eps) < 0$ when~$u \in (b_k,v_{k+1}]$. So, the right tail of the cup or multicup built on~$c_k$ always contains~$v_{k+1}$. Similarly, the left tail of the cup or multicup built over~$c_{k+1}$ contains~$v_{k+1}$. The following lemma says that this result is asymptotically sharp as~$\eps \to 0$.
\begin{Le}\label{TailForSmallEps}
Let~$(A_k(\eps),B_k(\eps))$ be the upper chord of the cup or the multicup built over~$c_k$\textup, let~$\tr_k = \tr_k(\eps)$ be the endpoint of its right tail. Then~$\tr_k \to v_{k+1}$ as~$\eps \to 0$. Similarly\textup, the endpoint~$\tl_k(\eps)$ of the left tail tends to~$v_k$. A similar convergence statement holds for the forces coming from the infinities.
\end{Le} 
\begin{proof}
It suffices to prove that for each point~$w_{+}$ such that~$v_{k+1} < w_{+}$ (we also assume that~$w_+$ is not far from~$v_{k+1}$, we want~$f''$ to increase on~$(v_{k+1},w_+)$), the inequality~$\tr_k < w_{+}$ holds eventually as~$\eps \to 0$. We recall the formula for the force from Definition~\ref{Forces},
\begin{equation}\label{forcie}
\Fr(u;a_k,b_k;\eps) = \eps^{-1}\Srt(a_k,b_k)e^{(b_k-u)/\eps} + 
		\eps^{-1}e^{-u/\eps}\int\limits_{b_k}^u e^{t/\eps}\,df''(t).
\end{equation}
It suffices to prove~$\Fr(w_{+}) > 0$. 
% The non-integral summand in formula~\eqref{forcie} tends to zero,
%\begin{equation*}
%\eps^{-1}\Srt(a_k,b_k)e^{(b_k-w_{+})/\eps} \rightarrow 0 \quad \hbox{as} \quad \eps \rightarrow 0,
%\end{equation*}
%because~$\Srt(a_k,b_k)$ is bounded and~$w_{+} - b_k > \dist(b_k, v_{k+1})$. 
We first deal with the integral summand (as usual, by~$v^{\mathrm{r}}_{k+1}$ we denote the right endpoint of~$v_{k+1}$),
\begin{multline*}
\eps^{-1}e^{-w_{+}/\eps}\int\limits_{b_k}^{w_{+}} e^{t/\eps}\,df''(t) = \eps^{-1}e^{-w_{+}/\eps}\bigg(\int\limits_{b_k}^{v^{\mathrm{r}}_{k+1}} e^{t/\eps}\,df''(t) + \int\limits_{v^{\mathrm{r}}_{k+1}}^{w_{+}} e^{t/\eps}\,df''(t)\bigg) \geq \\
 \eps^{-1}e^{-w_{+}/\eps}\Big(- \var(df''\mid_{[b_k,v^{\mathrm{r}}_{k+1}]})e^{v^{\mathrm{r}}_{k+1}/\eps} + e^{\omega/\eps}(f''(w_{+}) - f''(\omega))\Big),
\end{multline*}
where~$\omega \in (v_{k+1}^{\mathrm{r}},w_+)$. Then,
\begin{equation*}
\eps e^{(w_+-\omega)/\eps}\Fr(w_{+}) \geq 
\Srt(a_k,b_k)e^{(b_k-\omega)/\eps} - \var(df''\mid_{[b_k,v^{\mathrm{r}}_{k+1}]})e^{(v^{\mathrm{r}}_{k+1} -\omega)/\eps}+(f''(w_{+}) - f''(\omega)),
\end{equation*}
which tends to the positive number~$f''(w_{+}) - f''(\omega)$ as~$\eps$ tends to zero and~$\omega$ is fixed, because~$\Srt(a_k,b_k)$ is bounded.
%
%Surely, for sufficiently small~$\eps$, this expression exceeds~$\eps^{-1}e^{\frac{v^{\mathrm{r}}_{k+1} - w_+ + \delta}{\eps}}\frac{f''(w_{+}) - f''(v^{\mathrm{r}}_{k+1} + \delta)}{2}$, which is positive if~$\delta$ is sufficiently small.
\end{proof}
The notation in the following lemma is the same as in the previous one.
\begin{Le}\label{moninttails}
The sum of forces\textup,~$\Fr(u;a_k,b_k;\eps) + \Fl(u;a_{k+1},b_{k+1};\eps)$\textup, is increasing \textup(as a function of~$u$\textup) on the interval~$(\tl_{k+1},\tr_k) \cap (b_k,a_{k+1})$. 
\end{Le}
\begin{proof}
We differentiate the function in question with respect to~$u$, use formulas~\eqref{DifForcePoU}, and get 
\begin{equation}\label{difpowsum}
\big(\Fr(u) + \Fl(u)\big)' = \eps^{-1}\big(-\Fr(u) + \Fl(u)\big).
\end{equation}
This derivative is positive because~$\Fr(u) < 0$,~$\Fl(u) > 0$ on~$(\tl_{k+1},\tr_k)\cap (b_k,a_{k+1})$. The proposition is proved.
\end{proof}
\begin{Cor}\label{BalanceEquationForSmallEps}
The balance equation 
\begin{equation}
\Fr(w;a_k,b_k;\eps) + \Fl(w;a_{k+1},b_{k+1};\eps) = 0
\end{equation}
has a unique root~$w = w_{k+1}$ in~$(\tl_{k+1},\tr_k)$ for sufficiently small~$\eps$. 
\end{Cor}
\begin{proof}
First, by Lemma~\ref{TailForSmallEps},  we have~$[\tl_{k+1},\tr_k]\subset(b_k,a_{k+1})$ for sufficiently small~$\eps$. By Definition~\ref{Tail} and continuity of forces,~$\Fr(\tr_k;a_k,b_k;\eps) = 0$ and~$\Fl(\tl_{k+1};a_{k+1},b_{k+1};\eps) = 0$. Therefore,
\begin{equation*}
\Fr(\tl_{k+1};a_k,b_k;\eps) + \Fl(\tl_{k+1};a_{k+1},b_{k+1};\eps) = \Fr(\tl_{k+1};a_k,b_k;\eps) < 0,
\end{equation*}
because~$\tl_{k+1} \subset (b_k,\tr_k)$. Similarly,
\begin{equation*}
\Fr(\tl_k;a_k,b_k;\eps) + \Fl(\tl_k;a_{k+1},b_{k+1};\eps) = \Fl(\tl_k;a_k,b_k;\eps) > 0.
\end{equation*}
By the Bolzano--Weierstrass principle, the balance equation has a root~$w_{k+1}$ on~$(\tl_{k+1},\tr_k)$. By Lemma~\ref{moninttails}, this root is unique.
\end{proof}
\begin{Rem}\label{InfiniteVersion}
The results of the preceding lemma and the corollary hold true if one of the cups \textup(or both\textup) sit at infinity\textup, i.e. the corresponding~$c_k$ is infinite.
\end{Rem}
\paragraph{Proof of Theorem~\ref{SimplePicture}.}
Now we have all the ingredients to prove Theorem~\ref{SimplePicture}. First, we take~$\eps$ to be smaller than a half of the minimal length among solid roots of~$f'''$, and smaller than a half of the minimal distance between two distinct essential roots of~$f'''$. Then, by the results of Subsection~\ref{s411} (Lemma~\ref{CupNearRoot} and Remark~\ref{MulticupInARoot}), one can build either a full cup or a multicup on each~$c_k$. But then, if~$\eps$ satisfies the assumptions of Corollary~\ref{BalanceEquationForSmallEps} (together with Remark~\ref{InfiniteVersion}) for each~$k$, one can paste an angle between each pair of consecutive cups or multicups. The relation~$w_{k+1}\to v_k$ is an immediate consequence of Lemma~\ref{TailForSmallEps} and the inclusion~$w_{k+1} \in (\tl_{k+1},\tr_k)$.\qed

\subsection{Examples}\label{s413}
We omit any calculations in this subsection  and only describe the behavior of simple pictures for two specific examples.

\paragraph{Polynomial of sixth degree: simple picture.}
The calculations for this example may be found in Section~\ref{s45} below. Let~$f$ be a polynomial of sixth degree. It will be explained in Section~\ref{s45} that it suffices to consider the typical cases~$f'''(t) = t^3 - 3t + c$ and~$f'''(t) = -t^3 + 3t - c$,~$0 \leq c \leq 2$. The latter inequality makes~$f'''$ have exactly three essential roots.

If the leading coefficient is positive, then, by Theorem~\ref{SimplePicture}, the foliation consists of two angles and a cup for sufficiently small~$\eps$. As we increase~$\eps$, the cup grows, and the angles move. It appears that there exists a moment~$\eps_1 =\eps_1(c)\leq \frac{\sqrt{35}}{9}$ (the first critical point of the evolution) at which the right angle attacks the cup (i.e. its vertex coincides with the right endpoint of the upper chord of the cup). This picture (a long chord and an angle) may be interpreted as a left trolleybus with the help of formula~\eqref{CupPlusAngleL}.

If the leading coefficient is negative, then, by Theorem~\ref{SimplePicture}, the foliation consists of two cups and an angle. Again, as we increase~$\eps$, the cups grow and the angle moves.

If~$c=0$, then the picture is symmetric, and the angle is stable. It appears that both cups attack the angle at~$0$ at the moment~$\eps_1=\sqrt{\frac{15}{8}}$. Here we may use formulas~\eqref{CupPlusAngleR} and~\eqref{ChordalDomainPlusMultitrolleybusR} (or symmetrically, formulas~\eqref{CupPlusAngleL} and~\eqref{ChordalDomainPlusMultitrolleybusL}) and think of the resulting figure as of a multicup on three points.

If~$c > 0$, then the angle attacks the right cup at some moment~$\eps_1 = \eps_1(c) \leq \sqrt{5\big(1 - \big(\frac{c}{2}\big)^{\frac 23}\big)}$ forming a right trolleybus by formula~\eqref{CupPlusAngleR}.
 
The graph representation of the evolution see on Fig.~\ref{evolutionP6+} and Fig.~\ref{evolutionP6-}.

\paragraph{The sine monster: simple picture.} Consider the function
\begin{equation}\label{SineMonsterFormula}
f(t) = \begin{cases}
-\cos t, \quad &|t| \leq \alpha;\\
\frac 12(t^2 - \alpha^2) \cos \alpha + (\sin \alpha - \alpha\cos\alpha)(|t|-\alpha) - \cos \alpha,\quad &|t| \geq \alpha.
\end{cases}
\end{equation}
Here~$\alpha \geq 0$ is a parameter we are going to vary. 
We have~$f'''(t) = -\sin t$ when~$|t| \leq \alpha$ and~$f'''(t) = 0$ otherwise. So, this function may have a lot of roots of~$f'''$ (and the bigger~$\alpha$ is, the more complicated the Bellman candidate will be). This example was considered in paper~\cite{CrazySine}, and the reader is welcome to consult it for calculations. Here we only outline the behavior of the simple picture (before the first critical value) for different~$\alpha$.

\emph{Case~$0 < \alpha \leq \pi$.} In this case, we have only one essential root at zero. So, the foliation consists of a symmetric cup at zero and two tangent domains that last up to infinity. The picture is simple for all~$\eps$.

\emph{Case~$\pi < \alpha \leq 2\pi$.} Now we have five essential roots:~$(-\infty,-\alpha], -\pi, 0, \pi$, and~$[\alpha,\infty)$. So, for sufficiently small~$\eps$ there are two multicups lasting to infinities, two angles near the points~$\pm\pi$, and a symmetric cup at zero. As we increase~$\eps$, the cup grows, and the angles move. It appears that there might be two evolutional scenarios. Namely, there exists a point~$\alpha_1 \in (\pi,2\pi]$ ($\alpha_1 \approx 4.49341$) such that for all~$\alpha \in (\pi,\alpha_1)$ the angles join the multicups at the first critical moment (and thus forming multitrolleybuses by formulas~\eqref{AnglePlusMulticupLeft},~\eqref{AnglePlusMulticupRight}); for all~$\alpha \in (\alpha_1,2\pi]$, the angles simultaneously attack the cup (forming a birdie by formulas~\eqref{CupPlusAngleR} and~\eqref{RTrolleybusPlusAngle}); for~$\alpha = \alpha_1$ we have a huge multicup that contains two infinities (here we use many formulas from Subsection~\ref{s344})!

\emph{Case~$2\pi < \alpha \leq 3\pi$.} We have five essential roots:~$-2\pi,-\pi,0,\pi$,~and~$2\pi$. At the very beginning, there are three cups at the points~$-2\pi,0$, and~$2\pi$, two stable angles at~$-\pi$ and~$\pi$, and tangent domains that fill the gaps between these figures. The angles sit at their points until the moment~$\eps = \alpha-2\pi$, when the border cups begin to grow asymmetrically. Now the angles move towards the middle cup and attack it at the first critical moment forming a birdie (see formulas~\eqref{CupPlusAngleR} and~\eqref{RTrolleybusPlusAngle}). If~$\alpha = 3\pi$, then the angles are stable up to the moment~$\eps = 3\pi$ when the cups attack them (all three cups meet together forming a multicup on four points).

\emph{Case~$3\pi < \alpha \leq 4\pi$.} We have nine essential roots (two of which are infinite rays). There exist  points~$\beta_3< \mu_3 \approx 11.4912$ such that for~$\alpha < \beta_3$ the first critical value occurs when the border angles join the multicups on infinite rays forming infinite multitrolleybuses, for~$\beta_3< \alpha < \mu_3$ the first critical point appears when the three cups at the points~$-2\pi,0$, and~$2\pi$ meet together at the moment~$\eps = \pi$ forming a multicup on four points; for~$\alpha > \mu_3$ the first critical point appears when the border angles (those which started their ways from the points~$\pm 3\pi$) attack their neighbor cups (sitting at~$\pm 2\pi$ correspondingly) forming trolleybuses; when~$\alpha = \mu_3$ these two scenarios happen simultaneously, i.e. all three cups and four angles meet together at the moment~$\eps = \pi$ and form a multibirdie on four points!

Fortunately, for each~$\alpha > 4\pi$, the evolution before the first critical value is similar to the one of the surveyed cases (however, this is not the case for all evolution).

\section{Preparation to evolution}\label{s42}
In this section, we collect technical lemmas that are useful for the evolution. There will be three groups of lemmas. The first group consists of lemmas that describe the places where the fictious vertices of the third type may occur, the second is about tails and forces, and the third one works with the balance equation.

\subsection{Structural lemmas for chords}\label{s421}\index{domain! chordal domain}
We make a convention on chordal domains: the inequalities~$\Slt < 0$ and~$\Srt<0$ hold true inside the chordal domain. Note that the same inequalities are required to build a standard candidate in a chordal domain.
\begin{Le}\label{zerodifrootsInside}
Let~$\Ch([a_0,b_0],*)$ be a chordal domain. If~$\Slt(a_0,b_0) = 0$\textup, then~$f''$  decreases on the right of~$a_0$\textup; if~$\Srt(a_0,b_0) = 0$\textup, then~$f''$ increases on the left of~$b_0$. 
\end{Le}
\begin{proof}
We treat the case of the right differential only. The remaining case is symmetric. As usual, we may assume that~$f'(b_0) = f'(a_0) = 0$ ($L_{a_0,b_0}$ coincides with the~$x$-axis). Then the assumption~$\Srt(a_0,b_0) = 0$ is reformulated as~$f''(b_0) = 0$.

Let us assume the contrary: suppose that~$f''$ does not increase in a left neighborhood of~$b_0$. In particular, it means that~$f' \leq 0$ in a left neighborhood of~$b_0$. First, we claim that in such a case~$f''(a_0) = 0$ and~$f''$ increases in a right neighborhood of~$a_0$.  

To prove the claim, we note that~$f''(a_0) \leq 0$ (because~$\Slt(a_0,b_0) \leq 0$), so, by Condition~\ref{reg}, the claim is equivalent to the fact that~$f' > 0$ in a right neighborhood of~$a_0$. Assume the contrary, let~$f' < 0$ in a right neighborhood of~$a_0$ (the strict inequality sign here is not a misprint, we will comment on it a bit later). We are going to find a contradiction with the cup equation inside~$\Ch([a_0,b_0],*)$, i.e. for~$l < l_0$ (here~$l_0 = b_0 - a_0$, we also use the notation related to chordal domains). Consider the points~$a_\delta=a(l_{\delta})$,~$b_\delta=b(l_{\delta})$,~$l_{\delta}\df l_0-\delta$ where~$\delta$ is a small number. Since~$f'$ is negative in a right neighborhood of~$a_0$ and is non-positive in a left neighborhood of~$b_0$, we have
\begin{equation*}
\frac{f'(a_{\delta}) + f'(b_{\delta})}{2} < 0.
\end{equation*}   
On the other hand, the same inequalities (together with the cup equation~$\av{f'}{[a_0,b_0]} = 0$) lead to
\begin{equation*}
\av{f'}{[a_{\delta},b_{\delta}]} > 0.
\end{equation*}
So, the cup equation~\eqref{vallun} cannot be fulfilled for~$\big(a_{\delta},b_{\delta}\big)$, and we see that~$f' \geq 0$ in a right neighborhood of~$a_0$. We claimed a slightly stronger statement:~$f' > 0$. Examining the reasoning above once more, we see that if~$f' = 0$ on the right of~$a_0$  and the cup equation for~$\big(a_{\delta},b_{\delta}\big)$ holds true, then~$f' = 0$ in a left neighborhood of~$b_0$. But this contradicts the conditions that the differentials inside the chordal domain are strictly negative. So, we have finally proved that~$f''(a_0) = 0$ and~$f''$ increases in a right neighborhood of~$a_0$. 
\begin{figure}[h!]
\begin{center}
\vskip-30pt
\includegraphics[width = 0.99 \linewidth]{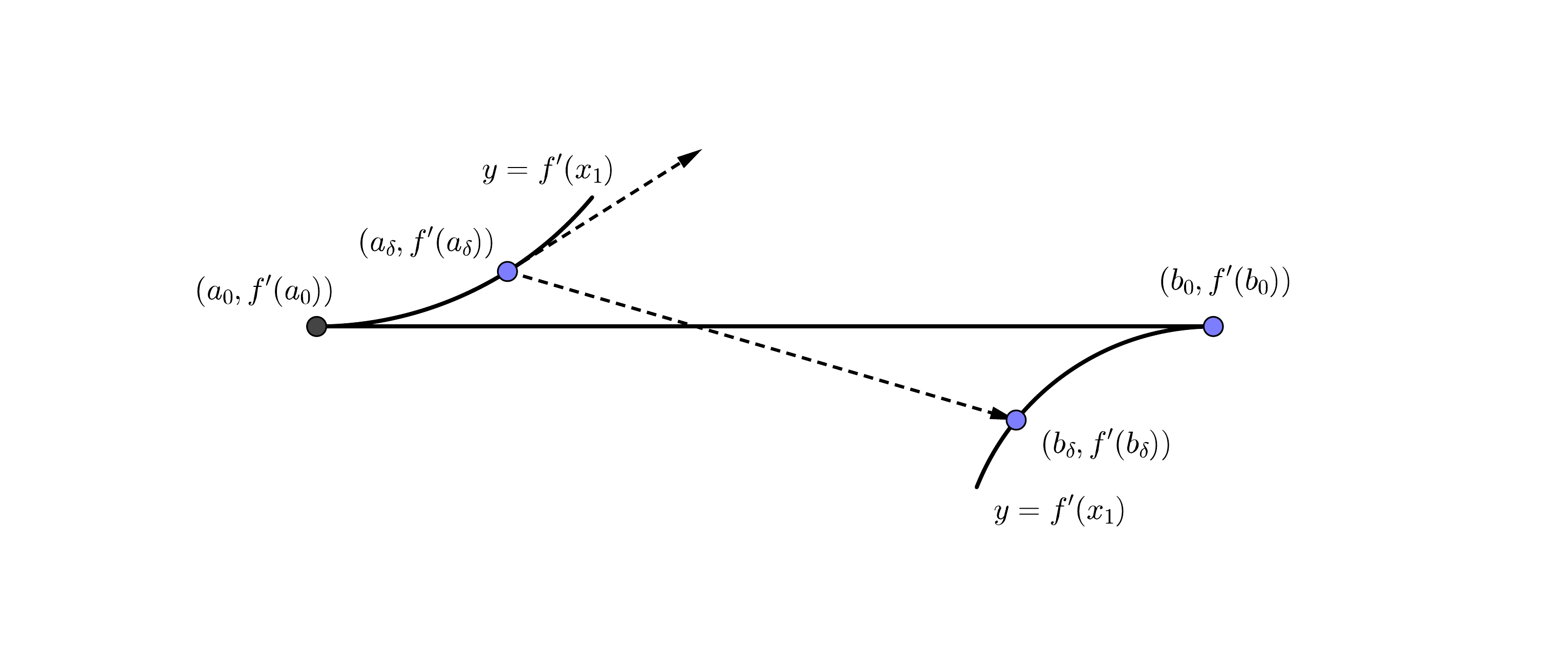}
\vskip-30pt
\caption{Left differential has wrong sign.}
\label{fig:leftdif}
\end{center}
\end{figure}

Now we will come to a contradiction with the inequality~$\Slt < 0$ inside the chordal domain. We will argue geometrically and hope that Figure~\ref{fig:leftdif} may help the reader to understand the heuristics. We see that the slope of the line~$L_{a_{\delta},a_{\delta}}$ increases, because~$f''$ increases in a right neighborhood of~$a_0$. So, it is greater than zero. On the other hand, the slope of the line~$L_{a_{\delta},b_{\delta}}$ is negative, because~$f'(a_{\delta}) > 0$ and~$f'(b_{\delta}) \leq 0$. Recalling the geometric interpretation of the differentials (discussed at the end of Subsection~\ref{s331}), we see that~$\Slt(a_{\delta},b_{\delta}) > 0$, which contradicts the assumptions about the chordal domain. So, we have found a contradiction and finally have proved the lemma. 
\end{proof}
\begin{Le}\label{zerodifrootsOutside}
Suppose that the pair~$(a_0,b_0)$ satisfies the cup equation and the chord~$[A_0,B_0]$ has nonzero tails. If~$\Slt(a_0,b_0) = 0$\textup, then~$f''$ increases on the left of~$a_0$\textup; if~$\Srt(a_0,b_0) = 0$\textup, then~$f''$ decreases on the right of~$b_0$. 
\end{Le}
\begin{proof}
We treat the case of the right differential only. The remaining case is symmetric. 

We look at formula~\eqref{RightForce} and see that it consists of the integral term only. By Condition~\ref{reg} for the function~$f$, the measure~$f'''$ is either negative or non-negative in a right neighborhood of~$b_0$. If it is non-negative, then the force is non-negative as well. But it must be negative in a right neighborhood of~$b_0$, because~$[A_0,B_0]$ has nonzero right tail. Therefore,~$f''$ decreases on the right of~$b_0$.
\end{proof}
We use notation from Definition~\ref{roots} in the corollary below. This lemma says that during the evolution, the differentials can vanish only in very special situations.
\begin{Cor}\label{RootCorollary}
Suppose that the chordal domain~$\Ch([a_0,b_0],*)$ has nonzero tails. If~$\Srt(a_0,b_0) = 0$\textup, then~$b_0 = c_i$ for some~$i$\textup; if~$\Slt(a_0,b_0) = 0$\textup, then~$a_0 = c_i$ for some~$i$.
\end{Cor}

\subsection{Tails growth lemmas}\label{s422}\index{force! function}
\begin{Le}\label{biggercupsmallerforce}
Suppose~$\Ch([a_1,b_1],*)$ is embedded into~$\Ch([a_0,b_0],*)$\textup, i.e. the foliation of the former chordal domain coincides with some part of the foliation of the latter. Then the left force of~$\Ch([a_1,b_1],*)$ is not less than the left force of~$\Ch([a_0,b_0],*)$\textup, whereas the right force of~$\Ch([a_1,b_1],*)$ does not exceed the right force generated by~$\Ch([a_0,b_0],*)$. Outside~$[a_1,b_1]$ the inequalities are strict.
\end{Le}
\begin{proof}
This is a straightforward consequence of formulas~\eqref{Fpoell}.
\end{proof}

\begin{Cor}\label{TailsGrowthNonEv}
Suppose~$\Ch([a_1,b_1],*)$ is embedded into~$\Ch([a_0,b_0],*)$ in the sense of Lemma\textup{~\ref{biggercupsmallerforce}}. Then the tails of the former chordal domain strictly contain the tails of the latter. 
\end{Cor}
\begin{proof}
This follows from Definition~\ref{Tail} and Lemma~\ref{biggercupsmallerforce}. 
\end{proof}

In other words, the previous statements can be reformulated informally: the less the chordal domain is, the larger the tails are and the bigger the absolute values of the forces inside the tails are. 

\begin{Le}\label{EvolutionalTailGrowthLemma}
Let~$\Ch([a_0,b_0],*)$ be a chordal domain\textup, let~$\Fl$ and~$\Fr$ be its left and right forces respectively. If~$u$ belongs to the closure of the left tail of~$\Ch([a_0,b_0],*)$\textup, then 
\begin{equation*}
\frac{\partial \Fl(u)}{\partial \eps} \geq 0,
\end{equation*}
and if~$u$ belongs to the closure of the right tail\textup, then
\begin{equation*}
\frac{\partial \Fr(u)}{\partial \eps} \leq 0.
\end{equation*}
The inequalities are strict outside~$[a_0,b_0]$.
\end{Le}
\begin{proof}
We treat the case of the right force only, the other one is symmetric. First, we note that inside~$[a_0,b_0]$ the force does not depend on~$\eps$, therefore,~$\frac{\partial \Fr(u)}{\partial \eps} = 0$. Second, we use formula~\eqref{leftder} and see that  
\begin{equation*}
\Fr'(u) + \frac{\Fr(u)}{\eps} - f'''(u)=0
\end{equation*}
outside this interval.
We differentiate this equation with respect to~$\eps$ and see that 
\begin{equation*}
\Big(\Fr'(u)\Big)_{\eps}' + \frac{\Big(\Fr(u)\Big)_{\eps}'}{\eps} - \frac{\Fr(u)}{\eps^2} = 0.
\end{equation*}
After interchanging the differentiations with respect~$u$ and~$\eps$ in the first summand we see that~$\Big(\Fr(u)\Big)_{\eps}'$ is a solution of the first-order differential equation with respect to~$u$: 
\begin{equation*}
\varkappa'(u) + \frac{\varkappa(u)}{\eps} - \frac{\Fr(u)}{\eps^2} = 0.
\end{equation*}
It is similar to equation~\eqref{difeq}, we can solve it: 
\begin{equation*}
\Big(\Fr(u)\Big)_{\eps}' = \eps^{-2}e^{-u / \eps}\int\limits_{b_0}^u \Fr(t)e^{t /\eps}dt,
\end{equation*}
we begin the integration from~$b_0$, because~$\lim_{z \to b_0+}\Big(\Fr(z)\Big)_{\eps}' = 0$ (this limit relation can be verified by a straightforward calculation using formula~\eqref{RightForce}).
Now, the result follows immediately, because~$\Fr < 0$ inside the tail.
\end{proof}
\begin{Cor}\label{TailsGrowthCorollary}
Let~$\Ch([a_0,b_0],*)$ be a chordal domain with nonzero tails. If we grow~$\eps$ a little\textup, its tails strictly grow. 
\end{Cor}
\begin{proof}
By Lemma~\ref{EvolutionalTailGrowthLemma}, the forces increase in absolute value on the corresponding tails, therefore, the tails cannot decrease. What is more, if we grow~$\eps$ a little, the force at the end of the tail will grow, and the tail will enlarge.
\end{proof}
\begin{Rem}\label{TailGrowthInfinity}
The results of the previous lemma and corollary hold for the forces coming from the infinities and multicups as well.
\end{Rem}
\begin{Le}\label{FullCupEvolutionTails}
Let~$\Ch([a(2\eps),b(2\eps)],*)$ be a full chordal domain\textup, i.e.~$b(2\eps) - a(2\eps) = 2\eps$. Then the tails of~$\Ch([a(2\eps),b(2\eps)],*)$ strictly grow in~$\eps$. 
\end{Le}
\begin{proof}
Indeed, consider some point~$u$ that belongs, say, to the left tail of~$\Ch([a(2\eps),b(2\eps)],*)$ for some~$\eps$. First, we need to prove that~$\Fl(u; a(2\eps), b(2\eps);\eps)$ grows in~$\eps$. But its derivative with respect to~$\eps$ is non-negative, because, by Lemma~\ref{difequal}, for each~$\eps$ it equals the corresponding derivative taken as if the chordal domain had fixed upper chord, which is nonnegative by Lemma~\ref{EvolutionalTailGrowthLemma}. Thus, the tails do not decrease. Moreover, the derivative of the corresponding force is nonzero at the end of each tail, again by Lemma~\ref{EvolutionalTailGrowthLemma}, so, the tail must grow.
\end{proof}
\begin{Def}\label{FlowOfChordalDomains}\index{domain! chordal domain! flow of chordal domains}
Let~$0 \leq \eps_1 < \eps_2 < \eps_{\infty}$. We call a chordal domain~$\Ch([a_0,b_0],[a_1,b_1])$ with a continuous function~$\mathfrak{l}\colon(\eps_1,\eps_2) \to (b_1-a_1,b_0-a_0]$ a flow of chordal domains. We define the forces of such a flow as
\begin{equation*}
\begin{aligned}
&\Fr\Big(u;\Ch([a_0,b_0],[a_1,b_1]),\mathfrak{l};\eps\Big) = \Fr\Big(u;\Ch\big(\big[a(\mathfrak{l}(\eps)),b(\mathfrak{l}(\eps))\big],[a_1,b_1]\big);\eps\Big), \quad &u \in (b_1,+\infty);\\
&\Fl\Big(u;\Ch([a_0,b_0],[a_1,b_1]),\mathfrak{l};\eps\Big) = \Fl\Big(u;\Ch\big(\big[a(\mathfrak{l}(\eps)),b(\mathfrak{l}(\eps))\big],[a_1,b_1]\big);\eps\Big), \quad &u \in (-\infty, a_1).
\end{aligned} 
\end{equation*}
We also define the tails of the flow as the intervals~$(b_1,\tr(\eps))$ and~$(\tl(\eps),a_1)$\textup, where~$\tr$ and~$\tl$ are the right and the left endpoints of the tails of~$[A(\mathfrak{l}(\eps)),B(\mathfrak{l}(\eps))]$. We say that a flow is decreasing if~$\mathfrak{l}$ is decreasing. We say that a flow is full if~$\mathfrak{l}(\eps) = 2\eps$ for all~$\eps \in (\eps_1,\eps_2)$.
\end{Def}
\begin{Cor}\label{FlowMonotonicityCorollary}
Consider a decreasing or full flow of chordal domains. The forces strictly increase in absolute value as functions of~$\eps$ on its tails outside the chordal domain and are constant inside the chordal domain. As a consequence\textup, the tails grow.
%the left force does not decrease. The inequalities are strict outside~$[a_0,b_0]$. As a consequence\textup, the tails of a decreasing flow grow along the flow. The right force of a full flow decreases on its tail outside~$[a_0,b_0]$\textup, whereas the left force increases. As a consequence\textup, the tails of a full flow grow.
\end{Cor}
\begin{proof}
For the case of a decreasing flow, we use Lemmas~\ref{EvolutionalTailGrowthLemma} and~\ref{biggercupsmallerforce}. For the case of a full flow, we use Lemma~\ref{FullCupEvolutionTails}.
\end{proof}
\begin{Def}\label{MFF}\index{force! monotone force flow}
Let~$0 \leq \eps_1 < \eps_2 <\eps_{\infty}$\textup,~$u \in \mathbb{R}\cup\{-\infty\}$. We say that a continuous function~$\FFr\colon (u,+\infty)\times[\eps_1,\eps_2]\to \mathbb{R}$ is a right monotone force flow if
\textup{\begin{itemize}
\item \emph{for any~$\eps \in [\eps_1,\eps_2]$\textup, the function~$\FFr(\cdot\,;\eps)$ is a force of some chord situated on the left of~$u$ or of~$-\infty$\textup, for~$\eps \in (\eps_1,\eps_2)$\textup, the tail of this force contains~$u$};
\item \emph{for any~$\eta_1$ and~$\eta_2$ such that~$\eps_1 \leq \eta_1 < \eta_2 \leq \eps_2$\textup, we have~$\FFr(v;\eta_1) > \FFr(v;\eta_2)$ whenever~$v \in [u,\tr(\eta_1)]$\textup, here~$\tr(\eta_1)$ is the right endpoint of the tail of~$\FFr(\cdot\,;\eta_1)$.}
\end{itemize}}
Let~$u \in \mathbb{R}\cup \{+\infty\}$. We say that a continuous function~$\FFl\colon (-\infty,u)\times[\eps_1,\eps_2]\to \mathbb{R}$ is a left monotone force flow if
\textup{\begin{itemize}
\item \emph{for any~$\eps \in [\eps_1,\eps_2]$\textup, the function~$\FFl(\cdot\,;\eps)$ is a force of some chord situated on the right of~$u$ or of~$+\infty$\textup, for~$\eps \in (\eps_1,\eps_2)$\textup, the tail of this force contains~$u$};
\item \emph{for any~$\eta_1$ and~$\eta_2$ such that~$\eps_1 \leq \eta_1 < \eta_2 \leq \eps_2$\textup, we have~$\FFl(v;\eta_1) < \FFl(v;\eta_2)$ whenever~$v \in [\tl(\eta_1),u]$\textup, here~$\tl(\eta_1)$ is the left endpoint of the tail of~$\FFl(\cdot\,;\eta_1)$.}
\end{itemize}}
\end{Def}
It follows immediately that~$\tr$ is an increasing function of~$\eps$ for the right force flow, whereas~$\tl$ is decreasing for the left force flow. A trivial example of a monotone force flow is given by a force of a chord that does not depend on~$\eps$ (the monotonicity follows from Lemma~\ref{EvolutionalTailGrowthLemma}). Another trivial remark is that one can take a restriction of a monotone force flow to a smaller domain to obtain another monotone force flow (this will be implicitly used several times in the sequel).  
\begin{Rem}\label{FCDMFF}
Let~$\{\Ch([a_0,b_0],*),\mathfrak{l}\}$ be a flow of chordal domains for~$\eps \in (\eta_1,\eta_2)$. Let
\begin{equation*}
\FFr(u\,;\eps) = \Fr(u\,;\Ch([a_0,b_0,*]),\mathfrak{l};\eps),\quad (u,\eps)\in (b_0,\infty)\times (\eta_1,\eta_2).
\end{equation*}
If the flow is either decreasing or full\textup, then~$\FFr$ is a right monotone force flow\textup(to fulfill Definition~\textup{\ref{MFF},} it has to be extended to the border cases~$\eps = \eta_1$ and~$\eps = \eta_2$ by continuity\textup). The same with the left force flow\textup:
\begin{equation*}
\FFl(u\,;\eps) = \Fl(u\,;\Ch([a_0,b_0,*]),\mathfrak{l};\eps),\quad (u,\eps)\in (-\infty,a_0)\times (\eta_1,\eta_2)
\end{equation*}
is a left monotone force flow if the generating flow of chordal domains is either decreasing or full. 
\end{Rem}

\subsection{Balance equation lemmas}\label{s423}
\begin{Def}\label{baleq}\index{equation! balance equation}
Let~$\Fr$ and~$\Fl$ be forces of chordal domains, infinities, or multicups such that their domains \textup(the definitions are given at the beginning of Subsection~\ref{s412}\textup) intersect. Then the balance equation is 
\begin{equation}\label{baleqformula}
\Fr(u) + \Fl(u) = 0,
\end{equation}
where~$u$ belongs to the intersection of the domains of the forces.
\end{Def} 
We are looking for solutions of balance equations. It was Proposition~\ref{moninttails} that helped us to establish the existence of the solution in Subsection~\ref{s412}. %We want to generalize it to newly defined forces. 
\begin{Le}\label{monbaleq}
Let~$\Fr$ and~$\Fl$ be two forces of chordal domains\textup, infinities\textup, multicups\textup, or simply chords such that their tails intersect. Then the function~$\Fr(u) + \Fl(u)$ increases on this intersection.
\end{Le}
\begin{proof}
Using formulas~\eqref{leftder} and~\eqref{rightder}, we see that
\begin{equation}\label{sumforceder}
\big(\Fr + \Fl\big)'(u) = - \frac{2\Fr(u)}{\ell_{\mathrm{L}}(u)} + \frac{2\Fl(u)}{\ell_{\mathrm{R}}(u)}
\end{equation}
if we assume~$\ell(u) = 2\eps$ when~$u$ lies outside the corresponding chordal domain, or multicup, or when the force comes from infinity. Both summands are positive in the intersection of the tails (except, possibly, for one or two points~$u$).
\end{proof}
%\begin{Cor}\label{RootOfBE}
%Let~$\Fr$ and~$\Fl$ be a right and a left force. Suppose that their tails intersect\textup, denote this intersection by~$(\tl,\tr)$. Then\textup, the corresponding balance equation has a unique root inside~$(\tl,\tr)$. 
%\end{Cor}
%\begin{proof}
%The function~$\Fl(u) + \Fr(u)$ is negative at the point~$u=\tl$ and positive at the point~$u=\tr$. Therefore, the balance equation has a root on~$(\tl,\tr)$. By Lemma~\ref{monbaleq}, this root is unique.
%\end{proof}
\begin{Le}\label{StrictMonotonicityBalanceEquation}
Let~$\Fr$ and~$\Fl$ be a right and a left force of chordal domains\textup, infinities\textup, or multicups. Suppose that the closures of their tails intersect\textup, and let~$w \in [\tl,\tr]$ be the root of the corresponding balance equation. Then\textup, the function~$\Fr+\Fl$ is strictly positive on the right of~$w$ and strictly negative on the left of it.
\end{Le}
\begin{proof}
By Lemma~\ref{monbaleq}, it suffices to consider the case where~$w$ coincides with the endpoint of one of two tails. So, let~$w = \tr$. If~$(\Fr+\Fl)'(w) > 0$, then the statement is obvious. Therefore~$(\Fr+\Fl)'(w) = 0$ and~$\Fl(w) = 0$ by formula~\eqref{sumforceder}. This can happen in two cases: either~$\tr = \tl$ (the closures of the tails intersect by one point) or~$\Fl$ is a force of a chordal domain with zero left differential, and~$w$ coincides with the left endpoint of its upper chord. In both cases~$\Fl > 0$ on the right of~$w$. 

Consider the following two cases:~$f''' \geq 0$ in a right neighborhood of~$w$ and~$f''' < 0$ in a right neighborhood of~$w$ (one of these two cases occurs by Condition~\ref{reg}). Looking at formula~\eqref{RightForce} (and~\eqref{RightForceInfinity}), we see that in the first case~$\Fr \geq 0$ in a right neighborhood of~$w$, and thus~$\Fr + \Fl$ is also positive there. In the second case,~$\Fr < 0$ in a right neighborhood of~$w$, thus, by formula~\eqref{sumforceder}, we have~$(\Fr + \Fl)' > 0$ there. Since~$\Fr(w)+ \Fl(w) = 0$, we also have~$\Fr+\Fl > 0$ on the right of~$w$.  
\end{proof}
\begin{Rem}\label{SolidRootRemark}
With the same reasoning\textup, one can prove the following. Suppose that~$\tr < \tl$\textup, but~$\Fr + \Fl = f''' = 0$ on~$[\tr,\tl]$ \textup(the root of the balance equation is solid and almost intersects the tails\textup). Then the function~$\Fr+\Fl$ is strictly positive in a  right neighborhood of~$[\tr,\tl]$ and strictly negative in a left.
\end{Rem}

\section{Local evolutional theorems}\label{s43}
The form of all theorems in this section is: if for some~$\eps$ one can build a Bellman candidate on a specific domain using specific formulas, then, for a slightly bigger~$\eps$, he can also build a Bellman candidate on a perturbed domain using similar formulas with perturbed parameters. The statements are rather formal and somewhat bulky, so, before each statement we give a short heuristic explanation. We also recall our convention that the differentials are strictly negative inside chordal domains and that inequalities~$m'' < 0$ and~$m'' > 0$ are fulfilled for right and left tangent domains correspondingly.

The proposition below says that any full chordal domain (with nonzero differentials of the upper chord) surrounded by tangent domains grows with~$\eps$ (in other words, in view of Definition~\ref{FlowOfChordalDomains}, a full chordal domain generates a full flow of chordal domains starting from it). 
 
\begin{St}[{\bf Induction step for a chordal domain}]\label{InductionStepForChordalDomain}\index{domain! chordal domain}
Let~$\eta_1 < \eps_{\infty}$\textup, let~$\Ch([a_0,b_0],*)$ be a full chordal domain \textup(i.e.~$b_0-a_0=2\eta_1$\textup)\textup, let~$u_1 < a_0$\textup,~$b_0 < u_2$. Let a continuous function~$B$ coincide with the standard candidates on~$\Ch([a_0,b_0],*)$\textup,~$\Lt(u_1,a_0;\eta_1)$\textup, and~$\Rt(b_0,u_2;\eta_1)$. 
%Define the function~$B$ by formula~\textup{\eqref{vallun}} inside the chordal domain\textup, by formulas~\textup{\eqref{linearity}, \eqref{ExplicitFormulaForm}, \eqref{m(b_0)}} inside~$\Rt(b_0,u_2)$\textup, and by formulas~\textup{\eqref{linearity}, \eqref{ExplicitFormulaForm2}, \eqref{m(a_0)}} inside~$\Lt(u_1,a_0)$. Suppose that it satisfies the inequalities~$m'' < 0$ and~$m'' > 0$ inside~$\Rt(b_0,u_2)$ and~$\Lt(u_1,a_0)$ correspondingly and the assumptions of Proposition~\textup{\ref{LightChordalDomainCandidate}} inside the chordal domain. 
If~$\Slt(a_0,b_0) < 0$ and~$\Srt(a_0,b_0) < 0$\textup, then there exists~$\eta_2 > \eta_1$ and functions~$a$ and~$b$ acting from~$[2\eta_1,2\eta_2]$ to~$\mathbb{R}$ such that~$a(2\eta_1) = a_0$\textup,~$b(2\eta_1) = b_0$\textup,~$b(2\eps) - a(2\eps) = 2\eps$\textup,~$\eps \in [\eta_1,\eta_2]$\textup, and these functions satisfy the assumptions of Proposition~\textup{\ref{LightChordalDomainCandidate}}\textup, and for each~$\eps \in [\eta_1,\eta_2]$ there exists a continuous function~$B_{\eps}$ which coincides with the standard candidates on $\Lt(u_1,a(2\eps);\eps)$\textup,~$\Ch([a(2\eps),b(2\eps)],*)$\textup, and~$\Rt(b(2\eps),u_2;\eps)$. 
%which is a standdefined on the domains~$\Ch([a(2\eps),b(2\eps)],*)$\textup,~$\Rt(u_1,a(2\eps))$\textup, and~$\Lt(b(2\eps),u_2)$ by the same formulas as before\textup, satisfies the same inequalities inside the tangent domains \textup(in particular\textup, the requirements  of Propositions~\textup{\ref{RightTangentsCandidate}} and~\textup{\ref{LeftTangentsCandidate})} on the corresponding domains.
\end{St}

\begin{proof}
By Lemma~\ref{SmoothChordalDomain}, there exist a number~$\eta_2$ and functions~$a$ and~$b$ acting from~$[2\eta_1,2\eta_2]$ to~$\mathbb{R}$ such that~$a(2\eta_1) = a_0$,~$b(2\eta_1) = b_0$,~$a(2\eps) - b(2\eps) = 2\eps$ and the pair~$(a(2\eps),b(2\eps))$ always satisfies the cup equation. By continuity, the inequalities~$\Slt(a(2\eps),b(2\eps)) < 0$ and~$\Srt(a(2\eps),b(2\eps)) < 0$ are also valid when~$\eps$ is not far from~$\eta_1$. Changing~$\eta_2$, we may assume that they hold on the whole interval~$[2\eta_1,2\eta_2]$.

By our assumptions,~$u_2$ belongs to the right tail of~$\Ch([a_0,b_0],*)$ and~$u_1$ belongs to its left tail. Due to Lemma~\ref{FullCupEvolutionTails},~$u_1$ and~$u_2$ belong to the left and the right tail of~$\Ch([a(2\eps), b(2\eps)],*)$,~$\eps \in [\eta_1,\eta_2]$, correspondingly. This yields the possibility to define~$B_\eps$ to be a standard candidate on~$\Ch([a(2\eps),b(2\eps)],*)\cup \Lt(u_1,a(2\eps);\eps) \cup \Rt(b(2\eps),u_2;\eps)$ for~$\eps \in [\eta_1,\eta_2]$.  
\end{proof}

By Remark~\ref{FCDMFF}, the functions
\begin{equation*}
\begin{aligned}
&\FFr(u;\eps) = \Fr(u;[a(2\eps),b(2\eps)];\eps),\quad \eps \in (\eta_1,\eta_2),\; u\in (b(2\eta_2),\infty);\\
&\FFl(u;\eps) = \Fl(u;[a(2\eps),b(2\eps)];\eps),\quad \eps \in (\eta_1,\eta_2),\; u\in (-\infty, a(2\eta_2))
\end{aligned}
\end{equation*}
are the right and left monotone force flows.

The following proposition says how a non-full multicup (i.e. a multicup~$\MTC(\{\mathfrak{a}_i\}_{i=1}^k;\eps)$ with~$\mathfrak{a}_k^{\mathrm{r}} - \mathfrak{a}_1^{\mathrm{l}} > 2\eps$) evolves in~$\eps$.

\begin{St}[{\bf Induction step for a multicup}]\label{InductionStepForMulticup}\index{multicup}
Let~$\eta_1 < \eps_{\infty}$\textup, let~$\MTC(\{\mathfrak{a}_i\}_{i=1}^k;\eta_1)$ be a multicup such that~$\mathfrak{a}_k^{\mathrm{r}} - \mathfrak{a}_1^{\mathrm{l}} > 2\eta_1$\textup, let~$u_1 < \mathfrak{a}_1^{\mathrm{l}}$ and~$\mathfrak{a}_k^{\mathrm{r}} < u_2$. Let a continuous function~$B$ coincide with the standard candidates on~$\MTC(\{\mathfrak{a}_i\}_{i=1}^k;\eta_1)$\textup,~$\Lt(u_1,a_0;\eta_1)$\textup, and~$\Rt(b_0,u_2;\eta_1)$. 
Then there exists~$\eta_2$\textup, $2\eta_1 < 2\eta_2 \leq \mathfrak{a}_k^{\mathrm{r}} - \mathfrak{a}_1^{\mathrm{l}}$\textup, such that for each~$\eps \in [\eta_1,\eta_2]$ there exists a continuous function~$B_{\eps}$ that coincides with the standard candidates on $\Lt(u_1,\mathfrak{a}_1^{\mathrm{l}};\eps)$\textup,~$\MTC(\{\mathfrak{a}_i\}_{i=1}^k;\eps)$\textup, and~$\Rt(\mathfrak{a}_k^{\mathrm{r}},u_2;\eps)$.
\end{St}

\begin{proof}
Actually, we may set~$\eta_2 = \frac12(\mathfrak{a}_k^{\mathrm{r}} - \mathfrak{a}_1^{\mathrm{l}})$. Take any~$\eps$ from the interval prescribed. The only non-obvious thing is why we can define the standard candidates on each domain and glue them continuously. The necessary and sufficient conditions for that are %are the inequalities~$m'' < 0$ and~$m'' > 0$ satisfied inside the domain of right and left tangents correspondingly. By Definition~\ref{ForceForMulticup}, this is the same as to prove that
\begin{equation*}
\begin{aligned}
&\Fr(u;\MTC(\{\mathfrak{a}_i\}_{i=1}^k);\eps) < 0,\quad & u \in [\mathfrak{a}_k^{\mathrm{r}}, u_2];\\
&\Fl(u;\MTC(\{\mathfrak{a}_i\}_{i=1}^k);\eps) > 0,\quad & u \in [u_1,\mathfrak{a}_1^{\mathrm{l}}].
\end{aligned}
\end{equation*}
By the assumptions, these inequalities hold true for~$\eps = \eta_1$. Then, by Lemma~\ref{EvolutionalTailGrowthLemma}, these inequalities are also true for all~$\eps > \eta_1$.
\end{proof}
Trivially, the functions
\begin{equation*}
\begin{aligned}
&\FFr(u;\eps) = \Fr(u;[\mathfrak{a}_1^{\mathrm l},\mathfrak{a}_k^{\mathrm r}];\eps),\quad \eps \in (\eta_1,\eta_2),\; u\in (\mathfrak{a}_k^{\mathrm r},\infty);\\
&\FFl(u;\eps) = \Fl(u;[\mathfrak{a}_1^{\mathrm l},\mathfrak{a}_k^{\mathrm r}];\eps),\quad \eps \in (\eta_1,\eta_2),\; u\in (-\infty,\mathfrak{a}_1^{\mathrm l})
\end{aligned}
\end{equation*}
are the right and left monotone force flows.

The following proposition says that a long chord with nonzero tails gives rise to a chordal domain. We note that this generalizes Proposition~\ref{InductionStepForChordalDomain} (in the latter case the differentials are nonzero, and thus the tails of the upper chord are nonzero). However, for didactic reasons, we prefer to separate these two propositions. In a sense, the cases where one or both differentials are zero differ from what is described in Proposition~\ref{InductionStepForChordalDomain}. Indeed, suppose that one of the differentials is zero and we have a chordal domain below~$[A_0,B_0]$. By the proposition below, after we increase~$\eps$, we can build a chordal domain above~$[A_0,B_0]$, but we cannot join the two chordal domains into a single one. We have to paste a fictious vertex corresponding to the chord~$[A_0,B_0]$ between them. We also note that the proposition below may be applied to the upper chord of a full multicup (i.e. a multicup with~$\mathfrak{a}_k^{\mathrm{r}} - \mathfrak{a}_1^{\mathrm{l}} = 2\eps$). 

\begin{St}[{\bf Induction step for a long chord}]\label{InductionStepForLongChords}
Let~$\eta_1$ be smaller than~$\eps_{\infty}$. Suppose that~$b_0 - a_0 = 2\eta_1$ and the pair~$(a_0,b_0)$ satisfies the cup equation. Let also~$u_1 < a_0$ and~$u_2 > b_0$ be numbers such that~$u_1$ belongs to the left tail of~$[A_0,B_0]$ and~$u_2$ belongs to the right tail of~$[A_0,B_0]$. 
Then there exists~$\eta_2 > \eta_1$\textup, and functions~$a,b\colon [2\eta_1,2\eta_2]\to\mathbb{R}$ satisfying the assumptions of Proposition~\textup{\ref{LightChordalDomainCandidate}} and such that~$b(2\eps) - a(2\eps) = 2\eps$\textup, and continuous functions~$B_\eps$ that coincide with the standard candidates on~$\Ch([a(2\eps),b(2\eps)],[a_0,b_0])$\textup, $\Lt(u_1,a(2\eps);\eps)$\textup, $\Rt(b(2\eps),u_2;\eps)$ for all~$\eps \in [\eta_1,\eta_2]$.
%
%
%constructed by formulas~\eqref{linearity}\textup,~\eqref{ExplicitFormulaForm}\textup, and~\eqref{m(b_0)} in~$\Rt(b(2\eps),u_2)$\textup, and by formulas~\eqref{linearity}\textup,~\eqref{ExplicitFormulaForm2}\textup, and~\textup{\eqref{m(a_0)}} inside~$\Lt(u_1,a(2\eps))$\textup, satisfies the inequalities~$m'' > 0$ inside~$\Lt(u_1,a(2\eps))$ and~$m''<0$ inside~$\Rt(b(2\eps),u_2)$ \textup(in particular\textup, it satisfies the assumptions of Propositions~\textup{\ref{RightTangentsCandidate}} and~\textup{\ref{LeftTangentsCandidate}} on the corresponding domains\textup) for all~$\eps \in [\eta_1,\eta_2]$.
\end{St}

\begin{proof}
Since the tails of~$[A_0,B_0]$ are not empty,~$\Slt(a_0,b_0) \leq 0$,~$\Srt(a_0,b_0) \leq 0$. If both these inequalities are strict, then we can simply apply Lemma~\ref{SmoothChordalDomain} and Proposition~\ref{InductionStepForChordalDomain}. So, assume that one or both differentials are zero. 

If~$\Srt(a_0,b_0) = 0$, then, by Lemma~\ref{zerodifrootsOutside}, the graph of~$f'$ in a right neighborhood of~$b_0$ lies below~$L_{a_0,b_0}$ (and similarly for~$a_0$ and~$\Slt(a_0,b_0)$). So, we are in the assumptions of Corollary~\ref{ChordDomFromAChord}. By our assumptions,~$u_1$ belongs to the left tail of~$[A_0,B_0]$ and~$u_2$ belongs to the right tail of~$[A_0,B_0]$. By Lemma~\ref{FullCupEvolutionTails},
$\tl(\Ch([a(2\eps),b(2\eps)],[a_0,b_0]);\eps)$
decreases in~$\eps$, so,
\begin{equation*}
u_1 > \tl(\Ch([a(2\eps),b(2\eps)],[a_0,b_0]);\eps)
\end{equation*}
for each~$\eps$. Similarly,
\begin{equation*}
u_2 < \tr(\Ch([a(2\eps),b(2\eps)],[a_0,b_0]);\eps).
\end{equation*}
These inequalities allows us to construct the desired function~$B_\eps$.
\end{proof}

Now we turn to trolleybuses. The next two propositions claim that the base of a trolleybus shrinks when~$\eps$ increases. On a more formal way, there exists a decreasing flow of chordal domains such that on each chordal domain we can build a trolleybus for its~$\eps$.

\begin{St}[{\bf Induction step for a right trolleybus}]\label{InductionStepForRightTrolleybus}\index{trolleybus}
Let~$\eta_1<\eps_{\infty}$. Suppose that~$u_1 < a_0 < b_0 \leq u_2$ and~$b_0 - a_0 \leq 2\eta_1$. Suppose that~$\FFr$ is a right monotone force flow on~$(u_1,+\infty)\times[\eta_1,\eta_3]$\textup, let~$a_0$ belong to the closure of the tail of~$\FFr(\cdot\,;\eta_1)$. Suppose that there exists a continuous function~$B$ that coincides with the standard candidates on~$\Rt(u_1,a_0;\eta_1)$\textup, $\RTroll(a_0,b_0;\eta_1)$\textup, $\Ch([a_0,b_0],*)$\textup, and $\Rt(b_0,u_2;\eta_1)$. Then there exists~$\eta_2$\textup, $\eta_1<\eta_2<\eps_\infty$\textup, and a decreasing continuous function~$\mathfrak{l}\colon [\eta_1,\eta_2] \to \mathbb{R}$ such that~$\mathfrak{l}(\eta_1) = b_0-a_0$\textup, and there exists a continuous function $B_\eps$ that coincides with the standard candidates on~$\Rt(u_1,a(\mathfrak{l}(\eps));\eps)$\textup, $\RTroll(a(\mathfrak{l}(\eps)),b(\mathfrak{l}(\eps));\eps)$\textup, $\Ch([a(\mathfrak{l}(\eps)),b(\mathfrak{l}(\eps))],*)$\textup, and $\Rt(b(\mathfrak{l}(\eps)),u_2;\eps)$ for any $\eps \in [\eta_1,\eta_2]$\textup, where $a$ and $b$ are the functions corresponding to the chordal domain~$\Ch([a_0,b_0],*)$.
\end{St}

%Here the functions~$a$ and~$b$ are the ones defined in the beginning of Subsection~\ref{s331}. 
We note that~$\Ch([a_0,b_0],*)$ equipped with~$\mathfrak{l}$ is a decreasing flow of chordal domains.
\begin{proof}
%First, we claim that the right tail of~$\Fr$ and the left tail of~$\Ch([a_0,b_0],*)$ intersect. The only case in which it is not so is when both tails end at~$a_0$. In such a case,~$\Slt(a_0,b_0) = 0$ (because the chordal domain has zero tail) and~$f''' > 0$ on the left of~$a_0$ (because the tail of~$\Fr$ ends at~$a_0$). But in such a case the tail of~$\Ch([a_0,b_0],*)$ is nonzero. So, the tails of~$\Fr$ and the chordal domain intersect. 
By Lemma~\ref{StrictMonotonicityBalanceEquation}\textup, the function~$u \mapsto \FFr(u;\eta_1) + \Fl(u;\Ch([a_0,b_0],*);\eta_1)$ is strictly positive on the right of~$a_0$, because~$a_0$ is the root of the balance equation for the two indicated forces. Fix~$a^-$ in a right neighborhood of~$a_0$ to be a point such that
\begin{equation*}
\FFr(a^-;\eta_1) + \Fl(a^-;\Ch([a_0,b_0],*);\eta_1) > 0.
\end{equation*}
We choose~$\eta_2$, $\eta_1<\eta_2< \min(\eta_3,\eps_\infty)$, to be such that
\begin{equation*}
\FFr(a^-;\eps) + \Fl(a^-;\Ch([a_0,b_0],*);\eps) > 0 \quad \hbox{for all~$\eps \in [\eta_1,\eta_2]$},
\end{equation*}
which can be done since the forces are continuous.
By Definition~\ref{MFF} of the monotone force flow,
\begin{equation*}
\FFr(a_0;\eps) + \Fl(a_0;\Ch([a_0,b_0],*);\eps) < 0,
\end{equation*}
since~$\Fl(a_0;\Ch([a_0,b_0],*);\eps)$ does not change with~$\eps$.
Therefore, there exists a point~$w = w(\eps) \in [a_0,a^-]$ that solves the balance equation for~$\FFr(\cdot\,;\eps)$ and~$\Fl(\cdot\,;\Ch([a_0,b_0],*);\eps)$. It is easy to see that we may take~$w(\eps) \in (a_0,\tr(\eps))$, where~$\tr(\eps)$ is the right endpoint of the tail of~$\FFr(\cdot\,;\eps)$ (because the sum of forces is positive at the point~$\tr(\eps)$ if~$\tr(\eps) < a^-$).

So we take~$\mathfrak{l}(\eps)$ to be such that
\begin{equation*}
a(\mathfrak{l}(\eps)) = w(\eps).
\end{equation*}
We note that the function~$w$ is increasing, thus~$\mathfrak{l}$ is decreasing. For the existence of the desired function~$B_\eps$ we only need to verify that %The function~$B$ constructed according to the rules prescribed, is a Bellman candidate on
%\begin{equation*}
%\Rt\big(u_1,a(\tilde{l}(\eps))\big)\cup\RTroll\big(a(\tilde{l}(\eps)),b(\tilde{l}(\eps))\big)\cup\Ch\big([a(\tilde{l}(\eps)),b(\tilde{l}(\eps))],*\big).
%\end{equation*} 
%We only have to prove that
\begin{equation*}
\Fr(u;a(\mathfrak{l}(\eps)),b(\mathfrak{l}(\eps));\eps) < 0, \quad \eps \in [\eta_1,\eta_2], u \in [b(\mathfrak{l}(\eps)); u_2],
\end{equation*}
which follows from Corollary~\ref{FlowMonotonicityCorollary}, because~$\big\{\Ch([a_0,b_0],*),\mathfrak{l}\big\}$ is a decreasing flow of chordal domains.
\end{proof}

By Remark~\ref{FCDMFF}, the functions
\begin{equation*}
\begin{aligned}
&\FFr(u;\eps) = \Fr\Big(u;a(\mathfrak{l}(\eps)),b(\mathfrak{l}(\eps));\eps\Big),\quad \eps \in (\eta_1,\eta_2),\; u\in (b_0,+\infty);\\
&\FFl(u;\eps) = \Fl\Big(u;a(\mathfrak{l}(\eps)),b(\mathfrak{l}(\eps));\eps\Big),\quad \eps \in (\eta_1,\eta_2),\; u\in (-\infty, a_0)
\end{aligned}
\end{equation*}
are the right and left monotone force flows.

\begin{St}[{\bf Induction step for a left trolleybus}]\label{InductionStepForLeftTrolleybus}
Let~$\eta_1<\eps_{\infty}$. Suppose that~$u_1 \leq a_0 < b_0 < u_2$ and~$b_0 - a_0 \leq 2\eta_1$. Suppose that~$\FFl$ is a left monotone force flow on~$(-\infty,u_2)\times[\eta_1,\eta_3]$\textup, let~$b_0$ belong to the closure of the tail of~$\FFl(\cdot\,;\eta_1)$. Suppose that there exists a continuous function~$B$ that coincides with the standard candidates on~$\Lt(u_1,a_0;\eta_1)$\textup, $\LTroll(a_0,b_0;\eta_1)$\textup, $\Ch([a_0,b_0],*)$\textup, and $\Lt(b_0,u_2;\eta_1)$. Then there exists~$\eta_2$\textup, $\eta_1<\eta_2<\eps_\infty$\textup, and a decreasing continuous function~$\mathfrak{l}\colon [\eta_1,\eta_2] \to \mathbb{R}$ such that~$\mathfrak{l}(\eta_1) = b_0-a_0$\textup, and there exists a continuous function $B_\eps$ that coincides with the standard candidates on~$\Lt(u_1,a(\mathfrak{l}(\eps));\eps)$\textup, $\LTroll(a(\mathfrak{l}(\eps)),b(\mathfrak{l}(\eps));\eps)$\textup, $\Ch([a(\mathfrak{l}(\eps)),b(\mathfrak{l}(\eps))],*)$\textup, and $\Lt(b(\mathfrak{l}(\eps)),u_2;\eps)$ for any $\eps \in [\eta_1,\eta_2]$\textup, where $a$ and $b$ are the functions corresponding to the chordal domain~$\Ch([a_0,b_0],*)$.
%
%Let~$\eta_1$ be smaller than~$\eps_{\infty}$. Suppose that~$u_1 < a_0 < b_0 < u_2$ and~$b_0 - a_0 \leq 2\eta_1$. Suppose that~$\FFl$ is a left monotone force flow on~$(-\infty,u_2)\times (\eta_1,\eta_1_+)$\textup, let~$b_0$ belong to the closure of the tail of~$\FFl(\cdot\,;\eta_1)$. Suppose that the function~$B$ built with the help of formulas~\textup{\eqref{linearity}} and~\textup{\eqref{difeq2}} in~$\Lt(u_1,a_0)$ and~$\Lt(b_0,u_2)$ \textup(with~$\eta_1m''(u) = \Fl(u;a_0,b_0;\eta_1)$ for the first domain and~$\eta_1 m''(u) = \FFl(u;\eta_1)$ for the second\textup)\textup, by formula~\textup{\eqref{vallun}} in~$\Ch([a_0,b_0],*)$\textup, and by formulas~\eqref{LinearInTrolleybus} and~\eqref{coef} in~$\LTroll(a_0,b_0)$\textup, falls under the scope of Proposition~\textup{\ref{S16}} and satisfies the inequality~$m'' >0$ inside~$\Lt(u_1,a_0)$. Then there exists~$\eta_2 > \eta_1$ and a non-increasing continuous function~$\tilde{l}: [\eta_1,\eta_2] \to \mathbb{R}$ such that~$\tilde{l}(\eta_1) = b_0-a_0$ and the function~$B$ build with the help of the same formulas on the union domain
%\begin{equation*}
%\Lt\big(u_1,a(\tilde{l}(\eps))\big)\cup\LTroll\big(a(\tilde{l}(\eps)),b(\tilde{l}(\eps))\big)\cup\Ch\big([a(\tilde{l}(\eps)),b(\tilde{l}(\eps)],*\big)\cup\Lt\big(b(\tilde{l}(\eps)),u_2\big)
%\end{equation*}
%\textup(with~$\eps m''(u) = \FFl(u;\eps)$ for~$\Lt(b(\tilde{l}(\eps)),u_2)$\textup) satisfies the assumptions of Proposition~\textup{\ref{S16}} and the inequality~$m'' > 0$ inside both tangent domains.
\end{St}
\begin{Rem}\label{UniqueTrolley}
It follows from Lemma~\textup{\ref{monbaleq}} that the functions~$\mathfrak{l}$ constructed in Propositions~\textup{\ref{InductionStepForRightTrolleybus}} and~\textup{\ref{InductionStepForLeftTrolleybus}} are unique \textup(at least when~$\eta_2 - \eta_1$ is sufficiently small\textup).
\end{Rem}
The next two propositions describe the evolutional behavior of multitrolleybuses. It appears that each multitrolleybus immediately splits into a trolleybus parade (by formulas~\eqref{RMultitrolleybusDesintegration} and~\eqref{LMultitrolleybusDesintegration}), and each of the trolleybuses decreases.

\begin{St}[{\bf Induction step for a right multitrolleybus}]\label{InductionStepForMultitrolleybus}\index{multitrolleybus}
Let~$\eta_1$ be smaller than~$\eps_{\infty}$. Consider a right multitrolleybus~$\MTTR(\{\mathfrak{a}_i\}_{i=1}^k)$. Let~$u_1 < \mathfrak{a}_1^{\mathrm{l}}$ and let~$\mathfrak{a}_{k}^{\mathrm{r}} \leq u_2$. Let~$\FFr$ be a right monotone force flow on~$(u_1,\infty)\times[\eta_1,\eta_3]$\textup, let~$\mathfrak{a}_1^{\mathrm{l}}$ belong to the closure of the tail of~$\FFr(\cdot\,;\eta_1)$. We also suppose that for each~$i = 1,2,\ldots,k-1$ there are chordal domains~$\Ch([\mathfrak{a}_i^{\mathrm{r}},\mathfrak{a}_{i+1}^{\mathrm{l}}],*)$ with the corresponding functions~$a_i$ and~$b_i$ that satisfy the assumptions of Proposition~\textup{\ref{LightChordalDomainCandidate}}. Suppose that there exists a continuous function~$B$ that coincides with the standard candidates on~$\Rt(u_1,\mathfrak{a}_1^{\mathrm{l}};\eta_1)$\textup, $\Rt(\mathfrak{a}_k^{\mathrm{r}},u_2;\eta_1)$\textup, $\MTTR(\{\mathfrak{a}_i\}_{i=1}^k;\eta_1)$\textup, and every~$\Ch([\mathfrak{a}_i^{\mathrm{r}},\mathfrak{a}_{i+1}^{\mathrm{l}}],*)$. Then\textup, there exists a number~$\eta_2 > \eta_1$ and a collection of continuous decreasing functions~$\mathfrak{l}_i\colon[\eta_1,\eta_2] \to \mathbb{R}$\textup,~$\mathfrak{l}_i(\eta_1) = \mathfrak{a}_{i+1}^{\mathrm{l}} - \mathfrak{a}_i^{\mathrm{r}}$\textup,~$i = 1,2,\ldots,k-1$\textup, such that for every~$\eps \in [\eta_1,\eta_2]$ there exists a continuous function~$B_{\eps}$ defined on the domain
\begin{equation*}
\begin{aligned}
\Rt\big(u_1,a_1(\mathfrak{l}_1(\eps));\eps\big) \cup \Big(\cup_{i=1}^{k-2} \Rt\big(b_i(\mathfrak{l}_i(\eps)),a_{i+1}(\mathfrak{l}_{i+1}(\eps));\eps\big)\Big) \cup \Big(\cup_{i=1}^{k-1}\RTroll(a_i(\mathfrak{l}_i(\eps)),b_{i}(\mathfrak{l}_i(\eps));\eps)\Big) \cup \\ \Big(\cup_{i=1}^{k-1}\Ch([a_i(\mathfrak{l}_i(\eps)),b_{i}(\mathfrak{l}_i(\eps))],*)\Big)\cup \Rt\big(b_{k-1}(\mathfrak{l}_{k-1}(\eps)),u_2;\eps\big)
\end{aligned}
\end{equation*}
that coincides with the standard candidate inside each subdomain of the partition.
\end{St}

In other words, there exist decreasing flows~$\big\{\Ch([\mathfrak{a}^{\mathrm r}_i,\mathfrak{a}^{\mathrm l}_{i+1}],*),\mathfrak{l}_i\big\}$ such that the~$i$th flow and the~$(i+1)$th are balanced at~$a_{i+1}(\mathfrak{l}_{i+1}(\eps))$ at the moment~$\eps$ (and also this point belongs to the right tail of the~$i$th flow).
We note that the function
\begin{equation*}
%\begin{aligned}
\FFr(u;\eps) = \Fr\Big(u;a_{k-1}(\mathfrak{l}_{k-1}(\eps)),b_{k-1}(\mathfrak{l}_{k-1}(\eps));\eps\Big),\quad \eps \in [\eta_1,\eta_2),\; u\in (\mathfrak{a}^{\mathrm l}_{k},+\infty)\\
%\FFl(u;\eps) = \Fl\Big(u;a(\mathfrak{l}(\eps)),b(\mathfrak{l}(\eps));\eps\Big),\quad \eps \in (\eta_1,\eta_2),\; u\in (-\infty, a_0)
%\end{aligned}
\end{equation*}
is a right monotone force flow, and $\FFr(u;\eta_1)=\Fr(u,\mathfrak{a}^{\mathrm l}_{1}, \mathfrak{a}^{\mathrm r}_{k},\eta_1)$.
\begin{proof}
To prove this proposition, one decomposes~$\MTTR(\{\mathfrak{a}_i\}_{i=1}^k)$ into a union of~$k-1$ right trolleybuses with the help of formula~\eqref{RMultitrolleybusDesintegration}. More or less, we will follow the lines of the proof of Proposition~\ref{InductionStepForRightTrolleybus} applying it with small modifications to each of these trolleybuses (note that we can apply it directly only to the leftmost trolleybus).

We use Lemma~\ref{StrictMonotonicityBalanceEquation} if~$\mathfrak{a}_1$ is a single point or Remark~\ref{SolidRootRemark} if it is an interval to pick a point~$a_{1}^-$ in a right neighborhood of~$\mathfrak{a}_1^{\mathrm r}$ such that
\begin{equation}\label{MTTRProofe1}
\FFr(a_{1}^-;\eps) + \Fl(a_{1}^-;\Ch([\mathfrak{a}_1^{\mathrm r},\mathfrak{a}_2^{\mathrm l}],*);\eps) > 0\qquad \hbox{for}\ \eps = \eta_1.
\end{equation}
Using the same Lemma or Remark for each~$i = 2,3,\ldots,k-1$, we pick a point~$a_{i}^-$ in a right neighborhood of~$\mathfrak{a}_i^{\mathrm r}$ such that
\begin{equation}\label{MTTRProofe2}
\Fr(a_i^-;\Ch([\mathfrak{a}_{i-1}^{\mathrm r},\mathfrak{a}_{i}^{\mathrm l}],*);\eps) + \Fl(a_i^-;\Ch([\mathfrak{a}_i^{\mathrm r},\mathfrak{a}_{i+1}^{\mathrm l}],*);\eps) > 0 \qquad \hbox{for}\ \eps = \eta_1.
\end{equation}
Now we take~$\eta_2-\eta_1$ to be so small that inequalities~\eqref{MTTRProofe1} and~\eqref{MTTRProofe2} (for all~$i=2,3,\ldots,k-1$) hold for any~$\eps \in [\eta_1,\eta_2]$. By Definition~\ref{MFF}, the tail of~$\FFr(\cdot;\eps)$ grows as a function of~$\eps$. In particular, it contains~$\mathfrak{a}_1^{\mathrm r}$ for any~$\eps > \eta_1$ (because~$f''' = 0$ on~$\mathfrak{a}_1$ if it is a solid root). On the other hand, by the same monotonicity property of the force flow,
\begin{equation*}
\FFr(\mathfrak{a}_1^{\mathrm r};\eps) + \Fl(\mathfrak{a}_1^{\mathrm r};\Ch([\mathfrak{a}_1^{\mathrm r},\mathfrak{a}_2^{\mathrm l}],*);\eps) < 0.
\end{equation*}
Thus, we may find a point~$w_1(\eps)$ that is the root of the balance equation for~$\FFr(\cdot\,;\eps)$ and~$\Fl(\cdot\,;\Ch([\mathfrak{a}_1^{\mathrm r},\mathfrak{a}_2^{\mathrm l}],*);\eps)$. Arguing as in the proof of Proposition~\ref{InductionStepForRightTrolleybus}, we may choose this point in such a way that it belongs to the right tail of~$\FFr(\cdot\,;\eps)$. So, we put~$a_1(\mathfrak{l}_1(\eps)) = w(\eps)$ and build the first trolleybus (note that the function~$\mathfrak{l}_1$ is decreasing).

We note that~$\big\{\Ch([\mathfrak{a}_1^{\mathrm r},\mathfrak{a}_2^{\mathrm l}]),\mathfrak{l}_1\big\}$ is a decreasing flow of chordal domains. Thus, by Corollary~\ref{FlowMonotonicityCorollary},
\begin{equation*}
\Fr(\mathfrak{a}_2^{\mathrm r};\Ch([a_1(\mathfrak{l}_1(\eps)),b_1(\mathfrak{l}_1(\eps))],*);\eps) + \Fl(\mathfrak{a}_2^{\mathrm r};\Ch([\mathfrak{a}_2^{\mathrm r},\mathfrak{a}_{3}^{\mathrm l}],*);\eps) < 0, \qquad \eps \in (\eta_1,\eta_2].
\end{equation*} 
We take~$\eta_2 - \eta_1$ to be so small (we diminish this quantity once more if needed) that inequality~\eqref{MTTRProofe2} for~$i=2$ holds true  with~$\Ch([\mathfrak{a}_{1}^{\mathrm r},\mathfrak{a}_{2}^{\mathrm l}],*)$ replaced by~$\Ch([a_1(\mathfrak{l}_1(\eps)),b_1(\mathfrak{l}_1(\eps))],*)$, i.e.
\begin{equation*}
\Fr\Big(a_2^-;\Ch\big([a_1(\mathfrak{l}_1(\eps)),b_1(\mathfrak{l}_1(\eps))],*\big);\eps\Big) + \Fl(a_2^-;\Ch([\mathfrak{a}_2^{\mathrm r},\mathfrak{a}_{3}^{\mathrm l}],*);\eps) > 0,\qquad \eps \in [\eta_1,\eta_2].
\end{equation*}
In such a case, we can find the root~$w_2(\eps)$ of the balance equation
\begin{equation*}
\Fr(\cdot\,;\Ch([a_1(\mathfrak{l}_1(\eps)),b_1(\mathfrak{l}_1(\eps))],*);\eps) + \Fl(\cdot\,;\Ch([\mathfrak{a}_2^{\mathrm r},\mathfrak{a}_{3}^{\mathrm l}],*);\eps) = 0
\end{equation*}
that belongs both to the right tail of~$[A_1(\mathfrak{l}_1(\eps)),B_1(\mathfrak{l}_1(\eps))]$ and the interval~$(\mathfrak{a}_2^{\mathrm r}, a_2^-)$. We choose the function~$\mathfrak{l}_2$ in such a way that~$w_2(\eps) = a_2(\mathfrak{l}_2(\eps))$. Reasoning consecutively in a similar fashion, we construct all the remaining trolleybuses.
\end{proof}

\begin{St}[{\bf Induction step for a left multitrolleybus}]\label{InductionStepForLeftMultitrolleybus}
Let~$\eta_1$ be smaller than~$\eps_{\infty}$. Consider the left multitrolleybus~$\MTTL(\{\mathfrak{a}_i\}_{i=1}^k)$. Let~$u_1 \leq \mathfrak{a}_1^{\mathrm{l}}$ and let~$\mathfrak{a}_{k}^{\mathrm{r}} < u_2$. Let~$\FFl$ be a left monotone force flow on~$(-\infty,u_2)\times[\eta_1,\eta_3]$\textup, let~$\mathfrak{a}_k^{\mathrm{r}}$ belong to the closure of the tail of~$\FFl(\cdot\,;\eta_1)$. We also suppose that for each~$i = 1,2,\ldots,k-1$ there are chordal domains~$\Ch([\mathfrak{a}_i^{\mathrm{r}},\mathfrak{a}_{i+1}^{\mathrm{l}}],*)$ with the corresponding functions~$a_i$ and~$b_i$ that satisfy the assumptions of Proposition~\textup{\ref{LightChordalDomainCandidate}}. Suppose that there exists a continuous function~$B$ that coincides with the standard candidates on~$\Lt(u_1,\mathfrak{a}_1^{\mathrm{l}};\eta_1)$\textup, $\Lt(\mathfrak{a}_k^{\mathrm{r}},u_2;\eta_1)$\textup, $\MTTL(\{\mathfrak{a}_i\}_{i=1}^k;\eta_1)$\textup, and every~$\Ch([\mathfrak{a}_i^{\mathrm{r}},\mathfrak{a}_{i+1}^{\mathrm{l}}],*)$. Then\textup, there exists a number~$\eta_2 > \eta_1$ and a collection of continuous decreasing functions~$\mathfrak{l}_i\colon[\eta_1,\eta_2] \to \mathbb{R}$\textup,~$\mathfrak{l}_i(\eta_1) = \mathfrak{a}_{i+1}^{\mathrm{l}} - \mathfrak{a}_i^{\mathrm{r}}$\textup,~$i = 1,2,\ldots,k-1$\textup, such that for every~$\eps \in [\eta_1,\eta_2]$ there exists a continuous function~$B_{\eps}$ defined on the domain
\begin{equation*}
\begin{aligned}
\Lt\big(u_1,a_1(\mathfrak{l}_1(\eps));\eps\big) \cup \Big(\cup_{i=1}^{k-2} \Lt\big(b_i(\mathfrak{l}_i(\eps)),a_{i+1}(\mathfrak{l}_{i+1}(\eps));\eps\big)\Big) \cup \Big(\cup_{i=1}^{k-1}\LTroll(a_i(\mathfrak{l}_i(\eps)),b_{i}(\mathfrak{l}_i(\eps));\eps)\Big) \cup \\ \Big(\cup_{i=1}^{k-1}\Ch([a_i(\mathfrak{l}_i(\eps)),b_{i}(\mathfrak{l}_i(\eps))],*)\Big)\cup \Lt\big(b_{k-1}(\mathfrak{l}_{k-1}(\eps)),u_2;\eps\big)
\end{aligned}
\end{equation*}
coinciding with the standard candidate on each subdomain of the partition.
\end{St}
\begin{Rem}\label{DegenerateTrolleybus}
We note that in the case where the multitrolleybus~$\MTTL(\{\mathfrak{a}\})$ is sitting on a single arc~$\mathfrak{a}$ with the incoming left monotone force flow~$\FFl$ on~$(-\infty,u_2)\times[\eta_1,\eta_3]$\textup, for any~$\eps \in (\eta_1,\eta_3]$\textup, we have the left tangent domain instead of this multitrolleybus with the force function given by~$\FFl$. %same force flow onthe whole construction will turn into a single tangent domain on the interval~$(u_1,u_2)$. We only say that the \textup`\textup`outcoming\textup'\textup' force flow here should be defined by a different formula. Namely\textup, let our multitrolleybus be~$\MTTL(\{\mathfrak{a}\})$, then this flow is defined as the restriction of~$\FFl$ to~$(-\infty,\mathfrak{a}_1^{\mathrm l})\times(\eta_1,\eta_2)$. 
Similar for the case of a right multitrolleybus.
 
In particular\textup, one can do the same for the case of a fictious vertex~$\Rt(w,w)$ or~$\Lt(w,w)$ of the fifth type.
\end{Rem}
In the next proposition, we show that angles move continuously.

\begin{St}[{\bf Induction step for an angle}]\label{InductionStepAngle}\index{angle}
Let~$\eta_1 < \eps_{\infty}$. Let~$w \in\mathbb{R}$\textup, consider the figure~$\Ang(w;\eta_1)$ with the domains~$\Rt(u_1,w;\eta_1)$ and~$\Lt(w,u_2;\eta_1)$ attached to it. Let also~$\FFr$ be a right monotone force flow on~$(u_1,\infty)\times[\eta_1,\eta_3]$ and let~$\FFl$ be a left monotone force flow on~$(-\infty,u_2)\times [\eta_1,\eta_3]$ such that~$[u_1,w)$ belongs to the tail of~$\FFr(\cdot\,;\eta_1)$ and~$(w,u_2]$ belongs to the tail of~$\FFl(\cdot\,;\eta_1)$. Then there exists~$\eta_2 > \eta_1$ and a continuous function~$w\colon [\eta_1,\eta_2] \to \mathbb{R}$ such that~$w(\eps)$ is the root of the balance equation for~$\FFr(\cdot\,;\eps)$ and~$\FFl(\cdot\,;\eps)$ and~$w(\eps)$ belongs to the intersection of the tails of these forces \textup(in particular\textup, the situation falls under the scope of Proposition~\textup{\ref{AngleProp}}\textup).
\end{St}

\begin{proof}
By Lemma~\ref{StrictMonotonicityBalanceEquation}, the function~$u \mapsto \FFr(u;\eta_1) + \FFl(u;\eta_1)$ is strictly negative on the left of~$w$ and strictly positive on the right. Thus, its small perturbation~$\FFr(u;\eps) + \FFl(u;\eps)$ always has a root in a small neighborhood of~$w$. It only remains to prove that this root can be found inside the intersection of the tails. If~$\tr(\eta_1)$ (the end of the tail of~$\FFr(\cdot\,;\eta_1)$) and~$\tl(\eta_1)$ are not equal to~$w$, then everything is fine provided~$\eps$ is not far from~$\eta_1$. Consider the case where~$\tr(\eta_1)=w(\eta_1)$. When we increase~$\eps$ a little starting from~$\eta_1$,~$\tr$ strictly grows, thus~$\tr(\eps) > w$. The function~$u \mapsto \FFr(u;\eps) + \FFl(u;\eps)$ is positive at~$\tr(\eps)$ (provided~$\tr(\eps)$ belongs to the domain of~$\FFl(\cdot\,;\eps)$; if not, then there is nothing to prove) and negative on the left of~$w$. Therefore,~$w(\eps)$ can be chosen to be smaller than~$\tr(\eps)$. A similar reasoning shows that~$w(\eps)$ may be chosen to belong to~$(\tl(\eps),\tr(\eps))$.
\end{proof}

\begin{St}[{\bf Induction step for a multibirdie}]\label{MultibirdieDesintegrationSt}\index{birdie}\index{multibirdie}
Let~$\eta_1$ be a positive number\textup,~$\eta_1 < \eps_{\infty}$. Let~$k$ be a natural number\textup, let~$\{\mathfrak{a}_i\}_{i=1}^{k}$ be a collection of disjoint intervals on the real line\textup, let~$u_1 < \mathfrak{a}_1^{\mathrm{l}}$\textup,~$u_2 > \mathfrak{a}_k^{\mathrm{r}}$. Let also~$\FFr$ be a right monotone force flow defined on~$(u_1,\infty)\times[\eta_1,\eta_3]$ and let~$\FFl$ be a left monotone force flow defined on~$(-\infty,u_2)\times [\eta_1,\eta_3]$. Let~$\mathfrak{a}_1^{\mathrm{l}}$ belong to the closure of the tail of~$\FFr(\cdot\,;\eta_1)$\textup, let~$\mathfrak{a}_k^{\mathrm{r}}$ belong to the closure of the tail of~$\FFl(\cdot\,;\eta_1)$. Consider the union domain
\begin{equation}\label{BeforeDesintegration}
\begin{aligned}
\Rt(u_1,\mathfrak{a}_1^{\mathrm{l}};\eta_1) \cup \MTB(\{\mathfrak{a}_i\}_{i=1}^k;\eta_1) \cup \Big(\cup_{i=1}^{k-1} \Ch([\mathfrak{a}_i^{\mathrm{r}},\mathfrak{a}_{i+1}^{\mathrm{l}}],*)\Big) \cup \Lt(\mathfrak{a}_k^{\mathrm{r}},u_2;\eta_1)
\end{aligned}
\end{equation}
and a continuous function~$B$ on this domain that coincides with the standard candidates on each subdomain. Then\textup, there exists a number~$\eta_2 > \eta_1$ and a collection of decreasing functions~$\mathfrak{l}_i\colon[\eta_1,\eta_2]\to \mathbb{R}$\textup,~$\mathfrak{l}_i(\eta_1) =  \mathfrak{a}_{i+1}^{\mathrm{l}} - \mathfrak{a}_i^{\mathrm{r}}$\textup,~$i = 1,2,\ldots,k-1$\textup, such that for every~$\eps \in [\eta_1,\eta_2]$ there exists an integer~$j = j(\eps)$\textup, $1 \leq j \leq k$\textup, and a real~$w(\eps)$\textup,~$b_{j-1}\big(\mathfrak{l}_{j-1}(\eps)\big) \leq w(\eps) \leq a_{j}\big(\mathfrak{l}_j(\eps)\big)$ \textup(here~$a_i$ and~$b_i$ are the functions~$a$ and~$b$ associated with~$\Ch([\mathfrak{a}_i^{\mathrm{r}},\mathfrak{a}_{i+1}^{\mathrm{l}}],*)$ for $1\leq i \leq k-1$\textup, see Subsection~\textup{\ref{s331}}\textup; with the additional agreement\textup: $b_0\df u_1$,~$a_k\df u_2$\textup) such that there exists a continuous function~$B_{\eps}$ on the domain
\begin{gather*}
%\Rt\big(u_1,a_1(\mathfrak{l}_1(\eps));\eps\big) \cup 
\Big(\cup_{i=0}^{j-2}\Rt\big(b_i\big(\mathfrak{l}_i(\eps)\big), a_{i+1}\big(\mathfrak{l}_{i+1}(\eps)\big);\eps\big)\Big) \cup\Big(\cup_{i=1}^{j-1}\RTroll\big(a_i\big(\mathfrak{l}_i(\eps)\big),b_i\big(\mathfrak{l}_i(\eps)\big);\eps\big)\Big) \cup
\\ 
\Rt\big(b_{j-1}\big(\mathfrak{l}_{j-1}(\eps)\big),w(\eps);\eps\big) \cup
\Ang\big(w(\eps);\eps\big) \cup \Lt\big(w(\eps),a_{j}\big(\mathfrak{l}_{j}(\eps)\big);\eps\big)\cup
\\
\Big(\cup_{i=j}^{k-1}\LTroll\big(a_i\big(\mathfrak{l}_i(\eps)\big),b_i\big(\mathfrak{l}_i(\eps)\big);\eps\big)\Big) \cup
\Big(\cup_{i=j}^{k-1}\Lt\big(b_i\big(\mathfrak{l}_i(\eps)\big), a_{i+1}\big(\mathfrak{l}_{i+1}(\eps)\big);\eps\big)\Big) \cup\\
\Big(\cup_{i=1}^{k-1} \Ch\big(\big[a_i\big(\mathfrak{l}_i(\eps)\big),b_i\big(\mathfrak{l}_i(\eps)\big)\big],*\big)\Big)
%\cup \Lt\big(b_{k-1}(\mathfrak{l}_{k-1}(\eps)),u_2;\eps\big)
\end{gather*}
%for~$j \ne 1,k$\textup, on the domain
%\begin{equation*}
%\begin{aligned}
%\Rt\big(u_1,w(\eps);\eps\big)\cup\Ang\big(w(\eps);\eps\big)\cup\Lt\big(w(\eps),a_1(\mathfrak{l}_1(\eps));\eps\big) \cup \Big(\cup_{i=1}^{k-2} \Lt\big(b_i(\mathfrak{l}_i(\eps)),a_{i+1}(\mathfrak{l}_{i+1}(\eps));\eps\big)\Big) \cup
%\\
%\Big(\cup_{i=1}^{k-1}\LTroll\big(a_i(\mathfrak{l}_i(\eps)),b_{i}(\mathfrak{l}_i(\eps));\eps\big)\Big) \cup  \Big(\cup_{i=1}^{k-1}\Ch([a_i(\mathfrak{l}_i(\eps)),b_{i}(\mathfrak{l}_i(\eps))],*)\Big)\cup \Lt\big(b_{k-1}(\mathfrak{l}_{k-1}(\eps)),u_2;\eps\big)
%\end{aligned}
%\end{equation*}
%for~$j=1$\textup, and on the domain
%\begin{equation*}
%\begin{aligned}
%\Rt\big(u_1,a_1(\mathfrak{l}_1(\eps));\eps\big) \cup \Big(\cup_{i=1}^{k-2} \Rt\big(b_i(\mathfrak{l}_i(\eps)),a_{i+1}(\mathfrak{l}_{i+1}(\eps));\eps\big)\Big) \cup \Big(\cup_{i=1}^{k-1}\RTroll\big(a_i(\mathfrak{l}_i(\eps)),b_{i}(\mathfrak{l}_i(\eps));\eps\big)\Big) \cup 
%\\
%\Big(\cup_{i=1}^{k-1}\Ch([a_i(\mathfrak{l}_i(\eps)),b_{i}(\mathfrak{l}_i(\eps))],*)\Big)\cup \Rt\big(b_{k-1}(\mathfrak{l}_{k-1}(\eps)),w(\eps);\eps\big) \cup \Ang\big(w(\eps)\big)\cup\Lt\big(w(\eps),u_2;\eps\big)
%\end{aligned}
%\end{equation*}
%for~$j=k$, 
that coincides with the standard candidate inside each subdomain of the partition. 
\end{St}

Before turning to the proof, we make several comments about Proposition~\ref{MultibirdieDesintegrationSt}. It is instructive to consider the case where~$k=2$ and~$\mathfrak{a}_1$ and~$\mathfrak{a}_2$ are points. Then, the initial domain~\eqref{BeforeDesintegration} reduces to
\begin{equation*}
\Rt(u_1,\mathfrak{a}_1) \cup \Bird(\mathfrak{a}_1,\mathfrak{a}_2) \cup \Ch([\mathfrak{a}_1,\mathfrak{a}_2],*) \cup \Lt(\mathfrak{a}_2,u_2).
\end{equation*}
In this case, Proposition~\ref{MultibirdieDesintegrationSt} says the following: first, the base of the birdie shrinks, second, the birdie itself desintegrates into a trolleybus and an angle (as prescribed by formulas~\eqref{RTrolleybusPlusAngle} and~\eqref{LTrolleybusPlusAngle}). Unfortunately, it is very hard to decide which of the two formulas~\eqref{RTrolleybusPlusAngle} or~\eqref{LTrolleybusPlusAngle} should the birdie choose. Similarly, the gist of the general case is formula~\eqref{MultibirdieDesintegration} (together with formula~\eqref{FirstFormula}), see Figure~\ref{fig:scenarios}. It should be noted that, though formally the case of a usual birdie is a subcase of the general situation, there is an effect that is present only there. Namely, a birdie can avoid desintegration and shrink by itself (e.g. see Section~\ref{s45} below), whereas a multibirdie of higher complexity dies. 
\begin{figure}%[16pt]{o}{300pt}
\hbox{
\hskip-40pt\includegraphics[width = 0.65 \linewidth]{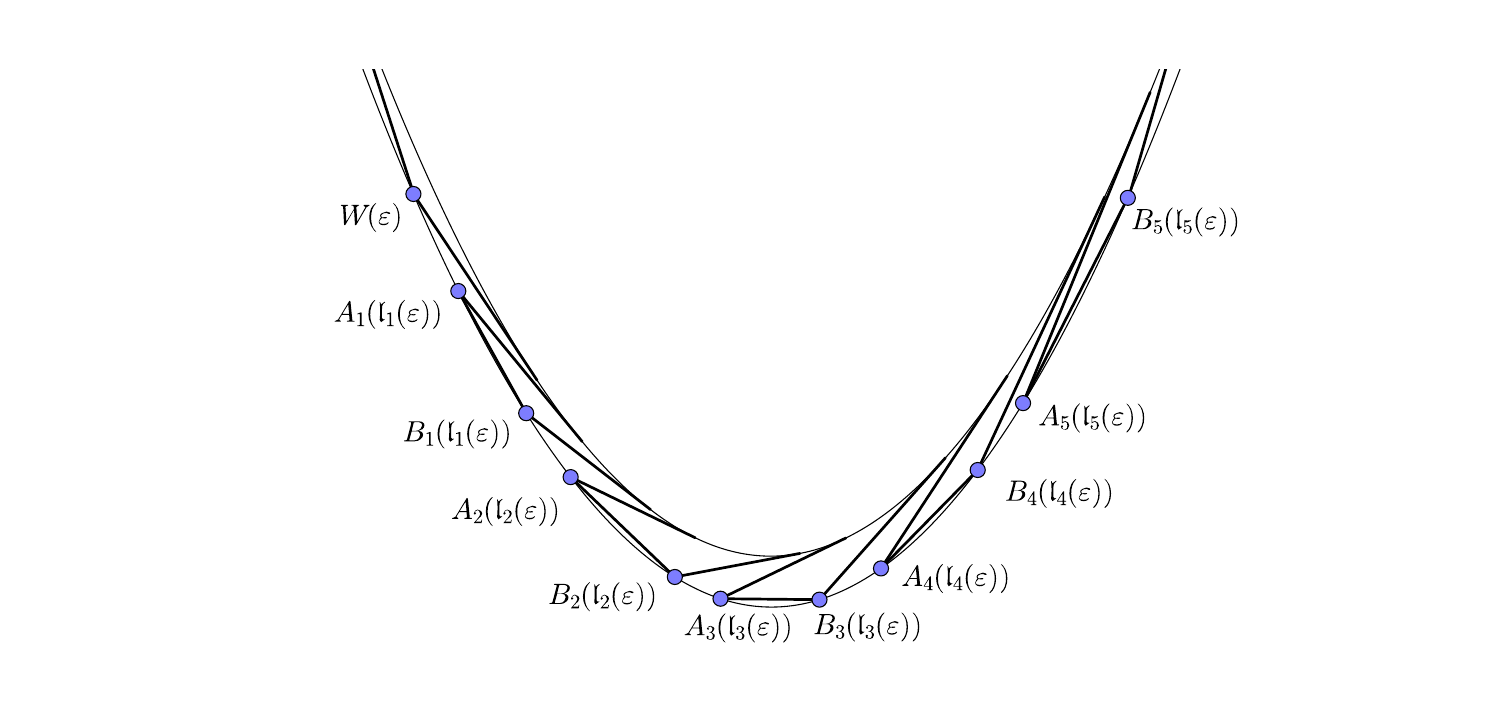}
\hskip-70pt
\includegraphics[width = 0.65 \linewidth]{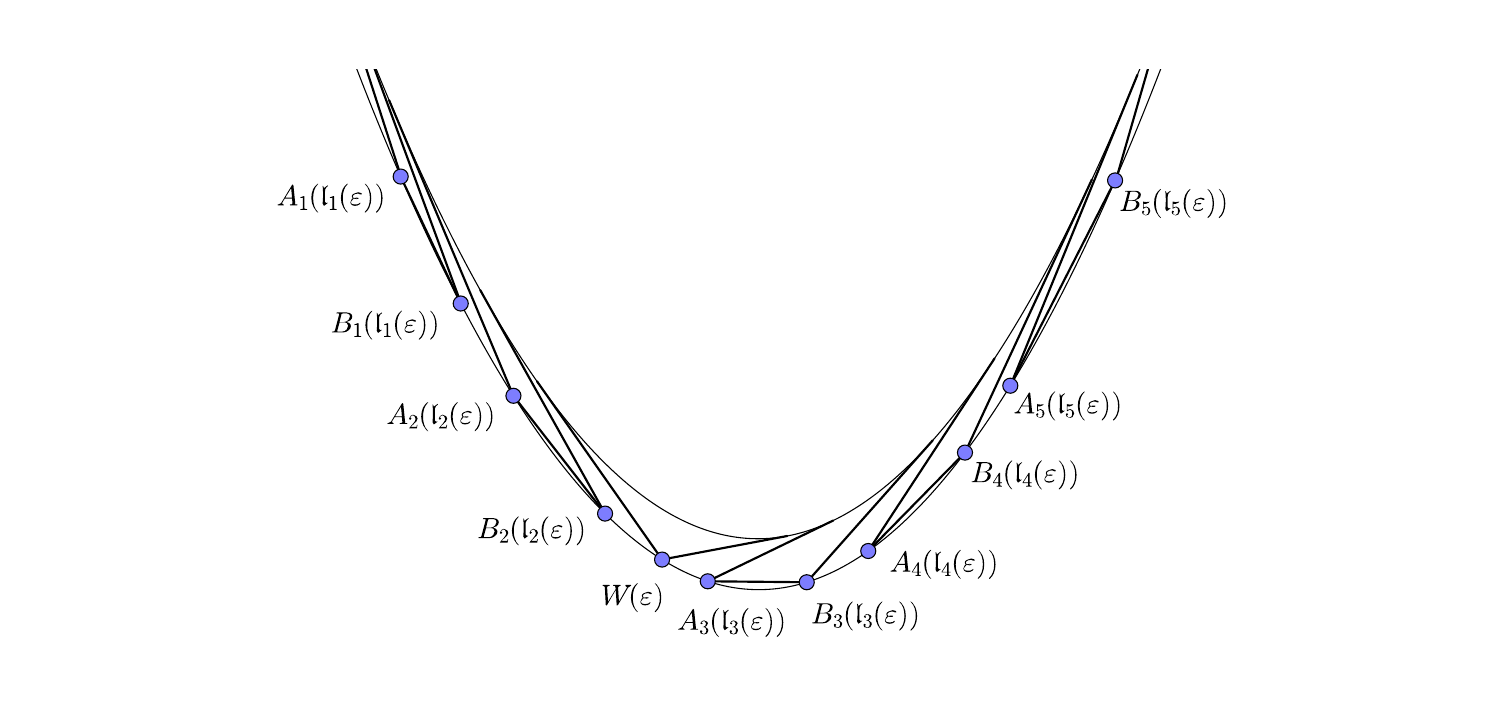}
}
\vskip-10pt
\caption{The cases~$k=5, j=1$ and~$k=5, j = 3$.}
\label{fig:scenarios}
\end{figure}

The proof of Proposition~\ref{MultibirdieDesintegrationSt} is more complicated than the preceding proofs in this section. We note that this complexity comes from the problem itself, see the second example in Subsection~\ref{s54} for the explanation. To make the reasoning more transparent, we first treat the easier case of a birdie. 

\paragraph{Proof of Proposition~\ref{MultibirdieDesintegrationSt} in the case~$k=2$, $\mathfrak{a}_1$ and~$\mathfrak{a}_2$ are points.} Consider the foliation
\begin{equation*}
\Rt(u_1,\mathfrak{a}_1) \cup \RTroll(\mathfrak{a}_1,\mathfrak{a}_2) \cup \Ch([\mathfrak{a}_1,\mathfrak{a}_2],*).
\end{equation*}
By Proposition~\ref{InductionStepForRightTrolleybus}, there exist a number~$\eta_2^{\mathrm{r}} > \eta_1$ and a continuous decreasing function~$\mathfrak{l}_{\mathrm{r}}\colon[\eta_1,\eta_2^{\mathrm{r}}]\to\mathbb{R}$ such that the foliation
\begin{equation*}
\Rt\Big(u_1,a(\mathfrak{l}_{\mathrm{r}}(\eps));\eps\Big) \cup \RTroll\Big(a(\mathfrak{l}_{\mathrm{r}}(\eps)),b(\mathfrak{l}_{\mathrm{r}}(\eps));\eps\Big) \cup \Ch\Big([a(\mathfrak{l}_{\mathrm{r}}(\eps)),b(\mathfrak{l}_{\mathrm{r}}(\eps))],*\Big) \cup \Rt\Big(b(\mathfrak{l}_{\mathrm{r}}(\eps)),\mathfrak{a}_2;\eps\Big)
\end{equation*}
and the continuous function that coincides with the standard candidate on each subdomain satisfy the assumptions of Proposition~\ref{S15}. Similarly, by Proposition~\ref{InductionStepForLeftTrolleybus}, there exists a number~$\eta_2^{\mathrm{l}} > \eta_1$ and a continuous decreasing function~$\mathfrak{l}_{\mathrm{l}}\colon[\eta_1,\eta_2^{\mathrm{l}}]\to\mathbb{R}$ such that the foliation
\begin{equation*}
\Lt\Big(\mathfrak{a}_1,a(\mathfrak{l}_{\mathrm{l}}(\eps));\eps\Big) \cup \LTroll\Big(a(\mathfrak{l}_{\mathrm{l}}(\eps)),b(\mathfrak{l}_{\mathrm{l}}(\eps));\eps\Big) \cup \Ch\Big([a(\mathfrak{l}_{\mathrm{l}}(\eps)),b(\mathfrak{l}_{\mathrm{l}}(\eps))],*\Big) \cup \Lt\Big(b(\mathfrak{l}_{\mathrm{l}}(\eps)),u_2;\eps\Big)
\end{equation*}
and the continuous function that coincides with the standard candidate on each subdomain satisfy the assumptions of Proposition~\ref{S16}. Choose~$\eta_2 = \min(\eta_2^{\mathrm{r}},\eta_2^{\mathrm{l}})$ and~$\mathfrak{l}(\eps) = \min(l_{\mathrm{r}}(\eps),l_{\mathrm{l}}(\eps))$,~$\eps \in [\eta_1,\eta_2]$. We prove that this choice of the function~$\mathfrak{l}$ allows us to construct the desired foliation. Fix~$\eps$ and assume that~$\mathfrak{l}(\eps) = \mathfrak{l}_{\mathrm{r}}(\eps)$ (the remaining case is symmetric). In such a case, we can build the right trolleybus, and we only have to paste the angle between~$b(\mathfrak{l}(\eps))$ and~$u_2$. To do this, consider the forces~$\Fr(\cdot\,;a(\mathfrak{l}(\eps)),b(\mathfrak{l}(\eps));\eps)$ and~$\FFl$ (the force flow coming from the right), we are interested in the balance equation for them. First, by Lemma~\ref{StrictMonotonicityBalanceEquation} and continuity of forces, the solution~$w = w(\eps)$ of the balance equation is a continuous function of~$\eps$ (in particular it is not greater then~$u_2$ provided~$\eps - \eta_1$ is sufficiently small). Second, by Lemma~\ref{biggercupsmallerforce},
\begin{equation*}
\Fr\Big(b(\mathfrak{l}_{\mathrm{l}}(\eps));a(\mathfrak{l}(\eps)),b(\mathfrak{l}(\eps));\eps\Big) + \FFl\Big(b(\mathfrak{l}_{\mathrm{l}}(\eps));\eps\Big) \leq \Fr\Big(b(\mathfrak{l}_{\mathrm{l}}(\eps));a(\mathfrak{l}_{\mathrm{l}}(\eps)),b(\mathfrak{l}_{\mathrm{l}}(\eps));\eps\Big) + \FFl\Big(b(\mathfrak{l}_{\mathrm{l}}(\eps));\eps\Big) = 0. 
\end{equation*} 
This shows that~$w(\eps) \geq b(\mathfrak{l}_{\mathrm l}(\eps)) \geq b(\mathfrak{l}(\eps))$. It is easy to see that~$w(\eps)$ can be chosen in such a way that it belongs to the right tail of~$\Ch([a(\mathfrak{l}(\eps)),b(\mathfrak{l}(\eps))],*)$ (we performed such a trick in the proof of Proposition~\ref{InductionStepForRightTrolleybus}). \qed

\begin{Rem}\label{UniqueBirdie}
The function~$\mathfrak{l}$ constructed in the proof is unique. Indeed\textup, by Remark~\textup{\ref{UniqueTrolley}}\textup, there are only two possibilities\textup: one for the case of a right trolleybus and one for a left\textup, and if we choose the trolleybus with the bigger base\textup, we are not able to paste the angle \textup(the balance point is inside the base of the trolleybus\textup).  
\end{Rem}
\paragraph{Proof of Proposition~\ref{MultibirdieDesintegrationSt}.} We consider two multitrolleybuses~$\MTTR(\{\mathfrak{a}_i\}_{i=1}^k)$ and~$\MTTL(\{\mathfrak{a}_i\}_{i=1}^k)$. Application of Propositions~\ref{InductionStepForMultitrolleybus} and~\ref{InductionStepForLeftMultitrolleybus} to these foliations gives us two collections of functions~$\{\mathfrak{l}_j\}_{j=1}^{k-1}$ (which act on an interval~$[\eta_1,\eta_2]$), which we call~$\{\mathfrak{l}_{\mathrm{r},j}\}$ (for~$\MTTR$) and~$\{\mathfrak{l}_{\mathrm{l},j}\}$ correspondingly. Let~$\eps \in (\eta_1,\eta_2]$ be fixed.

We claim that if~$\mathfrak{l}_{\mathrm{r},j} \leq \mathfrak{l}_{\mathrm{l},j}$,~$j > 1$, then~$\mathfrak{l}_{\mathrm{r},j-1} < \mathfrak{l}_{\mathrm{l},j-1}$. Indeed, by Lemma~\ref{biggercupsmallerforce},
\begin{equation*}
\begin{aligned}
\Fr\Big(b_{j-1}\big(\mathfrak{l}_{\mathrm{l},j-1}(\eps)\big);a_{j-1}\big(\mathfrak{l}_{\mathrm{l},j-1}(\eps)\big),b_{j-1}\big(\mathfrak{l}_{\mathrm{l},j-1}(\eps)\big);\eps\Big) + \Fl\Big(b_{j-1}\big(\mathfrak{l}_{\mathrm{l},j-1}(\eps)\big);a_{j}\big(\mathfrak{l}_{\mathrm{r},j}(\eps)\big),b_{j}\big(\mathfrak{l}_{\mathrm{r},j}(\eps)\big);\eps\Big) \geq\\
 \Fr\Big(b_{j-1}\big(\mathfrak{l}_{\mathrm{l},j-1}(\eps)\big);a_{j-1}\big(\mathfrak{l}_{\mathrm{l},j-1}(\eps)\big),b_{j-1}\big(\mathfrak{l}_{\mathrm{l},j-1}(\eps)\big);\eps\Big) + \Fl\Big(b_{j-1}\big(\mathfrak{l}_{\mathrm{l},j-1}(\eps)\big);a_{j}\big(\mathfrak{l}_{\mathrm{l},j}(\eps)\big),b_{j}\big(\mathfrak{l}_{\mathrm{l},j}(\eps)\big);\eps\Big) = 0.
\end{aligned}
\end{equation*}
On the other hand, by Lemma~\ref{StrictMonotonicityBalanceEquation} (we note that~$b_{j-1}\big(\mathfrak{l}_{\mathrm{l},j-1}(\eps)\big)$ is inside the left tail of~$\Ch([a_{j}\big(\mathfrak{l}_{\mathrm{r},j}(\eps)\big),b_{j}\big(\mathfrak{l}_{\mathrm{r},j}(\eps)\big)],*)$ by Corollary~\ref{TailsGrowthNonEv} and the fact that it lies inside the left tail of~$\Ch([a_{j}\big(\mathfrak{l}_{\mathrm{l},j}(\eps)\big),b_{j}\big(\mathfrak{l}_{\mathrm{l},j}(\eps)\big)],*)$),
\begin{equation*}
\begin{aligned}
\Fr\Big(b_{j-1}\big(\mathfrak{l}_{\mathrm{l},j-1}(\eps)\big);\Ch\big(\big[a_{j-1}\big(\mathfrak{l}_{\mathrm{r},j-1}(\eps)\big),b_{j-1}\big(\mathfrak{l}_{\mathrm{r},j-1}(\eps)\big)\big];*\big);\eps\Big) +\\ \Fl\Big(b_{j-1}\big(\mathfrak{l}_{\mathrm{l},j-1}(\eps)\big);a_{j}\big(\mathfrak{l}_{\mathrm{r},j}(\eps)\big),b_{j}\big(\mathfrak{l}_{\mathrm{r},j}(\eps)\big);\eps\Big) < 0,
\end{aligned}
\end{equation*}
since these forces are balanced at~$a_j\big(\mathfrak{l}_{\mathrm{r},j}(\eps)\big)$ and~$b_{j-1}\big(\mathfrak{l}_{\mathrm{l},j-1}(\eps)\big) < a_j\big(\mathfrak{l}_{\mathrm{r},j}(\eps)\big)$. Therefore,
\begin{equation*}
\begin{aligned}
\Fr\Big(b_{j-1}\big(\mathfrak{l}_{\mathrm{l},j-1}(\eps)\big);\Ch\big(\big[a_{j-1}\big(\mathfrak{l}_{\mathrm{r},j-1}(\eps)\big),b_{j-1}\big(\mathfrak{l}_{\mathrm{r},j-1}(\eps)\big)\Big];*\big);\eps\Big) <\\ \Fr\Big(b_{j-1}\big(\mathfrak{l}_{\mathrm{l},j-1}(\eps)\big);a_{j-1}\big(\mathfrak{l}_{\mathrm{l},j-1}(\eps)\big),b_{j-1}\big(\mathfrak{l}_{\mathrm{l},j-1}(\eps)\big);\eps\Big),
\end{aligned}
\end{equation*}
which, by Lemma~\ref{biggercupsmallerforce}, proves the claim.

Similarly, if~$\mathfrak{l}_{\mathrm{r},j} \geq \mathfrak{l}_{\mathrm{l},j}$,~$j < k-1$, then~$\mathfrak{l}_{\mathrm{r},j+1} > \mathfrak{l}_{\mathrm{l},j+1}$. These claims together show that there exists~$j = j(\eps) \in \{1,2,\ldots, k\}$ such that for any~$i < j$ the inequality~$\mathfrak{l}_{\mathrm{r},i} \leq \mathfrak{l}_{\mathrm{l},i}$ holds true and for any~$i \geq j$, one has~$\mathfrak{l}_{\mathrm{r},i} \geq \mathfrak{l}_{\mathrm{l},i}$. We define~$\mathfrak{l}_i(\eps) = \mathfrak{l}_{\mathrm{r},i}(\eps)$ for~$i < j$ and~$\mathfrak{l}_i(\eps) = \mathfrak{l}_{\mathrm{l},i}(\eps)$ for~$i \geq j$. Now we prove that the foliation constructed with this~$j$ as described in the formulation of the proposition, provides a function satisfying the conditions (we consider the case~$j\ne 1,k$ as a ``generic'' one, the two remaining cases are verbatim to the case of a simple birdie treated before). 

We only have to prove that the root~$w(\eps)$ of the balance equation for~$\Fr(\cdot\,;a_{j-1}(\mathfrak{l}_{j-1}(\eps)),b_{j-1}(\mathfrak{l}_{j-1}(\eps)),\eps)$ and~$\Fl(\cdot\,;a_{j}(\mathfrak{l}_{j}(\eps)),b_{j}(\mathfrak{l}_{j}(\eps)),\eps)$ lies between~$b_{j-1}(\mathfrak{l}_{j-1}(\eps))$ and~$a_{j}(\mathfrak{l}_{j}(\eps))$. We will prove that this root may be found even between~$b_{j-1}\big(\mathfrak{l}_{\mathrm{l},j-1}(\eps)\big)$ and~$a_{j}\big(\mathfrak{l}_{\mathrm{r},j}(\eps)\big)$ (we note that then~$w(\eps)$ automatically belongs to the intersection of the tails of the forces it balance). This will follow from the inequalities
\begin{equation*}
\begin{aligned}
\Fr\Big(b_{j-1}\big(\mathfrak{l}_{\mathrm{l},j-1}(\eps)\big);a_{j-1}\big(\mathfrak{l}_{j-1}(\eps)\big),b_{j-1}\big(\mathfrak{l}_{j-1}(\eps)\big),\eps\Big) + \Fl\Big(b_{j-1}\big(\mathfrak{l}_{\mathrm{l}, j-1}(\eps)\big);a_{j}\big(\mathfrak{l}_{j}(\eps)\big),b_{j}\big(\mathfrak{l}_{j}(\eps)\big),\eps\Big) \leq 0,\\
\Fr\Big(a_{j}\big(\mathfrak{l}_{\mathrm{r},j}(\eps)\big);a_{j-1}\big(\mathfrak{l}_{j-1}(\eps)\big),b_{j-1}\big(\mathfrak{l}_{j-1}(\eps)\big),\eps\Big) + \Fl\Big(a_{j}\big(\mathfrak{l}_{\mathrm{r},j}(\eps)\big);a_{j}\big(\mathfrak{l}_{j}(\eps)\big),b_{j}\big(\mathfrak{l}_{j}(\eps)\big),\eps\Big) \geq 0.
\end{aligned}
\end{equation*}
The first inequality follows from the fact~$\mathfrak{l}_{j-1} \leq \mathfrak{l}_{\mathrm{l},j-1}$ and~$\mathfrak{l}_{j} = \mathfrak{l}_{\mathrm{l},j}$ and Lemma~\ref{biggercupsmallerforce}, the second one is symmetric.\qed

\section{Global evolution}\label{s44}
Before passing to formal statements, we describe the rules of the evolution. %But first of all we make an agreement that all the linearity domains in the foliations are maximal by inclusion.  If in our foliation we have a domain that is not maximal by inclusion, then we can change the foliation (without changing the Bellman candidate) with the help of formulas from Subsection~\ref{s345} (this will be formalized in Subsection~\ref{s442}).

Consider any foliation with the graph~$\Gamma$ assigned to it. The vertices of~$\GammaFree$ are multicups, angles, trolleybuses, multitrolleybuses, birdies, multibirdies, fictious vertices of the first, third (corresponding to long chords), fourth, and fifth type. There are chordal domains attached to some of them from below. To each edge~$\mathfrak{E}$ of~$\GammaFree$, we assign a force that controls the slope coefficient~$m$ %(see Section~\ref{s32}) 
in the domain it represents. Surely, the force depends on the domain the tangent domain starts from. For example, it is natural to equip the edge~$\Lt(u_1,a_0)$ on Figure~\ref{fig:cup_graph} with~$\Fl(\cdot\,;a_0,b_0;\eps)$; and it is natural to assign the force~$\Fr(\cdot\,;a_0,b_0,*);\eps)$ to~$\Rt(b_0,u_2)$ on the left graph of Figure~\ref{fig:gtr}; the edges~$\Lt(u_1,\mathfrak{a}_1)$ and~$\Rt(\mathfrak{a}_4,u_2)$ of the graph on Figure~\ref{fig:multicup4} match~$\Fl(\cdot\,;\mathfrak{a}_1,\mathfrak{a}_4;\eps)$ and~$\Fr(\cdot\,;\mathfrak{a}_1,\mathfrak{a}_4;\eps)$ correspondingly. 

We describe formally the rules by which we assign forces to tangent domains in the table below (there is the type of vertex that is the beginning of the edge in the first colon, the name for its numerical parameters in the second, the force that is assigned to the tangent domain if it lies on the left of the figure in the third, and the force that is assigned to the tangent domain lying on the right of the figure in the last). 

\medskip
\centerline{
\begin{tabular}{|l|c|c|c|}
\hline 
{\bf Vertex type}&{\bf Parameters}&{\bf Left Force}&{\bf Right Force}\\
\hline
%Angle&$w$&&\\
%\hline
Right trolleybus&$(a_0,b_0)$&&$\Fr(\cdot\,;a_0,b_0;\eps)$\\
\hline
Left trolleybus&$(a_0,b_0)$&$\Fl(\cdot\,;a_0,b_0;\eps)$&\\
\hline
%Birdie&$(a_0,b_0)$&&\\
%\hline
Multicup&$\{\mathfrak{a}_i\}_{i=1}^k$&$\Fl(\cdot\,;\mathfrak{a}_1^{\mathrm l},\mathfrak{a}_k^{\mathrm r};\eps)$&$\Fr(\cdot\,;\mathfrak{a}_1^{\mathrm l},\mathfrak{a}_k^{\mathrm r};\eps)$\\
\hline
Right multitrolleybus&$\{\mathfrak{a}_i\}_{i=1}^k$&&$\Fr(\cdot\,;\mathfrak{a}_1^{\mathrm l},\mathfrak{a}_k^{\mathrm r};\eps)$\\
\hline
Left multitrolleybus&$\{\mathfrak{a}_i\}_{i=1}^k$&$\Fl(\cdot\,;\mathfrak{a}_1^{\mathrm l},\mathfrak{a}_k^{\mathrm r};\eps)$&\\
\hline
%Multibirdie&$\{\mathfrak{a}_i\}_{i=1}^k$&&\\
%\hline
Fictious vertex of the first type&$(a_0,b_0)$&$\Fl(\cdot\,;a_0,b_0;\eps)$&$\Fr(\cdot\,;a_0,b_0;\eps)$\\
\hline
Fictious vertex of the third type&$(a_0,b_0)$&$\Fl(\cdot\,;a_0,b_0;\eps)$&$\Fr(\cdot\,;a_0,b_0;\eps)$\\
\hline
Fictious vertex of the fourth type&$-\infty$&&$\Fr(\cdot\,;-\infty;\eps)$\\
\hline
Fictious vertex of the fourth type&$+\infty$&$\Fl(\cdot\,;+\infty;\eps)$&\\
\hline
Right fictious vertex of the fifth type&$c_i$&&$\Fr(\cdot\,;c_i,c_i;\eps)$\\
\hline
Left fictious vertex of the fifth type&$c_i$&$\Fl(\cdot\,;c_i,c_i;\eps)$&\\
\hline
\end{tabular}
}
\medskip

The rule all the foliations generated during the evolution satisfy, is the following: if~$\Rt(u_1,u_2)$ or~$\Lt(u_1,u_2)$ is represented by the edge~$\mathfrak{E}$ in~$\GammaFree$, then~$(u_1,u_2)$ belongs to the tail of the force corresponding to~$\mathfrak{E}$. This requirement for the foliation will be called the \emph{non-degeneracy force condition}\index{condition! non-degeneracy force condition}.

\begin{Cond}\label{NonDegeneracyForceCondition}
For any edge~$\mathfrak{E}$ in~$\GammaFree$ corresponding to a tangent domain~$\Rt(u_1,u_2)$ or~$\Lt(u_1,u_2)$\textup, the interval~$(u_1,u_2)$ belongs to the tail of the force assigned to~$\mathfrak{E}$.
\end{Cond}
A short inspection of definitions shows that Condition~\ref{NonDegeneracyForceCondition} holds true for all the graphs we have constructed. In other words, all the forces in right tangent domains are strictly negative, whereas in left tangent domains  they are strictly positive. %in all foliations that generate Bellman candidates, there is inequality~$m'' \leq 0$ for right tangent domains and~$m''\geq 0$ for left. The non-degeneracy force condition says that these inequalities are always strict inside the domains. 
In particullar, the following remark is important.
 
\begin{Rem}\label{nondegforcecondSimplePictureRem}
We note that the simple Bellman candidates constructed in Section~\textup{\ref{s41}} fulfill the non-degeneracy force Condition~\textup{\ref{NonDegeneracyForceCondition}}.
\end{Rem}
As has already been said, the main rule of the evolution is that the forces grow in absolute value on their tails, see Subsection~\ref{s422}. As a consequence, the tails themselves strictly grow (by this we mean that the~$\tr$ increase and the~$\tl$ decrease). Thus, full chordal domains grow (Proposition~\ref{InductionStepForChordalDomain}), the multicups are stable\footnote{In a sense, they also grow: the border tangents rise; however, the numerical parameters do not change.} (Proposition~\ref{InductionStepForMulticup}), the trolleybuses shrink (Propositions~\ref{InductionStepForRightTrolleybus} and~\ref{InductionStepForLeftTrolleybus}), the angles continuously wander from side to side (Proposition~\ref{InductionStepAngle}). These figures can be described as stable. If there are multitrolleybuses or multibirdies in the foliation for fixed~$\eps$, they immediately disintegrate (Propositions~\ref{InductionStepForMultitrolleybus},~\ref{InductionStepForLeftMultitrolleybus}, and~\ref{MultibirdieDesintegrationSt}). These figures are unstable. As for the birdie, it can shrink, but ``generically'' desintegrates (Proposition~\ref{MultibirdieDesintegrationSt} in its model case). So, it is half-stable. 

There is also one useful condition all our graphs will satisfy. It is of structural character (and thus relies on Definition~\ref{roots}) and concerns mostly fictious vertices. It is called the \emph{leaf-root condition}\index{condition! leaf-root condition}.
\begin{Cond}\label{Leaf-root condition}
Any arc of any multifigure that is not a single point coincides with one of the roots~$c_i$\textup;
numeric parameters of the fictious vertices of the second type are some roots~$c_i$ that are single points\textup;
each fictious vertex of the third type corresponding to~$[A_0,B_0]$ such that~$\Srt(a_0,b_0)=0$
satisfies the condition~$b_0 = c_i$ for some~$c_i$ that is a single point\textup, if~$\Slt(a_0,b_0) = 0$\textup, then~$a_0 = c_i$\textup;
numeric parameter of each vertex of the fifth type is a root~$c_i$ that is a single point.
\end{Cond}
\begin{Rem}
All simple graphs constructed in Section~\textup{\ref{s41}} fulfill the leaf-root Condition~\textup{\ref{Leaf-root condition}}.
\end{Rem}
\begin{Def}\label{Admissible graph}\index{graph! admissible graph}
Let~$\eps < \eps_{\infty}$. We say that a graph~$\Gamma$ is admissible for~$f$ and~$\eps$ if all figures corresponding to the vertices and edges of~$\Gamma$ satisfy their local propositions.%\textup, and~$\Gamma$\textup,~$f$\textup, and~$\eps$ also fulfill Conditions~\textup{\eqref{NonDegeneracyForceCondition}} and~\textup{\eqref{Leaf-root condition}}.  
\end{Def}
By ``all figures corresponding to the vertices and edges of~$\Gamma$ satisfy their local propositions''  we mean the following: for each vertex or edge in~$\Gamma$ the parameters satisfy the assumptions of the proposition which number is indicated for this vertex or edge in the table below (in the third colon).

\medskip
\centerline{
\begin{tabular}{|l|c|c|c|}
\hline 
{\bf Vertex or edge type}&{\bf Formulas}&{\bf Verification}&{\bf Evolutional rule}\\
\hline
Right tangent domain&\eqref{linearity},~\eqref{ExplicitFormulaForm},~\eqref{minfty}&\ref{RightTangentsCandidate}, \ref{RightTangentsCandidateInfty}&\\
\hline
Left tangent domain&\eqref{linearity},~\eqref{ExplicitFormulaForm2},~\eqref{minfty2}&\ref{LeftTangentsCandidate}, \ref{LeftTangentsCandidateInfty}&\\
\hline
Chordal domain&\eqref{vallun}&\ref{LightChordalDomainCandidate}&\\
\hline
Angle&%\eqref{ValAng},~
\eqref{StandardCandidateAngleFormula}&\ref{AngleProp}&\ref{InductionStepAngle}\\
\hline
Right trolleybus&\eqref{LinearInTrolleybus},~\eqref{coef}&\ref{S15}&\ref{InductionStepForRightTrolleybus}\\
\hline
Left trolleybus&\eqref{LinearInTrolleybus},~\eqref{coef}&\ref{S16}&\ref{InductionStepForLeftTrolleybus}\\
\hline
Birdie&\eqref{LinearInTrolleybus},~\eqref{coef}&\ref{S17}&\ref{MultibirdieDesintegrationSt}\\
\hline
Multicup&\eqref{CandidateInMultifigure},~\eqref{coef}&\ref{MulticupCandidate}&\ref{InductionStepForMulticup}\\
\hline
Full multicup&\eqref{CandidateInMultifigure},~\eqref{coef}&\ref{MulticupCandidate}&\ref{InductionStepForLongChords}\\
\hline
Right multitrolleybus&\eqref{CandidateInMultifigure},~\eqref{coef}&\ref{MultitrolleybusCandidate}&\ref{InductionStepForMultitrolleybus}\\
\hline
Left multitrolleybus&\eqref{CandidateInMultifigure},~\eqref{coef}&\ref{LeftMultitrolleybusCandidate}&\ref{InductionStepForLeftMultitrolleybus}\\
\hline
Multibirdie&\eqref{CandidateInMultifigure},~\eqref{coef}&\ref{MultibirdieCandidate}&\ref{MultibirdieDesintegrationSt}\\
\hline
Closed multicup&\eqref{CandidateInMultifigure},~\eqref{coef}&\ref{ClosedMulticupCandidate}&\\
\hline
Fictious vertex of the first type&\eqref{vallun}&\ref{CupMeetsRightTangents},~\ref{CupMeetsLeftTangents}&\ref{InductionStepForChordalDomain}\\
\hline
Fictious vertex of the third type, long chord&\eqref{vallun}&\ref{CupMeetsRightTangents},~\ref{CupMeetsLeftTangents}&\ref{InductionStepForLongChords}\\
\hline
Fictious vertex of the fifth type& \eqref{Beta2InTheFifthType},~\eqref{StandardCandidateAngleFormula}&%\ref{DomainOfLinearityWithTwoPointsOnTheLowerBoundaryMeetsTangentDomain}
&\ref{DegenerateTrolleybus}\\
\hline
\end{tabular}
}
\medskip

In the first colon, there is the type of the vertex or edge, in the second there is a reference to formulas with which the canonical function~$B$ is constructed in this figure,
 in the third colon the number of the proposition that guarantees that this~$B$ is a Bellman candidate are stored. Finally, the last colon contains the number of the proposition that describes the local evolution of the parameters for the figure. We have omitted fictious vertices of the second and fourth types (as well as the vertices of the third type that correspond to short chords), because the value of the function~$B$ in the domains corresponding to them is defined trivially, and these figures are stable and have no evolutional scenarios.

Now we describe how to construct the function~$B$ from a graph. First, one constructs this function for all the vertices and edges that participate in~$\Gamma \setminus \GammaFree$, because for their figures there is no additional information needed to construct~$B$ (no information from other figures). Second, we construct the function~$B$ for all multicups and all chordal domains (and thus all fictious vertices of the first, second, and third types). Again, this does not need any information from other figures. Note that we have built the function~$B$ for all figures that correspond to the roots of~$\GammaFree$. Now we will construct~$B$ consecutively, moving from the roots to leaves. We note that for each edge~$\mathfrak{E}$ in~$\GammaFree$, the values of~$B$ in the figure corresponding to its beginning define the boundary value of the slope coefficient~$m$, so one may construct~$B$ in the tangent domain corresponding to~$\mathfrak{E}$ if he knows the values of~$B$ on the domain of its source\footnote{There is one exception: for tangent domains coming from infinity, one does not need any boundary data.}. Similarly, to construct~$B$ for a vertex in~$\GammaFree$, it suffices to know the values of~$B$ in the tangent domains that correspond to its incoming edges. So, one can construct the function~$B$ consecutively, moving from the roots of~$\GammaFree$ to its leaves. Note that if~$\Gamma$ fulfills Condition~\ref{NonDegeneracyForceCondition}, then the restriction of~$B$ to each figure is a standard Bellman candidate there. 

\begin{Rem}\label{AdmissibleCondidate}
The function~$B$ constructed from an admissible graph is a Bellman candidate. 
\end{Rem} 
\begin{proof}
We look at the table above and use the corresponding verification proposition for each vertex or edge. We see that the function~$B$ is~$C^1$-smooth and locally concave (indeed, the verification propositions say that the function~$B$ is locally concave in a neighborhood of any point). The last condition of Definition~\ref{candidate} follows from the construction of the function~$B$. There only thing we have not discussed before is~$C^1$-smoothness of the concatenation of two chordal domains in a neighborhood of a fictious vertex of the third type, but such a smoothness is an easy consequence of formula~\eqref{BsubX2Chords}.
\end{proof}

Since during the evolution some figures grow and angles move, two figures might crash. For example, the vertex of an angle may coincide with the right endpoint of a long chord. In such a case, we look at formula~\eqref{CupPlusAngleR}, and see that now they form a trolleybus. Therefore, the graph of the foliation changes at this moment~$\eps$. We call such moments \emph{the critical points of the evolution} (the precise definition will be given later). The idea is that if a crash occurs, then the two crashed figures compose one new (with the help of formulas from Subsection~\ref{s344}), and we can proceed the evolution. Unfortunately, there might be infinitely many critical points (see the second example in Section~\ref{s54} below). However, if one restricts his attention only to those critical points, at which the structure of the graph~$\GammaFixed$ essentially changes, he finds only a finite number of critical points. Such points are called \emph{essentially critical}\index{essentially critical points} (we will give the rigorous definition in the proof of Theorem~\ref{BC}). The following definition is also useful.
\begin{Def}\label{Smoothgraph}\index{graph! smooth graph}
We say that a graph~$\Gamma$ is smooth if there are no vertices representing full multicups\textup, multitrolleybuses\textup, multibirdies\textup, fictious vertices of the third type that represent long chords\textup, and fictious vertices of the fifth type in~$\Gamma$. %We also require all multicups be not full.
\end{Def}
%Now we are ready to pass to formal statements.

\subsection{Induction step of the first type}\label{s441}
\begin{Th}\label{InductionStepOfTheFirstKind}
Let~$\eta_1 < \eps_{\infty}$\textup,~$\Gamma$ an admissible graph for~$f$ and~$\eta_1$ satisfying Conditions~\textup{\ref{NonDegeneracyForceCondition}} and~\textup{\ref{Leaf-root condition}}. Then\textup, there exists~$\eta_2 > \eta_1$ such that for every~$\eps \in (\eta_1,\eta_2]$ there exists a smooth graph~$\Gamma(\eps)$ admissible for~$f$ and~$\eps$ that satisfies Conditions~\textup{\ref{NonDegeneracyForceCondition}} and~\textup{\ref{Leaf-root condition}} as well. Moreover\textup, each vertex in~$\Gamma$ is replaced in~$\Gamma(\eps)$ by the local evolutional rule prescribed to its type. %The parameters of~$\Gamma(\eps)$ are obtained from those of~$\Gamma$ by applying to each vertex and edge its local evolutional proposition. 
\end{Th}
Informally, the theorem says that if the Bellman candidate of prescribed structure is given for some~$\eps$, then we can build a Bellman candidate for a slightly bigger~$\eps$ with almost the same formula, i.e. for each vertex of~$\Gamma$ we simply look at the big table after Definition~\ref{Admissible graph} and replace this vertex as prescribed by the corresponding proposition. The graph of the obtained Bellman candidate does not contain vertices that represent unstable figures.
\begin{proof}
The idea of the proof is to apply the local evolutional theorems consecutively along~$\GammaFree$. They will allow to change the parameters of the graph slightly to pass from an admissible graph for~$\eta_1$ to an admissible graph for~$\eps$. The numerical parameters of vertices in~$\Gamma \setminus \GammaFree$ will be unchanged, as well as all the edges that are not connected with~$\GammaFree$. The numerical parameters for vertices and edges of~$\GammaFree$ will be changed consecutively, going from the roots to the leaves. So, the proof naturally splits into four stages: the description of changes for the roots of~$\GammaFree$, for intermediate vertices and edges, and for the leaves, and finally, the verification of smoothness of the constructed graph as well as Conditions~\textup{\ref{NonDegeneracyForceCondition}} and~\textup{\ref{Leaf-root condition}}. For each vertex we will use its own local evolutional proposition (they are listed in the big table right after Definition~\ref{Admissible graph}). Each step of the construction will impose some restrictions on the smallness of~$\eta_2-\eta_1$. Since the parameters will be changed continuously, we may require that~$\eta_2-\eta_1$ is so small that every edge that is present in~$\Gamma$ still has non-zero length after the change (i.e. we make such a small perturbation that the figures corresponding to vertices of~$\Gamma$ do not meet each other).

\paragraph{Roots of~$\GammaFree$.} The roots of~$\GammaFree$ are multicups, infinities, and fictious vertices of the first and third types. Our aim here is to explain how do the numerical parameters (or the figures themselves) change when we vary~$\eps$, and explain that the forces assigned to the outcoming edges of the figures form monotone force flows (as functions of~$u$ and~$\eps$). 

In the case of a multicup, we consider two cases. If our multicup is not full (i.e. it is~$\MTC(\{\mathfrak{a}_i\}_{i=1}^k)$ and~$\mathfrak{a}_k^{\mathrm r} - \mathfrak{a}_1^{\mathrm l} > 2\eps$), we apply Proposition~\ref{InductionStepForMulticup} and see that if~$\eps-\eta_1$ is sufficiently small and the parameters of the multicup remain the same, then the forces in the two tangent domains adjacent to it, grow in absolute value on their tails (thus forming monotone force flows). If the multicup is full, we are going to apply Proposition~\ref{InductionStepForLongChords} to the chord~$[\mathfrak{A}_1^{\mathrm l},\mathfrak{A}_k^{\mathrm{r}}]$ (we use the traditional notation for the parameters of the multicup). To do this, we need to show that the tails of this chord are non-zero. This follows from the fact that~$\Gamma$ is under Condition~\ref{NonDegeneracyForceCondition}. So, we apply Proposition~\ref{InductionStepForLongChords} and build a full flow of chordal domains. We see that the full multicup becomes a closed one, and we also paste a new chordal domain and a long chord above it. The forces inside the adjacent domains form monotone force flows because they are forces of a full flow of chordal domains (see Remark~\ref{FCDMFF}).

In the case of the infinities, we do not have to change the figure. The forces assigned to the outcoming edges grow by Remark~\ref{TailGrowthInfinity} (and thus form a monotone force flow).

For fictious vertices of the first type, we apply Proposition~\ref{InductionStepForChordalDomain}. Such a fictious vertex generates a full flow of chordal domains. For each~$\eps$, we simply take the corresponding full chordal domain from it. The forces assigned to the outcoming edges form a monotone force flow (when we restrict each force to the intersection of their domains) because they are the forces of a full flow of chordal domains.

The case of a fictious vertex of the third type representing a long chord is considered similarly to the case of a full multicup (we add one more fictious vertex of the first type and link it to the considered vertex by a chordal domain). 

\paragraph{Intermediate vertices.} They are: trolleybuses, multitrolleybuses, and fictious vertices of the fifth type. Each of them has one incoming and one outcoming edge in~$\GammaFree$. Consider any vertex~$\mathfrak{L}$ of height~$1$ (see Definition~\ref{Height}). Since the beginning of its incoming edge is a root of~$\GammaFree$, we have already built modifications of its parameters for some interval~$\eps \in (\eta_1,\eta_2]$. So, we see that the forces assigned to the incoming edge of~$\mathfrak{L}$ for~$\eps \in (\eta_1,\eta_2]$ form a monotone force flow. Consider different cases of the type of~$\mathfrak{L}$. 

If~$\mathfrak{L}$ represents a trolleybus, then we apply Proposition~\ref{InductionStepForRightTrolleybus} or~\ref{InductionStepForLeftTrolleybus} to it (depending on the orientation of the trolleybus). Such an application is permitted since the tail of force for~$\eps = \eta_1$ of the incoming edge contains the whole tangent domain it represents (since~$\Gamma$ satisfies Condition~\ref{NonDegeneracyForceCondition}). Propositions~\ref{InductionStepForRightTrolleybus} or~\ref{InductionStepForLeftTrolleybus} say that a chordal domain below the trolleybus evolves as a decreasing flow of chordal domains. Therefore, its forces assigned to the outcoming edge form a monotone force flow due to Remark~\ref{FCDMFF}. 

If~$\mathfrak{L}$ represents a multitrolleybus, then we apply Proposition~\ref{InductionStepForMultitrolleybus} or~\ref{InductionStepForLeftMultitrolleybus} to it (depending on the orientation). In this case,~$\mathfrak{L}$ is changed for a subgraph as prescribed by formulas~\eqref{RMultitrolleybusDesintegration} and~\eqref{LMultitrolleybusDesintegration}, i.e. the multitrolleybus is split into several trolleybuses and tangent domains, and the trolleybuses are sitting on decreasing flows of chordal domains, and therefore the forces assigned to the outcoming edge form a monotone force flow. There is a peculiarity in the case where the multitrolleybus is sitting on a single arc. In this case we see that it immediately dissapears and its incoming and outcoming edges become a single one, and the peculiarity is that the monotone force flow for the outcoming edge is defined in a special way as described in Remark~\ref{DegenerateTrolleybus}. 

The case where~$\mathfrak{L}$ is a fictious vertex of the fifth type is similar to the case of a multitrolleybus on a single arc (use Remark~\ref{DegenerateTrolleybus}).

So, we have described how do the parameters change in the case when the height of~$\mathfrak{L}$ equals one. We see that the forces assigned to the outcoming edges form monotonic flows, and we also see that what we have used for the case of height one is Condition~\ref{NonDegeneracyForceCondition} for~$\Gamma$ and that the forces assigned to incoming edges form monotone flows. So, we can apply the same argument to the vertices of height two, then three, and so on (and surely, the procedure is finite, so there will be only a finite number of constraints on the smallness of~$\eta_2-\eta_1$). 

\paragraph{Leaves of~$\GammaFree$.} They are angles and multibirdies (including usual birdies). As we have seen from the previous construction, the forces that correspond to the incoming edges form monotone force flows. 

In the case of an angle, we apply Proposition~\ref{InductionStepAngle}. Again, we only have to verify the condition that the tails of the forces of incoming edges contain the vertex of the angle in their closure for~$\eps = \eta_1$. This follows from Condition~\ref{NonDegeneracyForceCondition}. We see that we can paste an angle between the two tangent domains corresponding to its incoming edges.

In the case of a multibirdie, we apply Proposition~\ref{MultibirdieDesintegrationSt} absolutely similar to the previous case. We only note that in this case the graph changes: instead of a single vertex representing a multibirdie, we get a small graph as prescribed by formulas~\eqref{MultibirdieDesintegration} and~\eqref{FirstFormula}. 

\paragraph{Conditions and smoothness verification.} First, we claim that the achieved graph~$\Gamma(\eps)$ is smooth. Indeed, we did not obtain new full multicups, multitrolleybuses, multibirdies, fictious vertices of the third type representing long chords, fictious vertices of the fifth type during the modification of our graph. And we have also seen that all figures of these types that were present in~$\Gamma$, have disappeared. Second, we see that Condition~\ref{Leaf-root condition} holds true for~$\Gamma(\eps)$, again due to the same reason (among the types vertices described in this condition,~$\Gamma(\eps)$ may possess only those ones belonging to~$\Gamma(\eps) \setminus \GammaFree(\eps)$, but these vertices are inherited from~$\Gamma$).

The verification of Condition~\ref{NonDegeneracyForceCondition} is left to the reader (in fact, its verification is contained in local induction steps of Section~\ref{s43}).
\end{proof}
\begin{Def}\label{SmoothFlow}\index{graph! smooth flow of graphs}
Let~$\Gamma$ be a smooth admissible for~$f$ and~$\eta_1$ graph\textup,~$\eta_1 < \eps_{\infty}$\textup, suppose that it satisfies Conditions~\textup{\ref{NonDegeneracyForceCondition}} and~\textup{\ref{Leaf-root condition}}. Let~$\eta_2$ be such that~$\eta_1 < \eta_2 <  \eps_{\infty}$. Let the edges of~$\Gamma$ representing full chordal domains be the~$\mathfrak{E}_{i}^+$\textup, let the edges representing chordal domains adjacent to trolleybuses and birdies be the~$\mathfrak{E}_i^-$. Let~$\{\mathfrak{E}_i^+(\eps)\}_{\eps \in [\eta_1,\eta_2)}$ and~$\{\mathfrak{E}_i^-(\eps)\}_{\eps \in [\eta_1,\eta_2)}$ be the full and decreasing flows of chordal domains starting from them. We call a collection of smooth admissible graphs~$\Gamma(\eps)$\textup, $\eps \in [\eta_1,\eta_2)$\textup, a \emph{smooth flow of graphs} if each~$\Gamma(\eps)$ satisfies Conditions~\textup{\ref{NonDegeneracyForceCondition}} and~\textup{\ref{Leaf-root condition}} and is such that~$\Gamma\setminus\GammaFree$ differs from~$\Gamma(\eps)\setminus\GammaFree(\eps)$ only in the edges~$\mathfrak{E}_{i}^+$ and~$\mathfrak{E}_i^-$, which are changed for $\mathfrak{E}_{i}^+(\eps)$ and~$\mathfrak{E}_i^-(\eps)$ correspondingly. We also require that the multicups and fictious vertices of the fourth type in~$\Gamma(\eps)$ are the same as in~$\Gamma$. 
%We call a collection of full chordal domain flows starting from the~$\mathfrak{E}_i^+$ parametrized by~$(\eta_1,\eta_2)$ and decreasing flows of chordal domain starting from the~$\mathfrak{E}_i^{-}$ parametrized by~$(\eta_1,\eta_2)$ a \emph{smooth flow of graphs} if for any~$\eps \in (\eta_1,\eta_2)$ there exists a smooth admissible graph~$\Gamma(\eps)$ satisfying Conditions~\textup{\ref{NonDegeneracyForceCondition}} and~\textup{\ref{Leaf-root condition}} such that~$\Gamma\setminus\GammaFree$ differs from~$\Gamma(\eps)\setminus\GammaFree(\eps)$ only in the edges~$\mathfrak{E}_{i}^+$ and~$\mathfrak{E}_i^-$\textup, and the parameters for each such an edge are taken from the corresponding flow of chordal domains\textup; we also require that the multicups and fictious vertices of the fourth type in~$\Gamma(\eps)$ are the same as in~$\Gamma$. 
\end{Def}

The chordal domains~$\mathfrak{E}_{i}^{\pm}(\eps)$ completely determine the graphs~$\Gamma(\eps)$. Indeed, we know everything about~$\Gamma(\eps)$ except for the orientation of trolleybuses and the places where the angles are situated. But this information is provided by the roots of the balance equations for chordal domains assigned to~$\mathfrak{E}_i^{\pm}(\eps)$. Each such balance equation has not more than one root in the intersection of the tails of forces by Lemma~\ref{monbaleq}, thus the graph is uniquely defined by the corresponding chordal domains. %So we may (and often will) call the collection of graphs~$\{\Gamma(\eps)\}_{\eps \in (\eta_1,\eta_2)}$ a flow of graphs.

Definition~\ref{SmoothFlow} might seem abstract and sophisticated. Informally, a smooth flow is the description of the local in time evolution of a smooth admissible graph satisfying Conditions~\textup{\ref{NonDegeneracyForceCondition}} and~\textup{\ref{Leaf-root condition}}. As we have already said, full chordal domains grow, while trolleybuses shrink. However, we do not require the trolleybuses to preserve their orientation, moreover, some of them may become birdies. Only angles may change their location (i.e. an angle may attack a trolleybus, form a birdie, which may desintegrate, etc.). We also note that the angles cannot attack long chords (i.e. full chordal domains), because the collection of these edges is preserved along the flow.

Lemma~\ref{monbaleq} also shows that a trolleybus cannot change its orientation immediately, i.e. if~$\eps-\eta_1$ is sufficiently small (depending on the flow in question), then all the trolleybuses in~$\Gamma(\eps)$ have the same orientation as in~$\Gamma$. Indeed, consider a trolleybus whose base is represented by the edge~$\mathfrak{E}_i^-$ in~$\Gamma$, let it be a right trolleybus. Then, we see that the balance equation for the right force of~$\mathfrak{E}_i^-$ and the force coming to it from the right has the root somewhere strictly on the right of the trolleybus for~$\eps = \eta_1$. Since all the forces are continuous functions of~$\eps$, by Lemma~\ref{monbaleq}, this root is also continuous as a function of~$\eps$. So, it cannot immediately come to the base of the trolleybus.  
\begin{Th}\label{SmoothFlowTheorem}
Let~$\eta_1 < \eps_{\infty}$. Let~$\Gamma$ be a smooth graph admissible for~$f$ and~$\eta_1$ satisfying Conditions~\textup{\ref{NonDegeneracyForceCondition}} and~\textup{\ref{Leaf-root condition}}. Let~$\eta_2$\textup,~$\eta_1 < \eta_2 < \eps_{\infty}$\textup, be such that~$\eta_2 - \eta_1$ is sufficiently small. Then there exists a unique smooth flow of graphs defined on~$(\eta_1,\eta_2)$ starting from~$\Gamma$. 
\end{Th} 
\begin{proof}
The existence of such a flow follows from Theorem~\ref{InductionStepOfTheFirstKind}\footnote{To be more honest, it follows from the construction used in the proof}. So, we turn to the uniqueness. Assume the contrary, let there be two smooth flows of graphs~$\digamma_1$ and~$\digamma_2$. Without loss of generality, we may assume that~$\digamma_1(\eps_n)\ne\digamma_2(\eps_n)$ when~$\eps_n\searrow \eta_1$. As it was said, the orientations of all trolleybuses in~$\digamma_1$ and~$\digamma_2$ are the same when~$\eps-\eta_1$ is sufficiently small.

The uniqueness for full chordal domain flows follows from Lemma~\ref{SmoothChordalDomain}. So, the collections of full chordal domain flows in~$\digamma_1$ and~$\digamma_2$ are the same. Therefore, the corresponding monotone force flows generated by them also coincide. 

Consider any trolleybus of height one in~$\Gamma$. We know that in~$\digamma_1(\eps)$ and~$\digamma_2(\eps)$,~$\eps-\eta_1$ sufficiently small, it has one and the same orientation. By Remark~\ref{UniqueTrolley}, it is also unique in the sense that the decreasing chordal domain flow that generates its base is the same for~$\digamma_1$ and~$\digamma_2$. Again, we see that the force of this decreasing flow is also the same in~$\digamma_1$ and~$\digamma_2$. Arguing in that way (increasing the height step by step), we prove that all the decreasing chordal domain flows whose starting chordal domain sat under a trolleybus in~$\Gamma$, are the same in~$\digamma_1$ and~$\digamma_2$.    

It remains to deal with birdies and angles. The birdies are treated by Remark~\ref{UniqueBirdie}, for the angles we use Lemma~\ref{monbaleq}. 

So,~$\digamma_1(\eps)=\digamma_2(\eps)$ provided~$\eps -\eta_1$ is sufficiently small, a contradiction.
\end{proof}

\subsection{Induction step of the second type}\label{s442}
Let~$\Gamma(\eps)$,~$\eps \in (\eta_1,\eta_2)$, be a smooth flow of graphs. In this section, we are concerned with the question: what does the limit of~$\Gamma(\eps)$ as~$\eps \to \eta_2$ may look like? By the very definition, a smooth flow is generated by a collection of several chordal domain flows (either full or decreasing), and for them there is only one way to pass to the limit. The question is if one can define a reasonable limit for the corresponding graphs~$\Gamma(\eps)$,~$\eps\to \eta_2$.

There are not very many figures that depend on~$\eps$ in~$\Gamma(\eps)$. We see that the problem of passing to the limit is actual only for angles, birdies, and trolleybuses (and for the latter figures, only the orientation is unclear, i.e. we know the limit of their bases).
\begin{Rem}
Let~$\{\Ch([a_1^{\mathrm{top}},b_1^{\mathrm{top}}],*),\mathfrak{l}_1\}$ and~$\{\Ch([a_2^{\mathrm{top}},b_2^{\mathrm{top}}],*),\mathfrak{l}_2\}$ be two flows of chordal domains with the corresponding functions $a_1,b_1$ and $a_2,b_2$ \textup($b_1\circ\mathfrak{l}_1\leq a_2\circ\mathfrak{l}_2$\textup)\textup, either decreasing or full\textup, parametrized by~$(\eta_1,\eta_2)$\textup,~$0\leq \eta_1 < \eta_2< \eps_{\infty}$. Let~$\Fr$ be the right force of the first flow\textup, let~$\Fl$ be the left force of the second flow. Suppose that for any~$\eps \in (\eta_1,\eta_2)$ the root~$w(\eps)$ of the balance equation between these two forces belongs both to the intersection of the closures of their tails  and to~$\big[b_1(\mathfrak{l}_1(\eps)),a_2(\mathfrak{l}_2(\eps))\big]$. Then\textup, the function~$\eps \mapsto w(\eps)$ has a limit as~$\eps \to \eta_2$. One or both chordal domains may be replaced by a multicup or the corresponding infinity \textup(in the second case\textup, the limit may be infinite\textup).
\end{Rem}
\begin{proof}
We consider the case of two chordal domains only, the other cases are even easier. It suffices to prove that the function~$w$ cannot have two limit points at~$\eta_2$ (because it definitely has a limit point). The functions~$\mathfrak{l}_1$ and~$\mathfrak{l}_2$ are monotone, thus, they have limits at~$\eta_2$, which we denote by~$\mathfrak{l}_1(\eta_2)$ and~$\mathfrak{l}_2(\eta_2)$. Then, by the continuity of forces, any limit point of~$w$ at~$\eta_2$ is a root of the balance equation for
\begin{equation*}
\Fr\Big(\cdot\,; a_1\big(\mathfrak{l}_1(\eta_2)\big),b_1\big(\mathfrak{l}_1(\eta_2)\big);\eta_2\Big)\quad \hbox{and}\quad \Fl\Big(\cdot\,; a_2\big(\mathfrak{l}_2(\eta_2)\big),b_2\big(\mathfrak{l}_2(\eta_2)\big);\eta_2\Big).
\end{equation*}
We note that since the tails grow, a limit point belongs to the intersection of the closures of the tails. So, if~$w$ has two limit points at~$\eta_2$, we get a contradiction with Lemma~\ref{monbaleq}.
\end{proof}
This remark shows that we may pass to the limit for roots of balance equations. So, we can define the limit of the graph~$\Gamma(\eps)$ as~$\eps \to \eta_2$. Call it simply~$\Gamma(\eta_2)$.

Of course,~$\Gamma(\eta_2)$ can still be a smooth admissible graph satisfying Conditions~\ref{NonDegeneracyForceCondition} and~\ref{Leaf-root condition}. But what if not? It is still a graph (i.e. one can construct a Bellman candidate from it using local formulas), but two unusual things may happen: some vertex may change its type or dissappear, or some edge may become ``of zero length'', i.e. a chordal domain may shrink to a single chord, or a tangent domain may shrink to a single tangent. However, we are going to prove that after certain reconstruction with the help of formulas from Subsection~\ref{s344}, the limit graph is admissible and verifies Conditions~\ref{NonDegeneracyForceCondition} and~\ref{Leaf-root condition}.

Let us look at the evolution of local figures along a smooth flow. We are interested in those types of vertices whose representatives may change their type by themselves after passing to the limit. These are multicups (some of them may become full) and fictious vertices of the first type (some of them may become fictious vertices of the third type).

As for the edges, the situation is more difficult. For example, suppose that there is a decreasing flow of chordal domains in our flow of graphs, whose limit is a single chord. Then the obtained picture in a neighborhood of this chord depends heavily on what was below it (either a closed multicup or a fictious vertex of the second or third type). The situation with edges representing tangent domains is even more difficult, because if several consecutive tangent domains vanish, we may get a rather complicated ensemble of joint linearity domains. We claim that what we get as a result of such a union is one of our domains described in Section~\ref{s34}. We describe all elementary unions (those in which only two figures participate) in three tables below.

The first table describes the concatenation of two figures adjacent to a chord. In each row there is the type of figure attached to the chord from below, in each colon --- the type of figure attached from above; the formula that we use for the concatenation is written in the cell on of their intersection (if the result of such a concatenation is obvious, we do not have a formula but simply write the result).

\medskip
\centerline{\large\bf Concatenation of figures adjacent to a chord.}\nopagebreak
\centerline{
\begin{tabular}{|l|c|c|}
\hline 
&Trolleybus&Birdie\\
\hline
Closed multicup& \eqref{ClosedMulticupTrolleybusR} and~\eqref{ClosedMulticupTrolleybusL} &\eqref{ClosedMulticupBirdie}\\
\hline
Fictious vertex of the second type& Fictious vertex of the fifth type& Angle\\
\hline
Fictious vertex of the third type& Trolleybus& Birdie\\ 
\hline
\end{tabular}
}
\medskip

Now let us turn to the case of a right tangent. As we have already said, there might be multiple concatenations. So, some figures that are not present in smooth graphs (e.g. multitrolleybuses) might be involved in formulas. Generally, any figure that is present in general graphs is under consideration. We see that the figure that is attached to the right tangent from the left might be either a right (multi)trolleybus, or a multicup, or fictious vertex of the first, third,  fourth, and fifth types. From the right there might be an angle, a right (multi)trolleybus, a (multi)birdie, or a fictious vertex of the fifth type. Again, the table below explains how to concatenate the figures in all these cases.

\medskip
\centerline{\large\bf Concatenation of figures adjacent to a right tangent.}\nopagebreak
\centerline{
\begin{tabular}{|l|c|c|c|c|c|}
\hline 
&Angle&Right trolleybus&Right multitrolleybus&Multibirdie&Fictious vertex 5\\
\hline
Right trolleybus&\eqref{RTrolleybusPlusAngle}&\eqref{RMultitrolleybusDesintegration}&\eqref{RMultitrolleybusDesintegration}&\eqref{MultibirdieDesintegration}&Right trolleybus\\
\hline
Multicup&\eqref{AnglePlusMulticupRight}&\eqref{MTCplusMTTR}&\eqref{MTCplusMTTR}&\eqref{MTCplusBirdieR}&Multicup\\
\hline
Right multitrolleybus&\eqref{AnglePlusMultiTrollebusR}&\eqref{RMultitrolleybusDesintegration}&\eqref{RMultitrolleybusDesintegration}&\eqref{MultibirdieDesintegration}&Right multitrolleybus\\
\hline
Fictious vertex 1&\eqref{CupPlusAngleL}&\eqref{ChordalDomainPlusMultitrolleybusR}&\eqref{ChordalDomainPlusMultitrolleybusR}&\eqref{ChordalDomainPlusMultibirdieL}&Fictious vertex 1\\
\hline
Fictious vertex 3&\eqref{CupPlusAngleL}&\eqref{ChordalDomainPlusMultitrolleybusR}&\eqref{ChordalDomainPlusMultitrolleybusR}&\eqref{ChordalDomainPlusMultibirdieL}&Fictious vertex 3\\
\hline
Fictious vertex 4&$-\infty$&&&&\\
\hline
Fictious vertex 5&Angle&Right trolleybus&Right multitrolleybus&Multibirdie&\\
\hline
\end{tabular}
}
\medskip

The last but one row is almost empty, because the~$-\infty$ vertex may be attached only to an angle (this means that the angle tends to minus infinity as~$\eps \to \eta_2$). Here is the table for a left tangent.

\medskip
\centerline{\large\bf Concatenation of figures adjacent to a left tangent.}\nopagebreak
\centerline{
\begin{tabular}{|l|c|c|c|c|c|}
\hline 
&Angle&Left trolleybus&Left multitrolleybus&Multibirdie&Fictious vertex 5\\
\hline
Left trolleybus&\eqref{LTrolleybusPlusAngle}&\eqref{LMultitrolleybusDesintegration}&\eqref{LMultitrolleybusDesintegration}&\eqref{MultibirdieDesintegration}&Left trolleybus\\
\hline
Multicup&\eqref{AnglePlusMulticupLeft}&\eqref{MTCplusMTTL}&\eqref{MTCplusMTTL}&\eqref{MTCplusBirdieL}&Multicup\\
\hline
Left multitrolleybus&\eqref{AnglePlusMultiTrollebusL}&\eqref{LMultitrolleybusDesintegration}&\eqref{LMultitrolleybusDesintegration}&\eqref{MultibirdieDesintegration}&Left multitrolleybus\\
\hline
Fictious vertex 1&\eqref{CupPlusAngleR}&\eqref{ChordalDomainPlusMultitrolleybusL}&\eqref{ChordalDomainPlusMultitrolleybusL}&\eqref{ChordalDomainPlusMultibirdieR}&Fictious vertex 1\\
\hline
Fictious vertex 3&\eqref{CupPlusAngleR}&\eqref{ChordalDomainPlusMultitrolleybusL}&\eqref{ChordalDomainPlusMultitrolleybusL}&\eqref{ChordalDomainPlusMultibirdieR}&Fictious vertex 3\\
\hline
Fictious vertex 4&$+\infty$&&&&\\
\hline
Fictious vertex 5&Angle&Left trolleybus&Left multitrolleybus&Multibirdie&\\
\hline
\end{tabular}
}
\medskip

Now we explain how to make a reconstruction of a general limit graph~$\Gamma(\eta_2)$. First, we look at all fictious vertices of the first type, find those of them whose differentials vanish, and rename them into the vertices of the third type. Second, we temporarily forget about all the edges in our graph that have non-zero length, and consider the connectivity components~$C_i$ inside the remaining graph. We are interested in the components consisting of more than one vertex (they are exactly the ones that will become new vertices, i.e. each~$C_i$ will be changed for a single vertex in the modified graph). For each~$i$,~$C_i$ can be described as a chain of edges representing tangent domains with several single edges corresponding to chords attached to this chain from below (we note that these edges representing chords are separated, they do not have common endpoints). So, we first use the rules for chords from the first table above to get rid of edges corresponding to chords. After that we use formulas for right and left tangents to join all the vertices in the connectivity component into a single vertex (this is done step by step, each time we join two vertices). After such procedure is applied to each component, we get a graph that has no edges of zero length. Since formulas from Subsection~\ref{s344} do not change the Bellman candidate, the achieved graph is admissible.  

On the level of graphs, the reconstruction procedure is nothing but a gluing of all the vertices inside~$C_i$ together into a single vertex. %In fact, it is very easy to decide for each~$C_i$, by what type of vertex it will be replaced: we simply look at the (non-zero) edges that connect~$C_i$ to its neighbor vertices in~$\Gamma$. If there are at least two chordal domains among them, it will become a multifigure. And the edges representing tangent domains also fully define the orientation of its border tangents.
\begin{Def}\label{LimitOfSmmoothFlow}\index{graph! smooth flow of graphs! modified limit of a smooth flow of graphs}
Let~$0 \leq \eta_1 < \eta_2 < \eps_{\infty}$. Let~$\Gamma = \Gamma(\eps)$ be a smooth flow of graphs parametrized by~$(\eta_1,\eta_2)$. Let~$\Gamma(\eta_2)$ be the limit graph \textup(with\textup, possibly some zero length edges\textup). Let the~$C_i$ be the connectivity components of the graph spanned by zero-length edges of~$\Gamma(\eta_2)$. The graph~$\G$ obtained from~$\Gamma(\eta_2)$ by changing those vertices of the first type whose one or both differentials vanish by fictious  vertices of the third type\textup, and replacing each~$C_i$ by the appropriate type vertex as described above\textup, is called the \emph{modified limit} of the flow~$\Gamma(\eps)$.  
\end{Def}
\begin{Rem}
As has already been said\textup, a modified limit of a smooth flow of graphs parametrized with~$(\eta_1,\eta_2)$ is admissible for~$f$ and~$\eta_2$. 
\end{Rem}
\begin{Th}\label{InductionStepOfSecondType}
Let~$0 \leq \eta_1< \eta_2 < \eps_{\infty}$\textup, let~$\Gamma = \Gamma(\eps)$ be a smooth flow of graphs. Then its modified limit~$\G$ satisfies Conditions~\textup{\ref{NonDegeneracyForceCondition}} and~\textup{\ref{Leaf-root condition}}.
\end{Th}
\begin{proof}
The proof naturally splits into two parts.

\paragraph{Verification of Condition~\ref{NonDegeneracyForceCondition}.} We remind the reader that Condition~\ref{NonDegeneracyForceCondition} says that for each edge~$\mathfrak{E}$ in~$\GFree$ that represents a tangent domain on the interval~$(u_1,u_2)$,~$(u_1,u_2)$ is contained in the tail of the force assigned to~$\mathfrak{E}$. So, let~$\mathfrak{E}$ be an arbitrary edge in~$\GFree$, let it represent~$\Rt(u_1,u_2)$,~$u_1 < u_2$, (here~$u_1$ may equal~$-\infty$,~$u_2$ may equal~$\infty$); as usual, the case of a left tangent domain is symmetric.

Since during the modification of~$\Gamma(\eta_2)$ we did not add new tangent domains,~$\mathfrak{E}$ is present in~$\Gamma(\eta_2)$, and thus for any~$\eps$ sufficiently close to~$\eta_2$, there exists a tangent domain~$\Rt(u_1(\eps),u_2(\eps))$ in~$\Gamma(\eps)$ such that~$u_1(\eps) \to u_1$ and~$u_2(\eps) \to u_2$. Let us look first on the forces assigned to~$\mathfrak{E}$ in~$\Gamma(\eta_2)$ and in~$\G$. Both forces are the forces of some chords with the right endpoint~$u_1$ (see the table right before Condition~\ref{NonDegeneracyForceCondition}), let these chords be~$[A_1,U_1]$ and~$[A_2,U_1]$. Let~$\mathfrak{L}$ be the vertex that is the beginning of~$\mathfrak{E}$ in~$\G$. A short inspection of the modification procedure shows that the both~$A_1$ and~$A_2$ belong to the figure corresponding to~$\mathfrak{L}$. So, by Lemma~\ref{ThreePointsOnOneLine} and formula~\eqref{RightForce}, the two forces are equal. Thus, it suffices to prove that~$(u_1,u_2)$ belongs to the tail of the force assigned to~$\mathfrak{E}$ in~$\Gamma(\eta_2)$. We note that this force is the limit of forces assigned to~$\Rt(u_1(\eps),u_2(\eps))$. Let these forces be~$\Fr(\cdot\,;\eps)$. As we have seen in the proof of Theorem~\ref{InductionStepOfTheFirstKind}, for every~$\eps^* \in (\eta_1,\eta_2)$, the function~$\FFr(u\,;\eps) = \Fr(u\,;\eps)$ defined on the intersection of the domains of the forces~$\{\Fr(\cdot\,;\eps)\}_{\eps \in [\eps^*,\eta_2]}$, is a right monotone force flow (see Definition~\ref{MFF}). Since when~$\eps^* \to \eta_2$, the domain of the force flow covers~$(u_1,u_2)$, the tails of a monotone force flow grow, and~$(u_1(\eps),u_2(\eps))$ lies in the tail of~$\Fr(\cdot\,;\eps)$ for each~$\eps$ (because Condition~\ref{NonDegeneracyForceCondition} is required for a smooth flow of graphs),~$(u_1,u_2)$ belongs to the tail of~$\Fr(\cdot\,;\eta_2)$.

\paragraph{Verification of Condition~\ref{Leaf-root condition}.} This condition splits into several requirements. First, we note that we both did not add new solid arcs to the multifigures (i.e. each multifigure present in~$\G$ inherits its arcs from some multifiures in~$\Gamma(\eta_2)$, which were present in~$\Gamma(\eps)$) and did not add new fictious vertices of the second type. Since~$\Gamma(\eps)$ fulfills Condition~\ref{Leaf-root condition}, we see that any solid arc of any multifigure in~$\G$ coincides with some solid root~$c_i$ of~$f'''$, as well as all fictious vertices of the second type coincide with some roots~$c_i$ that are points.

Next, we have to verify the assertion about the endpoints of the chords corresponding to the fictious veritces of the third type. It is easy to see that fictious vertices of the third type that represent short chords are inherited from~$\Gamma(\eps)$ and thus satisfy the condition. Let~$[A_0,B_0]$ be a long chord corresponding to a fictious vertex of the third type in~$\G$. Suppose that its right differential vanishes,~$\Srt(a_0,b_0) = 0$ (the case of the left differential is symmetric). By the modification construction, there is a full flow of chordal domains~$\{\Ch([a_0,b_0],*),\eps \mapsto 2\eps\}$ in the flow~$\Gamma(\eps)$. By Lemma~\ref{zerodifrootsInside},~$f''$ increases on the left of~$b_0$. Thus, it suffices to prove that~$[A_0,B_0]$ has non-zero right tail, because in such a case, by Lemma~\ref{zerodifrootsOutside},~$f''$ is decreasing on the right of~$b_0$, and thus~$b_0 = c_i$ for some~$i$. Assume the contrary, let~$[A_0,B_0]$ have zero tail. This means that the tail of the chord~$[A(2\eps),B(2\eps)]$ is always contained in~$[B(2\eps),B_0]$, because these tails grow in~$\eps$. But there must be the root of the balance equation (= the vertex of the angle in~$\Gamma(\eps)$) with the force coming from the right inside this tail. Therefore, there is an angle in~$\Gamma(\eps)$, whose vertex tends to~$b_0$ as~$\eps \to \eta_2$. That means that~$[A_0,B_0]$ is not a fictious vertex of the third type in~$\G$, a contradiction.

Finally, a fictious vertex of the fifth type in~$\G$ might appear only from from a fictious vertex of the second type in~$\Gamma(\eta_2)$, which is present in~$\Gamma(\eps)$ for all~$\eps$. Since the graphs inside a smooth flow fulfill Condition~\ref{Leaf-root condition}, any tangent corresponding to a fictious vertex of the fifth type in~$\G$ sits on some~$c_i$. 
\end{proof}

\subsection{Construction of Bellman candidate}\label{s443}
\begin{Th}\label{BC}
For any~$\eps < \eps_{\infty}$\textup, there exists a graph~$\Gamma(\eps)$ admissible for~$f$ and~$\eps$.
\end{Th}
\begin{proof}
We will prove a stronger statement: for any~$\eps$ there exists a graph~$\Gamma(\eps)$ admissible for~$f$ and~$\eps$ that satisfies Conditions~\ref{NonDegeneracyForceCondition} and~\ref{Leaf-root condition}.

By Theorem~\ref{SimplePicture}, for every~$\eps$ sufficiently small, there exists a simple graph~$\Gamma(\eps)$ that is admissible for~$f$ and~$\eps$ and it also satisfies the aforementioned conditions (see Remark~\ref{nondegforcecondSimplePictureRem}). Fix some~$\eps_0$ that is sufficiently small such that the graph~$\Gamma(\eps_0)$ is smooth and satisfies Conditions~\ref{NonDegeneracyForceCondition} and~\ref{Leaf-root condition}. We construct a sequence $\{\eps_n\}_{n\geq 0}$ inductively. If $\eps_n$ is given, we apply Theorem~\ref{InductionStepOfTheFirstKind} to get smooth admissible for~$f$ and~$\eps$ graphs~$\G(\eps)$ that satisfy Conditions~\ref{NonDegeneracyForceCondition} and~\ref{Leaf-root condition}, here~$\eps \in (\eps_n,\eps_n^+)$ for some $\eps_n^+>\eps_n$. We fix a moment~$\eps_n' \in (\eps_n,\eps_n^+)$ and consider the set~$\evo_{n+1} \subset \mathbb{R}$ given by the formula
\begin{equation*}
\evo_{n+1}\df\Bigg\{\eta \in (\eps_n',\eps_{\infty})\,\Bigg|\; \begin{aligned}\hbox{there exists a smooth flow of graphs starting from~$\G(\eps_n')$}\\ \hbox{and parametrized by~$(\eps_n',\eta)$}\end{aligned}\Bigg\}.
\end{equation*} 
By Theorem~\ref{SmoothFlowTheorem} (the existence part), the set~$\evo_{n+1}$ is non-empty. We define its supremum by~$\eps_{n+1}$. If~$\eps_{n+1} =\eps_{\infty}$, we finish. Otherwise, by the uniqueness part of Theorem~\ref{SmoothFlowTheorem}, there exists a smooth flow of graphs starting from~$\G(\eps_n')$ and parametrized by~$(\eps_{n}',\eps_{n+1})$. We apply Theorem~\ref{InductionStepOfSecondType}  and see that the modified limit~$\G = \G(\eps_{n+1})$ of this flow is admissible for~$f$ and~$\eps_{n+1}$ and satisfies Conditions~\ref{NonDegeneracyForceCondition} and~\ref{Leaf-root condition}. The induction step is finished.

It suffices to prove that the algorithm stops eventually, i.e.~$\eps_n = \eps_{\infty}$ for some~$n$. This happens due to Condition~\ref{Leaf-root condition}, the rules of evolution (written down in the table right after Definition~\ref{Admissible graph}), and the modification procedure. We will find out that several combinatorial characteristics of~$\G(\eps_n)$ satisfy certain monotonicity laws, are bounded, and thus have to stabilize when~$n$ is sufficiently big. After they all have stabilized, the set~$\evo_n$ has~$\eps_{\infty}$ as its supremum. The first characteristics is the number of roots in~$\G$. 

\paragraph{The number of roots in~$\G(\eps_n)$ does not increase.} It follows from local evolutional rules that  (in the above terminology)
\begin{equation*}
\Big|\big\{\hbox{Roots of~$\G(\eps_{n})$}\big\}\Big| = \Big|\big\{\hbox{Roots of~$\G(\eps_{n}')$}\big\}\Big|.
\end{equation*}  
Indeed, a root of a graph is either a multicup, or a fictious vertex of the first, third, or fourth type. There are no fictious vertices of the third type representing long chords in~$\G(\eps_n')$ since this graph is smooth. As for multicups, they all correspond to some non-full multicups in~$\G(\eps_n)$. The same with fictious vertices of the fourth type (they are the same). Fictious vertices of the first type originated in~$\G(\eps_n')$ either from a full multicup or from a fictious vertex of the first or third types. So, we see that the roots of~$\G(\eps_n')$ are in bijection with the roots of~$\G(\eps_n)$.

It is also clear that two graphs extracted from a smooth flow of graphs have the same collections of roots. Thus, to prove that
\begin{equation*}
\Big|\big\{\hbox{Roots of~$\G(\eps_{n+1})$}\big\}\Big| \leq \Big|\big\{\hbox{Roots of~$\G(\eps_{n})$}\big\}\Big|,
\end{equation*} 
it suffices to show that the number of the roots of a graph does not increase during the modification procedure. This is clear, since the modification is nothing but gluing several vertices of a graph into a single vertex. Such a change of a graph clearly does not increase the number of roots (we use the fact that our graphs are oriented trees).    

\paragraph{The number of leaves in~$\G(\eps_n)$ does not increase.} As in the previous case, it suffices to prove an inequality and an equality,
\begin{equation*}
\Big|\big\{\hbox{Leaves of~$\G(\eps_{n+1})$}\big\}\Big| \leq \Big|\big\{\hbox{Leaves of~$\G(\eps_{n}')$}\big\}\Big| = \Big|\big\{\hbox{Leaves of~$\G(\eps_{n})$}\big\}\Big|.
\end{equation*}
The proof of the inequality is totally similar to the previous case and follows from the fact that gluing together several vertices in an oriented tree does not increase the number of leaves in it. To prove the equality, we need to consider the leaves in~$\G(\eps_n')$. They are angles, closed multicups sitting on single arcs, fictious vertices of the second and, maybe, fourth types. Angles in~$\G(\eps_n')$ appear from angles, birdies, and multibirdies of~$\G(\eps_n)$. As for fictious vertices of the second and fourth types, they are the same in~$\G(\eps_n)$ and~$\G(\eps_n')$. Therefore, we see that the number of leaves in~$\G(\eps_n')$ equals the same number for~$\G(\eps_n)$.

We may pick a number~$n_0$ such that  the number of roots and leaves in~$\G(\eps_n)$ does not change for~$n > n_0$. In what follows, we assume that~$n > n_0$. Consider now all the leaves in~$\G(\eps_n)$ that have only one incoming edge. They are the closed multicups sitting on single arcs, fictious vertices of the second type, and infinities. By Condition~\ref{Leaf-root condition}, the closed multicups sitting on single arcs as well as fictious vertices of the second type coincide with some roots~$c_i$ (in a sense, the infinities also satisfy such a condition).  Let~$\mathfrak{L}$ be a vertex in~$\GFree(\eps_n)$, consider the set~$\mathfrak{L}^{\mathrm{sons}}$ of all leaves in~$\G(\eps_n)$ subordinated to it in the sense Definition~\ref{Ordering}. As we have seen,~$\mathfrak{L}^{\mathrm{sons}}$ may be identified with a subset of the set~$\{c_i\}$. Thus, for every~$n$ we have a collection of non-intersecting subsets of the set~$\{c_i\}$ generated by the sets~$\mathfrak{L}^{\mathrm{sons}}$. A set from this collection is called full if the vertex~$\mathfrak{L}$ that generates it, represents a root of~$\G(\eps_n)$ (i.e. it is either a multicup, or a fictious vertex of the first and third type).

\paragraph{The full subsets of~$\{c_i\}$ do not shrink in~$n$.} By this we mean that if~$A(n)$ is a full subset of~$\{c_i\}$ for~$\eps_n$, then for any~$m > n$ there exists a full subset~$A(m)$ generated by~$\GFree(\eps_m)$ such that~$A(n) \subset A(m)$. We note that in general, such  statement is incorrect. However, if we assume that the number of roots and leaves of~$\G(\eps_n)$ has stabilized, it becomes true. Surely, it suffices to prove the statement for the case~$m=n+1$.

Let~$A(n)$ be generated by the vertex~$\mathfrak{L}$ of~$\G(\eps_n)$. The vertex~$\mathfrak{L}$ may represent a multicup (full or not) or a fictious vertex of the first or third type. In any case, we can easily find a root~$\mathfrak{L}'$ of~$\G(\eps_n')$ that generates~$A(n)$ for~$\eps_n'$. The roots of the graph do not change along a smooth flow of graphs. So, it suffices to prove that there is a root of~$\G(\eps_{n+1})$ that generates a set that contains~$A(n)$. Consider the component~$C_i$ the vertex~$\mathfrak{L}'$ belongs to in the limit graph before the modification. Since the number of the roots of the graph does not decrease during the modification (we have assumed that it had already stabilized), the vertex that is present in the modified graph instead of~$C_i$, is a root (because~$C_i$ contains a root vertex, therefore, if~$C_i$ is replaced by a non-root vertex, the number of roots decreases). This modified vertex generates a set that contains~$A(n)$.

We have proved that full sets do not decrease during the evolution. Thus, there is a moment~$n_1 \geq n_0$ such that for any~$n > n_1$ the collection of full subsets of~$\{c_i\}$ does not change\footnote{In fact, there is not more than two full subsets.}. In what follows, we assume that~$n > n_1$. Consider now any root~$c_i$ that does not belong to a full set. Let~$c_i^{\mathrm{parent}}(n)$ be the set of the partition of~$\{c_i\}$ generated by~$\Gamma(\eps_n)$ such that~$c_i \subset c_i^{\mathrm{parent}}(n)$.

\paragraph{The set~$c_i^{\mathrm{parent}}(n)$ does not increase in~$n$ by inclusion.} By this we mean that~$c_i^{\mathrm{parent}}(n)$ generated by~$\G(\eps_n)$ contains the set~$c_i^{\mathrm{parent}}(n+1)$ generated by~$\G(\eps_{n+1})$ (as in the previous case, such a claim is not true in general, the assumption that the full sets have stabilized is required). It is easy to see that the sets of the partition do not increase when we pass from~$\G(\eps_n)$ to~$\G(\eps_n')$ (because~$c_i$ is situated somewhere under a trolleybus, or a multitrolleybus, or a multibirdie, or a fictious vertex of the fifth type; in the first case the corresponding set~$c_i^{\mathrm{parent}}(n)$ does not change, in the second, the third, and the fourth it shrinks and may even disappear from the graph). While we move along a smooth flow, nothing changes. Thus, we have to show that the set~$c_i^{\mathrm{parent}}(n)$ cannot catch other roots~$c_j$ during the modification. By the assumption that the full sets in the partition are stable, we see that tangent domains cannot have zero length in the limit graph (except, maybe, for tangent domains coming from the infinities and a trivial case of a tangent domain between a trolleybus and an angle). Indeed, if we have a tangent domain that has zero length (and it is not a tangent domain between a trolleybus and an angle), then this is not a tangent domain between the two decreasing flows of chordal domains (simply because they decrease); therefore, it was connected either to a full chordal domain flow, or a multicup; in both cases the collection of full sets changed after the modification. So, the only modification available is a splitting of a trolleybus or a birdie into several similar figures, which leads to the diminishing of the partition set~$c_i^{\mathrm{parent}}(n)$.

There exists a moment~$n_2 \geq n_1$ such that for all~$n > n_2$ the partition of~$\{c_i\}$ does not change at all. After that we may have only a very limited number of evolutional scenarios: during the passage from~$\G(\eps_n)$ to~$\G(\eps_n')$ a full multicup may become closed (with a chordal domain and a fictious vertex on its top) and a fictious vertex of the third type that represents a long chord may give rise to a new chordal domain above it together with the fictious vertex of the first type representing the upper chord. As for the modification of the limiting graph (the creation of~$\G(\eps_{n+1})$), there might be only a renaming of fictious vertex of the first type into a fictious vertex of the third type. 

It is clear that there might be only a finite number of multicup closings. The situation with the fictious vertices of the third type is a bit more difficult. It suffices to prove that there might be only a finite number of them born during modifications. For that we use Condition~\ref{Leaf-root condition} to see that each fictious vertex of the third type has one of the points~$(c_i,c_i^2)$ as its endpoint. Thus, there might be only a finite number of them in a graph, because when~$n > n_2$, each such new-born chord generates one of the full subsets of~$\{c_i\}$. So, the number of such chords does not exceed the number of the full subsets of~$\{c_i\}$ times the number of the roots~$c_i$.

We have proved that~$\eps_n = \eps_{\infty}$ when~$n$ is sufficiently big and thus built admissible graphs for all~$\eps \in (0,\eps_{\infty})$.
\end{proof}
By Remark~\ref{AdmissibleCondidate}, we have built a Bellman candidate for all~$\eps \in (0,\eps_{\infty})$. After we prove that the Bellman candidate built from an admissible graph coincides with the Bellman function~\eqref{Bf}, we will prove the uniqueness of such a Bellman candidate. Therefore, the moments~$\eps_n$ do not depend on the construction, but on the function~$f$ only. It is natural to call them \emph{essential critical points}\index{essentially critical points} of the evolution. We have proved that there is only a finite number of them.

One may say that the definition of an essential critical point is unfair, because at the moment when an angle attacks a trolleybus, the graph of the foliation changes. It is natural to consider the moment~$\eps$ when such a crash happens as a critical point. However, the second example of Section~\ref{s54} shows that the number of such critical points may be infinite. 

The proof of Theorem~\ref{BC} provides an algorithm for calculating the Bellman candidate. It constructs the Bellman candidate for all~$\eps \in (0,\eps_{\infty})$ except, possibly, for a finite number of intervals (each interval starts at essentially critical point), which can be chosen to be as small as we want, in a finite number of steps. By a step of this algorithm we mean a differentiation, an integration, or a solution of an algebraic equation. We will not describe it rigorously, because such a description will be essentially a repetition of the proof of Theorem~\ref{BC}. The reader is welcome to read the next section (and especially the paper~\cite{CrazySine}) to see how the algorithm works in practice.

\section{Examples}\label{s45}
\paragraph{The sine monster.}  Here we list several especially complicated foliations that appear
during the evolution with the boundary condition~\eqref{SineMonsterFormula}. We only survey the results, the reader is welcome to consult~\cite{CrazySine} for the hints to calculations. 

Let~$\alpha_n$ be the~$n$th positive root of the equation
\begin{equation}
\label{tan}
\tan \zeta=\zeta,
\end{equation}
i.e.~$\alpha_n\in\big(n\pi,(n+1)\pi\big)$. Consider the case when
$\alpha\in\big((2N-1)\pi,\alpha_{2N-1}\big)$. Then at some moment $\eps=\eps_1(\alpha)\in(0,\alpha)$ 
the most left and the most right angles escape to infinity. Or more exactly, at this moment these
angles touch the corresponding infinite multicups on the solid roots $(-\infty,-\alpha)$ and 
$(\alpha,+\infty)$ turning them into multitrolleybusses which immediately disappear.
All other figures evolve independently. Namely, at the moment~$\eps=\pi$ all~$2N-1$ cups touch 
their neighbors and the~$2N-2$ angles in between form a common multicup. This multicup completes 
its evolution at the moment $\eps=(2N-1)\pi$, when it becomes a full multicup. After this moment
the full chordal domain grows above this multicup till infinity. But for $\alpha=\alpha_{2N-1}$
the situation is very special, because the chordal domain above the closed multicup grows 
till the moment $\eps=\alpha_{2N-1}$, when it catches the angles at the last moment of escaping 
to the corresponding infinities, that is the moment when they form infinite linearity domains. 
So, a domain of linearity is formed above the chord of length~$2\alpha_{2N-1}$ 
and it is a stable infinite multicup. This means that for~$\eps>\alpha_{2N-1}$ the Bellman 
function does not change anymore, only its domain is increasing. The graph of the corresponding
foliation is presented on Fig.~\ref{InfiniteMcup}
\begin{figure}[h]
\vskip-20pt
\centering{\includegraphics{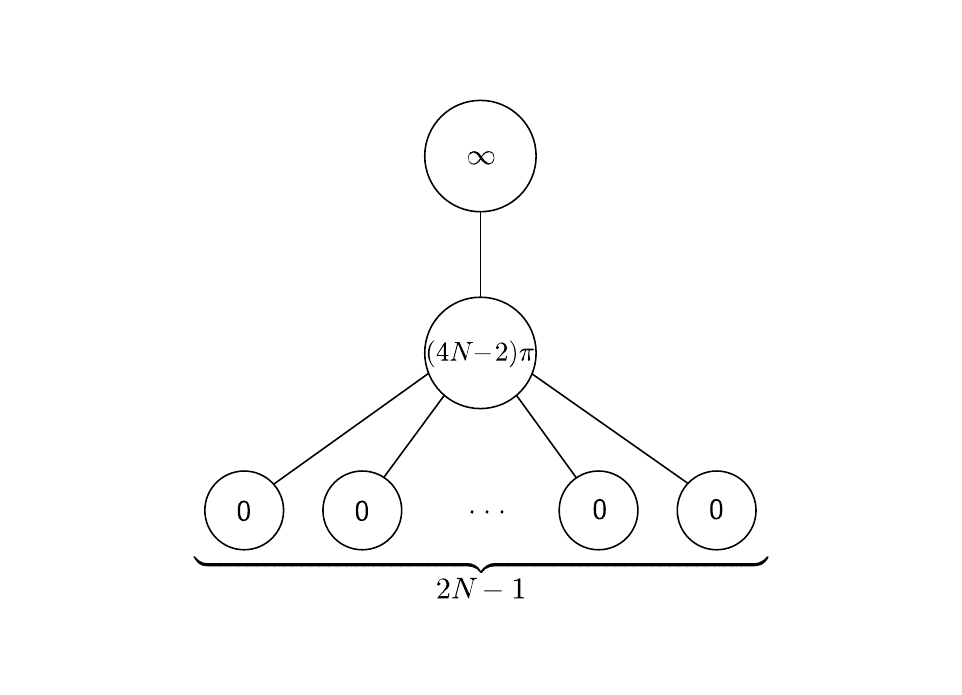}}
\vskip-20pt
\caption{The foliation for $\eps>\alpha=\alpha_{2N-1}$.}
\label{InfiniteMcup}
\end{figure}

If the parameter $\alpha$ is slightly bigger, the vertex of the rightmost angle attains its rightmost 
position (which is smaller than~$\alpha$) and turns back moving to the left until the moment 
when the angle meets the chordal domain above the big multicup. Due to symmetry, the left and
right angles meet the chordal domain simultaneously forming a birdie. It shrinks as 
$\eps$ increases, but never disappears if $\alpha_{2N-1}<\alpha\le(2N-\frac12)\pi$.
If~$\alpha>(2N-\frac12)\pi$, then one more essentially critical point of evolution appears. 
Namely, this is the moment when the base of the birdie shrinks to the size~$(4N-2)\pi$, i.e.
when the birdie turns into a multibirdie. At the same moment the 
multibirdie splits into~$2N-1$ pieces: a central birdie and~$N-1$ pairs of symmetrical 
trolleybuses, and after that these figures are shrinking, but none of them dies. 

Such kind of evolution occurs if $\alpha<\lambda_{2N-1}$, where $\lambda_n$ is the 
unique solution~$\zeta$ of the equation
\begin{equation}
\label{lambda-n}
n^2\pi^2\cos\zeta+n\pi\sin\zeta-e^{\frac\zeta{n\pi}-1}=0
\end{equation}
belonging to~$\big((n+\frac12)\pi,(n+1)\pi\big)$.
If~$\alpha>\lambda_{2N-1}$, the evolution changes slightly: the multicup does not have sufficiently much time 
to complete its evolution, and two symmetrical angles hit the multicup forming a
multibirdie. But except this moment, the rest of the evolution is the same: the multibirdie
immediately splits into a central birdie and~$N-1$ pairs of symmetric
trolleybuses, and after that these figures are shrinking, but none of them dies.

Such evolution occurs if $\alpha<\mu_{2N-1}$, where $\mu_n$ is the unique solution 
$\zeta$ of the equation
\begin{equation}
\label{mu-n}
\pi^2\cos\zeta+\pi\sin\zeta-e^{\frac\zeta\pi-n}=0\,
\end{equation}
belonging to~$\big((n+\frac12)\pi,(n+1)\pi\big)$. For bigger values of~$\alpha$, the angle
hits the neighbor cup earlier (at a moment $\eps<\pi$), when all the cups are still 
separated and there is no multicup. For this last value of the parameter
$\alpha=\mu_{2N-1}$ we have the only essentially critical point~$\eps=\pi$, when all~$4N-1$ figures ($2N-1$ cups and $2N$ angles) meet together 
and all angles turn into a common multibirdie. Due to symmetry, this multibirdie immediately 
splits into a central birdie and~$N-1$ symmetric pairs of trolleybuses. 

Now we can add a few comments about the crumbling of the multibirdie in all described cases. 
Since our boundary values are symmetric,  the multibirdie disintegrates into a symmetric
picture with a birdie in the middle.
However, as it is described in Section~\ref{s54} below, we can add an arbitrarily small and smooth 
perturbation to the boundary values in such a way that the crumbling of the multibirdie chooses 
an arbitrary scenario.
Namely, in the result, we can get an angle at any place with an arbitrary number of trolleybuses on
the right and left (only their total number~$2N-1$ is fixed). Moreover, before the position
of the angle stabilizes, it could wander between the trolleybuses in arbitrary way changing the
number of the left and right trolleybuses. 

Let us return to the evolution in the case when $\alpha>\mu_{2N-1}$ (or, more exactly,
when $\mu_{2N-1}<\alpha\le2N\pi$). Now, as in all cases considered before, except 
$\alpha=\alpha_{2N-1}$, two infinite multicups on the solid roots $(-\infty,-\alpha)$ and
$(\alpha,\infty)$ are stable and do not participate in the process of evolution. For small values of~$\eps$, when we have a simple picture, all  the angles sit constantly at the points $\pm (2k-1)\pi$,
$1\le k\le N-1$. But the angles originating from the points $\pm(2N-1)\pi$ move to their neighbor
cups and hit them at some moment~$\eps_1(\alpha)$, transforming the cups into trolleybus. For $\eps>\eps_1$,
the next pair of angles sitting at $\pm(2N-3)\pi$ begins to move towards the origin and hits their neighbor
cups at a moment~$\eps_2(\alpha)$. This is the moment when the second pair of trolleybuses arises.
The evolution continues as a ripple effect: we get an increasing sequence of moments~$\eps_k$ such
that for~$\eps<\eps_k<\pi$ the angles sitting at $\pm (2j-1)\pi$, $j\le N-k$, are stable, while the
angles originating from~$\pm (2N-2k+1)\pi$ move towards  the origin. At the moment~$\eps=\eps_k$ they
hit their neighbor cups, a new pair of trolleybuses arises, and the next pair of angles (sitting
at~$\pm (2N-2k-1)\pi$) begins to move. Finally, at the moment $\eps_N$ the last pair of angles
originating from $\pm\pi$ hits the central cup, forming a birdie. After that all the figures shrink
but none of them die.

The same ripple effect occurs for $\alpha\in(2N\pi,(2N+1)\pi)$, but now instead of two
stable multicups at infinite rays we have growing cups originating from the points $\pm2N\pi$.
Due to symmetry, we will discuss the behavior of the cup originating from~$2N\pi$ only.
If the value of $\alpha$ is sufficiently close to~$2N\pi$, then this cup grows infinitely in such a way
that its left end tends to a finite limit not touching the shrinking neighbor trolleybus
that appears as a result of the described ripple effect. But if $\alpha>\alpha_2+2(N-1)\pi$, then 
this cup catches the neighbor trolleybus, forming a multicup. This multicup
completes its evolution, and the chordal domain begins to grow above it. 
If~$\alpha\le\alpha_4+2(N-2)\pi$, this chordal domain grows in such a way that its left
end tends to a finite limit not touching the next trolleybus. However, if
$\alpha>\alpha_4+2(N-2)\pi$, this chordal domain catches the next trolleybus.
So, we have, in a sense, a second ripple effect: when~$\alpha_{2N}<\alpha<(2N-1)\pi$,
all the trolleybuses are caught and at some moment both huge chordal domains
touch the central birdie forming a common multicup. After that moment the evolution is
fairly simple: the multicup at some moment turns into a closed multicup, and a
 chordal domain starts to grow till infinity above it.

Finally, we note that all the critical points of evolution described above
tend to~$\pi$ as $\alpha\to(2N+1)\pi$. And if~$\alpha=(2N+1)\pi$, all these processes
merge into one, when  all cups and angles collide forming a multicup
with $2(N+1)$ boundary points at the moment $\eps=\pi$.

\paragraph{Polynomial of sixth degree: positive leading coefficient.}

This example is written down in full detail to show how to apply the developed machinery to a ``lively'' problem. This material is unnecessary for the general theory, an uninterested reader may skip it without any possible loss of understanding.

Consider a polynomial~$f$ such that~$f'''(t)=t^3-3dt+c$ (the case of an arbitrary polynomial with positive leading coefficient may be reduced to this one with the help of a linear change of variable and Remark~\ref{remm2}).
If~$d\le0$ or~$d>0$ and~$|c|\ge2d^{3/2}$, then~$f'''$ has only one essential root 
and, by Theorem~\ref{AngleProp}, the foliation consists of an angle and two tangent domains. In what follows we will
assume~$d>0$ and~$|c|<2d^{3/2}$. Moreover, without loss of generality we 
may assume~$d=1$ and~$0\le c<2$. The picture for arbitrary~$d$ can
be obtained by rescaling~$c\to d^{3/2}c$,~$x_1\to\sqrt{d}x_1$,~$x_2\to dx_2$,~$\eps\to\sqrt{d}\eps$, see Remark~\ref{remm2}.

Under this assumption~$f'''$ has three essential roots and, by Theorem~\ref{SimplePicture}, the foliation consists of two angles and one cup for sufficiently small~$\eps$. Let us find 
the first critical value~$\eps=\eps_1$ such that the foliation is as stated for all~$\eps\le\eps_1$, but it is not the case for~$\eps>\eps_1$. For this aim we need to calculate  the sum of forces, we call this function~$F$ for brevity.

The cup equation is
\begin{equation}
\label{cupeq}
(a-b)^2\Big[\frac1{120}(2a^3+3a^2b+3ab^2+2b^3)-\frac18(a+b)+\frac{c}{12}\Big]=0.
\end{equation}
The differentials of the cup are
\begin{align*}
\Slt&=(a-b)\Big[\frac1{20}(4a^3+3a^2b+2ab^2+b^3)-\frac12(2a+b)+\frac{c}2\Big],
\\
\Srt&=(b-a)\Big[\frac1{20}(a^3+2a^2b+3ab^2+4b^3)-\frac12(a+2b)+\frac{c}2\Big].
\end{align*}
Using the cup equation we rewrite them in the form
\begin{align*}
\Slt&=\frac1{20}(a-b)^2(2a^2+2ab+b^2-5),
\\
\Srt&=\frac1{20}(a-b)^2(a^2+2ab+2b^2-5).
\end{align*}
If we introduce\footnote{Attention: here~$\ell$ is not the one we usually use! Indeed, when treating a chordal domain in Subsection~\ref{s331} we had~$\ell = b-a$. However, to get rid of giant numerical coefficients, we divide~$\ell$ by two in this example.}~$w\df(b+a)/2$ and~$\ell\df(b-a)/2$, then the cup equation turns into
\begin{equation}
\label{cupeq1}
w^3-3w\Big(1-\frac{\ell^2}5\Big)+c=0
\end{equation}
and the expressions for the differentials are
\begin{align*}
\Slt&=\frac15\ell^2(5w^2-2w\ell+\ell^2-5),
\\
\Srt&=\frac15\ell^2(5w^2+2w\ell+\ell^2-5).
\end{align*}
Since the origin of the cup is the middle solution of the equation
$$
t^3-3t+c=0,
$$
we need 
to take the middle solution of equation~\eqref{cupeq1} (when~$\ell^2<5\big[1-(c/2)^{2/3}\big]$), because it tends to the origin of the cup as~$\ell\to0$. So, by writing~$w$ or~$w(\ell)$ we always mean this solution. Since we assume that~$c>0$, the cup 
equation~\eqref{cupeq1} has two positive and one negative root, and therefore, the midpoint
of our cup~$w=w(\ell)$ is always positive. In the case~$c=0$ there is a symmetry and we have~$w=0$
identically.

Now, we calculate the function~$F$.
On the half-line~$u\in[b,\infty)$ we have:
\begin{align*}
\Fl(u;\infty;\eps)&=\eps^{-1}e^{u/\eps}\int\limits^{\infty}_ue^{-t/\eps}\,df''(t)
\\
&=(u^3-3u+c)+3\eps(u^2-1)+6\eps^2u+6\eps^3,
\\
\Fr(u;a,b;\eps)&=\eps^{-1}\Srt e^{-(u-b)/\eps}+\eps^{-1}e^{-u/\eps}\int\limits_b^ue^{t/\eps}\,df''(t)
\\
&=\frac1{5\eps}\ell^2(5w^2+2w\ell+\ell^2-5)e^{-(u-b)/\eps}
\\
&\quad+\big[(u^3-3u+c)-3\eps(u^2-1)+6\eps^2u-6\eps^3\big]
\\
&\quad-\big[(b^3-3b+c)-3\eps(b^2-1)+6\eps^2b-6\eps^3\big]e^{-(u-b)/\eps},
\end{align*}
\begin{equation}
\label{F+}
\begin{aligned}
F(u)&=\mr{u}{b}+\ml{u}{\infty}
\\
&=\frac1{5\eps}\ell^2(5w^2+2w\ell+\ell^2-5)e^{-(u-b)/\eps}
\\
&\quad-\big[(b^3-3b+c)-3\eps(b^2-1)+6\eps^2b-6\eps^3\big]e^{-(u-b)/\eps}
\\
&\quad+2(u^3-3u+c)+12\eps^2u.
\end{aligned}
\end{equation}

Under the assumption~$\ell=\eps$ we find the value~$\eps$ such that the balance equation has a root at~$b$, i.e.~$F(b)=0$:
\begin{align*}
F(b)&=\frac15\eps(5w^2+2w\eps+\eps^2-5)
\\
&\quad-\big[(b^3-3b+c)-3\eps(b^2-1)+6\eps^2b-6\eps^3\big]
\\
&\quad+2(b^3-3b+c)+12\eps^2b
\\
&=\eps w^2+\frac25w\eps^2+\frac15\eps^3-\eps
\\
&\quad+\big((w+\eps)^3-3(w+\eps)+c\big)+3\eps\big((w+\eps)^2-1\big)+6\eps^2(w+\eps)+6\eps^3
\\
&=7\eps w^2+\frac{74}5\eps^2 w+\frac{81}5\eps^3-7\eps.
\end{align*}
Here we have used equation~\eqref{cupeq1}. As a result, we get the following system
of two equations:
\begin{eqnarray}
\label{1+}
35w^2+74\eps w+81\eps^2-35=0;
\\
\label{2+}
5w^3-3(5-\eps^2)w+5c=0,
\end{eqnarray}
whose solution~$\eps=\eps(c)$ is the desired critical value of~$\eps$. We cannot 
find this function explicitly, but it is easy to write down the inverse function~$c(\eps)$
if we express~$w$ in terms of~$\eps$ from equation~\eqref{1+},
$$
w_+(\eps)=\frac{-37\eps+\sqrt{1225-1466\eps^2}}{35}
$$
(we need a positive root) and then plug it into equation~\eqref{2+}:
\begin{equation}\label{c+}
c_+(\eps)=w_+(\eps)\Big(3-\frac35\eps^2-w_+(\eps)^2\Big).
\end{equation}
\begin{figure}[ht]
\centering{\includegraphics[width=12cm]{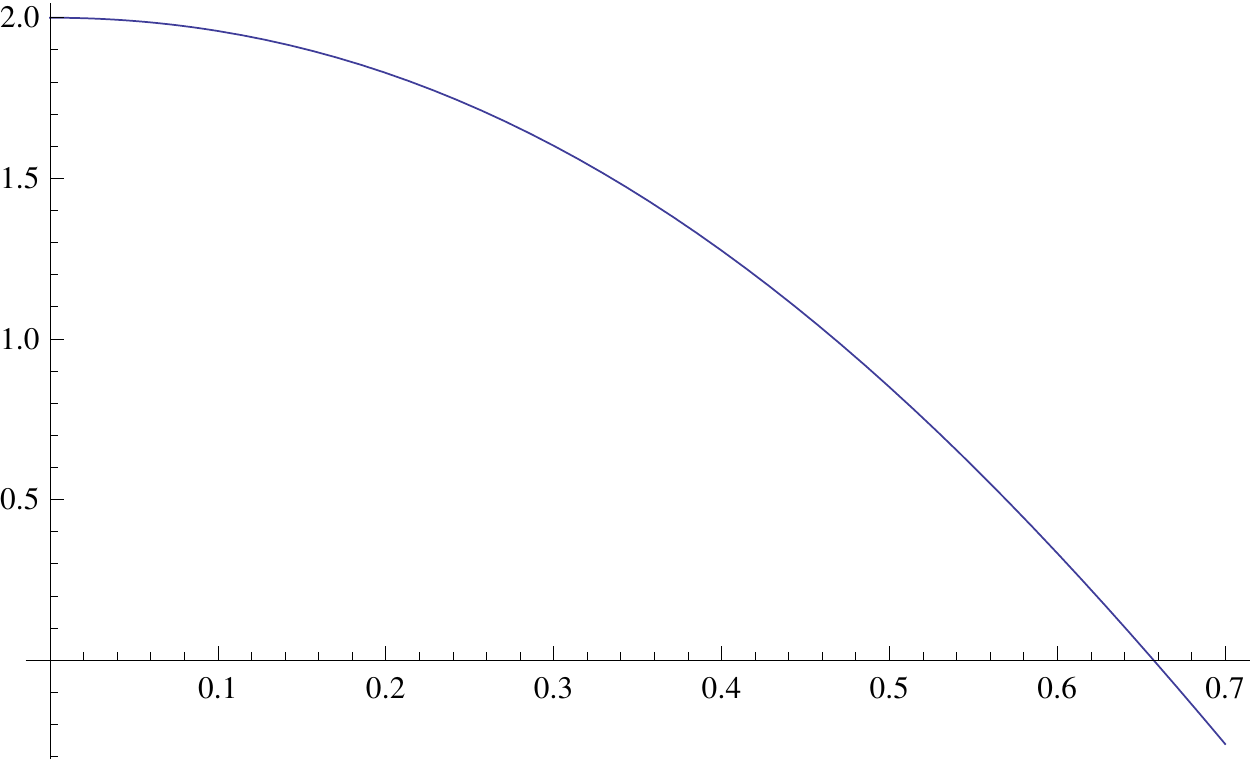}}
\caption{Graph of the function $c_+(\eps)$.}
\label{figc_+}
\end{figure}

We must verify that the function~$c_+$ is strictly decreasing on the interval~$[0,\frac{\sqrt{35}}9]$ from~$2$ to~$0$. First of all, we note that the function~$w_+$ decreases from~$1$ to~$0$ when~$\eps$ runs from~$0$ to~$\frac{\sqrt{35}}9$. 
Furthermore,
$$
c'_+(\eps)=3w'_+(\eps)\Big(1-\frac15\eps^2-w_+(\eps)^2\Big)-\frac65\eps w_+(\eps),
$$
$w_+(\eps)\ge0$, $w'_+(\eps)<0$. Therefore, to verify that $c_+$ is 
decreasing, it is sufficient to check that
$$
1-\frac15\eps^2-w_+(\eps)^2\ge0.
$$
But, using equation~\eqref{1+} for~$w_+$, we get
$$
1-\frac15\eps^2-w_+(\eps)^2=\frac{74\eps\big(w_+(\eps)+\eps\big)}{35}\ge0.
$$

Thus, we have proved that~$c_+$ maps the interval~$[0,\frac{\sqrt{35}}9]$ bijectively
onto~$[0,2]$. Therefore, the inverse function is well-defined on the interval~$[0,2]$,
i.e. for every~$c\in[0,2)$ we define~$\eps_1=\eps_1(c)$ as the solution of the equation~$c_+(\eps)=c$. When~$c$ runs over~$[0,2)$ the corresponding value~$\eps_1$ runs over~$(0,\frac{\sqrt{35}}9]$.

Now, consider the half-line~$(-\infty,a]$:
\begin{align*}
\Fr(u;-\infty;\eps)&=\eps^{-1}e^{-u/\eps}\int\limits_{-\infty}^ue^{t/\eps}\,df''(t)
\\
&=(u^3-3u+c)-3\eps(u^2-1)+6\eps^2u-6\eps^3,
\\
\Fl(u;a,b;\eps)&=-\eps^{-1}\Slt e^{-(a-u)/\eps}+\eps^{-1}e^{u/\eps}\int\limits_u^ae^{-t/\eps}\,df''(t)
\\
&=-\frac1{5\eps}\ell^2(5w^2-2w\ell+\ell^2-5)e^{-(a-u)/\eps}
\\
&\quad-\big[(a^3-3a+c)+3\eps(a^2-1)+6\eps^2a+6\eps^3\big]e^{-(a-u)/\eps}
\\
&\quad+\big[(u^3-3u+c)+3\eps(u^2-1)+6\eps^2u+6\eps^3\big],
\end{align*}
\begin{equation}
\label{F-}
\begin{aligned}
F(u)&=\mr{u}{-\infty}+\ml{u}{a}
\\
&=-\frac1{5\eps}\ell^2(5w^2-2w\ell+\ell^2-5)e^{-(a-u)/\eps}
\\
&\quad-\big[(a^3-3a+c)+3\eps(a^2-1)+6\eps^2a+6\eps^3\big]e^{-(a-u)/\eps}
\\
&\quad+2(u^3-3u+c)+12\eps^2u.
\end{aligned}
\end{equation}

We solve the symmetric equation~$F(a)=0$ under assumption~$\ell=\eps$.
\begin{align*}
F(a)&=-\frac15\eps(5w^2-2w\eps+\eps^2-5)
\\
&\quad-\big[(a^3-3a+c)+3\eps(a^2-1)+6\eps^2a+6\eps^3\big]
\\
&\quad+2(a^3-3a+c)+12\eps^2a
\\
&=-\eps w^2+\frac25w\eps^2-\frac15\eps^3+\eps
\\
&\quad+\big((w-\eps)^3-3(w-\eps)+c\big)-3\eps\big((w-\eps)^2-1\big)+6\eps^2(w-\eps)-6\eps^3
\\
&=-7\eps w^2+\frac{74}5\eps^2 w-\frac{81}5\eps^3+7\eps.
\end{align*}
Here we used relation~\eqref{cupeq1}. As a result, we get a similar system
of two equations,
\begin{gather}
%\begin{cases}
\label{1-}
35w^2-74\eps w+81\eps^2-35=0;
\\
\label{2-}
5w^3-3(5-\eps^2)w+5c=0.
%\end{cases}
\end{gather}
It turns into system~\eqref{1+}--\eqref{2+} if we replace~$\eps$ by~$-\eps$.
As before, we express~$w$ in terms of~$\eps$ from~\eqref{1-},
$$
w_-(\eps)=\frac{37\eps+\sqrt{1225-1466\eps^2}}{35}
$$
(again, we take the only positive root when~$\eps<\frac{\sqrt{35}}9$). Plugging it 
into~\eqref{2-} we get
$$
c_-(\eps)=w_-(\eps)\Big(3-\frac35\eps^2-w_-(\eps)^2\Big).
$$
\begin{figure}[ht]
\centering{\includegraphics[width=12cm]{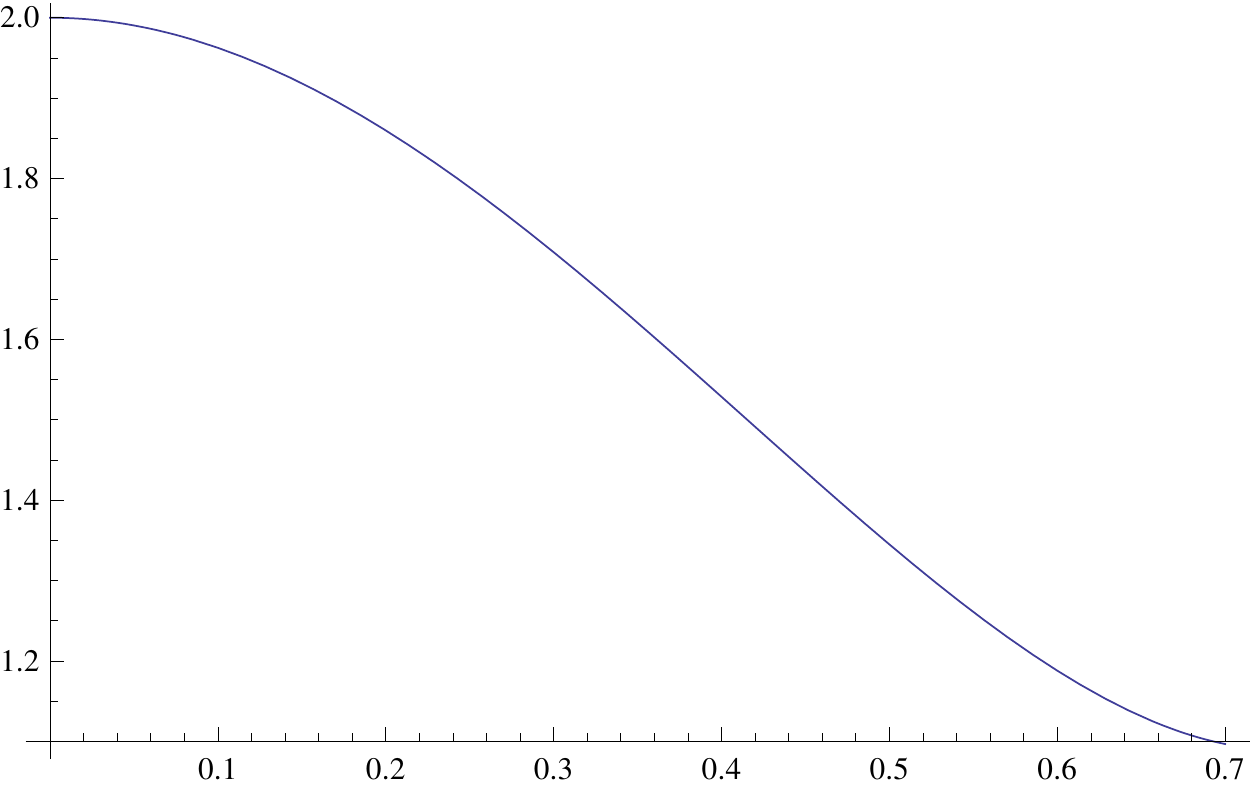}}
\caption{Graph of the function $c_-(\eps)$.}
\label{figc_-}
\end{figure}

We will not check that the function~$c_-$ is a monotone function and the equation~$c_-(\eps)=c$ has a unique solution. What we need to check is that all the possible  
solutions are greater than~$\eps_1$, i.e. greater than the unique solution of the
equation~$c_+(\eps)=c$. To get this, it is sufficient to check that~$c_-(\eps)>c_+(\eps)$
for~$\eps\in(0,\frac{\sqrt{35}}9)$. Subtracting equation~\eqref{2+} (for~$w_+$ and~$c_+$) 
from equation~\eqref{2-} (for~$w_-$ and~$c_-$) we get
$$
5(c_--c_+)=(w_--w_+)(15-3\eps^2-5w_-^2-5w_+^2-5w_-w_+).
$$
Using relations~$w_--w_+=\frac{74}{35}\eps$ and~$w_-w_+=1-\frac{81}{35}\eps^2$, we obtain
$$
c_--c_+=\frac{169756}{35^3}\eps^3.
$$ 

Therefore,~$c_- > c_+$. We can conclude that for all~$\eps < \eps_1$ there is a Bellman candidate with a simple picture, whereas for~$\eps = \eps_1$ the right angle attacks the cup. 

Now we have to
determine the function $\ell(\eps)$, which is not equal to $\eps$ anymore. If $c=0$, then
starting from $\eps=\eps_1$ there is a birdie, for other $c$ we have a left trolleybus
and a separated angle on the left of it. Now, when $\eps_1$ is defined, we rename the 
function $c_+$ and will call it $c_1$, i.e. $\eps_1(c)$ is the unique solution of
the equation $c_1(\eps)=c$, $c\in[0,2)$.

We are going to investigate when the birdie can occur. The necessary and sufficient condition for this consists of two equations,
$F(a)=0$ and $F(b)=0$, see Proposition~\ref{S17}. It is more convenient to replace these two equations by the other two, $F(b)+F(a)=F(b)-F(a)=0$.
Using~\eqref{F+} for $u=b$ and~\eqref{F-} for $u=a$ we get
\begin{align*}
F(b)+F(a)&=\frac4{5\eps}w\ell^3+(b^3+a^3-3(b+a)+2c)+3\eps(b^2-a^2)+6\eps^2(b+a)
\\
&=\frac{2}{5\eps}w(30\eps^3+30\eps^2\ell+12\eps\ell^2+2\ell^3).
\end{align*}
Since the expression in parentheses is positive, the whole 
expression vanishes if and only if $w=0$, i.e. for $c=0$. We restrict ourselves to this case for a while.

Now, we consider the second equation:
\begin{align*}
F(b)-F(a)&=\frac{2\ell^2}{5\eps}(5w^2+\ell^2-5)+(b^3-a^3-3(b-a))
\\
&\qquad\qquad+3\eps(b^2+a^2-2)+6\eps^2(b-a)+12\eps^3
\\
&=\frac{2}{5\eps}\big((\ell^4-5\ell^2)+5\eps(\ell^3-3\ell)
+15\eps^2(\ell^2-1)+30\eps^3\ell+30\eps^4\big).
\end{align*}
Here we use the fact we have already assumed: $w=0$, i.e. $b=\ell$, $a=-\ell$. 
Put $\tau=\eps/\ell$, then 
$\tau\ge1$ and
$$
F(b)-F(a)=0\quad\Longleftrightarrow\quad
\ell^2(30\tau^4+30\tau^3+15\tau^2+5\tau+1)-(15\tau^2+15\tau+5)=0.
$$
As a result, we get a parametric representation of the function $\ell=\ell(\eps)$:
\begin{align*}
\ell&=\sqrt{\frac{15\tau^2+15\tau+5}{30\tau^4+30\tau^3+15\tau^2+5\tau+1}};
\\
\eps&=\sqrt{\frac{15\tau^4+15\tau^3+5\tau^2}{30\tau^4+30\tau^3+15\tau^2+5\tau+1}}.
\rule{0pt}{25pt}
\end{align*}
It is easy to see that $\ell$ is monotonically decreasing from $\frac{\sqrt{35}}9$ to $0$
and $\eps$ is monotonously increasing from $\frac{\sqrt{35}}9$ to $\frac1{\sqrt2}$
when $\tau$ runs from $1$ to $\infty$. Therefore, this relation correctly 
defines a function $\ell(\eps)$ on the interval $[\frac{\sqrt{35}}9,\frac1{\sqrt2}]$, which is
decreasing from $\frac{\sqrt{35}}9$ to $0$. So, for these values of $\eps$ we
have a birdie, which shrinks to an angle for $\eps=\frac1{\sqrt2}$.

We have found the second critical value of $\eps$ for $c=0$, when geometrical picture
of extremal lines changes, $\eps_2=\frac1{\sqrt2}$. Now, we find $\eps_2$ for all other values
of $c$. As we know, for other values of $c$ a birdie does not occur, i.e. the trolleybus dies before
it could be touched by the left angle. Now we find the moment $\eps_2$ of its death.
For $\eps=\eps_2$ the equality $F(b)=0$ occurs with $\ell=0$ (i.e. $a=b=w$). So, in this
case we have
$$
F(b)=3\eps(w^2-1)+6\eps^2w+6\eps^3=0,
$$
hence
$$
w=-\eps+\sqrt{1-\eps^2},
$$
and
\begin{equation}\label{c2}
c=c_2(\eps)=3w-w^3=2(1-\eps^2)^{3/2}-2\eps^3.
\end{equation}
\begin{figure}[ht]
\centering{\includegraphics[width=12cm]{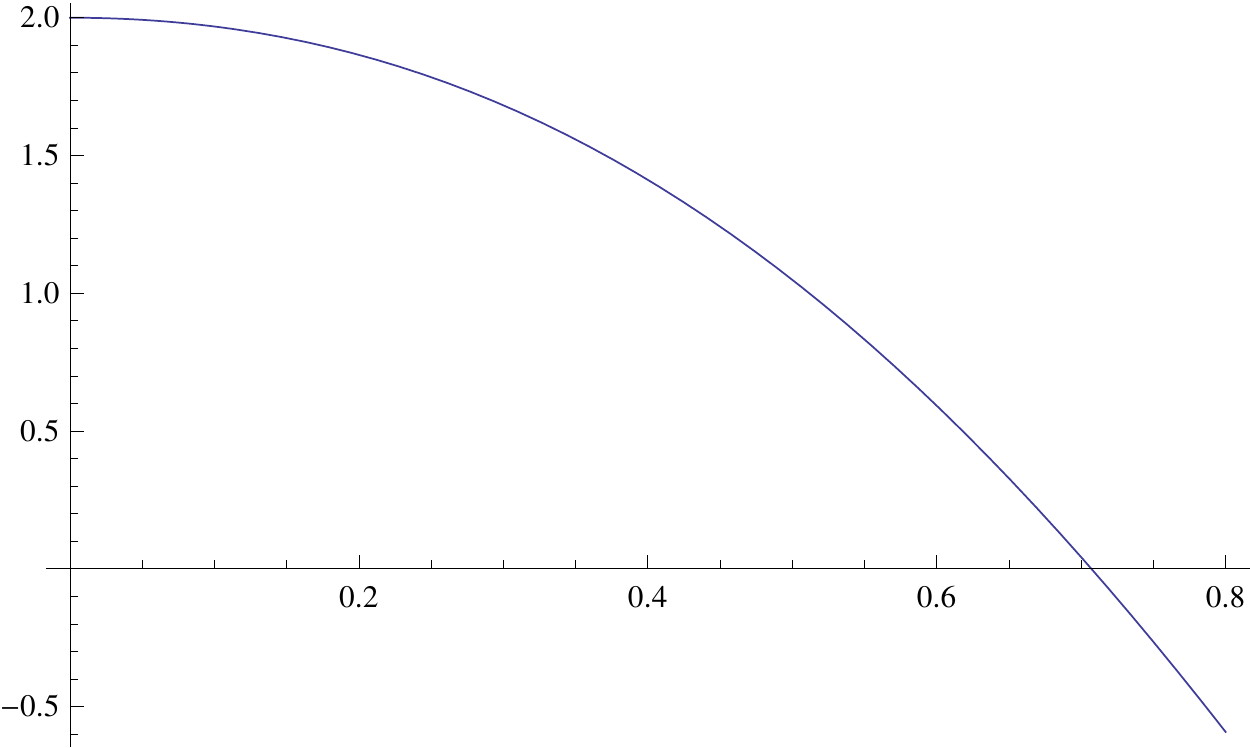}}
\caption{Graph of the function $c_2(\eps)$.}
\label{figc_2}
\end{figure}

The function $\eps\mapsto c_2(\eps)$ is monotonically decreasing and maps $[0,\frac1{\sqrt2}]$
onto $[0,2]$, therefore, the inverse function correctly defines $\eps_2(c)$ as the unique solution
of the equation $c_2(\eps)=c$. For all $\eps\ge\eps_2$ the tails from $\pm\infty$ have nonempty intersection despite that fact that $f'''$ has three essential roots. The foliation consists of a lonely angle. We must verify the latter assertion.
Recall that
\begin{align*}
\Fr(u;-\infty;\eps)&=(u^3-3u+c)-3\eps(u^2-1)+6\eps^2u-6\eps^3,
\\
\Fl(u;\infty;\eps)&=(u^3-3u+c)+3\eps(u^2-1)+6\eps^2u+6\eps^3.
\end{align*}
If $\eps\ge1$, then both these functions are strictly increasing, $\Fr(u;-\infty;\eps)$ 
has a positive root and $\Fl(u;\infty;\eps)$ has a negative one. The interval between these
two roots is nothing but the intersection of tails. 

If $\eps\in[\frac1{\sqrt2},1)$, then 
$\Fl(u;\infty;\eps)$ changes its sign only once, since
$$
\Fl(u;\infty;\eps)=(u+\eps)^3-3(1-\eps^2)(u+\eps)+2\eps^3+c
$$
and
$$
2\eps^3+c\ge2\eps^3\ge2(1-\eps^2)^{3/2}.
$$
Due to relation
\begin{equation}
\label{diff_m}
\Fl(u;\infty;\eps)-\Fr(u;-\infty;\eps)=6\eps u^2+6\eps(2\eps^2-1)
\end{equation}
we see that $\Fl(u;\infty;\eps)\ge\Fr(u;-\infty;\eps)$ for $\eps\ge\frac1{\sqrt2}$, i.e. 
$\Fr(u;-\infty;\eps)$ is negative where $\Fl(u;\infty;\eps)$ is. Again, this means that
the tails from $\pm\infty$ have nonempty intersection.

Finally, we consider the case when $\eps<\frac1{\sqrt2}$. The function 
$\Fl(u;\infty;\eps)$ has local minimum at the point $u=-\eps+\sqrt{1-\eps^2}$.
Its value at this point is
$$
\Fl(-\eps+\sqrt{1-\eps^2};\infty;\eps)=c-2\big((1-\eps^2)^{3/2}-\eps^3\big)=c_2(\eps_2)-c_2(\eps).
$$
Therefore, if $\eps\ge\eps_2$, then $\Fl(u;\infty;\eps)\ge0$ at the point of local minimum,
and therefore, it changes sign only once. On the left of that point the function $\Fr(u;-\infty;\eps)$
is negative. Indeed, $\Fr(u;-\infty;\eps)$ is increasing till the point $u=\eps-\sqrt{1-\eps^2}$
of its local maximum, which is on the right of the root of $\Fr(u;\infty;\eps)$. Hence, it is
sufficient to check that $\Fr(u;-\infty;\eps)\le0$ at the root of $\Fl(u;\infty;\eps)$. We use
equation~\eqref{diff_m} once more. The point $u$, where $\Fl(u;\infty;\eps)$ vanishes, is on
the left of the local maximum of this function, which occurs at the point $-\eps-\sqrt{1-\eps^2}$.
Therefore, $u<-\eps-\sqrt{1-\eps^2}<-\sqrt{1-\eps^2}$, i.e. $u^2>1-\eps^2$, and by equation~\eqref{diff_m}
we have $-\Fr(u;-\infty;\eps)>6\eps^3>0$. Again, the tails have nonempty intersection.

Note that the tail of $+\infty$ has a positive jump at the moment $\eps=\eps_2$, for
$\eps<\eps_2$ this tail ends on the right of the local minimum point $u=-\eps+\sqrt{1-\eps^2}$,
i.e. the tail ends at some positive point. Note, that the tail of $-\infty$ ends at a negative
point, because $\Fr(0;-\infty;\eps)=c+3\eps(1-2\eps^2)>0$ for $\eps<\frac1{\sqrt2}$. Thus, their
intersection is empty.

We know the whole evolution for the case of $f'''(t) = t^3 -3t + c$ (we have seen that the case of an arbitrary sixth degree polynomial with positive leading coefficient follows from this one with the help of linear transformations). If $|c| \geq 2$, then there is no evolution at all: for all $\eps$ the foliation consists of a lonely angle and two families of tangents adjacent to it. If $|c| \leq 2$, the situation is more complicated. Again, by virtue of Remark~\ref{remm2}, we can obtain the foliation for $-c$ from the one obtained for $c$ reflecting it in $x_1$. Thus, we can assume that $c \geq 0$. 

If $c = 0$, then for all $\eps \in [0, \sqrt{\frac{35}{9}})$ the picture is simple, with a cup and two symmetric angles. At the point $\eps = \sqrt{\frac{35}{9}}$ the angles attack the cup, forming a birdie. It decreases as $\eps$ grows, until the moment $\eps = \eps_2$, where $\eps_2 = \frac{1}{\sqrt{2}}$, when it dies. For all bigger $\eps$ there is an angle at $0$ and two families of tangents.

If $c \ne 0$, then $\eps_1$ is the unigue solution of equation~\eqref{c+}. For $\eps \in [\eps_1,\eps_2)$ there is a left trolleybus and an angle on the left of it. The point $\eps_2$ is determined by equation~\eqref{c2}. For $\eps \geq \eps_2$ there is a lonely angle. 
\begin{figure}[ht]
\vskip-30pt
\centering{\includegraphics[width=0.7\linewidth]{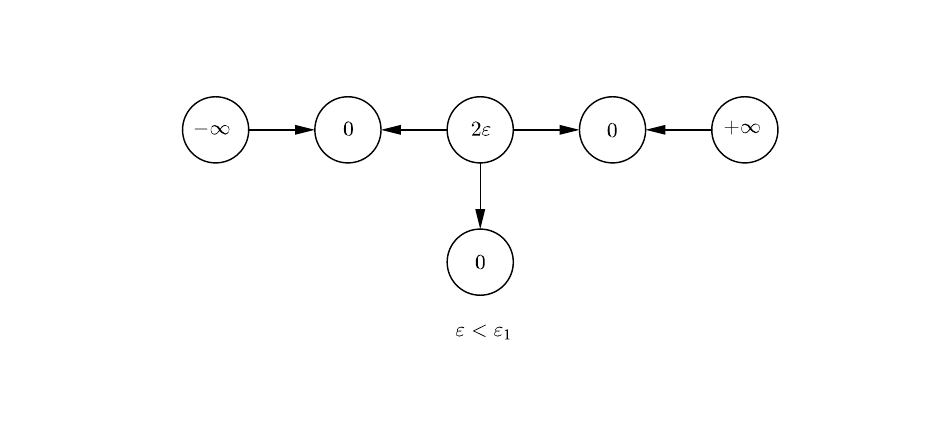}}
\vskip-40pt
\hbox{\hskip-90pt{\includegraphics[width=1.4\linewidth]{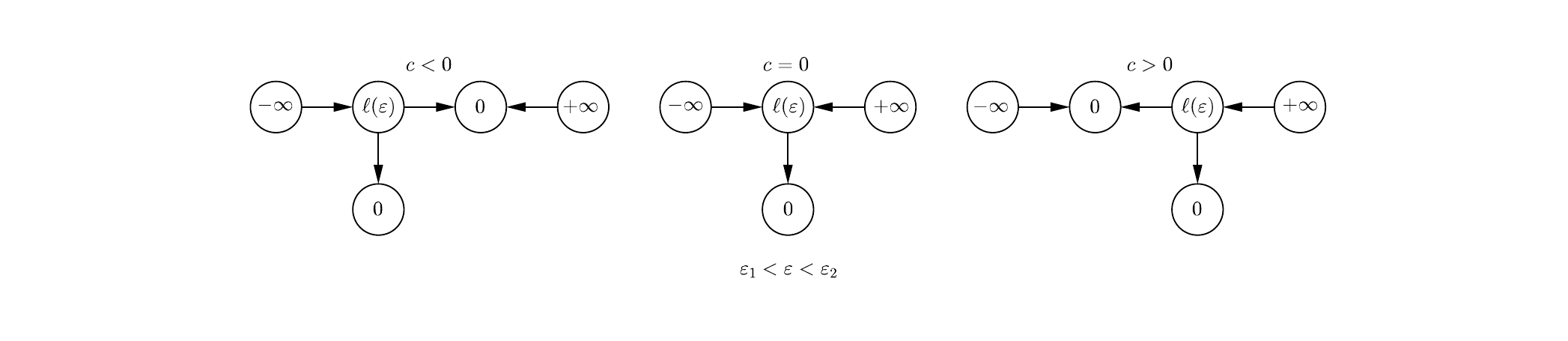}}}
\vskip-40pt
\centering{\includegraphics[width=0.7\linewidth]{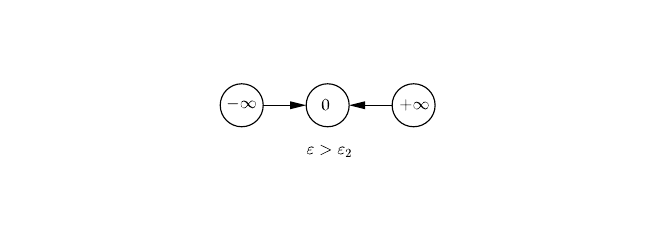}}
\vskip-50pt
\caption{Evolution of the graphs for $f(t)=t^3-3t+c$, $|c|<2$.}
\label{evolutionP6+}
\end{figure}

\paragraph{Polynomial of sixth degree: negative leading coefficient.}
We have built the Bellman function for the polynomial of sixth degree with positive leading coefficient.
Now, we turn to investigation of the case where $f'''=-t^3+3dt-c$.
Again, if $d\le0$ or $d>0$ and $|c|\ge2d^{3/2}$, then $f'''$ has only one root 
and by Theorem~\ref{SimplePicture} and general evolutional principles the foliation consists of a unique cup. So, in what follows we shall
assume $d>0$ and $|c|<2d^{3/2}$. By the same reason as before, without loss 
of generality, we may consider the case $d=1$ and $0\le c<2$. The picture for arbitrary 
$d$ can be obtained by the rescaling $c\to d^{3/2}c$, $x_1\to\sqrt{d}x_1$, 
$x_2\to dx_2$, $\eps\to\sqrt{d}\eps$.

Under this assumption $f'''$ has three roots and the foliation consists of two cups and one angle between them
for sufficiently small $\eps$. Let us find the first critical value $\eps=\eps_1$. For this aim we need to calculate the compatibility 
function $F$.

Now, we have two cups $\Ch([a_i,b_i],*)$, $i=1,2$. Their origins are the
left and the right roots of the equation
$$
-t^3+3dt-c=0.
$$
The cup equation is the same as before, i.e. it is equation~\eqref{cupeq}, because 
it is invariant under replacing $f\to-f$:
$$
(a_i-b_i)^2\Big[\frac1{120}(2a_i^3+3a_i^2b_i+3a_ib_i^2+2b_i^3)-\frac18(a_i+b_i)+
\frac{c}{12}\Big]=0,
$$
and in the expression for the differentials of the cup we have to change sign:
\begin{align*}
D_{{\rm L},i}&=-\frac1{20}(a_i-b_i)^2(2a_i^2+2a_ib_i+b_i^2-5),
\\
D_{{\rm R},i}&=-\frac1{20}(a_i-b_i)^2(a_i^2+2a_ib_i+2b_i^2-5).
\end{align*}

As before, we introduce $w_i\df(b_i+a_i)/2$ and $\ell_i\df(b_i-a_i)/2$, 
then the cup equation turns into
\begin{equation}
\label{cupeq2}
w_i^3-3w_i\Big(1-\frac{\ell_i^2}5\Big)+c=0.
\end{equation}
The expressions for the differentials are
\begin{align*}
D_{{\rm L},i}&=-\frac15\ell_i^2(5w_i^2-2w_i\ell_i+\ell_i^2-5),
\\
D_{{\rm R},i}&=-\frac15\ell_i^2(5w_i^2+2w_i\ell_i+\ell_i^2-5).
\end{align*}
Since the origins of the cups are the left and right solutions of the equation
$$
t^3-3t+c=0,
$$
among three possible solutions of~\eqref{cupeq2} (when $\ell_i^2<5\big[1-(c/2)^{2/3}\big]$) 
we need to take the left one for $i=1$ and the right one for $i=2$. 
These numbers tend to the roots of the cups as $\ell_i\to0$. So, when writing $w_i$ or
$w_i(\ell_i)$ we always assume these solutions. Since we consider $c>0$, the cup 
equation~\eqref{cupeq2} has two positive and one negative root. Therefore, the midpoint
of the left cup $w_1=w_1(\ell_1)$ is always negative and the midpoint of the right 
cup $w_2=w_2(\ell_2)$ is always positive. In the symmetrical case $c=0$ we have 
$\ell_1=\ell_2$ and $w_1=-w_2$ identically.

Now, calculate the function $F$ (this is the sum of two forces arising in the balance equation~\eqref{baleqformula}) on the interval $u\in[b_1,a_2]$:
\begin{align*}
\Fr(u;a_1,b_1;\eps)&=\eps^{-1}D_{{\rm R},1} e^{-(u-b_1)/\eps}+
\eps^{-1}e^{-u/\eps}\int\limits_{b_1}^ue^{t/\eps}\,df''(t)
\\
&=-\frac1{5\eps}\ell_1^2(5w_1^2+2w_1\ell_1+\ell_1^2-5)e^{-(u-b_1)/\eps}
\\
&\quad-\big[(u^3-3u+c)-3\eps(u^2-1)+6\eps^2u-6\eps^3\big]
\\
&\quad+\big[(b_1^3-3b_1+c)-3\eps(b_1^2-1)+6\eps^2b_1-6\eps^3\big]e^{-(u-b_1)/\eps},
\\
\Fl(u;a_2,b_2;\eps)&=-\eps^{-1}D_{{\rm L},2} e^{-(a_2-u)/\eps}+
\eps^{-1}e^{u/\eps}\int\limits^{a_2}_ue^{-t/\eps}\,df''(t)
\\
&=\frac1{5\eps}\ell_2^2(5w_2^2-2w_2\ell_2+\ell_2^2-5)e^{-(a_2-u)/\eps}
\\
&\quad-\big[(u^3-3u+c)+3\eps(u^2-1)+6\eps^2u+6\eps^3\big]
\\
&\quad+\big[(a_2^3-3a_2+c)+3\eps(a_2^2-1)+6\eps^2a_2+6\eps^3\big]e^{-(a_2-u)/\eps}.
\end{align*}
To find the first critical value $\eps_1$ we assume that $\eps<\eps_1$, then both
cups are maximal possible, i.e. $\ell_1=\ell_2=\eps$:
\begin{equation}
\label{F}
\begin{aligned}
F(u)&=\Fr(u;a_1,b_1;\eps)+\Fl(u;a_2,b_2;\eps)
\\
&=\eps\big(-w_1^2+2w_1\eps-\frac{11}5\eps^2+1\big)e^{-(u-w_1-\eps)/\eps}
\\
&\quad+\eps\big(w_2^2+2w_2\eps+\frac{11}5\eps^2-1\big)e^{-(w_2-\eps-u)/\eps}
\\
&\quad-2(u^3-3u+c)-12\eps^2u.
\end{aligned}
\end{equation}
The first critical value $\eps_1$ occurs when the angle touches one of the cups 
forming a trolleybus. Now we will show that if $c>0$, then the angle touches
the right cup. To prove this we calculate the sign of $F(\frac{w_1+w_2}{2})$.
We use Vi\`ete's formula to get $c=w_1w_2(w_1+w_2)$ and $3(1-\frac{\eps^2}{5})=w_1^2+w_2^2+w_1w_2$.

\begin{multline}
F(\frac{w_1+w_2}{2})=e^{-\frac{w_2-w_1-2\eps}{2\eps}}\eps(w_2^2-w_1^2+2\eps(w_2+w_1))-\\
2\left(\Big(\frac{w_1+w_2}{2}\Big)^3-3\frac{w_1+w_2}{2}+c \right)\\
-6\eps^2(w_1+w_2)=(w_1+w_2)\Big(\eps e^{-\frac{w_2-w_1-2\eps}{2\eps}}(w_2-w_1+2\eps)-\\
\frac{(w_1+w_2)^2}{4}+3-2w_1w_2-6\eps^2\Big).
\end{multline}

Let $R=w_2-w_1$. Then $\frac{(w_1+w_2)^2}{4}+2w_1w_2=3(1-\frac{\eps^2}{5})-\frac{3}{4}R^2$. Thus
$$F(\frac{w_1+w_2}{2})=(w_1+w_2)\left(\eps e^{-\frac{R-2\eps}{2\eps}}(R+2\eps)+\frac{3}{4}R^2-\frac{27}{5}\eps^2\right).$$
We want to prove that $F(\frac{w_1+w_2}{2})\leq 0$. Let 
$$\theta(R)=\eps e^{-\frac{R-2\eps}{2\eps}}(R+2\eps)+\frac{3}{4}R^2-\frac{27}{5}\eps^2.$$
Then 
$$\theta'(R)=\frac{R}{2}(3-e^{-\frac{R-2\eps}{2\eps}})>0$$ if $R\geq 2\eps$. Thus $\theta(R)\geq \theta(2\eps)=\frac{8}{5}\eps^2>0$ if $R\geq 2\eps$. But if $\eps\leq \eps_1$, then both cups are full and $R=w_2-w_1\geq 2\eps$. Since $w_1<0$, $w_2>0$ and $c=w_1w_2(w_1+w_2)>0$, we have $w_1+w_2<0$, thus 
$F(\frac{w_1+w_2}{2})=(w_1+w_2)\theta(w_2-w_1)<0$. This means that the vertex of the angle is on the right of $\frac{w_1+w_2}{2}$. So, at the first critical moment $\eps_1$ the angle touches the right cup.

We find the first critical value $\eps=\eps_1$. We know it satisfies the equation $F(a_2)=0$, or
\begin{equation}
\label{eps1}
\Big(w_1^2-2\eps w_1+\frac{11}5\eps^2-1\Big)e^{-(w_2-w_1-2\eps)/\eps}=
7w_2^2-\frac{74}5\eps w_2+\frac{81}5\eps^2-7.
\end{equation}
Here $w_1$ is the minimal solution of the equation
$$
w^3-3w\Big(1-\frac{\eps^2}5\Big)+c=0
$$
and $w_2$ is the maximal one. For any $c$, $0\le c<2$, equation~\eqref{eps1} 
must have exactly one positive root~$\eps_1$ such that 
$\eps_1^2<5\Big[1-\big(\frac c2\big)^{2/3}\Big]$ (we can simply take the smallest one). But we postpone investigation on
this equation and find the next critical point. For $\eps>\eps_1$ we have
a trolleybus and a cup. Consequently, two situations can happen: either the trolleybus dies or
the left cup touches the trolleybus forming a multicup. We find out when
second situation takes place.

Suppose that at the moment $\eps=\eps_2$ the left cup touches the trolleybus, i.e.
$a_2=b_1\df s$. Moreover, the left cup is full, i.e. $\ell_1=\eps$, $w_1=s-\eps$.
For the right cup we denote $\ell_2\df\ell$, $w_2=s+\ell$. Recall that $w_1$ is the
minimal root of the equation
$$
w_1^3-3w_1\Big(1-\frac{\eps^2}5\Big)+c=0
$$
and $w_2$ is the maximal root of the equation
$$
w_2^3-3w_2\Big(1-\frac{\ell^2}5\Big)+c=0.
$$
Subtracting one equation from another we get
$$
(s+\ell)^3-(s-\eps)^3-3(s+\ell)\big(1-\frac{\ell^2}5\big)+3(s-\eps)\big(1-\frac{\eps^2}5\big)=0,
$$
or
\begin{equation}
\label{fff}
s^2-\frac65(\eps-\ell)s+\frac8{15}(\ell^2-\ell\eps+\eps^2)-1=0.
\end{equation}
Now, we write down the balance equation:
\begin{align*}
\mr ss&=-\frac\eps5(5w_1^2+2\eps w_1+\eps^2-5)
\\
\ml ss&=\frac{\ell^2}{5\eps}(5w_2^2-2\ell w_2+\ell^2-5),
\end{align*}
and therefore,
$$
\eps^2(5w_1^2+2\eps w_1+\eps^2-5)-\ell^2(5w_2^2-2\ell w_2+\ell^2-5)=0,
$$
or
$$
5(\eps^2-\ell^2)s^2-8(\eps^3+\ell^3)s+4(\eps^4-\ell^4)-5(\eps^2-\ell^2)=0.
$$
Comparing with~\eqref{fff} we have
\begin{align*}
5(\eps^2-\ell^2)(s^2-1)&=8(\eps^3+\ell^3)s-4(\eps^4-\ell^4)
\\
&=(\eps^2-\ell^2)\Big(6(\eps-\ell)s-\frac83(\eps^2-\eps\ell+\ell^2)\Big).
\end{align*}
Solving this linear equation we get
$$
s=\frac23(\eps-\ell).
$$
After plugging this solution in~\eqref{fff} we obtain a relation between $\ell$ and $\eps$:
$$
\ell^2+\eps\ell+\eps^2-\frac{45}8=0.
$$
We see that this equation has a positive solution $\ell$ if and only if $\eps^2<\frac{45}8$.
For such $\eps$ we have
$$
\ell=-\frac\eps2+\frac32\sqrt{\frac52-\frac{\eps^2}3}.
$$
Assumption $\ell\le\eps$ yields $\eps^2\ge\frac{15}8$.
As a result, we get
\begin{align*}
s&=\frac23(\eps-\ell)=\eps-\sqrt{\frac52-\frac{\eps^2}3}\,,
\\
w_1&=s-\eps=-\sqrt{\frac52-\frac{\eps^2}3}\,,
\\
w_2&=s+\ell=\frac\eps2+\frac12\sqrt{\frac52-\frac{\eps^2}3}\,,
\\
c&=3w_1\Big(1-\frac{\eps^2}5\Big)-w_1^3=
\frac4{15}\Big(\eps^2-\frac{15}8\Big)\sqrt{\frac52-\frac{\eps^2}3}\,.
\end{align*}

We consider the following function (see Fig.~\ref{psi}),
$$
\psi(t)=\frac4{15}\Big(t-\frac{15}8\Big)\sqrt{\frac52-\frac{t}3},
$$
on the interval $t\in[\frac{15}8,\frac{45}8]$.
\begin{figure}[ht]
\centering{\includegraphics[width=12cm]{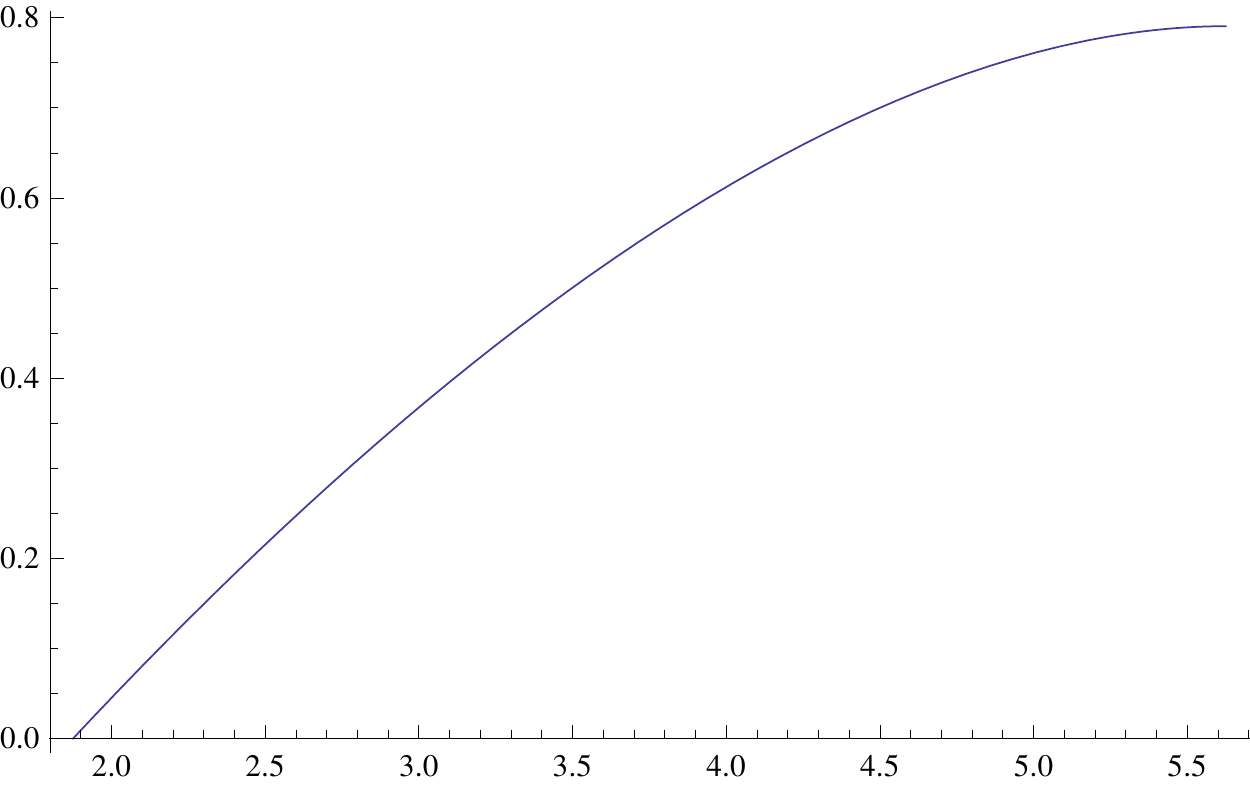}}
\caption{Graph of the function $\psi(t)$.}
\label{psi}
\end{figure}

Since
$$
\psi'(t)=\frac4{30}\cdot\frac{\frac{45}8-t}{\sqrt{\frac52-\frac t3}}\ge0\,,
$$
the function $\psi$ monotonically increases from $\psi(\frac{15}8)=0$ to 
$\psi(\frac{45}8)=\sqrt{\frac58}$. Therefore, if $c\in\Big[0,\sqrt{\frac58}\Big)$, then
the equation $\psi(\eps^2)=c$ has exactly one solution, 
$\eps=\eps_2\in\Big[\sqrt{\frac{15}8},\sqrt{\frac{45}8}\Big)$. If $c=0$, then the picture
is symmetric {\bf(}$w_1=-w_2=\sqrt{3(1-\frac{\eps^2}5)}$\;{\bf)}, the angle is stable with the
vertex at the origin ($s=0$), and both cups touch it simultaneously at the moment
$\eps=\eps_1=\eps_2=\sqrt{\frac{15}8}$, when $w_2=\eps$. They form a multicup,
which lives until $\eps_3=2\eps_2=\sqrt{\frac{15}2}$, when it
becomes closed. The chordal domain over it grows up to infinity.
If $c\in\Big(0,\sqrt{\frac58}\Big)$, then $\eps_1<\eps_2$ and for $\eps\in(\eps_1,\eps_2)$
there exists a complete left cup and a trolleybus separated by a domain foliated
by right tangent lines. At the moment $\eps_2$ the cup touches the trolleybus forming
a non-symmetric multicup, which lives until 
$\eps_3=\eps_2+\ell(\eps_2)=\frac{\eps_2}2+\frac32\sqrt{\frac52-\frac{\eps_2^2}3}$, 
when the chordal domain appears on this multicup. The bigger $c$ is, the smaller is the trolleybus
(the second chordal domain) at the moment of building a multicup. Finally, for $c=\sqrt{\frac58}$
the left cup touches the trolleybus at the moment of its death ($\ell=0$). Thus, for
$c\in\Big[\sqrt{\frac58},2\Big)$ there are only two critical points: $\eps=\eps_1$,
when the angle and the right cup form a trolleybus, and $\eps=\eps_2$, when trolleybus
dies. This value $\eps_2$ is the solution of the equation $\Fr(w_2;a_1,b_1;\eps)=0$, i.e. it occurs
when the right tail of the left cup attains the root of the right cup and therefore jumps at 
that moment till $+\infty$. This equation can be rewritten as follows:
$$
\big(w_1^2-2\eps w_1+\frac{11}5\eps^2-1\big)e^{-(w_2-w_1-\eps)/\eps}=3w_2^2-6\eps w_2+6\eps^2-3\,,
$$
where $w_1$ is the minimal root of the equation $w_1^3-3w_1\bigl(1-\frac{\eps^2}5\bigr)+c=0$
and $w_2$ is the maximal root of the equation $w_2^3-3w_2+c=0$.
The left cup continue its growth till infinity. 

If $c>2$, as we know, there are
no critical points at all, there exists only one cup all the time. 

For $c<0$ the picture is
symmetric and the trolleybus appears after gluing the angle with the left cup.

\begin{figure}[ht]
\vskip-30pt
\centering{\includegraphics[width=0.7\linewidth]{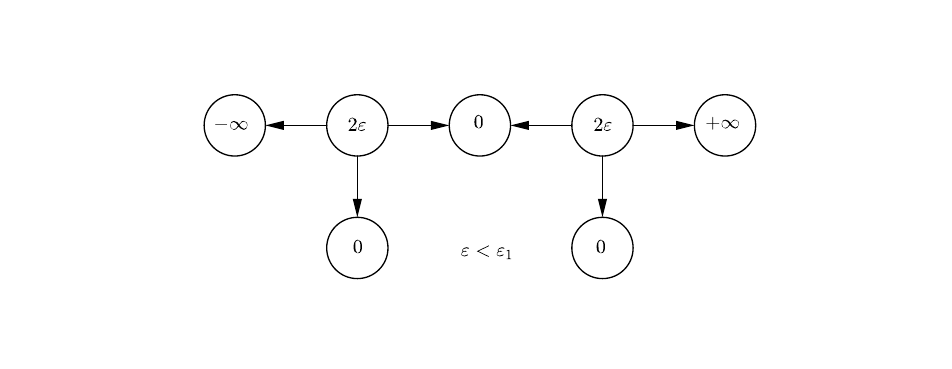}}
\vskip-50pt
\hbox{\hskip-90pt{\includegraphics[width=1.4\linewidth]{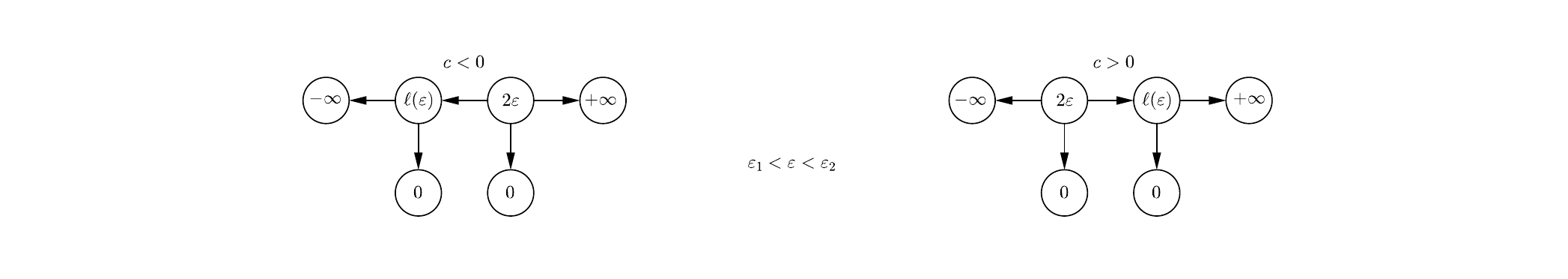}}}
\vskip-30pt
\centering{\includegraphics[width=0.4\linewidth]{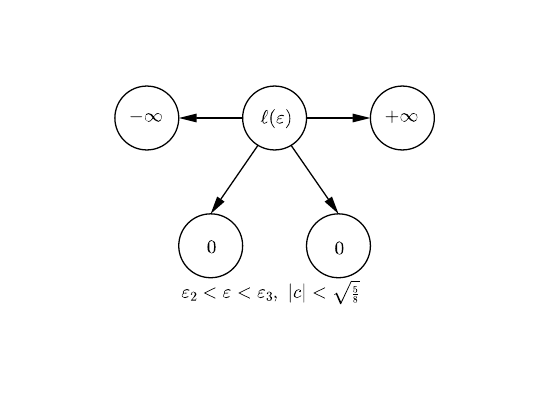}
\includegraphics[width=0.43\linewidth]{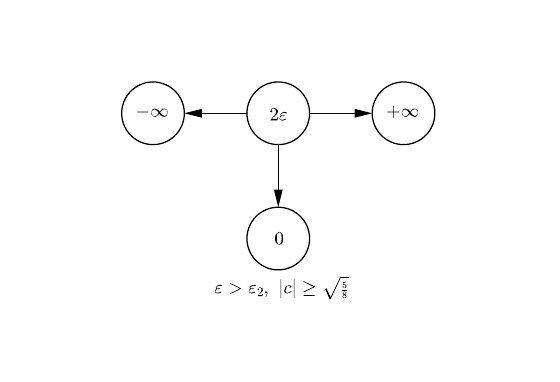}}
\vskip-40pt
\hbox{\hskip52pt{\includegraphics[width=0.35\linewidth]{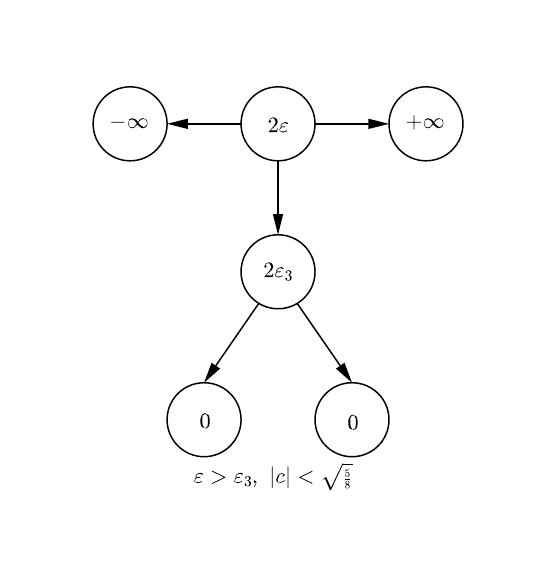}}}
\vskip-30pt
\caption{Evolution of the graphs for $f(t)=-t^3-3t+c$, $|c|<2$.}
\label{evolutionP6-}
\end{figure}

\chapter{Optimizers}
In the previous chapter, we constructed a Bellman candidate~$B$ of a special form (Theorem~\ref{BC} and Remark~\ref{AdmissibleCondidate}). We claim that it coincides with the Bellman function~$\Bell$. Subsection~\ref{s224} suggests a method to prove the claim.  We have to construct an optimizer\index{optimizer}~$\varphi_x$ for each~$x \in \Omega_{\eps}$ (see Definition~\ref{Opt}). Here we will follow the same strategy as when we were constructing Bellman candidates: we will first study the local behavior of the optimizers (i.e. how do optimizers vary when~$x$ runs through one figure), this is done in Section~\ref{s52}, and then ``glue'' these local scenarios together in Section~\ref{s53}. The optimizers for Bellman candidates with simple graphs (and trolleybuses) were built in~\cite{5A} (see the preprint~\cite{5AOld} as well). Here we also describe them (with somehow modified proofs) for the sake of completeness.

\section{Abstract theory}\label{s51}
We begin with an abstract description of how do optimizers look like. First, as it was mentioned in Subsection~\ref{s224}, it is natural to construct monotone optimizers. Second, all our optimizers will be concatenations of different constants and logarithms (however, it is not clear from the very beginning, why should the optimizer be of such a form). It is not difficult to build a monotone function~$\varphi_x$ such that~$\bp{\varphi_x} = x$ and~$B(\varphi_x) = f[\varphi_x]$ (we use the notation introduced in formula~\eqref{functional} and Definition~\ref{BellmanPoint}). The main difficulty is to verify that~$\varphi_x \in \BMO_{\eps}$. It was noticed in~\cite{5A} that it is more natural to argue geometrically. The notion of a \emph{delivery curve}\index{delivery curve} is useful in this context.
\begin{Def}\label{DelCurve}
Let~$\eps < \eps_{\infty}$\textup, let~$B$ be a Bellman candidate. Suppose~$\varphi\colon [l,r] \to \mathbb{R}$ is a summable function. The curve~$\gamma\colon [l,r] \to \Omega_\eps$ given by the formula
\begin{equation}\label{CurveGenerator}
\gamma(\tau) \df \bp{\varphi\big|_{[l,\tau]}} \!\!\!\!\!= \av{(\varphi,\varphi^2)}{[l,\tau]},\quad \tau \in [l,r],
\end{equation}
is called a delivery curve if~$B(\gamma(\tau)) = \av{f(\varphi)}{[l,\tau]}$ for any~$\tau \in [l,r]$ \textup(in particular\textup,~$\gamma \subset \Omega_{\eps}$\textup). The function~$\varphi$ is called the generating function for~$\gamma$.
\end{Def} 
In other words,~$\gamma$ is a curve that ``delivers'' optimizers to the point. The word ``curve'' here means a parametrized curve, because the definition depends on the parametrization. The advantage of considering such a curve is that it allows to verify the condition that~$\varphi$ is a test function (that~$\varphi \in \BMO_{\eps}$). We begin with an obvious remark.
\begin{Rem}\label{Graph}
The curve~$\gamma$ generated by a function~$\varphi$ with the help of formula~\textup{\eqref{CurveGenerator}} is a graph of a function \textup(in the standard coordinates\textup) if and only if the function~$\tau\mapsto \av{\varphi}{[l,\tau]}$ is monotone on~$[l,r]$.
\end{Rem}
However, we can say more. We will not prove the fact below.
\begin{Fact}\label{MonConv}
A curve given by formula~\textup{\eqref{CurveGenerator}} is a graph of a convex function in the standard coordinates if and only if its generating function is monotone.
\end{Fact}
The main feature we will use is the formula
\begin{equation}\label{FirstDerivative}
 \gamma(\tau) + (\tau-l)\gamma'(\tau) = \big(\varphi(\tau),\varphi^2(\tau)\big),
\end{equation}
which can be obtained by differentiation of formula~\eqref{CurveGenerator}. In particular, this formula shows that the tangent to~$\gamma$ at the point~$\tau$ looks in the direction of the point~$(\varphi(\tau),\varphi^2(\tau))$. So, one can reconstruct the values of~$\varphi$ by looking at the points on the lower parabola that ``are indicated'' by the tangents of the corresponding delivery curve. We will use this principle very often. Moreover, equation~\eqref{FirstDerivative} allows to reconstruct~$\varphi$. %by looking at the geometric form of~$\gamma$ (i.e. treating~$\gamma$ as a subset of~$\mathbb{R}^2$).

\begin{Le}\label{Monotone_Convex}
A curve given by formula~\textup{\eqref{CurveGenerator}} is convex if its generating function is monotone.
\end{Le}
\begin{proof}
Let us assume for a while that the function~$\varphi$ is sufficiently smooth ($C^2$ will do). In such a case, we may differentiate equation~\eqref{FirstDerivative} once more and get
\begin{equation*}
(\tau-l)\gamma''(\tau) = -2\gamma'(\tau) + \varphi'(\tau)(1,2\varphi(\tau)). 
\end{equation*}
Thus, the curvature of~$\gamma$, which is~$|\gamma'|^{-3} \det\binom{\gamma'}{\gamma''}$, has the same sign as~$\varphi'\det\binom{\gamma'}{(1,2\varphi)}$. We use equation~\eqref{FirstDerivative} once more to express~$\gamma'$ and rewrite the determinant as
\begin{equation*}
\det\binom{\gamma'}{(1,2\varphi)}=
\frac{1}{\tau-l}\det
\begin{pmatrix}
\varphi - \gamma_1 & \varphi^2 - \gamma_2\\
1 & 2\varphi
\end{pmatrix}.
\end{equation*}
This expression is positive,~because~$(1,2\varphi(\tau))$ is the tangent vector to~$\FixedBoundary\Omega_{\eps}$ at~$(\varphi(\tau),\varphi^2(\tau))$ and~$\gamma(\tau)$ belongs to~$\Omega_{\eps}$. Therefore, the sign of the curvature coinsides with the sign of~$\varphi'$.

We only have to get rid of the smoothness assumption. For that purpose we use an easy observation: if~$\gamma_n\colon[l,r] \to \mathbb{R}^2$ is a sequence of convex curves (in the sense that these curves are graphs of convex functions in the standard coordinates) and~$\gamma_n \to \gamma$ pointwise, where~$\gamma$ is a graph of a function, then this function is convex. Thus, our aim is to approximate~$\varphi$ by smooth monotone functions~$\varphi_n$ in such a way that the curves~$\gamma_n$ generated by~$\varphi_n$ converge to~$\gamma$ pointwise. This can be done in many ways, for example, like this: first truncate the function~$\varphi$ on a big level from below and above (i.e. consider the functions~$\varphi_{-n,n}$,~$n\to \infty$, given by formula~\eqref{truncation}), second, extend the obtained truncation to~$\mathbb{R}$ by the formula
\begin{equation*}
\bar{\varphi}_n(\tau) = \begin{cases} \lim\limits_{t\to l+}\varphi_{-n,n}(t),\quad &\tau < l;\\
\varphi_{-n,n}(\tau),\quad &\tau \in [l,r];\\
\lim\limits_{t \to r-}\varphi_{-n,n}(t),\quad &\tau > r,
\end{cases}
\end{equation*}
and finally, convolve the extension~$\bar{\varphi}_n$ with sufficiently sharp approximation of identity to obtain the function~$\varphi_n$ (we restrict the convolution to~$[l,r]$). 
\end{proof}
%\begin{proof}
%We begin with differentiating the curve~$\gamma$ and see that
%\begin{equation}\label{FirstDerivative}
 %\gamma(\tau) + (\tau-l)\gamma'(\tau) = \big(\varphi(\tau),\varphi^2(\tau)\big)
%\end{equation}
%for almost all~$\tau$.  Differentiating this formula once more,
%\begin{equation*}
%(\tau-l)\gamma''(\tau) = -2\gamma'(\tau) + 2\varphi'(\tau)(1,2\varphi(\tau)). 
%\end{equation*}
%Thus, the curvature of the curve~$\gamma$ has the same sign as~$\varphi'(\tau)\det\big(\gamma',(1,2\varphi(\tau))\big)$. It remains to notice that the determinant~$\big(\gamma',(1,2\varphi(\tau))\big)$ always has one and the same sign by virtue of formula~\eqref{FirstDerivative}. Thus, the sign of the curvature coincides with the sign of~$\varphi'$.
%{\bf [the differentiation seem a bit suspicious.]}
%\end{proof}

\begin{figure}[h!]
\begin{center}
\includegraphics{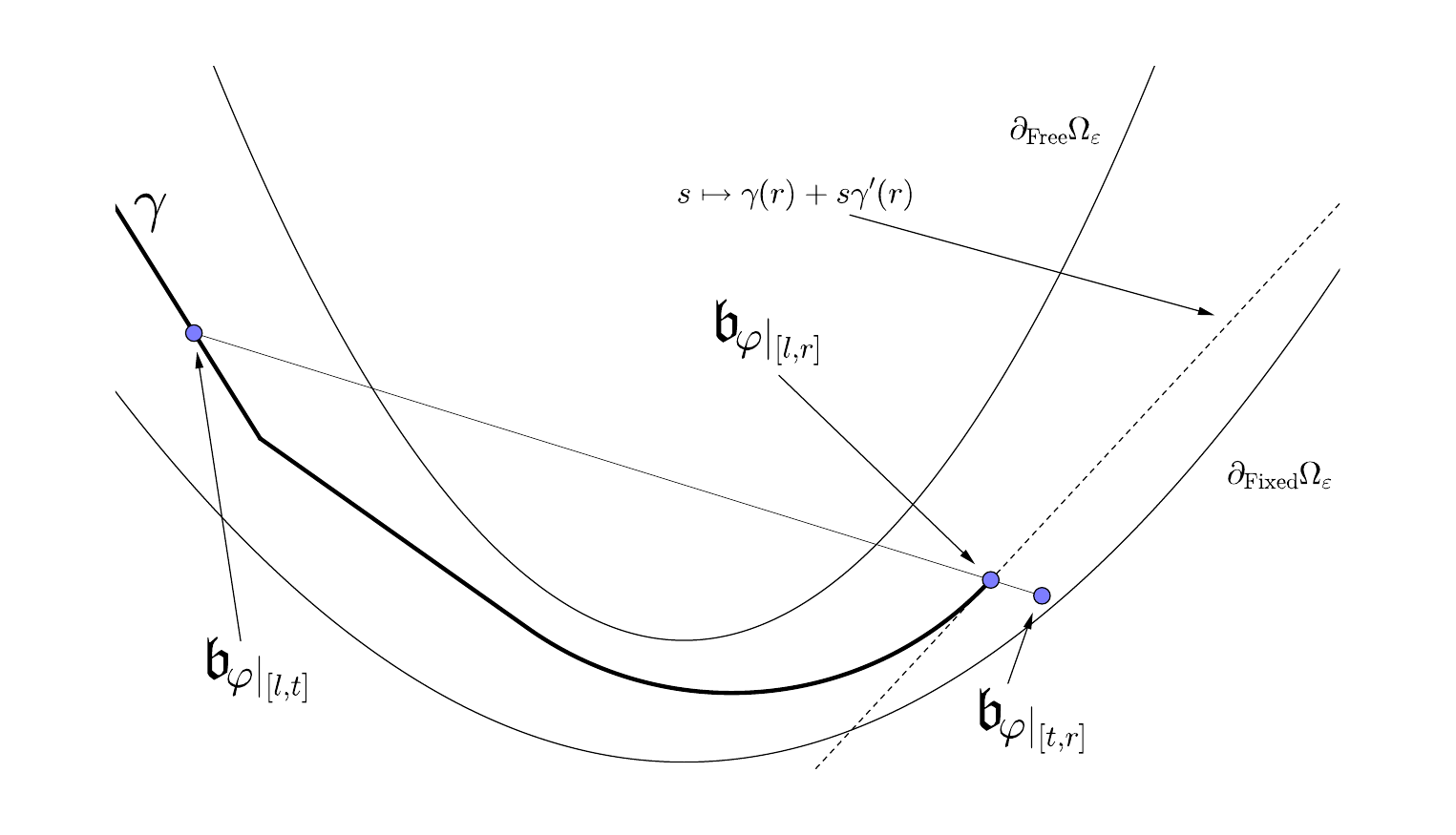}
\caption{Illustration to the proof of Lemma~\ref{DeliveryCurveLemma}.}
\label{fig:DCLemma}
\end{center}
\end{figure}

The following lemma links the condition that~$\gamma$ is convex with the condition~$\varphi \in \BMO_{\eps}$ (see Figure~\ref{fig:DCLemma} for visualization of the proof).

\begin{Le}\label{DeliveryCurveLemma}
Suppose~$\gamma$ to be a curve on~$[l,r]$ given by formula~\textup{\eqref{CurveGenerator}}. Let it be convex in the sense that it is a graph of a convex function in the standard coordinates. Suppose also that the tangent line~$s\mapsto \gamma(r) + s\gamma'(r)$\textup,~$s\in\mathbb{R}$\textup, does not cross the free boundary~$\FreeBoundary\Omega_{\eps}$. Then\textup, for any~$t \in [l,r)$\textup, we have~$\bp{\!\varphi\big|_{[t,r]}}\!\!\!\!\!\in\Omega_{\eps}$.
\end{Le}
\begin{proof}
We note that~$\bp{\!\varphi\big|_{[l,r]}}$ is a convex combination of~$\bp{\!\varphi\big|_{[t,r]}}$ and~$\bp{\!\varphi\big|_{[l,t]}}$. Thus,~$\bp{\!\varphi\big|_{[t,r]}}$ is separated from~$\{x \in \mathbb{R}\mid x_2 > x_1^2 + \eps^2\}$ by the line~$s\to \gamma(r) + s\gamma'(r)$,~$s\in\mathbb{R}$. On the other hand,~$\bp{\varphi\big|_{[t,r]}}$ surely belongs to~$\{x \in \mathbb{R}\mid x_1^2 \leq x_2\}$, so it belongs to~$\Omega_{\eps}$. 
\end{proof}

\begin{Cor}\label{DeliveryCurveCorollary}
Suppose~$\gamma$ to be a delivery curve on~$[l,r]$. Let it be convex in the sense that it is a graph of a convex function in the standard coordinates. Suppose also that the tangent line~$s\mapsto \gamma(\tau) + s\gamma'(\tau)$\textup,~$s\in\mathbb{R}$\textup, does not cross the free boundary~$\FreeBoundary\Omega_{\eps}$ for any~$\tau \in [l,r]$. Then\textup, the function~$\varphi$ that generates~$\gamma$ belongs to~$\BMO_{\eps}$.
\end{Cor}

Before we pass to constructing specific delivery curves, we should postulate a heuristic principle that will help us to guess them. Since a delivery curve ``consists of optimizers'', it has to avoid the directions in which the Bellman candidate is non-linear. So, we guess that delivery curves should go either along the extremals, or along the free boundary (the first case corresponds to the parts of the optimizer where it is constant, in the second case it is locally a logarithm). 

\section{Local behavior of optimizers}\label{s52}

\subsection{Optimizers for tangent domains}\label{s521}\index{domain! tangent domain}
Consider a tangent domain~$\Rt(u_1,u_2)$ and a standard candidate~$B$ on it. Suppose~$\psi$ is an optimizer for the point~$\big(u_1-\eps,(u_1-\eps)^2 + \eps^2\big)$ (which is the upper parabola endpoint of the right tangent passing through~$(u_1,u_1^2)$; see Figure~\ref{fig:RtOpt} below) defined on the interval~$[l,l_1]$. Our aim is to build the optimizers for all the points inside~$\Rt(u_1,u_2)$. We start with the points lying on the free boundary. For that purpose, we look for a function~$\varphi$ on~$[l,r]$, for some $r>l_1$, such that~$\varphi = \psi$ on~$[l,l_1]$ and its delivery curve~$\gamma$ goes along the upper parabola from~$\big(u_1-\eps,(u_1-\eps)^2 + \eps^2\big)$ to~$\big(u_2-\eps,(u_2-\eps)^2 + \eps^2\big)$ on~$[l_1,r]$. %We have to find a function that generates it in the sense of formula~\eqref{CurveGenerator}. Let~$\varphi$ be a function on~$[l,r]$,~$r > l_1$, let~$\varphi = \psi$ on~$[l,l_1]$. Suppose that~$\varphi$ generates~$\gamma$.

Our curve satisfies the equation~$\gamma_2 = \gamma_1^2 + \eps^2$ and~\eqref{FirstDerivative} on $[l_1,r]$. Therefore,
\begin{equation*}
(\tau-l)(\gamma_1',2\gamma_1'\gamma_1) + (\gamma_1,\gamma_1^2+ \eps^2) = (\varphi,\varphi^2), \quad \gamma = \gamma(\tau);\varphi = \varphi(\tau); \tau \in(l_1,r).
\end{equation*}
Thus,
\begin{equation*}
\Big((\tau-l)\gamma_1' + \gamma_1\Big)^2 = 2(\tau-l)\gamma_1'\gamma_1 + \gamma_1^2 + \eps^2,
\end{equation*}
which shows that
\begin{equation*}
\gamma_1(\tau) = \pm\eps\log(\tau-l) + c
\end{equation*}
for some constant~$c$. The curve~$\gamma$ moves to the right, thus we choose the sign~``$+$''. Since we wanted our curve to be at the point~$\big(u_1-\eps,(u_1-\eps)^2 + \eps^2\big)$ when~$\tau = l_1$, we adjust~$c$ to satisfy this condition:
\begin{equation*}
\gamma_1(\tau) = \eps\log\Big(\frac{\tau-l}{l_1-l}\Big) + u_1 - \eps.
\end{equation*}
We recover the function~$\varphi$:
\begin{equation}\label{OptimizerRightTangentsUpperParabola}
\varphi(\tau) = 
\begin{cases}
\psi(\tau),\quad & \tau \in [l,l_1);\\
\eps\log(\frac{\tau-l}{l_1-l}) + u_1,\quad & \tau \in [l_1,r].
\end{cases}
\end{equation}
In particular,~$r$ is defined by the equation~$\varphi(r) = u_2$ (one can see that from equation~\eqref{FirstDerivative}). We claim that~$\gamma$ is a delivery curve.

To prove the claim, it suffices to show that~$B(\gamma(s)) = \av{f(\varphi)}{[l,s]}$ for~$s \in [l_1,r]$, because for $s \in [l,l_1]$ this follows from the fact that~$\psi$ is an optimizer. The function~$B$ is given by formulas~\eqref{linearity},~\eqref{ExplicitFormulaForm} (see Proposition~\ref{RightTangentsCandidate}) and our assumption~$B(u_1-\eps,(u_1-\eps)^2 + \eps^2) = \av{f(\psi)}{[l,l_1]}$. The equality wanted can be rewritten as follows:
\begin{equation}\label{NewtonLeibniz}
(s-l)B(\gamma(s)) - (l_1-l)B(\gamma(l_1)) = \int\limits_{l_1}^sf(\varphi(\tau))\,d\tau.
\end{equation}
Consider the right-hand side and make the change of variable~$t = \varphi(\tau)$ (in other words,~$\frac{\tau - l}{l_1 - l}=e^{\frac{t-u_1}{\eps}}$):
\begin{equation}\label{NewtonLeibnizRewritten}
\int\limits_{l_1}^sf(\varphi(\tau))\,d\tau = \int\limits_{l_1}^s f\Big(\eps\log\big(\frac{\tau-l}{l_1-l}\big) + u_1\Big)\,d\tau = \frac{l_1 - l}{\eps}\int\limits_{\varphi(l_1)}^{\varphi(s)} f(t)e^{\frac{t-u_1}{\eps}}\,dt.
\end{equation}
To compute the left-hand side of~\eqref{NewtonLeibniz}, we see that in our case formula~\eqref{linearity} becomes (we denote~$u = \varphi(s)$)
\begin{equation*}
\begin{aligned}
B(\gamma(s)) = -\eps m(\varphi(s)) + f(\varphi(s)) \stackrel{\scriptscriptstyle{\eqref{ExplicitFormulaForm}}}{=} -\eps e^{-\frac{u}{\eps}}\Big(e^{\frac{u_1}{\eps}}m(u_1) + \eps^{-1}\int\limits_{u_1}^u f'(t)e^{\frac{t}{\eps}}\,dt\Big) + f(u)=\\
-\eps e^{-\frac{u}{\eps}}\Big(e^{\frac{u_1}{\eps}}m(u_1) + \eps^{-1}f(u)e^{\frac{u}{\eps}} - \eps^{-1}f(u_1)e^{\frac{u_1}{\eps}} - \eps^{-2}\int\limits_{u_1}^u f(t)e^{\frac{t}{\eps}}\,dt\Big) + f(u) = \\
-\eps e^{\frac{u_1-u}{\eps}}m(u_1) + e^{\frac{u_1-u}{\eps}}f(u_1) + \eps^{-1}\int\limits_{u_1}^u f(t) e^{\frac{t-u}{\eps}}\,dt.
\end{aligned}
\end{equation*}
So, by equality~$\frac{s - l}{l_1 - l}=e^{\frac{u-u_1}{\eps}}$, the left-hand side of~\eqref{NewtonLeibniz} equals
\begin{equation*}
(s-l)\Big(-\eps e^{\frac{u_1-u}{\eps}}m(u_1) + e^{\frac{u_1-u}{\eps}}f(u_1) + \eps^{-1}\!\!\int\limits_{u_1}^u f(t) e^{\frac{t-u}{\eps}}dt\Big) - (l_1-l)\Big(-\eps m(u_1) + f(u_1)\Big) = \frac{s-l}{\eps}\!\int\limits_{u_1}^u f(t) e^{\frac{t-u}{\eps}}dt,
\end{equation*}
which coinsides with the right-hand side of~\eqref{NewtonLeibnizRewritten} due to the same equality. Formula~\eqref{NewtonLeibniz} is proved. And we have constructed the optimizers for the points on the upper parabola (under the assumption that~$\gamma$ is convex). How to construct the optimizers for all the other points? The answer is suggested by formula~\eqref{linearity}.

\begin{figure}[h!]
\begin{center}
\includegraphics{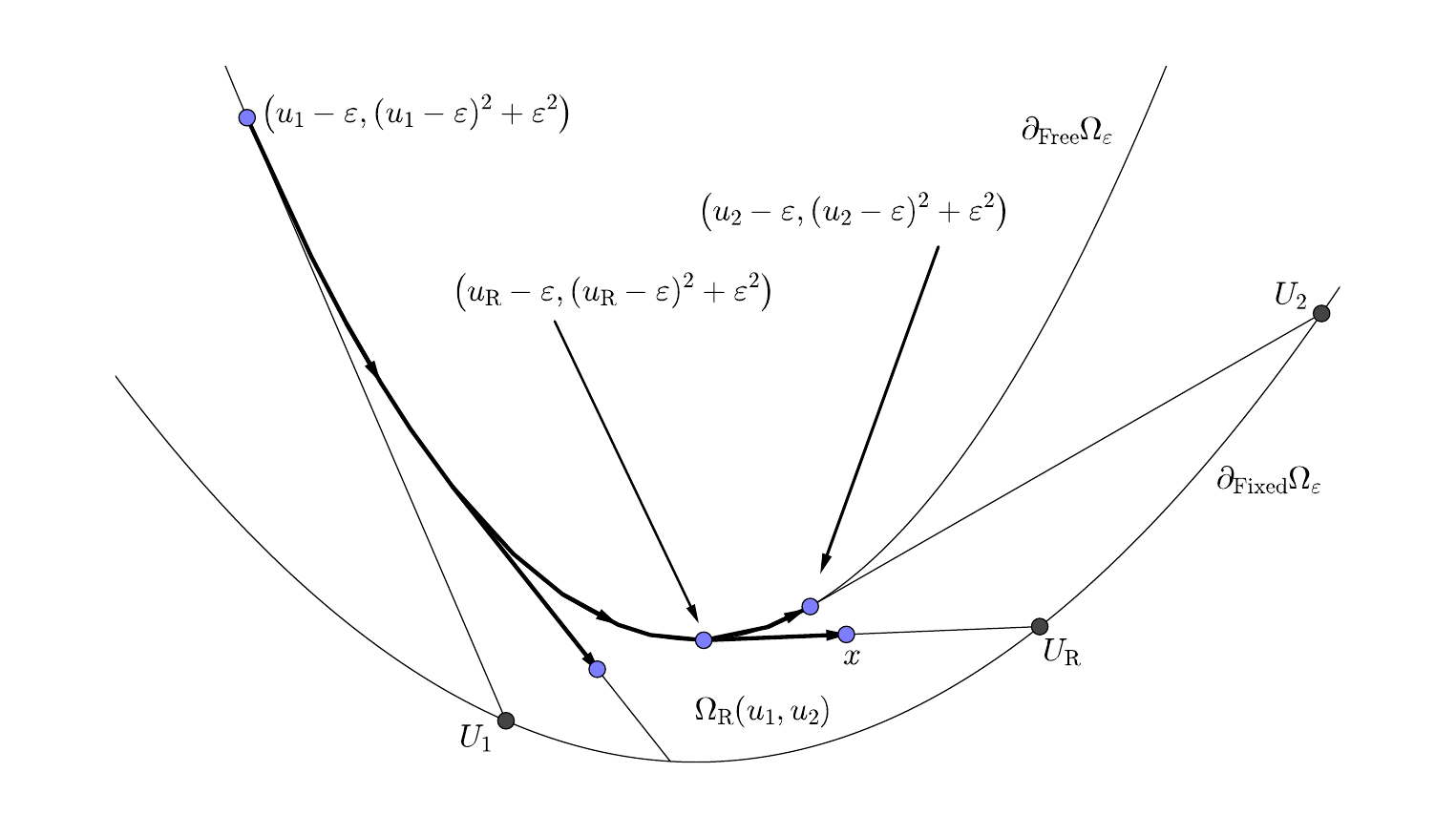}
\caption{Delivery curves inside a tangent domain.}
\label{fig:RtOpt}
\end{center}
\end{figure}

Let~$x \in \Rt(u_1,u_2)$, let~$\Ur$ be the point where the right tangent passing through~$x$ intersects the fixed boundary, i.e~$\ur$ is given by formula~\eqref{ur}. Suppose~$\psi \in \BMO_{\eps}([l,l_1])$ to be an optimizer for the point~$(\ur - \eps, (\ur - \eps)^2 + \eps^2)$. Formula~\eqref{linearity} suggests us the view of the optimizer~$\varphi\colon [l,r] \to \mathbb{R}$ for~$x$ (the delivery curve first goes from~$(u_1 - \eps, (u_1 - \eps)^2 + \eps^2)$ to~$(\ur - \eps, (\ur - \eps)^2 + \eps^2)$ and then to~$x$ along the tangent, see Figure~\ref{fig:RtOpt} for a visualization):
\begin{equation}\label{OptimizerRightTangentsAll}
\varphi(\tau) = 
\begin{cases}
\psi(\tau),\quad & \tau \in [l,l_1);\\
\ur,\quad & \tau \in [l_1,r].
\end{cases}
\end{equation}
The value of the parameter~$r$ is defined uniquely by the equation~$\av{\varphi}{[l,r]} = x_1$. The equality~$\av{B(\varphi)}{[l,r]} = B(x)$ follows from formula~\eqref{linearity}. So, we have constructed the optimizers for all the points in~$\Rt(u_1,u_2)$ (under the assumption that the obtained curve is convex).
\begin{St}\label{OptimizersRightTangentDomain}
Let~$B$ be a standard candidate on~$\Rt(u_1,u_2)$. Suppose that there exists an optimizer~$\psi$ for~$B$ at the point~$(u_1 - \eps, (u_1 - \eps)^2 + \eps^2)$\textup, which is non-decreasing and such that~$\psi \leq u_1$. Then\textup, there exists a non-decreasing optimizer~$\varphi_x$ for~$B$ at every point~$x \in \Rt(u_1,u_2)$\textup, moreover\textup,~$\varphi_x \leq u_2$.
\end{St}
\begin{proof}
Consider the delivery curve~$\gamma_{\psi}$ that corresponds to~$\psi$ on~$[l,l_1]$. We build~$\varphi$ with the help of formula~\eqref{OptimizerRightTangentsUpperParabola} on~$[l,r]$ for some~$r>l_1$, and consider the delivery curve it generates. This delivery curve~$\gamma_{\varphi}$ is a continuation of~$\gamma_{\psi}$ along the free boundary. Since~$\psi \leq u_1$, the function~$\varphi$ is monotone, and thus, by Lemma~\ref{Monotone_Convex},~$\gamma_{\varphi}$ is convex. We have to prove that~$\bp{\varphi|_{[s,t]}} \in \Omega_\eps$ for $l\leq s <t \leq r$. But for $t\leq l_1$ this follows from the fact that~$\psi$ is an optimizer and for $t>l_1$ this follows from Lemma~\ref{DeliveryCurveLemma}. Thus the function~$\varphi$ is an optimizer for the point~$x$ on the upper parabola.

To treat the points that are not on the upper parabola, we use formula~\eqref{OptimizerRightTangentsAll}. Again, the delivery curve is convex, and Lemma~\ref{DeliveryCurveLemma} justifies that the function we have built is an optimizer.
\end{proof}
We briefly state a symmetric proposition
\begin{St}\label{OptimizersLeftTangentDomain}
Let~$B$ be a standard candidate on~$\Lt(u_1,u_2)$. Suppose that there exists an optimizer~$\psi$ for~$B$ at the point~$(u_2 + \eps, (u_2 + \eps)^2 + \eps^2)$\textup, which is non-increasing and such that~$\psi \geq u_2$. Then\textup, there exists a non-increasing optimizer~$\varphi_x$ for~$B$ at every point~$x \in \Lt(u_1,u_2)$\textup, moreover\textup,~$\varphi_x \geq u_1$.
\end{St}
As usual, all the constructions are symmetric. We only mention that formula~\eqref{OptimizerRightTangentsUpperParabola} should be replaced with
\begin{equation}\label{OptimizerLeftTangentsUpperParabola}
\varphi(\tau) = 
\begin{cases}
\psi(\tau),\quad & \tau \in [l,l_1);\\
-\eps\log(\frac{\tau - l}{l_1-l}) + u_2,\quad & \tau \in [l_1,r],
\end{cases}
\end{equation}
provided~$\psi$ is defined on~$[l,l_1]$, formula~\eqref{OptimizerRightTangentsAll} is changed for
\begin{equation}\label{OptimizerLeftTangentsAll}
\varphi(\tau) = 
\begin{cases}
\psi(\tau),\quad & \tau \in [l,l_1);\\
\ul,\quad & \tau \in [l_1,r],
\end{cases}
\end{equation}
where~$\ul$ is given by formula~\eqref{ul} and~$\psi$ is an optimizer for~$(\ul+\eps, (\ul+\eps)^2 + \eps^2)$.

Similar propositions for infinite domains need additional study.
\begin{St}\label{OptimizerRightTangentsInfty}
Let~$B$ be the standard candidate on~$\Rt(-\infty,u_2)$. There exists a non-decreasing optimizer~$\varphi_x$ for~$B$ at every point~$x \in \Rt(-\infty,u_2)$\textup, moreover\textup,~$\varphi_x \leq u_2$.
\end{St} 
\begin{proof}
We begin with the points on the upper parabola. The finite case formula~\eqref{OptimizerRightTangentsUpperParabola} suggests us to take the function~$\varphi_x\colon (l,r] \to \mathbb{R}$ as follows:
\begin{equation*}
\varphi_x(\tau) = \eps\log(\tau-l),
\end{equation*}
where the value~$r$ is such that~$\varphi(r) = x_1 + \eps$,~$r = e^{\frac{x_1+\eps}{\eps}} + l$. Indeed, in such a case, the corresponding curve~$\gamma$ goes along the upper parabola from~$-\infty$ to~$x$. It remains to verify the identity~$B(x) = \av{f(\varphi_x)}{[l,r]}$. Again, we use the change of variable~$t = \varphi(\tau)$ or~$\tau = e^{\frac{t}{\eps}} + l$ for the right-hand side (and as usually, write~$u = x_1+\eps$)
\begin{equation*}
\frac{1}{r-l}\int\limits_{l}^r f(\varphi(\tau))\,d\tau = \frac{1}{r-l}\int\limits_{l}^r f\Big(\eps\log(\tau-l)\Big)\,d\tau = \frac{1}{\eps(r-l)}\int\limits_{-\infty}^{u} f(t)e^{\frac{t}{\eps}}\,dt.
\end{equation*}
As for the left-hand side, we take formula~\eqref{linearity}:
\begin{equation*}
B(x) = -\eps m(u) + f(u) \stackrel{\scriptscriptstyle{\eqref{minfty}}}{=} -e^{-u/\eps}\int\limits_{-\infty}^u f'(t)e^{t/\eps}\,dt + f(u) \stackrel{\hbox{\tiny Lem.}~\scriptscriptstyle{\ref{emb}}}{=} \frac{e^{-u/\eps}}{\eps}\int\limits_{-\infty}^u f(t)e^{t/\eps}\,dt.
\end{equation*}
Therefore, the right-hand and left-hand sides are equal by the formula~$r-l = e^{\frac{u}{\eps}}$.

The optimizers for all the other points of~$\Rt(-\infty,u_2)$ are constructed by formula~\eqref{OptimizerRightTangentsAll}.
\end{proof}
\begin{St}\label{OptimizerLeftTangentsInfty}
Let~$B$ be the standard candidate on~$\Lt(u_1,\infty)$. There exists a non-increasing optimizer~$\varphi_x$ for~$B$ at every point~$x \in \Lt(u_1,\infty)$\textup, moreover\textup,~$\varphi_x \geq u_1$.
\end{St}

\subsection{Optimizers for all other figures}\label{s522}
It is very easy to construct the optimizers for chordal domains. Indeed, the proposition below is a straightforward consequence of formula~\eqref{vallun}. The function~$\ell$ is defined at the beginning of Subsection~\ref{s331}.
\begin{St}\label{OptimizersChordalDomain}\index{domain! chordal domain}
Let~$B$ be the standard candidate on~$\Ch([a_0,b_0],[a_1,b_1])$. Then\textup, for any point~$x \in \Ch([a_0,b_0],[a_1,b_1])$\textup, the optimizer~$\varphi_x\colon \big[a\big(\ell(x)\big),b\big(\ell(x)\big)\big] \to \mathbb{R}$ for~$B$ at~$x$ is given by the formula
\begin{equation*}
\varphi_x(\tau) = 
\begin{cases}
b\big(\ell(x)\big), \quad & \tau\in \big[a\big(\ell(x)\big),x_1\big);\\
a\big(\ell(x)\big), \quad & \tau\in \big[x_1,b\big(\ell(x)\big)\big].
\end{cases}
\end{equation*}
\end{St}
\begin{proof}
It is obvious that~$B(x) = \av{f(\varphi_x)}{[a(\ell(x)),b(\ell(x))]}$ and~$\bp{\varphi_x} = x$. It is easy to see that all the Bellman points of~$\varphi$ lie on $[a(l(x)),b(l(x))] \subset \Omega_\eps$, therefore,~$\varphi_x \in \BMO_{\eps}$.%Surely, the function~$\varphi_x$ is monotone, and the corresponding curve~$\gamma$ is the segment~$[\big(b(\ell(x)),b^2(\ell(x))\big), x\big]$. Thus, the situation falls under the scope of Corollary~\ref{DeliveryCurveCorollary}, and~$\varphi_x \in \BMO_{\eps}$.
\end{proof}

\begin{Rem}
The optimizer~$\varphi_x$ we suggest for a chordal domain is non-increasing. One can construct a non-decreasing optimizer simply taking the function~$\tau\mapsto \varphi_x\big(a(\ell(x)) + b(\ell(x)) - \tau\big)$ defined on the same interval.
\end{Rem}

Let us now pass to the case of a closed multicup.
\begin{St}\label{OptimizersClosedMulticup}\index{multicup! closed multicup}
Let~$B$ be the standard candidate on~$\ClMTC(\{\mathfrak{a}_i\}_{i=1}^k)$. Then\textup, there exists a monotone  optimizer~$\varphi_x$ for~$B$ at any point~$x \in \ClMTC(\{\mathfrak{a}_i\}_{i=1}^k)$\textup, moreover\textup,~$\mathfrak{a}_1^{\mathrm l} \leq \varphi_x \leq \mathfrak{a}_k^{\mathrm r}$. 
\end{St}

\begin{proof}
Fix~$x\in \ClMTC(\{\mathfrak{a}_i\}_{i=1}^k)$. By elementary geometry considerations (alternatively, one may invoke the Caratheodory theorem),~$x$ is a convex combination of the three points~$A_1,A_2,A_3 \in \cup_{i=1}^k\mathfrak{A}_i$:
\begin{equation*}
x = \alpha_1 A_1 + \alpha_2 A_2 + \alpha_3 A_3;\, \quad \alpha_1+\alpha_2+\alpha_3=1, \quad \alpha_j \geq 0.
\end{equation*}
Without loss of generality, we may assume that~$a_1 \leq a_2 \leq a_3$. We put
\begin{equation*}
\varphi_x(\tau)=
\begin{cases}
a_1,\quad & \tau \in [0,\alpha_1);\\
a_2,\quad & \tau \in [\alpha_1,\alpha_1+\alpha_2);\\
a_3,\quad & \tau \in [\alpha_1+\alpha_2,1).
\end{cases}
\end{equation*}
The equality~$\av{\varphi_x}{[0,1]} = x$ is evident. The equality~$B(x) = \av{f(\varphi_x)}{[0,1]}$ follows from linearity of the Bellman candidate inside the closed multicup. The Bellman points of~$\varphi_x$ lie in the multicup, thus,~$\varphi_x \in \BMO_{\eps}$.
\end{proof}

\begin{St}\label{OptimizersMulticup}\index{multicup}
Let~$B$ be the standard candidate on~$\MTC(\{\mathfrak{a}_i\}_{i=1}^k)$. Then\textup, there exists a monotone optimizer~$\varphi_x$ for~$B$ at any point~$x \in \MTC(\{\mathfrak{a}_i\}_{i=1}^k)$\textup, moreover\textup,~$\mathfrak{a}_1^{\mathrm l} \leq \varphi_x \leq \mathfrak{a}_k^{\mathrm r}$. 
\end{St}

\begin{proof}
Consider the convex open set~$\Omega'$ that is the convex hull of~$\{x \in \mathbb{R}^2\mid x_1^2 + \eps^2 < x_2\}$ and the points~$\mathfrak{A}_1^{\mathrm{l}}$ and~$\mathfrak{A}_k^{\mathrm{r}}$. Since~$x \notin \Omega'$, by the separation theorem, there exists a line~$\varkappa = \varkappa(x)$ that passes through~$x$ and does not intersect~$\Omega'$. This line crosses the boundary of~$\MTC(\{\mathfrak{a}_i\}_{i=1}^k)$ twice. There may be two variants of such a crossing: either~$\varkappa$ crosses an arc~$\mathfrak{A}_i$,~$i \in \{1,2,\ldots,k\}$, or~$\varkappa$ crosses a segment~$[\mathfrak{A}_i^{\mathrm{r}},\mathfrak{A}_{i+1}^{\mathrm{l}}]$,~$i \in \{1,2,\ldots,k-1\}$ (by the condition~$\varkappa \cap \Omega' = \varnothing$, this line cannot cross the part of the boundary consisting of tangents and the arc of the upper parabola). Anyway, let these two points of intersection be~$y$ and~$z$, suppose that~$y_1 < z_1$. 

\begin{figure}[h!]
\begin{center}
\includegraphics[width = 0.9\linewidth]{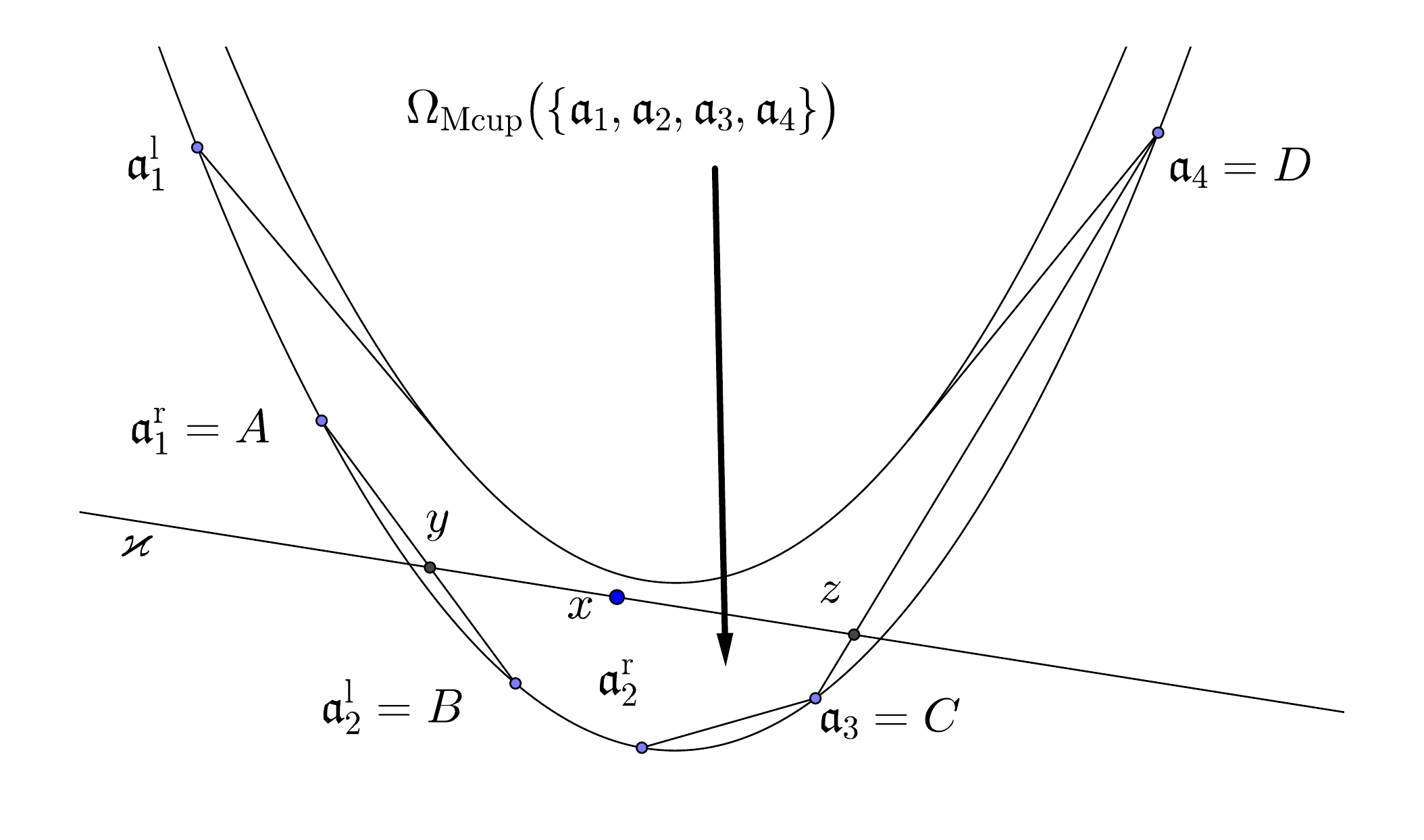}
\caption{Construction of optimizer in multicup.}
\label{fig:DCMC}
\end{center}
\end{figure}

If~$y$ lies on the chord~$[\mathfrak{A}_i^{\mathrm{r}},\mathfrak{A}_{i+1}^{\mathrm{l}}]$, then we can write
\begin{equation*}
y= \alpha_y \mathfrak{A}_i^{\mathrm{r}} + \beta_y \mathfrak{A}_{i+1}^{\mathrm{l}};\quad\alpha_y + \beta_y = 1, \quad \alpha_y,\beta_y \geq 0.
\end{equation*}
Similarly, if~$z$ lies on the chord~$[\mathfrak{A}_j^{\mathrm{r}},\mathfrak{A}_{j+1}^{\mathrm{l}}]$, then
\begin{equation*}
z= \alpha_z \mathfrak{A}_j^{\mathrm{r}} + \beta_z \mathfrak{A}_{j+1}^{\mathrm{l}};\quad\alpha_z + \beta_z = 1, \quad \alpha_z,\beta_z \geq 0.
\end{equation*}
So, in any case,~$y= \alpha_y A_1 + \beta_y A_2$,~$z = \alpha_z A_3 + \beta_z A_4$ (see Figure~\ref{fig:DCMC}), where~$A_1,A_2,A_3,A_4$ are some points on~$\MTC(\{\mathfrak{a}_i\}_{i=1}^k) \cap \FixedBoundary \Omega_{\eps}$ (if~$y$ is an intersection of~$\varkappa$ with an arc, then we may take~$A_1=A_2$; similarly with~$z$) such that
\begin{equation*}
a_1 \leq y_1 \leq a_2 \leq a_3 \leq z_1 \leq a_4.
\end{equation*} 
Define the optimizer on the interval~$[y_1,z_1]$ by the formula
\begin{equation}\label{OptimizerFormulaMulticup}
\varphi_x(t) = 
\begin{cases}
a_1,\quad &t \in [y_1,y_1 + \alpha_y(x_1 - y_1));\\
a_2,\quad &t \in [y_1 + \alpha_y(x_1 - y_1), x_1);\\
a_3,\quad &t \in [x_1, x_1+ \alpha_z(z_1 - x_1));\\
a_4,\quad &t \in [x_1+ \alpha_z(z_1 - x_1),z_1].
\end{cases}
\end{equation}
As usual, the equalities~$\bp{\varphi_x} = x$ and~$\av{f(\varphi_x)}{[y_1,z_1]} = B(x)$ are evident. However, if one draws the delivery curve for~$\varphi_x$, he sees that in some cases it does not fall under the scope of Lemma~\ref{DeliveryCurveLemma} (the tangent may cross the free boundary), so, we have to verify the condition~$\varphi_x \in \BMO_{\eps}$.

We claim that a point~$\av{\varphi_x}{J}$, where~$J \subset [y_1,z_1]$, either belongs to one of the segments~$[A_1,A_2]$ and~$[A_3,A_4]$, or is separated from the upper parabola by~$\varkappa$. Once the claim is proved, we see that~$\varphi_x \in \BMO_{\eps}$.

We will consider different cases of disposition of~$J$ inside~$[y_1,z_1]$. If~$J$ intersects all four intervals in formula~\eqref{OptimizerFormulaMulticup}, then we may represent~$x$ as a linear combination of~$\av{\varphi_x}{J}$,~$A_1$, and~$A_4$. Since~$A_1$ and~$A_4$ lie above~$\varkappa$ (i.e. in the same half-plane with the free boundary),~$\av{\varphi_x}{J}$ lies below~$\varkappa$.

So, we may suppose that~$J$ intersects at most three intervals from formula~\eqref{OptimizerFormulaMulticup}. Without loss of generality, we may assume that~$J \cap [x_1+ \alpha_z(z_1 - x_1),z_1] = \varnothing$. Then,~$\av{\varphi_x}{J}$ is a linear combination of the points~$A_1, A_2, A_3$.  Again, let~$J \cap [x_1, x_1+ \alpha_z(z_1 - x_1)) \ne \varnothing$ and~$J \cap[y_1,y_1 + \alpha_y(x_1 - y_1)) \ne \varnothing$. In such a case,~$\av{\varphi_x}{J}$ is a linear combination of~$A_3$ and a point from~$[y,A_2]$ (since~$[y_1 + \alpha_y(x_1 - y_1), x_1)\subset J$). Thus, if~$J \cap [x_1, x_1+ \alpha_z(z_1 - x_1)) \ne \varnothing$ and~$J \cap[y_1,y_1 + \alpha_y(x_1 - y_1)) \ne \varnothing$, then~$\av{\varphi_x}{J}$ is separated from the free boundary by~$\varkappa$.

Finally, if~$J \cap [x_1, x_1+ \alpha_z(z_1 - x_1)) = \varnothing$, then~$\av{\varphi_x}{J} \in [A_1,A_2]$; if~$J \cap[y_1,y_1 + \alpha_y(x_1 - y_1)) = \varnothing$, then~$\av{\varphi_x}{J} \in [A_2,A_3]$. So, we have verified the condition~$\varphi_x \in \BMO_{\eps}$.
\end{proof}

\begin{St}\label{OptimizersAngle}\index{angle}
Let~$B$ be the standard candidate in~$\Ang(w)$. Let~$\psi_1$ be a non-decreasing optimizer for the point~$(w-\eps, (w-\eps)^2 + \eps^2)$ such that~$\psi_1 \leq w$\textup, let~$\psi_2$ be a non-increasing optimizer for the point~$(w+\eps, (w+\eps)^2 + \eps^2)$ such that~$\psi_2 \geq w$. Then\textup, there exists an optimizer for every point~$x \in \Ang(w)$.
\end{St}
\begin{proof}
The proof of this proposition is very similar to the proof of Proposition~\ref{OptimizersMulticup}. First, there exist numbers~$\alpha_1,\alpha_2,\alpha_3$ such that
\begin{equation*}
x = \alpha_1(w-\eps, (w-\eps)^2 + \eps^2) + \alpha_2(w,w^2) + \alpha_3(w+\eps, (w+\eps)^2 + \eps^2);\quad \alpha_1 + \alpha_2+\alpha_3 = 1,\quad \alpha_j \geq 0.
\end{equation*}
Second, by Remark~\ref{Rescaling}, we may model~$\psi_1$ and~$\psi_2$ on any interval. Suppose that~$\psi_1$ is adjusted to~$[0,\alpha_1]$,~$\psi_2$ is adjusted to~$[0,\alpha_3]$. Define the optimizer~$\varphi_x$ on~$[0,1]$ by the formula
\begin{equation*}
\varphi_x(\tau) = 
\begin{cases}
\psi_1(\tau),\quad & \tau \in [0,\alpha_1);\\
w,\quad & \tau \in [\alpha_1,\alpha_1+\alpha_2);\\
\psi_2(1 - \tau),\quad & \tau \in [\alpha_1+\alpha_2,1].
\end{cases}
\end{equation*}
Again, the equalities~$\bp{\varphi_x} = x$ and~$\av{f(\varphi_x)}{[0,1]} = B(x)$ are evident. We have to verify that~$\varphi_x \in \BMO_{\eps}$. Let~$\varkappa$ be a line passing through~$x$ that separates it from~$\{x \in \mathbb{R}^2 \mid x_1^2 + \eps^2 < x_2\}$. First, we prove that~$\gamma_{\psi_1}$ lies above~$\varkappa$ and~$\gamma_{\psi_2}$ lies above~$\varkappa$ as well. As usual, it suffices to prove the claim about the former curve only. Since~$\psi_1$ is non-increasing, Lemma~\ref{Monotone_Convex} says that the curve~$\gamma_{\psi_1}$ is a graph of a convex function. The condition~$\psi_1 \leq w$ shows that the derivative of this function is not greater than~$2(w-\eps)$, by virtue of formula~\eqref{FirstDerivative}. Therefore,~$\gamma_{\psi_1}$ lies above the tangent to the upper parabola at the point~$(w-\eps,(w-\eps)^2 + \eps^2)$ and on the left of this point. So,~$\gamma_{\psi_1}$ lies above~$\varkappa$.

Let~$J \subset [0,1]$ be an interval, we have to prove that~$\av{\varphi_x}{J} \in \Omega_{\eps}$. Consider several cases of disposition of~$J$ inside~$[0,1]$. 

Suppose that~$[\alpha_1,\alpha_1+\alpha_2] \subset J$. Then we claim that~$\av{\varphi_x}{J}$ is separated from the upper parabola by~$\varkappa$. Since~$x$ is a linear combination of~$\av{\varphi}{[0,\alpha_1) \setminus J}$,~$\av{\varphi}{[\alpha_1+\alpha_2,1] \setminus J}$, and~$\av{\varphi}{J}$, and the first two points lie above~$\varkappa$,~$\av{\varphi}{J}$ lies below~$\varkappa$ indeed.

So, we may suppose that~$[\alpha_1,\alpha_1+\alpha_2]$ is not contained in~$J$. In such a case, Lemma~\ref{DeliveryCurveLemma} may be applied to~$\varphi|_{J}$, because~$\varphi|_{J}$ is a monotone function, and the tangent to the corresponding curve at the endpoint of~$J$ does not intersect the upper parabola.
\end{proof}

\begin{figure}[h!]
\begin{center}
\includegraphics{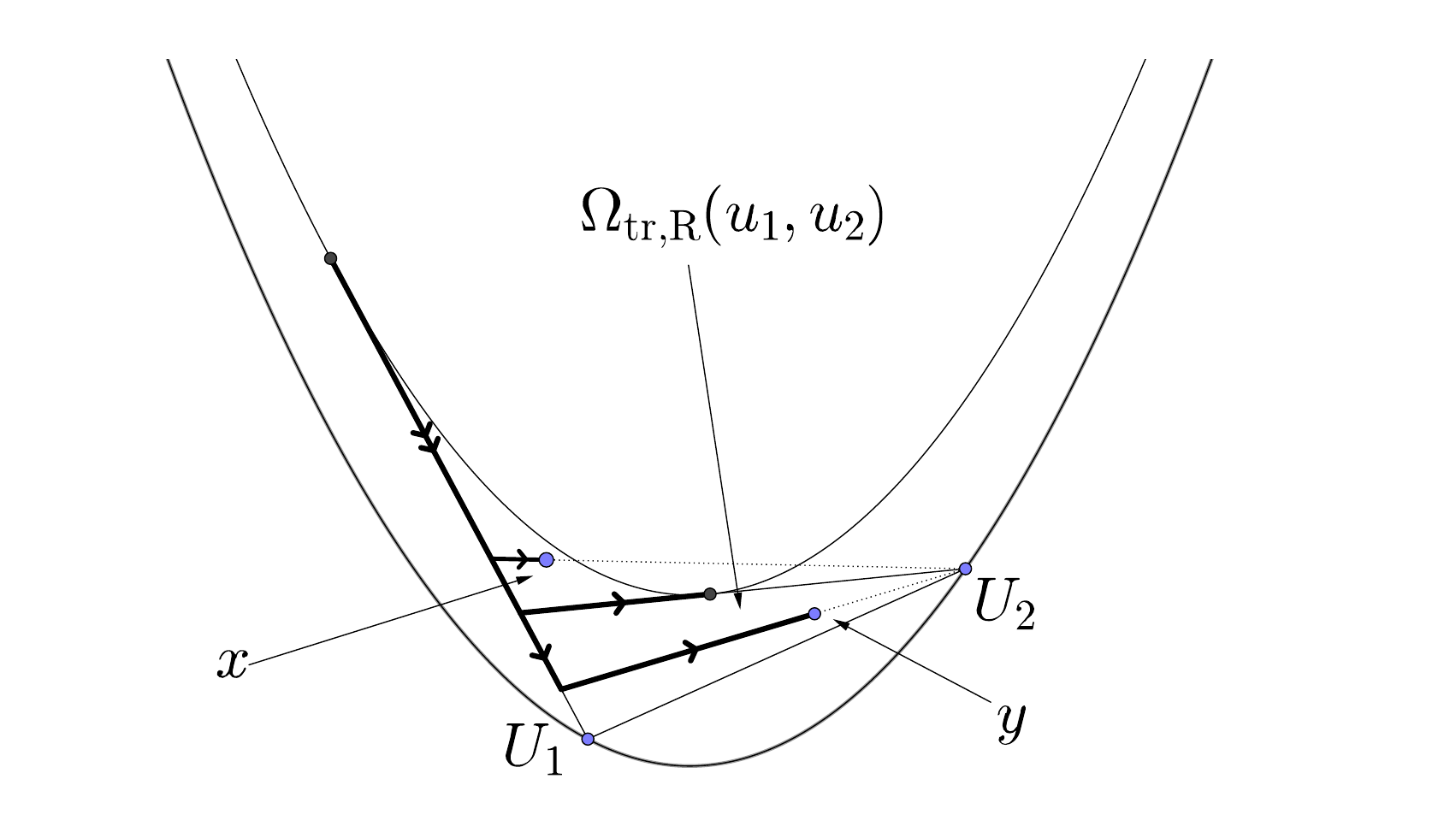}
\caption{Delivery curves inside a trolleybus.}
\label{fig:RTroll_DC}
\end{center}
\end{figure}

\begin{St}\label{OptimizersTrolleybusR}\index{trolleybus}
Let~$B$ be the standard candidate in~$\RTroll(u_1,u_2)$. Suppose~$\psi$ to be a non-decreasing optimizer for the point~$(u_1 - \eps, (u_1 - \eps)^2 + \eps^2)$ such that~$\psi \leq u_1$. Then\textup, for any~$x \in \RTroll(u_1,u_2)$ there exists a non-decreasing optimizer that does not exceed~$u_2$.  
\end{St}
\begin{proof}
Choose any point~$x \in \RTroll(u_1,u_2)$. The trolleybus~$\RTroll(u_1,u_2)$ lies inside the triangle with the vertices~$U_1,U_2, (u_1 - \eps, (u_1 - \eps)^2 + \eps^2)$. So, there exist~$\alpha_0,\alpha_1,\alpha_2$ such that
\begin{equation*}
x = \alpha_0(u_1 - \eps, (u_1 - \eps)^2 + \eps^2) + \alpha_1 U_1 + \alpha_2 U_2;\quad \alpha_0+\alpha_1+\alpha_2 = 1, \quad \alpha_j \geq 0.
\end{equation*}
Again, by Remark~\ref{Rescaling}, we may assume that~$\psi$ is adjusted to~$[0,\alpha_0]$. Define the optimizer~$\varphi_x\colon [0,1] \to (-\infty,u_2]$ by the formula
\begin{equation*}
\varphi_x(\tau) = 
\begin{cases}
\psi(\tau),\quad & \tau \in [0,\alpha_0);\\
u_1,\quad & \tau \in [\alpha_0,\alpha_0+\alpha_1);\\
u_2,\quad & \tau \in [\alpha_0+\alpha_1,1].
\end{cases}
\end{equation*}
As usual, we only have to verify the condition~$\varphi_x \in \BMO_{\eps}$, or~$\av{\varphi_x}{J} \in \Omega_{\eps}$,~$J \subset[0,1]$. Select a line~$\varkappa = \varkappa(x)$ that passes through~$x$ and separates it from the free boundary and the point~$U_2$. As in the proof of Propositions~\ref{OptimizersMulticup} and~\ref{OptimizersAngle}, we will have to consider different cases of location of~$J$ inside~$[0,1]$.

If~$[\alpha_0,\alpha_0+\alpha_1) \subset J$, then~$\av{\varphi_x}{J}$ is separated from the upper parabola by~$\varkappa_x$ (this standard reasoning had already been done during the proof of Propositions~\ref{OptimizersMulticup} and~\ref{OptimizersAngle}).

If~$J \cap (\alpha_0+\alpha_1,1] = \varnothing$, then the situation falls under the scope of Lemma~\ref{DeliveryCurveLemma}.

If~$J \cap [0,\alpha_0) = \varnothing$, then~$\av{\varphi_x}{J} \subset [U_1,U_2]$.

So, we have considered all the cases and verified that~$\varphi_x \in \BMO_{\eps}$.
\end{proof}
As usual, we have a symmetric proposition.

\begin{St}\label{OptimizersTrolleybusL}
Let~$B$ be the standard candidate in~$\LTroll(u_1,u_2)$. Suppose~$\psi$ to be a non-increasing optimizer for the point~$(u_2 + \eps, (u_1 + \eps)^2 + \eps^2)$ such that~$\psi \geq u_2$. Then\textup, for any~$x \in \LTroll(u_1,u_2)$ there exists a non-increasing optimizer that is not less than~$u_1$.  
\end{St}
It remains to construct the optimizers for multitrolleybuses, birdies, and multibirdies. Formulas from Subection~\ref{s344} will help us in this business.

\begin{St}\label{OptimizersMultitrolleybusR}\index{multitrolleybus}
Let~$B$ be the standard candidate in~$\MTTR(\{\mathfrak{a}_i\}_{i=1}^k)$. Suppose~$\psi$ to be a non-decreasing optimizer for the point~$(\mathfrak{a}_1^{\mathrm{l}} - \eps, (\mathfrak{a}_1^{\mathrm{l}} - \eps)^2 + \eps^2)$ such that~$\psi \leq \mathfrak{a}_1^{\mathrm{l}}$. Then\textup, for any~$x \in \MTTR(\{\mathfrak{a}_i\}_{i=1}^k)$ there exists a non-decreasing optimizer that does not exceed~$\mathfrak{a}_k^{\mathrm r}$.  
\end{St}

\begin{proof}
We apply formula~\eqref{RMultitrolleybusDesintegration} and decompose a multitrolleybus in an alternating sequence of right trolleybuses and right tangent domains. Applying Proposition~\ref{OptimizersTrolleybusR} to the trolleybuses and Proposition~\ref{OptimizersRightTangentDomain} to the tangent domains consecutively, we build optimizers for all the points inside~$\MTTR(\{\mathfrak{a}_i\}_{i=1}^k)$. 
\end{proof}

\begin{St}\label{OptimizersMultitrolleybusL}
Let~$B$ be the standard candidate in~$\MTTL(\{\mathfrak{a}_i\}_{i=1}^k)$. Suppose~$\psi$ to be a non-increasing optimizer for the point~$(\mathfrak{a}_k^{\mathrm{r}} + \eps, (\mathfrak{a}_k^{\mathrm{r}} + \eps)^2 + \eps^2)$ such that~$\psi \geq \mathfrak{a}_k^{\mathrm{r}}$. Then\textup, for any~$x \in \MTTL(\{\mathfrak{a}_i\}_{i=1}^k)$ there exists a non-increasing optimizer that is not less than~$\mathfrak{a}_1^{\mathrm l}$.  
\end{St}

\begin{St}\label{OptimizersMultiBirdie}\index{multibirdie}
Let~$B$ be the standard candidate in~$\MTB(\{\mathfrak{a}_i\}_{i=1}^k)$. Let~$\psi_1$ be a non-decreasing optimizer for the point~$(\mathfrak{a}_1^{\mathrm{l}}-\eps, (\mathfrak{a}_1^{\mathrm{l}}-\eps)^2 + \eps^2)$ such that~$\psi_1 \leq \mathfrak{a}_1^{\mathrm l}$\textup, let~$\psi_2$ be a non-increasing optimizer for the point~$(\mathfrak{a}_k^{\mathrm{r}}+\eps, (\mathfrak{a}_k^{\mathrm{r}}+\eps)^2 + \eps^2)$ such that~$\psi_2 \geq \mathfrak{a}_k^{\mathrm r}$. Then\textup, there exists an optimizer for every point~$x \in \MTB(\{\mathfrak{a}_i\}_{i=1}^k)$.
\end{St}

\begin{proof}
We apply formula~\eqref{MultibirdieDesintegration} (with any choice of~$j$) and decompose the multibirdie into a collection of trolleybuses, tangent domains, and angle. Applying Propositions~\ref{OptimizersTrolleybusR} and~\ref{OptimizersTrolleybusL} to the trolleybuses, Propositions~\ref{OptimizersRightTangentDomain} and~\ref{OptimizersLeftTangentDomain} to the tangent domains, and Proposition~\ref{OptimizersAngle} to the angle consecutively, we build optimizers for all the points inside~$\MTB(\{\mathfrak{a}_i\}_{i=1}^k)$.  
\end{proof}

\section{Global optimizers}\label{s53}
Before passing to global optimizers, we need to introduce the notions of \index{incoming node}incoming and \index{outcoming node}outcoming nodes for a vertex or an edge in the foliation graph. These nodes are points on the upper parabola needed to transfer the delivery curve from one domain to another. As the reader might noticed, for some figures, the optimizers inside them depend on the optimizer coming from the right or from the left (for example, to build the optimizers in the trolleybus, we need an optimizer for a special point on the upper parabola, see Propositions~\ref{OptimizersTrolleybusR} and~\ref{OptimizersTrolleybusL}). So, the incoming node is the point where we start the delivery curve from, whereas the outcoming node is the point where it leaves the figure. In the tables below, we give a precise definition.

\medskip
\centerline{
\begin{tabular}{|p{2.0cm}|p{2.3cm}|p{2.3cm}|p{2.3cm}|p{2.3cm}|p{2.3cm}|}
\hline 
&{\tiny Long chord}~$\scriptstyle[A,B]$&$\scriptstyle\MTC(\{\mathfrak{a}_i\}_{i=1}^k)$&$\RTroll(u_1,u_2)$&$\LTroll(u_1,u_2)$&$\Rt(u_1,u_2)$\\
\hline
Outcoming&$\scriptstyle\frac{A+B}{2}$&$\scriptscriptstyle(\mathfrak{a}_1^{\mathrm{l}}+\eps,(\mathfrak{a}_1^{\mathrm{l}}+\eps)^2 + \eps^2)$, $\scriptscriptstyle(\mathfrak{a}_k^{\mathrm{r}}-\eps,(\mathfrak{a}_k^{\mathrm{r}}-\eps)^2 + \eps^2)$&$\scriptscriptstyle(u_2 - \eps,(u_2-\eps)^2+\eps^2)$&$\scriptscriptstyle(u_1 + \eps,(u_1+\eps)^2+\eps^2)$&$\scriptscriptstyle(u_2 - \eps,(u_2-\eps)^2+\eps^2)$\\
\hline
Incoming&&&$\scriptscriptstyle(u_1 - \eps,(u_1-\eps)^2+\eps^2)$&$\scriptscriptstyle(u_2 + \eps,(u_2+\eps)^2+\eps^2)$&$\scriptscriptstyle(u_1 - \eps,(u_1-\eps)^2+\eps^2)$\\
\hline
\end{tabular}
}
\medskip
\medskip
\centerline{
\begin{tabular}{|p{2.0cm}|p{2.3cm}|p{2.3cm}|p{2.3cm}|p{2.3cm}|p{2.3cm}|}
\hline 
&$\Lt(u_1,u_2)$&$\scriptstyle \MTTR(\{\mathfrak{a}_i\}_{i=1}^k)$&$\scriptstyle \MTTL(\{\mathfrak{a}_i\}_{i=1}^k)$&$\scriptstyle\Ang(w)$&$\scriptstyle\MTB(\{\mathfrak{a}_i\}_{i=1}^k)$\\
\hline
Outcoming&$\scriptscriptstyle(u_1 + \eps,(u_1+\eps)^2+\eps^2)$&$\scriptscriptstyle (\mathfrak{a}_k^{\mathrm{r}}-\eps,(\mathfrak{a}_k^{\mathrm{r}}-\eps)^2 + \eps^2)$&$\scriptscriptstyle (\mathfrak{a}_1^{\mathrm{l}}+\eps,(\mathfrak{a}_1^{\mathrm{l}}+\eps)^2 + \eps^2)$&&\\
\hline
Incoming&$\scriptscriptstyle(u_2 + \eps,(u_2+\eps)^2+\eps^2)$&$\scriptscriptstyle (\mathfrak{a}_1^{\mathrm{l}}-\eps,(\mathfrak{a}_1^{\mathrm{l}}-\eps)^2 + \eps^2)$&$\scriptscriptstyle (\mathfrak{a}_k^{\mathrm{r}}+\eps,(\mathfrak{a}_k^{\mathrm{r}}+\eps)^2 + \eps^2)$&$\scriptscriptstyle (w-\eps,(w-\eps)^2+\eps^2)$ $\scriptscriptstyle (w+\eps,(w+\eps)^2+\eps^2)$&$\!\!\scriptscriptstyle (\mathfrak{a}_1^{\mathrm{l}}-\eps,(\mathfrak{a}_1^{\mathrm{l}}-\eps)^2+\eps^2)$ $\scriptscriptstyle (\mathfrak{a}_k^{\mathrm{r}}+\eps,(\mathfrak{a}_k^{\mathrm{r}}+\eps)^2+\eps^2)$\\
\hline
\end{tabular}
}
\medskip

All the other figures do not have incoming or outcoming nodes at all. Now we see that the propositions of Subsections~\ref{s521} and~\ref{s522} are of the following form: if there is a monotone optimizer for the incoming node satisfying a certain bound  (if there are incoming nodes for the figure in question), then we can build monotone optimizers for the whole figure satisfying a similar bound for the outcoming nodes (if there are any). So, these propositions are very well suited for induction. To state the general theorem about optimizers, we need Definition~\ref{Admissible graph} (see Remark~\ref{AdmissibleCondidate} and the paragraph before it, where it is explained that an admissible graph generates a Bellman candidate).
\begin{Th}\label{GlobalOptimizer}
Let~$\eps < \eps_{\infty}$\textup, let~$\Gamma(\eps)$ be an admissible for~$\eps$ and~$f$ graph. The Bellman candidate~$B$ generated by~$\Gamma(\eps)$ admits an optimizer at every point of~$\Omega_{\eps}$.  
\end{Th}
\begin{proof}
We begin with building optimizers for the points~$x\in \Omega_{\eps}$ that belong to chordal domains, closed multicups, and multicups. For them, the optimizers are given by Proposition~\ref{OptimizersChordalDomain} (chordal domains and fictious vertices of the first and third types),~\ref{OptimizersClosedMulticup} (closed multicups), and~\ref{OptimizersMulticup} (multicups), and do not depend on the global structure of~$\Gamma(\eps)$.

Note that all other points belong to the figures from~$\GammaFree(\eps)$. We have built the optimizers for the roots of this graph (which are fictious vertices of the first type, fictious vertices of the third type representing long chords, multicups, and fictious vertices of the fourth type). Now we turn to the intermediate vertices of~$\GammaFree(\eps)$, i.e. vertices that are neither roots nor leaves, and edges. We are going to build the optimizers (delivery curves) consecutively, moving from roots to the leaves along~$\GammaFree(\eps)$ and inducting on height of the vertex in question (see Definition~\ref{Height}). It is natural to extend the notion of height to edges of~$\GammaFree$: the height of an edge is the height of its beginning plus one.

We will prove the following statement by induction: for any point inside the figure corresponding to an edge or an intermediate vertex of height~$h$, there exists an optimizer; moreover, if this edge or vertex is right-oriented (i.e. a right tangent domain, or a right trolleybus, or a right multitrolleybus, or a right fictious vertex of the fifth type), then the optimizer does not exceed~$u+\eps$, where~$u$ is the first coordinate of the outcoming node of the figure in question; if the figure is left oriented, then the optimizer is not less than~$u-\eps$.

Let~$h=1$. Consider the edges first, let~$\Rt(u_1,u_2)$ be represented by such an edge. Since we have built the optimizers for all the roots of~$\GammaFree(\eps)$, there is an optimizer~$\psi$ for the incoming node of~$\Rt(u_1,u_2)$. A brief inspection of the optimizers we have built for the roots of~$\GammaFree(\eps)$ shows that~$\psi$ is bounded from above by~$u_1$ and also may be chosen to be non-decreasing. Thus, we may apply  Proposition~\ref{OptimizersRightTangentDomain} and proceed this optimizer to the whole tangent domain. The optimizer in the outcoming node will be bounded by~$u_2$. If~$u_1 = -\infty$, then we simply use Proposition~\ref{OptimizerRightTangentsInfty}. So, we have built all the optimizers for all tangent domains of height one.

Now we are able to pass to intermediate vertices of height~$1$. Each such vertex is either a trolleybus, or multitrolleybus, or fictious vertex of the fifth type (for which the optimizers are already built). We see that its incoming node coincides with the outcoming node of its incoming edge. Since we have already built the monotone optimizer that fulfills the required bound for this point, we may apply Propositions~\ref{OptimizersTrolleybusR} and~\ref{OptimizersTrolleybusL} to trolleybuses and Propositions~\ref{OptimizersMultitrolleybusR} and~\ref{OptimizersMultitrolleybusL} to multitrolleybuses to build the optimizers in these figures. These optimizers are monotone and satisfy the bound from below or above (depending on the orientation) by the first coordinate of the outcoming node shifted by~$\pm\eps$ correspondingly. So, we have considered the case~$h=1$.

We note that the case~$h=2$ is totally similar: we first proceed the optimizers from outcoming nodes of the vertices of height~$1$ to edges of height~$2$ with the help of Prpositions~\ref{OptimizersRightTangentDomain} and~\ref{OptimizersLeftTangentDomain}, and then proceed the optimizers further inside the domains represented by vertices of height~$2$. Inducting on height, we build the optimizers for all the edges of~$\GammaFree(\eps)$ and all intermediate vertices of this graph. 

It remains to construct the optimizers for the leaves of~$\GammaFree(\eps)$. They are angles, multibirdies, and, maybe, fictious vertices of the fourth type (for which we do not have to construct anything). For any of the leaves, we have constructed the optimizers for its incoming nodes, because they coincide with the outcoming nodes of its incoming edges. So, we simply apply Propositions~\ref{OptimizersAngle} and~\ref{OptimizersMultiBirdie}.  
\end{proof}
\begin{Th}\label{Final}
For any~$\eps < \eps_{\infty}$ there exists a Bellman candidate~$B_{\eps}$ that coincides with~$\Bell$. 
\end{Th}
\begin{proof}
By Theorem~\ref{BC}, there exists an admissible for~$\eps$ and~$f$ graph~$\Gamma(\eps)$. By Remark~\ref{AdmissibleCondidate} it generates a Bellman candidate~$B_{\eps}$. By Definition~\ref{candidate},~$B_{\eps}$ is locally concave, thus, by Lemma~\ref{LMaj},~$\Bell \leq B_{\eps}$. On the other hand, by Theorem~\ref{GlobalOptimizer}, for each~$x \in \Omega_{\eps}$, there exists an optimizer~$\varphi_x$ and, thus,~$B_{\eps} \leq \Bell$.
\end{proof}
\centerline{\bf In particular, we have proved Theorem~\ref{MT}.}

\section{Examples}\label{s54}
\paragraph{Wandering angle.} In this example, we show how the general knowledge about evolution helps to build the Bellman function. Our aim here is to construct a function~$f$ such that its foliation consists of a single angle for all~$\eps$, and this angle wanders around zero (i.e. there are infinite number of moments~$\eps$ such that the angle is sitting at the point~$0$).

Define~$f$ by the formula
\begin{equation*}
f''' = \sum\limits_{k=1}^{\infty} (-1)^k\chi_{(-1)^kI_k},\quad \hbox{where}\quad I_k = [2^{2^k},2^{2^k+1}]. 
\end{equation*} 
We see that the function~$f$ satisfies Conditions~\ref{reg} and~\ref{sum} (with~$\eps_{\infty} = \infty$), because the only essential root of~$f'''$ is the interval between~$-I_1$ and~$I_2$. Since~$f$ has only one essential root, which is~$v_1$ (see Definition~\ref{roots}), there are only two scenarios of the evolution here: for sufficiently small~$\eps$ there is an angle with two tangent domains coming from infinity, the angle may either live forever, or escape as it did in the first example of Subsection~\ref{s353}. We will prove that the angle lives forever and, moreover, it visits the point~$0$ infinitely many times.

Let us show that~$\Fr(0;-\infty;2^{2^l}) + \Fl(0;\infty;2^{2^l}) > 0$  if~$l$ is odd and~$\Fr(0;-\infty;2^{2^l}) + \Fl(0;\infty;2^{2^l}) < 0$ if~$l$ is even, provided~$l$ is sufficiently big. Using formulas~\eqref{RightForceInfinity} and~\eqref{LeftForceInfinity}, we see that
\begin{equation*}
\Fr(0;-\infty;\eps) + \Fl(0;\infty;\eps) = \int\limits_{\mathbb{R}}e^{\frac{-|t|}{\eps}}f'''(t)\,dt,
\end{equation*}
which leads to
\begin{equation}\label{BalanceForWandAng}
\Fr(0;-\infty;\eps) + \Fl(0;\infty;\eps) = \sum\limits_{k=1}^{\infty}(-1)^k\eps e^{\frac{-2^{2^k}}{\eps}}\Big(e^{\frac{-2^{2^k}}{\eps}}-1\Big).
\end{equation}
Now let~$\eps = 2^{2^l}$. If~$k < l$, then
\begin{equation*}
\Big|e^{\frac{-2^{2^k}}{2^{2^l}}}\Big(e^{\frac{-2^{2^k}}{2^{2^l}}}-1\Big)\Big| \leq\Big|e^{-2^{-2^{l-1}}}-1\Big| \leq 2^{-2^{l-1}}.
\end{equation*}
If~$k > l$, then
\begin{equation*}
\Big|e^{\frac{-2^{2^k}}{2^{2^l}}}\Big(e^{\frac{-2^{2^k}}{2^{2^l}}}-1\Big)\Big| \leq e^{-2^{2^{k-1}}}.
\end{equation*}
So, plugging these two estimates into formula~\eqref{BalanceForWandAng}, we see that
\begin{equation*}
\begin{aligned}
\Fr(0;-\infty;2^{2^l}) + \Fl(0;\infty;2^{2^l}) = 2^{2^{l}}\Big((-1)^{l}e^{-1}(e^{-1}-1) + O\big(\sum\limits_{k<l}2^{1-2^{l-1}}\big) + O\big(\sum\limits_{k>l}e^{-2^{2^{k-1}}}\big)\Big) =\\ 2^{2^{l}}\big((-1)^{l}e^{-1}(e^{-1}-1) + o(1)\big).
\end{aligned}
\end{equation*}
The claim is proved. 

It follows that there exists a sequence of numbers~$\{\eps_n\}_n$ tending to infinity such that
\begin{equation*}
\Fr(0;-\infty;\eps_n) + \Fl(0;\infty;\eps_n) = 0.
\end{equation*}
Thus,~$f$ and~$\eps_n$ fall under the scope of Proposition~\ref{AngleProp}, which means that there exists a Bellman candidate with the foliation~$\Rt(-\infty,0)\cup\Ang(0)\cup\Lt(0,\infty)$. By Theorem~\ref{GlobalOptimizer}, it coincides with the Bellman function. Therefore, the angle cannot dissapear during the evolution. Due to Lemma~\ref{monbaleq} and ``balance inequalities'' proved above, when~$\eps =2^{2^l}$ and~$l$ is even, the angle is situated on the right of~$0$, whereas for odd~$l$ it is on the left.

\paragraph{Oscillating birdie.} In this example, our aim is to construct a function~$f$ that has a birdie at the moment~$\eps$, and this birdie desintegrates in a non-regular way.

Let~$f_0 \in C^{\infty}$ be the polynomial of sixth degree that has a birdie (see Section~\ref{s45}, we take~$f_0'''(t) = t^3 -3t$, see Subsection~\ref{s413} as well). Take~$\eps' \in \big(\sqrt{\frac{35}{9}},\frac{1}{\sqrt{2}}\big)$, then for all~$\eps \in [\eps', \eps' + \delta)$ ($\delta$ is a small number) the foliation for~$f_0$ consists of a single symmetric birdie~$\Bird\big(a(\eps),b(\eps)\big)$, $a(\eps) = - b(\eps)$, and two families of tangents surrounding it. What is more, we recall that~$f_0'''$ and the differentials of the base of the birdie are separated from zero in neighborhoods of $a(\eps')$ and $b(\eps')$.

Let~$\{\eps_k\}_{k \in \mathbb{N}}$ be a decreasing sequence such that $\eps_1 < \eps' + \delta$ and $\eps_k \to \eps'$ as $k \to \infty$. For brevity, we denote~$b_k \df b(\eps_k)$ and~$a_k \df a(\eps_k)$. Therefore, $a_k$ decreases downto~$a \df a(\eps')$ and~$b_k$ increases to~$b \df b(\eps')$. We are going to perturb the function~$f_0$ by a~$C^{\infty}$ function in such a way that the foliation of the perturbed function does not change at the moments~$\eps_k$, but it consists of a trolleybus and an angle for some moment~$\eps$ between each~$\eps_{k+1}$ and~$\eps_k$ (what is more, we can construct the perturbation in such a fashion that the trolleybus changes its orientation each time). This will justify the difficulties we had with the evolutional step for birdies and multibirdies in Proposition~\ref{MultibirdieDesintegrationSt}. 

The perturbation will be of the form~$h = \sum_{k = 1}^{\infty}h_k$, where the sum converges in~$C^{3}$. The perturbed function~$f_0 + h$ is denoted by~$f$. Each function~$h_k$ is supported either inside~$(b_k,b_{k+1})$ or inside~$(a_{k+1},a_k)$, what is more, it has sufficiently small $C^{3}$-norm,~$\frac{C}{k^2}$ will do; here~$C$ is chosen to be so small that the perturbation neither harms the roots of~$f'''$ (we add a perturbation on the intervals where~$f'''$ is uniformly separated from zero) nor changes the sign of the differentials of~$f'''$. 

It is easy to see that the pairs~$(a_k,b_k)$ satisfy the cup equation~\eqref{urlun} for the function~$f$. Indeed, the functions~$h_k$ are supported inside the intervals~$(b_k,b_{k+1})$ or~$(a_{k+1},a_k)$, therefore,~$f'(a_k) = f'_0(a_k)$, and the same equalities hold for all higher derivatives and the points of the type~$b$ as well. However, to control the forces coming from the infinities, we need to ask for some additional properties of the functions~$h_k$. Suppose that~$h_k$ is supported on the right of zero, then these conditions look like this:
\begin{equation*}
\int\limits_{b_k}^{b_{k+1}} h_k'''(t)e^{-\frac{t}{\eps_l}}\, dt = 0, \quad l = 1,\ldots,k.
\end{equation*}
For those~$h_k$ that lie on the left of zero, the conditions are symmetric:
\begin{equation*}
\int\limits_{a_{k+1}}^{a_k} h_k'''(t)e^{\frac{t}{\eps_l}}\, dt = 0, \quad l = 1,\ldots,k.
\end{equation*}
Surely, each function~$h_k$ is under a finite number of integral conditions, therefore, such functions exist (what is more, they can have as small $C^3$-norm as wanted). One can easily see that the forces coming from the infinities do not change their values at the points~$a_k$ and~$b_k$ after the perturbation. Therefore, the foliation for the function~$f$ at the moment~$\eps_k$ consists of the birdie~$\Bird(a_k,b_k)$ and two families of tangents surrounding it. 

We only have to verify that for each~$k$ there is some~$\eps$,~$\eps \in (\eps_{k+1},\eps_k)$, such that the foliation consists of a trolleybus and an angle at the moment~$\eps$ (in such a case one can obtain the needed orientation of the trolleybus by reflecting the function~$h_k$ with respect to zero). For example, this can be done in the following way: we split the interval~$(b_k,b_{k+1})$ into two intervals,~$\big(b_k,b(\eps'_k)\big)$ and~$\big(b(\eps'_k),b_{k+1}\big)$ for some~$\eps'_k \in (\eps_{k+1},\eps_k)$ and say that~$h_k$ is supported in~$[b(\eps'_k),b_{k+1}]$. Then we ask for one more integral condition, namely,
\begin{equation*}
\int\limits_{b(\eps_k')}^{b_{k+1}} h_k'''(t)e^{-\frac{t}{\eps_k'}}\, dt > 0 \quad \text{ and } \quad \int\limits_{b_m}^{b_{m+1}} h_m'''(t)e^{-\frac{t}{\eps_k'}}\, dt \ = \int\limits_{a_{m+1}}^{a_m} h_m'''(t)e^{\frac{t}{\eps_k'}}\, dt = 0 \quad  \text{ for } m > k.
\end{equation*}
Surely, this condition breaks the birdie into a trolleybus and an angle. The functions~$h_k$ that satisfy all the conditions required exist, because the conditions are linearly independent and each function~$h_k$ is under a finite number of conditions.  

The evolution in~$[\eps', \eps' + \delta]$ for the function~$f$ is as follows: at the  points~$\eps_k$, the foliation is
\begin{equation*}
\Rt(-\infty,a_k) \cup \Bird(a_k,b_k) \cup\Ch([a_k,b_k],[0,0]) \cup \Lt(b_k,\infty),
\end{equation*}
whereas for the points between the~$\eps_k$, the foliation consists of a trolleybus, an angle, a chordal domain, and two tangent domains.

\chapter{Related questions and further development}

\section{Related questions}\label{s61}
\subsection{Analytic properties}\label{s611}
In Subsection~\ref{s443}, we proved that there is only a finite number of essential critical points of the evolution. However, this may seem a bit ``unfair'', because these critical point  are only those moments~$\eps$ at which~$\GammaFixed$ changes (see the beginning of Section~\ref{s44} and the proof of Theorem~\ref{BC}). We have not proved that there is only a finite number of moments~$\eps$ where~$\GammaFree$ changes when~$\eps$ belongs to an interval. In particular, it is interesting whether the pattern stabilizes as~$\eps \searrow \eps'$ in Proposition~\ref{MultibirdieDesintegrationSt}. Informally, does a birdie or a multibirdie choose one of many possible ways to desintegrate? %We believe that the answer is ``no'', however, we are not able to construct a counterexample (see Conjecture~\ref{BirdieOscCConj}). 
In general, the answer is ``no'', see the second example in Section~\ref{s54} (this emphasizes that the difficulties we had to overcome in Proposition~\ref{MultibirdieDesintegrationSt} had their origin in the nature of the evolution, not in the authors' heads).

The answer is ``yes'' if one assumes that~$f$ is piecewise analytic. Then, the algorithm is finite in the sense that the graph~$\Gamma(\eps)$ changes only a finite number of times when~$\eps \in (0,\eps_{\infty})$. Indeed, if~$f$ is piecewise analytic, then all the functions~$a$ and~$b$ associated with chordal domains are piecewise analytic (by the implicit function theorem). Thus, the forces are piecewise analytic as functions of~$u$. So, the roots of balance equations are also piecewise analytic as functions of~$\eps$ (again, by the implicit function theorem). Therefore, the graph~$\Gamma(\eps)$ changes only a finite number of times, since each such critical point is a root of an equation between several piecewise analytic functions (and there are only finite number of them). We do not mark this reasoning as a proof, because the precise statement and a rigorous reasoning will take much place. As the reader have seen, in all the examples, the picture stabilizes in a reasonably small number of steps (the only known example where there are many such steps is the one considered in~\cite{CrazySine}, but it was designed to illustrate as many evolutional scenarios as possible).

However, there is a small hope that the behavior of angles (birdies and multibirdies as well) is much ``smoother'' than one can conjecture from the first sight. The proposition below deals with the easiest case of a lonely angle surrounded by tangent domains that last to infinity. We show that such an angle cannot oscillate near a point.

\begin{St}\label{AnalyticAngle}
Let~$f$ be a function satisfying Conditions~\textup{\ref{reg},~\ref{sum}}. Suppose that for any~$\eps$ in the interval~$[\eps',\eps''] \subset (0,\eps_{\infty})$ there is a point~$w(\eps)$ such that the foliation
\begin{equation*}
\Rt(-\infty,w(\eps))\cup\Ang(w(\eps)) \cup \Lt(w(\eps),\infty)
\end{equation*}
provides a Bellman candidate for~$f$ and~$\eps$ \textup(i.e. the function~$f$ falls under the scope of Proposition~\textup{\ref{AngleProp}}\textup). Then\textup, either the equation~$w(\eps) = z$ has only finite number of roots when~$\eps \in [\eps',\eps'']$ for any point~$z \in \mathbb{R}$\textup, or the angle is stable \textup($w(\eps)$ is constant when~$\eps \in [\eps',\eps'']$\textup).
\end{St}

\begin{proof}
The point~$u=w(\eps)$ is the root of the balance equation
\begin{equation*}
\Fr(u;-\infty;\eps) + \Fl(u;\infty;\eps) = 0.
\end{equation*} 
With the help of formulas~\eqref{RightForceInfinity} and~\eqref{LeftForceInfinity}, this equation can be rewritten as
\begin{equation*}
\int\limits_{\mathbb{R}}e^{\frac{-|t-u|}{\eps}}\,df''(t) = 0.
\end{equation*}
We note that the convolution on the left is an analytic function of~$\eps$ for any~$u$ fixed. So, if
\begin{equation*}
\Fr(z;-\infty;\eps) + \Fl(z;\infty;\eps) = 0
\end{equation*} 
infinitely many times on a finite interval, then the same equality holds for all~$\eps \in [\eps',\eps'']$, since an analytic function having infinitely many roots in a finite domain equals zero. 
\end{proof}

We note that an angle may visit a point infinitely many times if we consider infinite time. See the first example in Section~\ref{s54}. %We believe that Proposition~\ref{AnalyticAngle} is an exception, and the general situation is worse.
%\begin{Conj}\label{BirdieOscCConj}
%There exists a function~$f$ satisfying conditions~\textup{\eqref{reg},~\eqref{sum},} a decreasing sequence~$\{\eps_n\}$\textup,~$\eps_n \to \eps$\textup, and sequences~$\{w_n^{+}\}_n$\textup,~$\{a_n\}_n$\textup,~$\{b_n\}_n$\textup, and~$\{w_n^{-}\}_n$\textup,~$w_n^{-} < a_{2n} < b_{2n}$\textup,~$a_{2n+1} < b_{2n+1} < w_n^{+}$\textup, such that the Bellman function for~$f$ and~$\eps_{2n}$ has the foliation
%\begin{equation*}
%\Rt(-\infty,w_n^{-})\cup\Ang(w_n^{-})\cup\Lt(w_n^{-},a_{2n})\cup\LTroll(a_{2n},b_{2n})\cup\Ch([a_{2n},b_{2n}],[c_1,c_1])\cup\Lt(b_{2n},\infty),
%\end{equation*}
%whereas the Bellman function for~$f$ and~$\eps_{2n+1}$ has the foliation
%\begin{equation*}
%\Rt(-\infty,a_{2n+1})\cup\Ch([a_{2n+1},b_{2n+1}],[c_1,c_1]) \cup\RTroll(a_{2n+1},b_{2n+1})\cup\Rt(b_{2n+1},w_n^+) \cup\Ang(w_n^+)\cup\Lt(w_n^+,\infty).
%\end{equation*}
%\end{Conj}

\subsection{Boundary behavior}\label{s612}
As it was said in Subsection~\ref{s212}, Condition~\ref{sum} is only sufficient for the finiteness of the Bellman function~$\Bell$. It can be weakened slightly. We cite Theorem~$6.2$ from~\cite{SZ}.

\begin{Th}
If a measurable function~$f$ satisfies the condition
\begin{equation*}
\sum\limits_{k\in\mathbb{Z}}e^{-\frac{|k|}{\eps}}\!\!\!\sup\limits_{[\eps(k-2),\eps(k+2)]}\!\!\!|f| < \infty,
\end{equation*}
then the Bellman function~\textup{\eqref{Bf}} is finite.
\end{Th} 

There is a related question: when is the expression~$\int f(\varphi)$ well defined (and thus formula~\eqref{Bf} is relevant)? Both Condition~\ref{sum} and the theorem above provide conditions on~$|f|$. Maybe, there are some functions~$f$ oscillating at infinity such that~$\int |f(\varphi)|$ is infinite for some~$\varphi \in \BMO_{\eps}$, but the expression~$\int f(\varphi)$ can still be defined correctly? To study this question, we introduce yet another Bellman function:
\begin{equation}\label{Bff}
\Bellb(x_1,x_2;\,f) \;\df 
		\sup\big\{f[\varphi] \mid\; \av{\varphi}{I} = x_1,\av{\varphi^2}{I} = x_2,\;\varphi \in \BMO_{\eps}(I)\cap L_{\infty}(I)\big\}.
\end{equation} 
We note that it satisfies all the trivial properties of the function~$\Bell$ listed in Subsection~\ref{s211} and also Lemma~\ref{LMaj}. Moreover, it is well defined for any measurable function~$f$ that is locally bounded, whereas the function~$\Bell$ needs summability conditions such as Condition~\ref{sum} to be defined (and one also has to prove that the integral is well defined). In the case where~$\Bell$ is defined,~$\Bellb=\Bell$. A natural question arises: is the Bellman function~$\Bell$ well defined provided there exists a finite function~$G \in \Lambda_{\eps,f}$ (see formula~\eqref{Lambda})? Our aim here is to give an example of~$f$ for which the answer is negative. We need Proposition~$6.1$ from~\cite{SZ} (we formulate it in our particular case).

\begin{St}\label{Definitness}
Suppose that~$f$ is locally bounded from below and the minimal function~$\BG$ from~$\Lambda_{\eps,f}$ is finite. Then\textup, the integral~$\int f(\varphi)$ is well defined for all~$\varphi \in \BMO_{\eps}$ if and only if the function~$\Bell(\cdot\,;f_+)$ is finite, where $f_+=\max(f,0)$.
\end{St}

Namely, we will construct a function~$f$ and a Bellman candidate~$B$ in~$\Omega_1$ such that~$B$ is finite, but there exists a sequence of intervals~$\{\Delta_k\}_k$ tending to infinity,~$|\Delta_k| = \frac{1}{10}$,~$f(t) \geq e^t$ on~$\Delta_k$. Surely, for such a function~$f$,~$\Bell(\cdot\,;f_+)$ is infinite. Indeed, the value~$\int_{[0,1]} f_+(-\ln t)$ is infinite. So, by Proposition~\ref{Definitness}, there exists a function~$\varphi \in \BMO_{1}$ such that~$\int f(\varphi)$ is not well defined. It remains to prove the lemma below.

\begin{Le}
There exists a~$C^1$-smooth function~$f$\textup, a finite Bellman candidate~$B$ in~$\Omega_1$ for~$f$\textup, and a sequence of intervals~$\{\Delta_k\}_k$ tending to infinity such that~$|\Delta_k| = \frac{1}{10}$ and~$f(t) \geq e^t$ on~$\Delta_k$ for all~$k \in \mathbb{N}$.
\end{Le}
\begin{proof}
We begin with an abstract construction. Let~$v \in \mathbb{R}$, consider the domain~$\Lt(v,u)$, where~$u$ is very big. Suppose also that the values~$f(v)$,~$f'(v)$ and~$m(v)$ are given (as usual,~$m$ is the slope function given by formula~\eqref{linearity}). By~\eqref{difeq2},
\begin{equation*}
-m' + m = f',
\end{equation*} 
thus,~$m'(v)$ is also given. Let us set~$m''(t) = (e^{e^t})''= e^{e^t}(e^{2t} + e^t)\geq 0$ on this domain. In such a case, we can use the data and the Newton--Leibniz formula to find~$m'(t)$,~$m(t)$, and~$f(t)$, and build a locally concave (on~$\Lt(v,u)$) function~$B$ by formula~\eqref{linearity} with the help of Proposition~\ref{LeftTangentsCandidate}. It is clear that~$f$ is~$C^2$--smooth, and the functions~$f$ and~$m$ satisfy the inequalities (for sufficiently big~$t$)
\begin{equation}\label{Inequalities}
\begin{aligned}
1.1e^{e^t} \geq m(t) \geq 0.9e^{e^t};\\
f'(t) \geq -1.1e^t e^{e^t};\\
f(t) \geq -1.2e^{e^t}. 
\end{aligned}
\end{equation}
We set~$u$ to be so big that all these inequalities are valid at the point~$t = u$. After that, we choose some $w$, $w>u+3$, and extend the function~$B$ linearly (preserving the~$C^1$--smoothness) into the domain~$\MTC(\{[u,w]\})$, thus, defining the value of~$f$ on~$[u,w]$. Using~\eqref{Inequalities}, we see that~$f(u+2) = f(u) + 2m(u) \geq 0.6e^{e^u}$. 

It follows from linearity that~$f \geq 0.2e^{e^u}$ either on~$[u+1.9,u+2]$ or on~$[u+2,u+2.1]$. Indeed, the restriction of~$f$ to~$[u,w]$ is a quadratic polynomial. We know that~$f(u) \geq -1.2e^{e^u}$ and~$f(u+2) \geq 0.6e^{e^u}$. If one considers the two cases of positive and negative leading coefficient, he sees that~$f \geq 0.2e^{e^u}$ either on~$[u+1.9,u+2]$ or on~$[u+2,u+2.1]$. So, one of these two intervals is a good candidate to be one of the~$\Delta_k$ (since~$e^{e^u} \geq e^{u+2.1}$). It remains to extend the function~$B$ linearly to~$\Ang(w)$ and return to the left tangents.

Let us set the parameters~$f(0),m(0)$, and~$f'(0)$, and extend the function~$B$ linearly on the left (so,~$f$ is a quadratic polynomial there). To construct~$B$ on the right, we apply the procedure described above consecutively (i.e. setting~$v := 0$ first, then~$v := w$, and so on). 
\end{proof}

%\subsection{Other measures}\label{s613}

%\subsection{Duality theorem for continuous functions}\label{s614}

\section{Further development}\label{s62}
\subsection{More historical remarks}\label{s621}
As it was said in the introduction, the main ideas and results of the present text were known to the authors at the time of preparation of~\cite{5A}. There was a strong development in the field since that moment. We describe it briefly.

Several examples of Bellman functions given by formula~\eqref{Bf} for specific~$f$ played a significant role in the papers~\cite{LSSVZ,Osekowski2} (however, the second paper treats the case of very non-smooth~$f$). Though the theory from~\cite{5A} formally was not used, both papers admit inspiration by it. 

From the very beginning, it was clear that the method is not limited by the space~$\BMO$. Indeed, see the papers~\cite{BR,Rez,Va3}, where similar questions are studied on Muckenhoupt classes and Gehring (so called reverse H\"older) classes. All these particular cases were unified into a single extremal problem in~\cite{5AShortNew}. The authors also claim that all the constructions of the~$\BMO$ setting can be transferred to this more general one. It appears that one has to replace each notion (force, differentials of a chord, etc.) by its differential geometric analog. For example, the differentials can be expressed as an outer product of certain three vectors. Some hints to these differential geometric constructions may be found in~\cite{ISZ}. However, there are many technicalities to overcome that are not present in the~$\BMO$ case. That is why we decided first to present the full treatment of the~$\BMO$ case.

Theorem~\ref{MT} is not a miracle now. A relatively short and transparent proof of it (omitting the construction of the Bellman function) was given in~\cite{SZ} (in the general geometric setting of~\cite{5AShortNew}). The idea is that there is a third function (other than~$\Bell$ and the minimal locally concave~$\BG$) built as a solution of a certain optimization problem. The optimization problem is posed on the class of martingales starting from a point, wandering inside~$\Omega_{\eps}$, and ending their way on~$\FixedBoundary \Omega_{\eps}$. The value~$\BM(x)$ is the supremum of~$\E f(M)$ over all the martingales starting from the point~$x$. It is relatively easy to show that~$\BM = \BG$ (and this is a much more general fact). Then one shows that~$\BM = \Bell$ by proving two embedding theorems: each function from~$\BMO_{\eps}$ generates a martingale with ``the same'' final distribution, and vice versa, each martingale of the prescribed type generates a function from~$\BMO_{\eps}$.

We also mention that now one can work with non-quadratic semi-norms on~$\BMO$. Historically, the trick that allows to pass to~$p$-semi-norms, was invented by Slavin. Now it is written down in~\cite{Slavin2}, where the sharp constant for the John--Nirenberg inequality with respect to~$p$-semi-norm was found;~$1 \leq p \leq 2$ there (the case $p>2$ is considered in~\cite{SlVa4}). The trick was used in~\cite{LSSVZ} as well. The Bellman function that plays the major role in~\cite{Slavin2, SlVa4} is defined on the exponential strip and falls under the scope of the general setting in~\cite{5AShortNew}.

One can ask what happens if the underlying space on which the~$\BMO$ space is defined, has dimension~$2$ or greater. The most natural question is: if one defines the~$\BMO$ space on the cube~$Q \in \mathbb{R}^d$ as
\begin{equation*}
	\|\varphi\|_{\BMO(Q)} \!\!\df\; \Big(\sup_{J\subset Q}\av{|\varphi-\av{\varphi}{J}|^2}{J}\Big)^{\frac{1}{2}},
\end{equation*}
where~$J$ runs through all subcubes of~$Q$, what is the supremum of~$\lambda$'s such that~$e^{\lambda \frac{\varphi}{\|\varphi\|_{\BMO}}}$ is summable over~$Q$? It is not hard to see that~$\lambda \leq 1$. However, it is natural to conjecture that~$\lambda < 1$. Unfortunately, it is not clear how to find the optimal~$\lambda$. But one can solve the problem for dyadic classes\index{dyadic problems}.

%Let~$d \geq 1$ be the dimension, l
Let~$Q$ be a cube in~$\mathbb{R}^d$. By~$\mathcal{D}$ we denote the collection of all its dyadic subcubes. Define the~$\BMO^{\mathrm{dyad}}$-semi-norm by the formula
\begin{equation}\label{DyadicBMONorm}
\|\varphi\|_{\BMO^{\mathrm{dyad}}} \!\!\df\; \Big(\sup_{J\in\mathcal{D}(Q)}\av{|\varphi-\av{\varphi}{J}|^2}{J}\Big)^{\frac{1}{2}}.
\end{equation}  
Define the Bellman function by the formula~\eqref{Bf} with the usual~$\BMO$-norm replaced with~$\|\cdot\|_{\BMO^{\mathrm{dyad}}}$. The case~$f(t) = e^{\lambda t}$ (corresponding to the integral form of the John--Nirenberg inequality~\eqref{intJN}) was treated in~\cite{SlVa3}. It appeared that the Bellman function is the restriction of~$\boldsymbol{B}_{\eps'}(\cdot\,;\exp)$ to~$\Omega_{\eps}$, where~$\eps' > \eps$ is some explicit parameter (depending on~$\eps$ and~$d$). It would be also interesting to find sharp constants in the analogs of inequalities~\eqref{bmoequivnorms} and~\eqref{JN} for dyadic-type spaces. A more demanding task is to build a theory for the optimization problems on dyadic classes similar to the one presented here.

In Subsection~\ref{s224}, we said that the monotonic rearrangement operator is related to the optimization problem~\eqref{Bf}. Indeed, Theorem~\ref{MonotonicRearrangement} (from~\cite{Ivo}) helped us to guess the optimizers. There is a reverse influence. It was noticed in~\cite{SZ} that the martingales introduced there allow to prove sharp inequalities for the monotonic rearrangement operator on the class in question (e.g.~$\BMO$, Muckenhoupt classes, Gehring classes, etc.). The paper~\cite{SVZ} provides sharp estimates for the monotonic rearrangement operator from dyadic classes in arbitrary dimension to the corresponding continuous class on the interval.

\subsection{Conjectures and suggestions}\label{s622}
\paragraph{Other measures.} Let~$\mu$ be a Borel measure with finite variation on~$[0,1]$. We may consider a similar problem over this measure. Namely, consider the~$\BMO$ space with the seminorm~\eqref{BMOnorm2} where all averages are taken with respect to~$\mu$, i.e.
\begin{equation*}
\av{\varphi}{J} \df \frac{1}{\mu(J)}\int\limits_{J}\varphi(t)\,d\mu(t).
\end{equation*}
One may then consider the Bellman function given by formula~\eqref{Bf} where all averages are with respect to~$\mu$. It is easy to see that if the measure~$\mu$ does not contain atoms, then the resulting Bellman function is the same as before. It is interesting to see what happens in the other cases.

\paragraph{Other manifolds.} There are only two one-dimensional compact connected manifolds: a segment and a circle. We have studied the problem for the segment, but there is an absolutely similar setting on the circle. Namely, define the space~$\BMO^{\circ}$ of functions on~$\mathbb{T}$ by the seminorm
\begin{equation*}
	\|\varphi\|_{\BMO^{\circ}} \!\!\df\; \Big(\sup_{J\subset \mathbb{T}}\av{|\varphi-\av{\varphi}{J}|^2}{J}\Big)^{\frac{1}{2}},
\end{equation*}
the supremum is taken over all the arcs~$J \subset \mathbb{T}$.
Similarly, one can consider the Bellman function~$\Bell^{\circ}(x;f)$ given by formula~\eqref{Bf} with~$I$ replaced by~$\mathbb{T}$. It is clear that~$\Bell^{\circ} \leq \Bell$, because every function in~$\BMO^{\circ}$ can be identified with a function in~$\BMO(I)$. In some cases, this inequality may turn into equality. Anyway, it is interesting to find~$\Bell^{\circ}$ at least for some reasonable~$f$ (such as the exponential function).

\paragraph{Dyadic problems.}\index{dyadic problems}
Define the dyadic~$\BMO$ on the cube~$Q\subset \mathbb{R}^d$ by formula~\eqref{DyadicBMONorm} and consider the dyadic Bellman function 
\begin{equation}\label{BfDyad}
		\Belld(x_1,x_2;\,f) \;\df 
		\sup\big\{f[\varphi] \mid\; \av{\varphi}{Q} = x_1,\av{\varphi^2}{Q} = x_2,\;\varphi \in \BMO_{\eps}^{\mathrm{dyad}}(Q)\big\}.
	\end{equation} 
	It is easy to see that the domain of this function is~$\Omega_{\eps}$ and it satisfies the usual boundary conditions. As for the main inequality, one can easily prove that
	\begin{equation}\label{DyadicConcavity}
	\Belld(x) \geq 2^{-d}\sum\limits_{k=1}^{2^d}\Belld(x_k)\quad \hbox{provided}\quad x = 2^{-d}\sum\limits_{k=1}^{2^d}x_k,\quad x \in \Omega_{\eps}, \forall k\ x_k\in \Omega_{\eps}.
	\end{equation}
	In particular, the function~$\Belld$ is locally concave on~$\Omega_{\eps}$. However, condition~\eqref{DyadicConcavity} is much stronger. We recall a definition from~\cite{SZ}.
	\begin{Def}
	We call a domain~$\Omega$ an extension of~$\Omega_{\eps}$ if there exists an open convex unbounded set~$\Omega'$ such that~$\Omega = \{x \in\mathbb{R}^2\mid x_2 \geq x_1^2\} \setminus \Omega'$ and the closure of~$\Omega'$ lies in~$\{x \in\mathbb{R}^2\mid x_2 > x_1^2 + \eps^2\}$.
	\end{Def}
	\begin{Conj}\label{ExtensionConjecture}
	For every~$f$\textup,~$\eps$\textup, and~$d$ there exists an extension~$\Omega$ of~$\Omega_{\eps}$ such that
	\begin{equation*}
	\Belld(x) = \inf\Big\{G(x)\,\Big|\; G\,\hbox{is locally concave on}\;\Omega,\quad G(x_1,x_1^2) \geq f(x_1)\;\hbox{for all}\; x_1 \in \mathbb{R}\Big\},\quad x \in \Omega_{\eps}.
	\end{equation*}
	\end{Conj}
	This conjecture is verified for the case~$f(t) = e^t$ in the paper~\cite{SlVa3}. In the light of~\cite{5AShortNew}, Conjecture~\ref{ExtensionConjecture} opens the road towards a theory for dyadic Bellman functions similar to the one described in the present paper (however, the conjecture says nothing about how to find~$\Omega$).  
	
	One can pose a similar conjecture for the continuous~$\BMO$ on the cube of arbitrary dimension, however, up to the moment, there is a very little understanding of what happens there.
	
	\paragraph{Campanato spaces.}\index{Campanato spaces} The space~$\BMO$ can be thought of as a point on the scale of Morrey--Campanato spaces (see~\cite{KK} for the general theory of such spaces). Namely, consider the seminorm
	\begin{equation*}
	\|\varphi\|_{C_{2}^{\alpha,1}(I)} \!\!\df\; \Big(\sup_{J\subset I}|J|^{-\alpha}\av{|\varphi-\av{\varphi}{J}|^2}{J}\Big)^{\frac{1}{2}}.
	\end{equation*}
	Here~$\alpha \in (-\frac12,1]$. The space~$C_2^{\alpha,1}$ generated by this seminorm coincides with~$\BMO$ if~$\alpha=0$ and with the Lipschitz class if~$\alpha=1$. It seems reasonable to state a similar Bellman function problem on this space:
	\begin{equation*}
	B(x,w;f,\eps) = \sup\Big\{\av{f(\varphi)}{I}\,\Big|\; \av{\varphi}{I} =x_1,\av{\varphi^2}{I} = x_2, |I|=w, \|\varphi\|_{C_2^{\alpha,1}(I)} \leq \eps\Big\}.
	\end{equation*}
	One has to introduce the third variable~$w$ due to the scaling properties of the problem. Though it may seem questionable whether this function has good analytic properties, some hope comes from the paper~\cite{Osekowski4}, where the sharp equivalence between the Campanato and the classical seminorms is established for the Lipschitz class ($\alpha=1$).

\printindex
\index{z@\addcontentsline{toc}{chapter}{Index}\quad\quad\quad\quad\quad\quad\quad\quad\quad\quad\quad\quad\quad\quad\quad\quad\quad\quad\quad\quad\quad\quad\quad\quad\quad\quad\quad\quad\quad\quad\quad\quad\quad\quad\quad\quad\quad\quad\quad\quad\quad\quad\quad\quad\quad\quad\quad\quad\quad\quad\quad\quad\quad\quad\quad\quad\quad\quad\quad\quad\quad\quad\quad\quad|phantom}


\begin{thebibliography}{99}
\addcontentsline{toc}{chapter}{References}

\bibitem{BR} O. Beznosova, A. Reznikov, \emph{Sharp estimates involving~$A_{\infty}$ and~$LlogL$ constants\textup, and their applications to PDE}, Alg. i Anal. {\bf26}:1 (2014), 40--67. 

\bibitem{Burk} D.L.~Burkholder,
	\emph{Boundary value problems and sharp inequalities for martingale transforms},
	Ann. Prob. {\bf 12}:3 (1984), 647--702.
	
\bibitem{Hanner} O.~Hanner,
	\emph{On the uniform convexity of $L^p$ and $\ell^p$},
	Arkiv Mat. {\bf 3}:19 (1955), 239--244.

\bibitem{Ivo} I.~Klemes, 
	\emph{A mean oscillation inequality},
	Proc. AMS {\bf 93}:3 (1985), 497--500.  

\bibitem{Koosis} P.~Koosis, \emph{Introduction to $H^p$ spaces}, Cambridge University press, 1998.

\bibitem{Kor} 	A. A. Korenovskii, \emph{On the connection between mean oscillation and exact integrability classes of functions}, Mat. Sb.~{\bf 181}:12 (1990), 1721--1727 (in Russian); translated in Math. of the USSR-Sbornik~{\bf 71}:2 (1992), 561--567.
	
	\bibitem{ISZ} P.~Ivanisvili, D.~M.~Stolyarov, P.~B.~Zatitskiy, \emph{Bellman VS Beurling\textup: sharp estimates of uniform convexity for~$L^p$ spaces}, Alg. i Anal. {\bf 27}:2 (2015), 218--231 (in Russian); to be translated in St.-Petersburg Math. J., http://arxiv.org/abs/1405.6229.
	
\bibitem{CR} P.~Ivanishvili, N.~N.~Osipov,  D.~M.~Stolyarov, V.~I.~Vasyunin, P.~B.~Zatitskiy, 
		\emph{On Bellman function for extremal problems in $\BMO$},
		C. R. Math. {\bf 350}:11 (2012), 561--564.

\bibitem{5A} P.~Ivanishvili, N.~N.~Osipov,  D.~M.~Stolyarov, V.~I.~Vasyunin, P.~B.~Zatitskiy, 
		\emph{Bellman function for extremal problems on $\BMO$}, to appear in Trans. AMS, http://arxiv.org/abs/1205.7018.
		
\bibitem{5AOld}  P.~Ivanishvili, N.~N.~Osipov,  D.~M.~Stolyarov, V.~I.~Vasyunin, P.~B.~Zatitskiy, \emph{Bellman function for extremal problems on $\BMO$}, PDMI preprints 19/2011 (in Russian).
		
\bibitem{Addendum} P.~Ivanishvili, N.~N.~Osipov,  D.~M.~Stolyarov, V.~I.~Vasyunin, P.~B.~Zatitskiy, \emph{Addendum to the preprint \textup`\textup`Bellman functions for extremal problems in $\BMO$\textup'\textup'}, PDMI preprints no.~10/2012 (in Russian).

\bibitem{5AShortNew} P.~Ivanisvili, N.~N.~Osipov, D.~M.~Stolyarov, V.~I.~Vasyunin, P.~B.~Zatitskiy, {\sl Sharp estimates of integral functionals on classes of functions with small mean oscillation}, to appear in C. R. Math., http://arxiv.org/abs/1412.4749.

\bibitem{KK} S. Kislyakov, N. Kruglyak, \emph{Extremal problems in interpolation theory\textup, Whitney--Besicovitch coverings\textup, and singular integrals}, Mon. Matem. IMPAN {\bf 74}, Springer 2013.

\bibitem{LSSVZ} A.~A~Logunov, L.~Slavin, D.~M.~Stolyarov, V.~Vasyunin, P.~B. ~Zatitskiy, \emph{Weak integral conditions for~$\BMO$},  Proc.  AMS {\bf 143} (2015), 2913--2926, http://arxiv.org/abs/1309.6780.

\bibitem{Melas} A.~D.~Melas, \emph{The Bellman function of dyadic-like maximal operators and related inequalities}, Adv. in Math.~{\bf 192} (2005), 310--340.
		
\bibitem{NaTr} F.~L.~Nazarov, S.~R.~Treil,
		\emph{Hunting the Bellman function\textup{:} application to estimates of singular integrals and other classical problems of harmonic analysis},
		Alg. i anal.~{\bf 8}:5 (1996), 32--162 (in Russian); translated in St.-Petersburg Math. J.~{\bf 8}:5 (1997), 721--824.	
		
\bibitem{NaTrVol} F.~Nazarov, S.~Treil, and A.~Volberg,~\emph{The Bellman functions and two-weight inequalities for Haar multipliers}, J. AMS~{\bf 12}:4 (1999), 909--928.

\bibitem{NTV} F.~Nazarov, S.~Treil, A.~Volberg, \emph{
   Bellman function in stochastic optimal control and harmonic analysis
   \textup(how our Bellman function got its name\textup)}, Oper. Th.: Adv. and Appl.~{\bf 129} (2001), 393--424, Birkh\"auser Verlag.
		
\bibitem{Os} A.~Os\c{e}kowski,
		\emph{Sharp Martingale and Semimartingale Inequalities},
		Mon. Matem. IMPAN~{\bf 72}, Springer Basel, 2012.
		
		\bibitem{Osekowski3} A.~Os\c{e}kowski, \emph{Survey Article\textup: Bellman function method and sharp inequalities for martingales}, Rocky Mountain J. Math.
{\bf 43}:6 (2013), 1759--1823.

\bibitem{Osekowski2}  A.~Os\c{e}kowski,
	\emph{Sharp inequalities for~$\BMO$ functions}, Chin. Ann. of Math. Ser. B. {\bf 36} (2015), 165--176.
	
	\bibitem{Osekowski4} A.~Os\c{e}kowski, \emph{Sharp estimates for Lipschitz class}, J. Geom. Anal. (2015).
		
\bibitem{Rez} A.~Reznikov
		\emph{Sharp weak type estimates for weights in the class $A_{p_1, p_2}$},
		Rev. Mat. Iberoam. {\bf 29}:2 (2013), 433--478.
		
\bibitem{RVV} A.~Reznikov, V.~Vasyunin, A.~Volberg,
		\emph{An observation\textup: cut-off of the weight $w$ does not increase the $A_{p_{1}, p_{2}}$-\textup`\textup`norm\textup'\textup' of $w$},
		http://arxiv.org/abs/1008.3635.
		
		
\bibitem{Slavin} L.~Slavin, \emph{Bellman function and $\BMO$}, Ph.D. thesis, Michigan
	State University, 2004.
	
	\bibitem{Slavin3} L.~Slavin, \emph{Best constants for a family of Carleson sequences}, http://arxiv.org/abs/1501.00093.
	
	\bibitem{Slavin2} L.~Slavin, \emph{The John--Nirenberg constant for~$\BMO^p$\textup,~$1 \leq p \leq 2$}, submitted, http://arxiv.org/abs/1506.04969.

\bibitem{SSV} L.~Slavin, A.~Stokolos, V.~Vasyunin, \emph{Monge--Amp\`ere
  equations and Bellman functions\textup: the dyadic maximal operator}, C.
  R. Math.~{\bf 346}:1 (2008), 585--588.

\bibitem{SlVa} L.~Slavin, V.~Vasyunin,
		\emph{Sharp results in the integral-form John--Ni\-ren\-berg inequality},
		Trans. Amer. Math. Soc. {\bf 363}:8 (2011), 4135--4169, http://arxiv.org/abs/0709.4332.
		
\bibitem{SlVa2} L.~Slavin and V.~Vasyunin,
		\emph{Sharp $L^p$ estimates on $\BMO$},  Ind. Univ. Math. J. {\bf 61} (2012), 1051--1110,
		http://arxiv.org/abs/1110.1771.

		
		\bibitem{SlVa3} L.~Slavin, V.~Vasyunin, \emph{Inequalities for~$\BMO$ on~$\alpha$-trees}, submitted, http://arxiv.org/abs/1501.00097.

\bibitem{SlVa4} L.~Slavin and V.~Vasyunin,
		\emph{The John--Ni\-ren\-berg constant of $BMO^p$, $p>2$},  in preparation.
		
	
		
\bibitem{Stein} I.~M.~Stein, \emph{Harmonic analysis\textup, real-variable methods\textup, orthogonality and oscillatory integrals}, 
    Princeton University press, 1993.
		
		\bibitem{SZ} D.~M.~Stolyarov, P.~B.~Zatitskiy, \emph{Theory of locally concave functions and its applications to sharp estimates of integral functionals}, submitted, http://arxiv.org/abs/1412.5350.
		
		\bibitem{SVZ} Dmitriy M. Stolyarov, Vasily I. Vasyunin, Pavel B. Zatitskiy, \emph{Monotonic rearrangements of functions with small mean oscillation}, submitted, http://arxiv.org/abs/1506.00502.
		
\bibitem{Va2}	V.~Vasyunin,
		\emph{The sharp constant in the John--Nirenberg inequality},
		preprint POMI no.~20, 2003.
		
\bibitem{Va3} V.~I.~Vasyunin, \emph{The sharp constant in the reverse H{\"o}lder inequality for Muckenhoupt weights}, Alg. i anal.~{\bf 15}:1 (2003), 73--117 (in Russian); translated in  St.-Petersburg Math. J.~{\bf 15}:1 (2004), 49--79.
			
\bibitem{Va} V.~Vasyunin, \emph{
    Sharp constants in the classical weak form of the John--Ni\-ren\-berg inequality},
    PDMI preprint, no.~10/2011 (in Russian).
		
		\bibitem{CrazySine} V.~Vasyunin, \emph{An example of constructing the Bellman function for extremal problems in~$\BMO$}, Zap. Nauchn. Sem. POMI {\bf 424} (2014), 33--125 (in Russian).
    
  \bibitem{VaVo} Vasily Vasyunin and Alexander Volberg,
		\emph{Monge--Amp\`ere Equation and Bellman Optimization of Carleson Embedding Theorems},
		Amer. Math. Soc. Transl. Ser.~2 {\bf 226} (2009), 195--238.
		
		\bibitem{VV2} V.~Vasyunin, A.~Volberg, \emph{Sharp constants in the classical
weak form of the John-Nirenberg inequality},  Proc. Lond. Math.
Soc.  {\bf 108}:6 (2014), 1417--1434.
			
			\bibitem{Volberg1} A.~Volberg, \emph{Bellman approach to some problems in harmonic analysis}, S\'eminaire \'Equations aux d\'eriv\'ees partielles (2001-2002), 1--14.
 
\bibitem{Volberg} A.~Volberg, \emph{Bellman function technique in Harmonic Analysis}, Lectures of INRIA Summer School in Antibes, June 2011, http://arxiv.org/abs/1106.3899.

	 

\end{thebibliography}
\end{document}